\documentclass{article}

\usepackage{amssymb, amsmath,amsthm}
\usepackage{xfrac}
\usepackage{mathrsfs}
\usepackage{bbm}
\usepackage{amscd}
\usepackage[active]{srcltx}
\usepackage{verbatim}
\usepackage{graphicx}

\usepackage[backend=bibtex,style=alphabetic]{biblatex}
\usepackage[utf8]{inputenc}
\usepackage[T1]{fontenc}
\usepackage{lmodern}

\usepackage{hyperref} 
\usepackage{todonotes}
\usepackage{enumerate}
\usepackage{mathtools}
\usepackage{bbm}

\usepackage{algorithm}
\usepackage{algpseudocode}
\usepackage{algorithmicx}
\algdef{SE}[DOWHILE]{Do}{doWhile}{\algorithmicdo}[1]{\algorithmicwhile\ #1}

\usepackage{array,multirow,float}
\usepackage{soul}

\usepackage{tikz}
\usetikzlibrary{arrows.meta}

\mathtoolsset{showonlyrefs,showmanualtags}

\theoremstyle{plain}
\newtheorem{theorem}{Theorem}[section]
\theoremstyle{remark}
\newtheorem{remark}[theorem]{Remark}
\newtheorem{example}[theorem]{Example}
\theoremstyle{plain}
\newtheorem{corollary}[theorem]{Corollary}
\newtheorem{lemma}[theorem]{Lemma}
\newtheorem{proposition}[theorem]{Proposition}
\newtheorem{definition}[theorem]{Definition}

\numberwithin{equation}{section}

\allowdisplaybreaks

\renewcommand{\d}{{\boldsymbol{d}}}
\renewcommand{\b}{{\bf{w}}}

\DeclareMathOperator{\GQ}{GQ}
\newcommand{\PP}{\mathbb{P}}

\title{Unified framework for asymptotically uniform iterative construction of generalised random graphs with  local constraints}
\author{Ivan Kryven, Rik Versendaal, Mike de Vries}
\date{\today}

\begin{document}

	\maketitle

	\begin{abstract}
		We develop a unified framework for constructing combinatorial structures under local constraints. Our approach extends the configuration model for random graphs with a prescribed degree sequence, and covers many special cases, including bipartite graphs, directed graphs, oriented graphs, edge-colored (bipartite) graphs, and (directed) hypergraphs.
		
		By reformulating half-edge matching as an independent set problem in an auxiliary graph, we identify \emph{2-uniformity}, a property characterising when greedy sampling preserves asymptotic uniformity. We classify all 2-uniform graphs and show that only two classes, the configuration space and the bipartite configuration space, have unbounded independence number, enabling the asymptotic regime. Our main theorem then gives the asymptotic sampling distribution and enumeration formulae for configurations, with error terms of order $O(d_{\max}^4\log m/m+d_{\max}^2(\log m)^2/m)$ as the number of edges $m$ tends to infinity with maximum degree $d_{\max}=O(m^{\sfrac14}/\log m)$. This settles the long-standing $O(m^{\sfrac14-\tau})$ bound (for some fixed $\tau > 0$), making the critical exponent explicit.
		
		Furthermore, our theorem accommodates forbidden edges, provided that each vertex participates in at most $O(m^{\sfrac14}/\log m)$ of them. In particular, this enables the sampling of edge-colored graphs with prescribed degree sequences for each color class by constructing the colored subgraphs one at a time.
	\end{abstract}

	\tableofcontents

	\section{Introduction}\label{section_introduction}

	The study of graphs with prescribed degree sequences began with the foundational combinatorial characterisation results of Havel~\cite{havel1955}, Hakimi~\cite{hakimi1962}, and Erdős--Gallai~\cite{erdos1960}. These results answered the basic existence question: which sequences of non-negative integers arise as degree sequences of simple graphs?  Havel and Hakimi provided constructive processes, while Erdős and Gallai gave a non-constructive characterisation in terms of linear inequalities. This line of inquiry was continued with the Gale--Ryser Theorem~\cite{Gale1957,Ryser1957} for bipartite graphs and Fulkerson's characterisation~\cite{fulkerson1960} for directed graphs. Tutte's $f$-factor theorem~\cite{Tutte1952} addresses a related question: whether a host graph contains a spanning subgraph satisfying prescribed per-vertex degree constraints.
	Further extensions include the sequence packing problem~\cite{West2012}, where one verifies whether multiple degree sequences can be realised by edge-disjoint subgraphs on the same vertex set. This can also be interpreted as the realisation problem for edge-colored graphs. The realisation problem for hypergraphs has also been addressed~\cite{Li2025,Deza2018}.

	As probabilistic methods gained prominence in combinatorics, the focus has shifted from realisation problems to two closely related follow-up questions: how many graphs have a given degree sequence, and how to sample uniformly at random from the set of all such graphs. 
	While the enumeration questions were treated in works of McKay, Bender, Greenhill and their co-authors~\cite{mckay1985asymptotics,mckay1984asymptotics,BENDER1978296,GREENHILL2008,greenhill2024enumerationdihypergraphsspecifieddegrees,greenhill2026asymptoticenumerationconstrainedbipartite},
	the first sampling results were addressed separately  with Markov chain Monte Carlo, using switching moves to walk on the space of realisations. For these approaches, rapid mixing has been established for bounded-degree graphs by McKay and Wormald~\cite{MCKAY199052} and later for directed graphs with irregular degree sequences by Greenhill and Sfragara~\cite{Greenhill2018}.
	
	Parallel to these developments, Bayati, Kim, and Saberi~\cite{Bayati} (who followed up on Steger and Wormald \cite{Steger1999}) considered constructing simple graphs sequentially. They introduced an iterative process that achieves a $1+o(1)$ deviation from uniformity when the maximum degree is bounded as $d_{\max} = O(m^{\sfrac14-\tau})$ for some fixed $\tau > 0$ as the number of edges $m$ tends to infinity. The advantage of sequential approaches is that they run in near-linear time and can simultaneously estimate the partition function, thereby yielding enumeration estimates. However, in contrast to the plethora of cases for which the realisation problem was resolved, few other generalisations of the iterative process are known to date, besides the extension to directed graphs~\cite{Ieperen2024}, and the reverse-process approach of Arman, Guo and Wormald~\cite{Arman2021}, which starts from a complete configuration model multigraph and iteratively repairs it via local edge switches.

	The iterative process introduced in~\cite{Bayati}  builds on the configuration model for random graphs with a given degree sequence~\cite{molloy1995}. In the configuration model, each vertex $v_i$ is assigned $d_i$ half-edges, and a uniformly random perfect matching is placed on the collection of all half-edges to obtain a multigraph. Conditioning on the event that no self-loops or multi-edges occur yields uniformly random simple graphs, but the probability of this event decays exponentially as the maximum degree grows. To overcome this probability-decay issue, the iterative process selects half-edge pairs sequentially, rejecting any pair that creates a self-loop or multi-edge. 
	By assigning non-uniform sampling weights to admissible pairs, and thereby compensating for the combinatorial bias introduced by the sequential constraints, the iterative process yields a distribution on simple graphs that is asymptotically uniform under the above-mentioned maximum degree assumption.

	\paragraph*{Unified framework.}
	In this paper, we introduce a unified framework for the asymptotically uniform
	iterative construction of generalised random graphs. To ensure broad generality, we reformulate the configuration model in a more abstract setting by viewing half-edge matching as the independent set problem in an auxiliary graph, which we refer to as the \emph{configuration space}.
	We then identify a fundamental property of the configuration space, which we call \emph{$2$-uniformity}, that characterises when greedy construction of a maximal independent set preserves asymptotic uniformity. In Theorem~\ref{theorem_classification}, we classify all $2$-uniform graphs. Apart from the configuration space, only one additional infinite family admits unbounded independence number, which is related to the configuration space for bipartite graphs. We thus refer to this family as the \emph{bipartite configuration space}. This demonstrates that $2$-uniform graphs are a natural way to generalize the configuration model. Indeed, any asymptotic statement about the configuration model for undirected or bipartite graphs translates directly to 2-uniform graphs. This division into two classes of configuration spaces is already implicitly present in the literature, wherein asymptotic results are typically handled separately for configuration models for undirected or bipartite (directed) graphs: see for instance~\cite{mckay1985asymptotics} vs.~\cite{mckay1984asymptotics} or~\cite{BENDER1978296} vs.~\cite{GREENHILL2008} for asymptotic enumeration of graphs, and~\cite{Bayati} vs.~\cite{Ieperen2024} for iterative sampling.

	To formulate our generalised iterative process on $2$-uniform graphs, we
	furthermore generalize the concepts of forbidden structures (such as self-loops or multi-edges in the configuration model for simple graphs) by introducing a set $F$ of forbidden vertices and an equivalence relation $R$ whose equivalence classes consist of forbidden pairs of vertices in the configuration space.  Then, the iterative graph sampling process translates to the greedy sampling of an independent set that does not contain forbidden structures from $F$ and $R$. By further introducing a sampling weight function $\b$, we define the \emph{iterative maximal feasible independent set (IMFIS) process} on a $2$-uniform graph. In our main theorem, Theorem~\ref{theorem_main}, we state the asymptotic sampling distribution of the IMFIS process on a $2$-uniform graph. By choosing $F$, $R$, and $\b$ appropriately, we show that the IMFIS process can be used on the (bipartite) configuration space to sample undirected graphs, bipartite graphs, directed graphs, oriented graphs, edge-colored (bipartite) graphs, and (directed) hypergraphs. Theorem~\ref{theorem_main} then allows us to determine the asymptotic sampling distribution of these processes. As an immediate consequence, we also obtain asymptotic enumeration formulae for the size of the sample space for each of the aforementioned graphs. 
	
	The freedom to choose arbitrary forbidden structures $F$ and $R$ is crucial for our framework. For instance, the extension to edge-coloured (bipartite) graphs is realised by choosing appropriate forbidden structures, so we can satisfy the degree sequence for each color class one by one. The extension to directed hypergraphs is then realised by embedding the incidence graph of the hypergraph into the bipartite configuration space and using two colours to represent in- and out-edges. We also extend the setting from~\cite{Bayati} by accommodating forbidden edges, essentially packing a degree sequence with a fixed graph, which is also related to Tutte's $f$-factor problem~\cite{Tutte1952}, as we look for an $f$-factor in a complement graph.
	
	Finally, we use our general setting of $2$-uniform graphs to sharpen the asymptotic bounds for the special cases available in the literature. For instance, when sampling undirected graphs, we show that our iterative process allows $d_{\max}=O(m^{\sfrac14}/\log m)$ rather than $d_{\max}=O(m^{\sfrac14-\tau})$ for some fixed $\tau>0$, and we show that the $1+o(1)$ deviation factor is bounded by $1+O(d_{\max}^4\log m/m+d_{\max}^2(\log m)^2/m)$. This settles the $O(m^{\sfrac14-\tau})$ bound from \cite{Bayati}, making the critical exponent of $\sfrac14$ explicit. These improved bounds generalize to each of the aforementioned types of graphs, and we expect that these bounds are tight up to a polylogarithmic factor. We also provide a secondary main result, Theorem~\ref{theorem_space}, on the probability of fixed structures appearing in the IMFIS process, which we use to bound the probability of certain forbidden structures appearing when sampling (directed) hypergraphs.

	\paragraph*{Structure of the paper and proof techniques.} Before stating our main theorems in the abstract setting of $2$-uniform graphs, Section~\ref{section_results} presents a list of corollaries of the main theorems applied to recognisable special cases of random graphs: undirected graphs, bipartite graphs, directed graphs, oriented graphs, edge-colored (bipartite) graphs, and (directed) hypergraphs. For each of these special cases, an iterative sampling  process is defined, and a theorem is given that provides the resulting asymptotic sampling distribution and a graph counting formula. The section is self-contained, in that it is not yet necessary to be familiar with the general IMFIS process for this section, nor is it necessary to read through this section to understand the rest of the paper. The section serves as a lookup table for cases of application, as well as for comparing implications of our main theorem, such as enumeration estimates, with results present in the literature.
	
	In Section~\ref{section_independent}, we build up the abstract machinery of the configuration space and the bipartite configuration space, as well as the IMFIS process, which we use throughout the rest of the paper. We show that the IMFIS process is well-behaved on the (bipartite) configuration space, and we use this discussion to motivate our definition of $2$-uniform graphs. We furthermore discuss the basic properties of $2$-uniform graphs, and we give their full classification. 
	
	In Section~\ref{section_main}, we use the abstract machinery from Section~\ref{section_independent} to state our main theorem, Theorem~\ref{theorem_main}, and our secondary theorem, Theorem~\ref{theorem_space}, including some remarks and an immediate corollary on uniformity and enumeration.
	
	In Section~\ref{section_applications}, we apply the results from Section~\ref{section_main} to prove the results from Section~\ref{section_results}. This section mainly serves to show how all the aforementioned types of graphs can be modelled by means of $2$-uniform graphs. 
	
	Finally, Section~\ref{section_proof} contains the proofs of all the results from Section~\ref{section_main}, dominated by the proof of Theorem~\ref{theorem_main}. The argument combines concentration inequalities, stochastic ordering theory, and a double-counting technique. To analyze the outcome distribution of the IMFIS process, we calculate the probability of an arbitrary specified sample outcome set $S$. We consider a uniformly random order in which the elements of $S$ can be chosen. After choosing $r$ elements, this yields a uniformly random subset of $S$ of size $r$, which we compare to a simplified model where each element of $S$ is added independently with probability $r/|S|$. This lets us apply Vu's concentration inequality for graph parameters~\cite{Vu}, controlling the deviation of our iterative selection process from its expected behavior. We then bound the probability that the IMFIS process terminates prematurely at an incomplete set $S$ (because no admissible pairs remain outside $S$) using a stochastic order argument~\cite{BookSO}. Finally, we compare the collection of all incomplete sets with the collection of all complete sets via a double-counting argument, using a concept of relatedness inspired by~\cite{MCKAY2003273}.

	\section{Sampling Various Types of Graphs}\label{section_results}
	This section catalogues applications of Theorem~\ref{theorem_main} to concrete 
	graph families: undirected, bipartite, directed, oriented, edge-colored, and hypergraphs. For each family we state an iterative sampling process achieving asymptotic uniformity, allow for forbidden edges, and give the corresponding 
	deviation bound. These results serve both as immediate corollaries of our main theorem and as a self-contained reference for readers focused on applying our Theorem~to a specific settings.
	
	The section is designed to be read independently from the rest of the paper. A reader seeking only the practical implications can consult this section alone, while those interested in the proof machinery should proceed to Sections~\ref{section_independent}--\ref{section_proof}.
	
	For each graph family, we report two quantities: The asymptotic deviation from uniformity and a counting formula for the number of graphs with the specified degree sequence. The deviation bounds either improve existing results or, where no prior bounds were known, fill a gap in the literature. Counting formulae agree with established results for undirected, bipartite, directed, oriented, and directed hypergraphs, and appear to be new for edge-colored (bipartite) graphs and hypergraphs.

	We use the following notation consistently throughout this section. Consider vertex sets $A=\{v_1,\ldots,v_n\}$ and $B=\{v_1',\ldots,v_{n'}'\}$ with degree sequences $\d=(d_1,\ldots,d_n)$ and $\d'=(d_1',\ldots,d_{n'}')$ of non-negative integers with sums $s:=d_1+\ldots+d_n>0$ and $s':=d_1'+\ldots+d_{n'}'>0$. We define the \emph{branching factor} of the degree sequence $\d$ by $$\lambda_{\d}:=\frac1s\sum_{i=1}^{n}d_i(d_i-1).$$ The interpretation of this quantity is that, if half‑edges were paired uniformly at random in the configuration model, the branching factor approximates the average number of second-order neighbours. For a set $X$ of unordered pairs of vertices from $A$, we define its \emph{degree-degree interaction weight} by $$\mu_{\d}^{X}:=\frac1s\sum_{\{v_i,v_j\}\in X}d_id_j.$$ This quantifies how strongly high‑degree vertices are over-represented in $X$ by the configuration model. Analogously, for a set $X\subset A\times B$, we define $$\mu_{\d,\d'}^{X}:=\frac1{\sqrt{ss'}}\sum_{(v_i,v_j')\in X}d_id_j'.$$ 
	If $n=n'$, then we write $\d\leq\d'$ if $d_i\leq d_i'$ for $i=1,\ldots,n$, we define $\d-\d':=(d_1-d_1',\ldots,d_n-d_n')$, and we define $$\lambda_{\d,\d'}:=\frac1{\sqrt{ss'}}\sum_{i=1}^{n}d_id_i'.$$ Finally, for real numbers $m,\Delta>0$, we define the quantity
	\begin{equation}\label{equation_script_e}
		\mathscr{E}(m,\Delta):=\frac{\Delta^2\log m}{m}+\frac{\Delta(\log m)^2}{m},
	\end{equation}
	which occurs in many of our asymptotic results.
	
	\begin{remark}\label{remark_asymptotics}
		Throughout the paper, we generally consider the asymptotics as one specific parameter tends to infinity. We emphasize for clarity that this does not rely on some underlying sequence of objects and values that depend on this parameter. So unless a value is explicitly mentioned to be fixed, the constants hidden by the asymptotics may never depend on such a value.
		
		For example, consider Theorem~\ref{theorem_simple} from Section \ref{subsection_results_simple}, where $m\to\infty$. The assumption $\Delta=O(m^{1/2}/(\log m)^2)$ simply means that we assume there exist constants $C,N>0$ such that $m\geq N$ implies $|\Delta|\leq Cm^{1/2}/(\log m)^2$, as defined by big O notation. Then the terms $O(\Delta/m)$ and $O(\mathscr{E})$ in the conclusion imply the existence of constants $C_1,N_1,C_2,N_2>0$ such that these terms are absolutely bounded by $C_1\Delta/m$ and $C_2\mathscr{E}$ if $m\geq N_1$ or $m\geq N_2$ respectively. Note that the constants $C_1,N_1,C_2,N_2$ may thus depend on $C$ and $N$, but not on anything else.
	\end{remark}
	
	\subsection{Undirected graphs}\label{subsection_results_simple}
	
	We consider the following iterative stochastic process for sampling a random simple undirected graph with a given degree sequence. This is essentially a reformulation of Procedure A from~\cite{Bayati}.
	
	\begin{definition}[Iterative undirected graph process]\label{definition_process_simple}
		Consider the vertex set $V=\{v_1,\ldots,v_n\}$, consider a degree sequence $\d=(d_1,\ldots,d_n)$ of non-negative integers with $2m:=d_1+\ldots+d_n>0$ even, and consider a set $X$ of forbidden edges.
		\begin{itemize}
			\item Start with no edges $E_0=\emptyset$.
			\item For $r=0,\ldots,m-1$, let $d_{i,r}$ denote the degree of vertex $v_i$ in $(V,E_r)$, and consider those pairs of vertices $v_i,v_j\in V$ with $i\neq j$, $d_{i,r}<d_i$, $d_{j,r}<d_j$, and $\{v_i,v_j\}\not\in E_r\cup X$.
			\begin{itemize}
				\item If no such pairs exist, let $E_{r+1}=E_r$.
				\item Otherwise, choose such a pair $v_i,v_j$ at random with probability proportional to $(d_i-d_{i,r})(d_j-d_{j,r})(1-\tfrac{d_id_j}{4m})$, and let $E_{r+1}=E_r\cup\{\{v_i,v_j\}\}$.
			\end{itemize}
			\item Finally, apply rejection sampling to sample $E=E_m$ with $m$ edges.
		\end{itemize}
		This guarantees that $(V,E)$ is a simple undirected graph with degree sequence $\d$ and with no edges from $X$.
	\end{definition}
	
	To clarify the rejection sampling, if $E_m$ does not have $m$ edges, then we redo the entire process, and we repeat this until we find that $E_m$ does have $m$ edges.
	
	\begin{theorem}\label{theorem_simple}
		Consider the iterative undirected graph process on $V$, $\d$, and $X$. Let $x_{\max}$ be the maximum degree in $(V,X)$, and let
		\begin{align}
			d_{\max}&:=\max\{d_1,\ldots,d_n\},
			\\\Delta&:=d_{\max}(d_{\max}+x_{\max}).
		\end{align}
		Consider the asymptotics as $m\to\infty$, and assume that $\Delta=O(m^{\sfrac12}/(\log m)^2)$. Then the probability of rejection is $O(\Delta/m)$. In particular, for $m$ large enough, there exists a simple undirected graph with degree sequence $\d$ and with no edges from $X$. Furthermore, any such graph is sampled with probability $(1+O(\mathscr{E}))\mathscr{P}$, where $\mathscr{E}:=\mathscr{E}(m,\Delta)$ and $$\mathscr{P}:=\exp(\tfrac12\lambda_\d+\tfrac14\lambda_\d^2+\mu_{\d}^{X})\frac1{(2m-1)!!}\prod_{i=1}^{n}d_i!,$$ which is asymptotically uniform. The total number of such graphs is thus $(1+O(\mathscr{E}))\mathscr{P}^{-1}$.
	\end{theorem}
	
	Note that the bounds $d_{\max},x_{\max}=O(m^{\sfrac14}/\log m)$ would be sufficient for the required bound on $\Delta$. The proof of Theorem~\ref{theorem_simple} is in Section~\ref{subsection_applications_simple}.
	
	Theorem~\ref{theorem_simple} strengthens the sampling result Theorem~1 from~\cite{Bayati}. Our new result adds the possibility of forbidden edges, allows larger values of $d_{\max}$, and gives stronger asymptotic bounds on the probability of rejection and the deviation from uniformity.
	
	As for our graph counting result, the same formula was given in Theorem~4.6 from~\cite{mckay1985asymptotics}. Our error bound $O(\mathscr{E})$ is only a logarithmic factor larger than their error bound of $O(\Delta^2/m)$.
	
	\subsection{Bipartite graphs and directed graphs}\label{subsection_results_bipartite}
	
	We consider the following iterative stochastic process for sampling a random simple bipartite graph with given degree sequences on each part.
	
	\begin{definition}[Iterative bipartite graph process]
		Consider the bipartite vertex set $V=A\cup B$ with $A=\{v_1,\ldots,v_n\}$ and $B=\{v_1',\ldots,v_{n'}'\}$, consider degree sequences $\d=(d_1,\ldots,d_n)$ and $\d'=(d_1',\ldots,d_{n'}')$ of non-negative integers with $m:=d_1+\ldots+d_n=d_1'+\ldots+d_{n'}'>0$, and consider a set $X$ of forbidden edges.
		\begin{itemize}
			\item Start with no edges $E_0=\emptyset$.
			\item For $r=0,\ldots,m-1$, let $d_{i,r}$ and $d_{j,r}'$ denote the degrees of vertices $v_i$ and $v_j'$ in $(V,E_r)$, and consider those pairs of vertices $(v_i,v_j')\in A\times B$ with $d_{i,r}<d_i$, $d_{j,r}'<d_j'$, and $(v_i,v_j')\not\in E_r\cup X$.
			\begin{itemize}
				\item If no such pairs exist, let $E_{r+1}=E_r$.
				\item Otherwise, choose such a pair $(v_i,v_j')$ at random with probability proportional to $(d_i-d_{i,r})(d_j'-d_{j,r}')(1-\tfrac{d_id_j'}{2m})$, and let $E_{r+1}=E_r\cup\{(v_i,v_j')\}$.
			\end{itemize}
			\item Finally, apply rejection sampling to sample $E=E_m$ with $m$ edges.
		\end{itemize}
		This guarantees that $(V,E)$ is a simple bipartite graph with degree sequences $\d$ and $\d'$ on its parts and with no edges from $X$.
	\end{definition}
	
	\begin{theorem}\label{theorem_bipartite}
		Consider the iterative bipartite graph process on $V$, $\d$, $\d'$, and $X$. Let $x_{\max}$ be the maximum degree in $(V,X)$, and let
		\begin{align}
			d_{\max}&:=\max\{d_1,\ldots,d_n,d_1',\ldots,d_{n'}'\},
			\\\Delta&:=d_{\max}(d_{\max}+x_{\max}).
		\end{align}
		Consider the asymptotics as $m\to\infty$, and assume that $\Delta=O(m^{\sfrac12}/(\log m)^2)$. Then the probability of rejection is $O(\Delta/m)$. In particular, for $m$ large enough, there exists a simple bipartite graph with degree sequences $\d$ and $\d'$ on its parts and with no edges from $X$. Furthermore, any such graph is sampled with probability $(1+O(\mathscr{E}))\mathscr{P}$, where $\mathscr{E}:=\mathscr{E}(m,\Delta)$ and $$\mathscr{P}:=\exp(\tfrac12\lambda_\d\lambda_{\d'}+\mu_{\d,\d'}^{X})\frac1{m!}\prod_{i=1}^{n}d_i!\prod_{j=1}^{n'}d_j'!,$$ which is asymptotically uniform. The total number of such graphs is thus $(1+O(\mathscr{E}))\mathscr{P}^{-1}$.
	\end{theorem}
	
	Note that the bounds $d_{\max},x_{\max}=O(m^{\sfrac14}/\log m)$ would be sufficient for the required bound on $\Delta$. The proof of Theorem~\ref{theorem_bipartite} is in Section~\ref{subsection_applications_bipartite}.
	
	In Remark 1 from~\cite{Bayati}, it is stated without proof that their sampling result can be modified to sample bipartite graphs. Theorem~\ref{theorem_bipartite} strengthens this claim. Again, our new result adds the possibility of forbidden edges, allows larger values of $d_{\max}$, and gives stronger asymptotic bounds on the probability of rejection and the deviation from uniformity.
	
	As for our graph counting result, the same formula was given in Theorem~4.6 from~\cite{mckay1984asymptotics}. Again, our error bound $O(\mathscr{E})$ is only a logarithmic factor larger than their error bound of $O(\Delta^2/m)$.
	
	The iterative bipartite graph process can also be adapted to sample directed graphs as follows.
	
	\begin{definition}[Iterative directed graph process]
		Consider the vertex set $V=\{v_1,\ldots,v_n\}$, consider degree sequences $\d=(d_1,\ldots,d_n)$ and $\d'=(d_1',\ldots,d_{n}')$ of non-negative integers with $m:=d_1+\ldots+d_n=d_1'+\ldots+d_{n}'>0$, and consider a set $X$ of forbidden arcs. We define $\widetilde{V}:=A \cup B$ with $A=\{w_1,\ldots,w_n\}$ and $B=\{w_1',\ldots,w_n'\}$ disjoint, and $$\widetilde{X}:=\{(w_i,w_j'):(v_i,v_j)\in X\}\cup\{(w_i,w_i'):i=1,\ldots,n\}.$$ Consider sampling a bipartite graph from the iterative bipartite graph process on $\widetilde{V}$, $\d$, $\d'$, and $\widetilde{X}$. Identify any pair $w_i$, $w_i'$ with the vertex $v_i$ and any edge $(w_i,w_j')$ with an arc from $v_i$ to $v_j$. This results in a simple directed graph with out- and in- degree sequences $\d$ and $\d'$ and with no arcs from $X$.
	\end{definition}
	
	This lets us apply Theorem~\ref{theorem_bipartite} to sample directed graphs. We get $$\mathscr{P}=\exp(\lambda_{\d,\d'}+\tfrac12\lambda_{\d}\lambda_{\d'}+\mu_{\d,\d'}^{X})\frac1{m!}\prod_{i=1}^{n}d_i!d_i'!.$$ This strengthens the sampling result Theorem~1.2 from~\cite{Ieperen2024}. Our result adds the possibility of forbidden arcs, allows larger values of $d_{\max}$, and gives stronger asymptotic bounds on the probability of rejection and the deviation from uniformity.
	
	As for the resulting counting result, the same formula follows from Corollary~3.3 from~\cite{greenhill2026asymptoticenumerationconstrainedbipartite}. Again, our error bound $O(\mathscr{E})$ is only a logarithmic factor larger than their error bound of $O(\Delta^2/m)$.
	
	\subsection{Oriented graphs}\label{subsection_results_oriented}
	
	An \emph{oriented graph} is a simple directed graph for which no two vertices are connected by arcs in both ways. We consider the following iterative stochastic process for sampling a random oriented graph with given out- and in-degree sequences.
	
	\begin{definition}[Iterative oriented graph process]
		Consider the vertex set $V=\{v_1,\ldots,v_n\}$, consider degree sequences $\d=(d_1,\ldots,d_n)$ and $\d'=(d_1',\ldots,d_{n}')$ of non-negative integers with $m:=d_1+\ldots+d_n=d_1'+\ldots+d_{n}'>0$, and consider a set $X$ of forbidden arcs.
		\begin{itemize}
			\item Start with no edges $E_0=\emptyset$.
			\item For $r=0,\ldots,m-1$, let $d_{i,r}$ and $d_{i,r}'$ denote the out- and in-degrees of vertex $v_i$ in $(V,E_r)$, and consider those pairs of vertices $(v_i,v_j)\in V^2$ with $i\neq j$, $d_{i,r}<d_i$, $d_{j,r}'<d_j'$, $(v_i,v_j)\not\in E_r\cup X$, and $(v_j,v_i)\not\in E_r$.
			\begin{itemize}
				\item If no such pairs exist, let $E_{r+1}=E_r$.
				\item Otherwise, choose such a pair $(v_i,v_j)$ at random with probability proportional to $(d_i-d_{i,r})(d_j'-d_{j,r}')(1-\tfrac{d_id_j'}{2m})$, and let $E_{r+1}=E_r\cup\{(v_i,v_j)\}$.
			\end{itemize}
			\item Finally, apply rejection sampling to sample $E=E_m$ with $m$ arcs.
		\end{itemize}
		This guarantees that $(V,E)$ is an oriented graph with out- and in-degree sequences $\d$ and $\d'$ and with no arcs from $X$.
	\end{definition}
	
	\begin{theorem}\label{theorem_oriented}
		Consider the iterative oriented graph process on $V$, $\d$, $\d'$, and $X$. Let $x_{\max}$ be the maximum out- or in-degree in $(V,X)$, and let
		\begin{align}
			d_{\max}&:=\max\{d_1,\ldots,d_n,d_1',\ldots,d_{n}'\},
			\\\Delta&:=d_{\max}(d_{\max}+x_{\max}).
		\end{align}
		Consider the asymptotics as $m\to\infty$, and assume that $\Delta=O(m^{\sfrac12}/(\log m)^2)$. Then the probability of rejection is $O(\Delta/m)$. In particular, for $m$ large enough, there exists an oriented graph with out- and in-degree sequences $\d$ and $\d'$ and with no arcs from $X$. Furthermore, any such graph is sampled with probability $(1+O(\mathscr{E}))\mathscr{P}$, where $\mathscr{E}:=\mathscr{E}(m,\Delta)$ and $$\mathscr{P}:=\exp(\lambda_{\d,\d'}+\tfrac12\lambda_{\d,\d'}^2+\tfrac12\lambda_{\d}\lambda_{\d'}+\mu_{\d,\d'}^{X})\frac1{m!}\prod_{i=1}^{n}d_i!d_i'!,$$ which is asymptotically uniform. The total number of such graphs is thus $(1+O(\mathscr{E}))\mathscr{P}^{-1}$.
	\end{theorem}
	
	Note that the bounds $d_{\max},x_{\max}=O(m^{\sfrac14}/\log m)$ would be sufficient for the required bound on $\Delta$. Furthermore, note that, compared to the iterative directed graph process, the resulting value of $\mathscr{P}$ only differs by the additional term $\frac12\lambda_{\d,\d'}^2$ in the exponent. The proof of Theorem~\ref{theorem_oriented} is in Section~\ref{subsection_applications_oriented}.
	
	As for the resulting counting results, the same formula follows from Corollary~5.3 from~\cite{greenhill2026asymptoticenumerationconstrainedbipartite}, though forbidden edges are not considered. Again, our error bound $O(\mathscr{E})$ is only a logarithmic factor larger than their error bound of $O(\Delta^2/m)$.
	
	\subsection{Edge-colored graphs}\label{subsection_results_colored}
	
	In an edge-colored graph, there is a set of colors, and any edge is assigned one color. A simple edge-colored graph is thus essentially a set of edge-disjoint simple graphs, one for each color, on a common set of vertices. We thus consider the problem where a set of degree sequences is given, one for each color, on a common set of vertices, and we need to sample edge-disjoint simple graphs with the given degree sequences.
	
	This relates to the packing problem of degree sequences, which asks whether this sampling is possible for given degree sequences. In general, this is an NP-complete decision problem~\cite{West2012}, even for only two degree sequences. However, by assuming certain asymptotic bounds on the given degrees, we can still give an iterative stochastic process that results in asymptotically uniform sampling. Since the process is successful with high probability, the degree sequences are eventually guaranteed to pack in the given regime.
	
	The main idea behind our process is to sample the graphs one by one. Sampling each graph uniformly from its degree sequence, we also have to forbid the edges from all previously sampled graphs to guarantee edge-disjointness. However, these forbidden edges affect the total number of possible graphs, so this results in a bias towards sequences of edge-disjoint graphs for which the earlier graphs leave less options for the later graphs. We fix this by introducing an opposite bias to the distribution of the earlier graphs, such that the two biases cancel each other asymptotically for each pair of sampled graphs.
	
	\begin{definition}[Iterative edge-colored graph process]
		Consider the vertex set $V=\{v_1,\ldots,v_n\}$, and for $g=1,\ldots,\chi$, consider a degree sequence $\d^{(g)}=(d_1^{(g)},\ldots,d_n^{(g)})$ of non-negative integers with $2m^{(g)}:=d_1^{(g)}+\ldots+d_n^{(g)}>0$ even, and consider a set $X^{(g)}$ of forbidden edges. For $g=1,\ldots,\chi$, define $$\b^{(g)}_{i,j}:=\frac{d_i^{(g)}d_j^{(g)}}{4m^{(g)}}+\sum_{\substack{h=g+1,\ldots,\chi\\\{v_i,v_j\}\not\in X^{(h)}}}\frac{d_i^{(h)}d_j^{(h)}}{2m^{(h)}}.$$ For $g=1,\ldots,\chi$, sample $E^{(g)}$ with the following steps:
		\begin{itemize}
			\item Start with no edges $E_0^{(g)}=\emptyset$.
			\item For $r=0,\ldots,m^{(g)}-1$, let $d_{i,r}^{(g)}$ denote the degree of vertex $v_i$ in $(V,E_r^{(g)})$, and consider those pairs of vertices $v_i,v_j\in V$ with $i\neq j$, $d_{i,r}^{(g)}<d_i^{(g)}$, $d_{j,r}^{(g)}<d_j^{(g)}$, and $\{v_i,v_j\}\not\in E_r^{(g)}\cup X^{(g)}\cup E^{(1)}\cup\ldots\cup E^{(g-1)}$.
			\begin{itemize}
				\item If no such pairs exist, let $E_{r+1}^{(g)}=E_r^{(g)}$.
				\item Otherwise, choose such a pair $v_i,v_j$ at random with probability proportional to $(d_i^{(g)}-d_{i,r}^{(g)})(d_j^{(g)}-d_{j,r}^{(g)})(1-\b^{(g)}_{i,j})$, and let $E_{r+1}^{(g)}=E_r^{(g)}\cup\{\{v_i,v_j\}\}$.
			\end{itemize}
			\item Finally, apply rejection sampling to sample $E^{(g)}=E_{m^{(g)}}^{(g)}$ with $m^{(g)}$ edges.
		\end{itemize}
		This guarantees that $(V,E^{(g)})$ for $g=1,\ldots,\chi$ are edge-disjoint simple graphs with degree sequence $\d^{(g)}$ and with no edges from $X^{(g)}$.
	\end{definition}
	
	\begin{theorem}\label{theorem_colored}
		Consider the iterative edge-colored graph process on $V$ and $\d^{(g)}$ and $X^{(g)}$ for $g=1,\ldots,\chi$. Let $x_{\max}^{(g)}$ be the maximum degree in $(V,X^{(g)})$, and let
		\begin{align}
			d_{\max}^{(g)}&:=\max\{d_1^{(g)},\ldots,d_n^{(g)}\},
			\\\Delta^{(g)}&:=d_{\max}^{(g)}\left(x_{\max}^{(g)}+\sum_{h=1}^{g}d_{\max}^{(h)}\right)+m^{(g)}\sum_{h=g+1}^{\chi}\frac{(d_{\max}^{(h)})^2}{m^{(h)}}.
		\end{align}
		Consider the asymptotics as $m_{\min}:=\min\{m^{(1)},\ldots,m^{(\chi)}\}\to\infty$, and assume that $\Delta^{(g)}=O((m^{(g)})^{\sfrac12}/(\log m^{(g)})^2)$ for $g=1,\ldots,\chi$. Then each iteration $g=1,\ldots,\chi$ has probability $O(\Delta^{(g)}/m^{(g)})$ of rejection. In particular, for $m_{\min}$ large enough, there exist edge-disjoint simple graphs for $g=1,\ldots,\chi$ with degree sequence $\d^{(g)}$ and with no edges from $X^{(g)}$. Furthermore, any such graphs are sampled with probability $e^{O(\mathscr{E})}\mathscr{P}$, where $\mathscr{E}:=\sum_{g=1}^{\chi}\mathscr{E}(m^{(g)},\Delta^{(g)})$ and
		\begin{align}
			\mathscr{P}&:=\exp\left(\sum_{g=1}^{\chi}\left(\tfrac12\lambda_{\d^{(g)}}+\tfrac14\lambda_{\d^{(g)}}^2+\tfrac12\sum_{h=g+1}^{\chi}\lambda_{\d^{(g)},\d^{(h)}}^2+\mu_{\d^{(g)}}^{X^{(g)}}\right)\right)
			\\&\quad\cdot\prod_{g=1}^{\chi}\left(\frac1{(2m^{(g)}-1)!!}\prod_{i=1}^{n}d_i^{(g)}!\right),
		\end{align}
		which is asymptotically uniform if $\mathscr{E}=o(1)$. The total number of such graphs is thus $e^{O(\mathscr{E})}\mathscr{P}^{-1}$.
	\end{theorem}
	
	Note that the number of colors $\chi$ need not be fixed, but the bound $\Delta^{(g)}=O((m^{(g)})^{\sfrac12}/(\log m^{(g)})^2)$ is uniform over $g=1,\ldots,\chi$. So we require that $\Delta^{(g)}\leq C(m^{(g)})^{\sfrac12}/(\log m^{(g)})^2$ for all $g=1,\ldots,\chi$ whenever $m_{\min}\geq N$ for some fixed constants $C,N>0$. However, we need $\mathscr{E}=o(1)$ to guarantee asymptotic uniformity. If $\chi$ is fixed, then $\mathscr{E}=o(1)$ follows, but if $\chi$ grows unboundedly as $m_{\min}\to\infty$, then $\mathscr{E}$ may also grow unboundedly.
	
	If $m^{(1)}\leq\ldots\leq m^{(\chi)}$, note that the bounds $$d_{\max}^{(g)}=O\left(\frac{(m^{(g)})^{\sfrac14}}{\chi^{\sfrac12}\log m^{(g)}}\right),\ x_{\max}^{(g)}=O\left(\frac{\chi^{\sfrac12}(m^{(g)})^{\sfrac14}}{\log m^{(g)}}\right)$$ for $g=1,\ldots,\chi$ would be sufficient for the required bound on $\Delta^{(g)}$. Indeed, since $x\mapsto x^{\sfrac14}/\log x$ is eventually increasing and $x\mapsto x^{-\sfrac12}/(\log x)^2$ is eventually decreasing, we get
	\begin{align}
		d_{\max}^{(h)}&=O\left(\frac{(m^{(g)})^{\sfrac14}}{\chi^{\sfrac12}\log m^{(g)}}\right)\ \text{for}\ h\leq g,
		\\\frac{(d_{\max}^{(h)})^2}{m^{(h)}}&=O\left(\frac{(m^{(g)})^{-\sfrac12}}{\chi(\log m^{(g)})^2}\right)\ \text{for}\ h\geq g.
	\end{align}
	This suggests that it helps to sort the colors by their required number of edges before applying the iterative edge-colored graph process. The proof of Theorem~\ref{theorem_colored} is in Section~\ref{subsection_applications_colored}.
	
	\subsection{Edge-colored bipartite graphs}\label{subsection_results_bipcol}
	
	For a set of colors, we consider the problem where a set of pairs of degree sequences is given, one pair for each color, on a common bipartite vertex set, and we need to sample edge-disjoint simple bipartite graphs with the given pairs of degree sequences on their parts.
	
	\begin{definition}[Iterative edge-colored bipartite graph process]
		Consider the bipartite vertex set $V=A\cup B$ with $A=\{v_1,\ldots,v_n\}$ and $B=\{v_1',\ldots,v_{n'}'\}$, and for $g=1,\ldots,\chi$, consider degree sequences $\d^{(g)}=(d_1^{(g)},\ldots,d_n^{(g)})$ and $\d'^{(g)}=(d_1'^{(g)},\ldots,d_{n'}'^{(g)})$ of non-negative integers with $m^{(g)}:=d_1^{(g)}+\ldots+d_n^{(g)}=d_1'^{(g)}+\ldots+d_{n'}'^{(g)}>0$, and consider a set $X^{(g)}$ of forbidden edges. For $g=1,\ldots,\chi$, define $$\b^{(g)}_{i,j}:=\frac{d_i^{(g)}d_j'^{(g)}}{2m^{(g)}}+\sum_{\substack{h=g+1,\ldots,\chi\\(v_i,v_j')\not\in X^{(h)}}}\frac{d_i^{(h)}d_j'^{(h)}}{m^{(h)}}.$$ For $g=1,\ldots,\chi$, sample $E^{(g)}$ with the following steps:
		\begin{itemize}
			\item Start with no edges $E_0^{(g)}=\emptyset$.
			\item For $r=0,\ldots,m^{(g)}-1$, let $d_{i,r}^{(g)}$ and $d_{j,r}'^{(g)}$ denote the degrees of vertices $v_i$ and $v_j'$ in $(V,E_r^{(g)})$, and consider those pairs of vertices $(v_i,v_j')\in A\times B$ with $d_{i,r}^{(g)}<d_i^{(g)}$, $d_{j,r}'^{(g)}<d_j'^{(g)}$, and $(v_i,v_j')\not\in E_r^{(g)}\cup X^{(g)}\cup E^{(1)}\cup\ldots\cup E^{(g-1)}$.
			\begin{itemize}
				\item If no such pairs exist, let $E_{r+1}^{(g)}=E_r^{(g)}$.
				\item Otherwise, choose such a pair $(v_i,v_j')$ at random with probability proportional to $(d_i^{(g)}-d_{i,r}^{(g)})(d_j'^{(g)}-d_{j,r}'^{(g)})(1-\b^{(g)}_{i,j})$, and let $E_{r+1}^{(g)}=E_r^{(g)}\cup\{(v_i,v_j')\}$.
			\end{itemize}
			\item Finally, apply rejection sampling to sample $E^{(g)}=E_{m^{(g)}}^{(g)}$ with $m^{(g)}$ edges.
		\end{itemize}
		This guarantees that $(V,E^{(g)})$ for $g=1,\ldots,\chi$ are edge-disjoint simple bipartite graphs with degree sequences $\d^{(g)}$ and $\d'^{(g)}$ on their parts and with no edges from $X^{(g)}$.
	\end{definition}
	
	\begin{theorem}\label{theorem_bipcol}
		Consider the iterative edge-colored graph process on $V$ and $\d^{(g)}$, $\d'^{(g)}$ and $X^{(g)}$ for $g=1,\ldots,\chi$. Let $x_{\max}^{(g)}$ be the maximum degree in $(V,X^{(g)})$, and let
		\begin{align}
			d_{\max}^{(g)}&:=\max\{d_1^{(g)},\ldots,d_n^{(g)},d_1'^{(g)},\ldots,d_{n'}'^{(g)}\},
			\\\Delta^{(g)}&:=d_{\max}^{(g)}\left(x_{\max}^{(g)}+\sum_{h=1}^{g}d_{\max}^{(h)}\right)+m^{(g)}\sum_{h=g+1}^{\chi}\frac{(d_{\max}^{(h)})^2}{m^{(h)}}.
		\end{align}
		Consider the asymptotics as $m_{\min}:=\min\{m^{(1)},\ldots,m^{(\chi)}\}\to\infty$, and assume that $\Delta^{(g)}=O((m^{(g)})^{\sfrac12}/(\log m^{(g)})^2)$ for $g=1,\ldots,\chi$. Then each iteration $g=1,\ldots,\chi$ has probability $O(\Delta^{(g)}/m^{(g)})$ of rejection. In particular, for $m_{\min}$ large enough, there exist edge-disjoint simple bipartite graphs for $g=1,\ldots,\chi$ with degree sequences $\d^{(g)}$ and $\d'^{(g)}$ on their parts and with no edges from $X^{(g)}$. Furthermore, any such graphs are sampled with probability $e^{O(\mathscr{E})}\mathscr{P}$, where $\mathscr{E}:=\sum_{g=1}^{\chi}\mathscr{E}(m^{(g)},\Delta^{(g)})$ and
		\begin{align}
			\mathscr{P}&:=\exp\Bigg(\sum_{g=1}^{\chi}\Bigg(\tfrac12\lambda_{\d^{(g)}}\lambda_{\d'^{(g)}}+\sum_{h=g+1}^{\chi}\lambda_{\d^{(g)},\d^{(h)}}\lambda_{\d'^{(g)},\d'^{(h)}}+\mu_{\d^{(g)},\d'^{(g)}}^{X^{(g)}}\Bigg)\Bigg)
			\\&\quad\cdot\prod_{g=1}^{\chi}\left(\frac1{m^{(g)}!}\prod_{i=1}^{n}d_i^{(g)}!\prod_{j=1}^{n'}d_j'^{(g)}!\right),
		\end{align}
		which is asymptotically uniform if $\mathscr{E}=o(1)$. The total number of such graphs is thus $e^{O(\mathscr{E})}\mathscr{P}^{-1}$.
	\end{theorem}
	
	If $m^{(1)}\leq\ldots\leq m^{(\chi)}$, note that the bounds $$d_{\max}^{(g)}=O\left(\frac{(m^{(g)})^{\sfrac14}}{\chi^{\sfrac12}\log m^{(g)}}\right),\ x_{\max}^{(g)},y_{\max},z_{\max}^{(g)}=O\left(\frac{\chi^{\sfrac12}(m^{(g)})^{\sfrac14}}{\log m^{(g)}}\right)$$ for $g=1,\ldots,\chi$ would be sufficient for the required bound on $\Delta^{(g)}$. Indeed, since $x\mapsto x^{\sfrac14}/\log x$ is eventually increasing and $x\mapsto x^{-\sfrac12}/(\log x)^2$ is eventually decreasing, we get
	\begin{align}
		d_{\max}^{(h)}&=O\left(\frac{(m^{(g)})^{\sfrac14}}{\chi^{\sfrac12}\log m^{(g)}}\right)\ \text{for}\ h\leq g,
		\\\frac{(d_{\max}^{(h)})^2}{m^{(h)}}&=O\left(\frac{(m^{(g)})^{-\sfrac12}}{\chi(\log m^{(g)})^2}\right)\ \text{for}\ h\geq g.
	\end{align}
	The proof of Theorem~\ref{theorem_bipcol} is in Section~\ref{subsection_applications_bipcol}.
	
	\subsection{Hypergraphs}\label{subsection_results_hyper}
	
	A \emph{simple hypergraph} is a tuple $(V,E)$, where $V$ is a set of vertices and $E$ is a set of \emph{hyperedges}, which are sets of at least two vertices from $V$. Note that there cannot be duplicated hyperedges or duplicated vertices within a hyperedge.
	
	Any simple hypergraph can be described by a simple bipartite graph between its vertices and hyperedges by connecting any hyperedge to all its vertices. We use this idea to randomly sample a simple hypergraph, given a degree sequence $\d$ and a sequence $\d'$ of hyperedge cardinalities.
	
	\begin{definition}[Iterative hypergraph process]
		Consider the vertex set $V=\{v_1,\ldots,v_n\}$ with a degree sequence $\d=(d_1,\ldots,d_n)$ of non-negative integers. Consider a sequence $\d'=(d_1',\ldots,d_m')$ of integers of at least two. Assume that $s:=d_1+\ldots+d_n=d_1'+\ldots+d_m'>0$. Finally, consider a set $X$ of forbidden hyperedges. For $V'=\{v_1',\ldots,v_m'\}$, use the iterative bipartite graph process to sample a simple bipartite graph $(V\cup V',\widetilde{E})$ with degree sequences $\d$ and $\d'$ on its parts. Apply rejection sampling to make sure that all vertices in $V'$ have pairwise distinct neighborhoods that are not in $X$. Finally, identify each such neighborhood with a hyperedge on $V$ to sample a simple hypergraph $(V,E)$ with degree sequence $\d$ and with hyperedges not in $X$ with cardinalities matching $\d'$.
	\end{definition}
	
	\begin{theorem}\label{theorem_hyper}
		Consider the iterative hypergraph process on $V$, $\d$, $\d'$, and $X$. Let $$d_{\max}:=\max\{d_1,\ldots,d_n,d_1',\ldots,d_m'\},$$ and for $k\geq2$, let $c_k$ be the required number of hyperedges of cardinality $k$, and let $x_k$ be the number of hyperedges in $X$ of cardinality $k$. Define
		\begin{align}
			\rho_k&:=k!\binom{c_k}{2}\left(\frac{\lambda_{\d}}{s}\right)^k\leq\binom{c_k}{2}\left(\frac{k\lambda_{\d}}{s}\right)^k,
			\\\rho_X&:=\sum_{e\in X}\frac{c_{|e|}|e|!}{s^{|e|}}\prod_{v_i\in e}d_i\leq\sum_{k=2}^{d_{\max}}x_{k}c_{k}\left(\frac{kd_{\max}}{s}\right)^{k}.
		\end{align}
		Consider the asymptotics as $s\to\infty$, and assume that $d_{\max}=O(s^{\sfrac14}/\log s)$ and $\rho_2,\rho_X=o(1)$. Then the probability of rejection in the iterative bipartite graph process is $O(d_{\max}^2/s)$, and the probability of rejection afterwards is $O(\rho_2+\rho_3+\rho_X+1/s)$. In particular, for $s$ large enough, there exists a simple hypergraph with degree sequence $\d$ and with hyperedges not in $X$ with cardinalities matching $\d'$. Furthermore, any such hypergraph is sampled with probability $(1+O(\mathscr{E}+\rho_2+\rho_X))\mathscr{P}$, where $\mathscr{E}:=\mathscr{E}(s,d_{\max}^2)$ and $$\mathscr{P}:=\exp\left(\tfrac12\lambda_{\d}\lambda_{\boldsymbol{\d'}}\right)\frac1{s!}\prod_{i=1}^{n}d_i!\prod_{k=2}^{d_{\max}}(k!^{c_k}\cdot c_k!),$$ which is asymptotically uniform. The total number of such hypergraphs is thus $(1+O(\mathscr{E}+\rho_2+\rho_X))\mathscr{P}^{-1}$.
	\end{theorem}
	
	Note that the terms $\rho_3$ and $1/s$ do not need to be controlled, and they disappear in the asymptotic error bound, because it can be shown that $\rho_3,1/s=O(\mathscr{E})$. Furthermore, note that the set $X$ only affects the asymptotic bounds, so in the given regime, it does not significantly affect the size of the sample space. The proof of Theorem~\ref{theorem_hyper} is in Section~\ref{subsection_applications_hyper}.
	
	The formula from our counting result appears to be new, though the same counting formula for $r$-uniform hypergraphs (where all hyperedges have cardinality $r\geq3$) and without forbidden edges is given in Theorem~1.1 from~\cite{BLINOVSKY2016287}. This setting simplifies our error bound to $O(\mathscr{E})$, which is only a logarithmic factor larger than their error bound of $O(d_{\max}^3/m)=O(rd_{\max}^3/s)$ if $d_{\max}=O(r)$.
	
	\subsection{Directed hypergraphs}\label{subsection_results_dihy}
	
	A \emph{simple directed hypergraph} is a tuple $(V,E)$, where $V$ is a set of vertices and $E$ is a set of hyperarcs, which are tuples $e=(D,C)$ of disjoint non-empty sets of vertices from $V$. The sets $D$ and $C$ are respectively called the domain and codomain of $e$, and the \emph{order} of $e$ is defined as the tuple of cardinalities $|e|:=(|D|,|C|)$. Note that there cannot be duplicated hyperarcs or duplicated vertices in the domain or codomain of a hyperarc. The out- or in-degrees of a vertex are the number of hyperarcs where the vertex appears in its domain or codomain respectively.
	
	Any simple directed hypergraph can be described by a simple edge-colored bipartite graph between its vertices and hyperarcs by connecting any hyperarc to all vertices in its domain with one color and to all vertices in its codomain with a second color. We use this idea to randomly sample a simple directed hypergraph, given out- and in-degree sequences $\d^{(1)}$ and $\d^{(2)}$ and sequences $\d'^{(1)}$ and $\d'^{(2)}$ of domain and codomain cardinalities of hyperarcs.
	
	\begin{definition}[Iterative directed hypergraph process]
		Consider the vertex set $V=\{v_1,\ldots,v_n\}$ with out- and in-degree sequences $\d^{(1)}=(d_1^{(1)},\ldots,d_n^{(1)})$ and $\d^{(2)}=(d_1^{(2)},\ldots,d_n^{(2)})$ of non-negative integers. Consider sequences $\d'^{(1)}=(d_1'^{(1)},\ldots,d_m'^{(1)})$ and $\d'^{(2)}=(d_1'^{(2)},\ldots,d_m'^{(2)})$ of positive integers with $s^{(i)}:=d_1^{(i)}+\ldots+d_n^{(i)}=d_1'^{(i)}+\ldots+d_m'^{(i)}>0$ for $i\in\{1,2\}$, and consider a set $X$ of forbidden hyperarcs. For $V':=\{v_1',\ldots,v_m'\}$, use the iterative edge-colored bipartite graph process to sample edge-disjoint simple bipartite graphs $(V\cup V',\widetilde{E}^{(i)})$ with degree sequences $\d^{(i)}$ and $\d'^{(i)}$ on its parts for $i\in\{1,2\}$. Apply rejection sampling to make sure that all vertices in $V'$ have distinct pairs of neighborhoods that are not in $X$. Finally, identify each such pair of neighborhoods with a hyperarc on $V$ to sample a simple directed hypergraph $(V,E)$ with out- and in-degree sequences $\d^{(1)}$ and $\d^{(2)}$ and with hyperarcs not in $X$ with orders $(d_1'^{(1)},d_1'^{(2)}),\ldots,(d_m'^{(1)},d_m'^{(2)})$.
	\end{definition}
	
	\begin{theorem}\label{theorem_dihy}
		Consider the iterative directed hypergraph process on $V$, $\d^{(1)}$, $\d^{(2)}$, $\d'^{(1)}$, $\d'^{(2)}$, and $X$. Let $$d_{\max}:=\max\{d_1^{(1)},\ldots,d_n^{(1)},d_1^{(2)},\ldots,d_n^{(2)},d_1'^{(1)},\ldots,d_m'^{(1)},d_1'^{(2)},\ldots,d_m'^{(2)}\},$$ and for $k^{(1)},k^{(2)}\geq1$, let $c_{k^{(1)},k^{(2)}}$ be the required number of hyperarcs of order $(k^{(1)},k^{(2)})$, and let $x_{k^{(1)},k^{(2)}}$ be the number of hyperarcs in $X$ of order $(k^{(1)},k^{(2)})$. Define
		\begin{align}
			\rho_{k^{(1)},k^{(2)}}&:=k^{(1)}!k^{(2)}!\binom{c_{k^{(1)},k^{(2)}}}{2}\left(\frac{\lambda_{\d^{(1)}}}{s^{(1)}}\right)^{k^{(1)}}\left(\frac{\lambda_{\d^{(2)}}}{s^{(2)}}\right)^{k^{(2)}}
			\\&\leq\binom{c_{k^{(1)},k^{(2)}}}{2}\left(\frac{k^{(1)}\lambda_{\d^{(1)}}}{s^{(1)}}\right)^{k^{(1)}}\left(\frac{k^{(2)}\lambda_{\d^{(2)}}}{s^{(2)}}\right)^{k^{(2)}},
			\\\rho_X&:=\sum_{(D,C)\in X}\frac{c_{|D|,|C|}|D|!|C|!}{(s^{(1)})^{|D|}(s^{(2)})^{|C|}}\prod_{v_{i^{(1)}}\in D}d_{i^{(1)}}^{(1)}\prod_{v_{i^{(2)}}\in C}d_{i^{(2)}}^{(2)}
			\\&\leq\sum_{k^{(1)}=1}^{d_{\max}}\sum_{k^{(2)}=1}^{d_{\max}}x_{k^{(1)},k^{(2)}}c_{k^{(1)},k^{(2)}}\left(\frac{k^{(1)}d_{\max}}{s^{(1)}}\right)^{k^{(1)}}\left(\frac{k^{(2)}d_{\max}}{s^{(2)}}\right)^{k^{(2)}}.
		\end{align}
		Assume that $s^{(1)}\leq s^{(2)}$ and consider the asymptotics as $s^{(1)}\to\infty$. Assume that $d_{\max}=O((s^{(1)})^{\sfrac14}/\log(s^{(1)}))$ and $\rho_{1,1},\rho_X=o(1)$. Then the total probability of rejection in the iterative edge-colored bipartite graph process is $O(d_{\max}^2/s^{(1)})$ in both iterations, and the probability of rejection afterwards is $O(\rho_{1,1}+\rho_{1,2}+\rho_{2,1}+\rho_X+1/s^{(1)})$. In particular, for $s^{(1)}$ large enough, there exists a simple directed hypergraph with out- and in degree sequences $\d^{(1)}$ and $\d^{(2)}$ and with hyperarcs not in $X$ and with orders $(d_1'^{(1)},d_1'^{(2)}),\ldots,(d_m'^{(1)},d_m'^{(2)})$. Furthermore, any such directed hypergraph is sampled with probability $(1+O(\mathscr{E}+\rho_{1,1}+\rho_X))\mathscr{P}$, where $\mathscr{E}:=\mathscr{E}(s^{(1)},d_{\max}^2)$ and
		\begin{align}
			\mathscr{P}&:=\exp\left(\tfrac12\lambda_{\d^{(1)}}\lambda_{\d'^{(1)}}+\tfrac12\lambda_{\d^{(2)}}\lambda_{\d'^{(2)}}+\lambda_{\d^{(1)},\d^{(2)}}\lambda_{\d'^{(1)},\d'^{(2)}}\right)
			\\&\quad\cdot\frac1{s^{(1)}!s^{(2)}!}\prod_{i=1}^{n}d_i^{(1)}!d_i^{(2)}!\prod_{k^{(1)}=1}^{d_{\max}}\prod_{k^{(2)}=1}^{d_{\max}}((k^{(1)}!k^{(2)}!)^{c_{k^{(1)},k^{(2)}}}(c_{k^{(1)},k^{(2)}})!),
		\end{align}
		which is asymptotically uniform. The total number of such directed hypergraphs is thus $(1+O(\mathscr{E}+\rho_{1,1}+\rho_X))\mathscr{P}^{-1}$.
	\end{theorem}
	
	Note that the terms $\rho_{1,2}$, $\rho_{2,1}$, and $1/s^{(1)}$ do not need to be controlled, and they disappear in the asymptotic error bound, because it can be shown that $\rho_{1,2},\rho_{2,1},1/s^{(1)}=O(\mathscr{E})$. Furthermore, note that the set $X$ only affects the asymptotic bounds, so in the given regime, it does not significantly affect the size of the sample space. Finally, in case $s^{(1)}\leq s^{(2)}$ does not hold, one can swap $\d^{(1)}$ with $\d^{(2)}$ and $\d'^{(1)}$ with $\d'^{(2)}$ and invert all the hyperarcs in the sampled directed hypergraph. The proof of Theorem~\ref{theorem_dihy} is in Section~\ref{subsection_applications_dihy}.
	
	The formula from our counting result is also given in Theorem~1.1 from~\cite{greenhill2024enumerationdihypergraphsspecifieddegrees}, though forbidden hyperarcs and hyperarcs of order $(1,1)$ are not considered. This setting simplifies our error bound to $O(\mathscr{E})$, which is only a logarithmic factor larger than a simplified version of their error bound, which is $O(d_{\max}^4/s^{(1)})$.

	\section{The Configuration Space and the IMFIS Process}\label{section_independent}

	We aim to generalize the process of iteratively randomly matching half-edges in the configuration model under certain constraints that guarantee that a simple graph is sampled. This involves capturing the similarity and the difference between the configuration models for undirected graphs and for bipartite graphs. For this purpose, we reformulate the problem of matching half-edges in the (bipartite) configuration model as an independent set problem in a graph, which we call the (bipartite) configuration space. This lets us generalize the iterative random matching of half-edges in the configuration model to the iterative random construction of independent sets in the configuration space, which we call the iterative maximal feasible independent set (IMFIS) process.
	
	We show that the configuration space and the bipartite configuration space share a special property, which we call $2$-uniformity, and we demonstrate that $2$-uniformity is what makes the IMFIS process well-behaved. An exact formulation of this statement is given in our main theorem, Theorem~\ref{theorem_main} in Section~\ref{section_main}. Finally, we discuss the basic properties and a classification of $2$-uniform graphs. We show that $2$-uniformity captures the similarity between the configuration models for simple graphs and for simple bipartite graphs, and we introduce a parameter $\ell(G)$ for $2$-uniform graphs $G$ that captures their difference.
	
	Throughout the rest of this paper, we use the following conventions for the definitions of graphs, matchings, and independent sets.
	
	\begin{definition}[Graph]
		A \emph{graph} is a tuple $G=(V,E)$ with $V$, a finite set of vertices, and $E$, a set of edges, which are unordered pairs of vertices. In a \emph{multigraph}, $E$ and its elements are multisets, which allows \emph{self-loops} and \emph{multi-edges}. To emphasize the distinction from a multigraph, a graph may be referred to as a \emph{simple graph}.
	\end{definition}
	
	\begin{definition}[Matching]
		A \emph{matching} in a graph is a set of pairwise disjoint edges. The matching is \emph{perfect} if the union of its edges contains all vertices. A \emph{(perfect) matching on a set $X$} is a (perfect) matching in the complete graph with vertex set $X$. For disjoint sets $X$ and $Y$, a \emph{(perfect) bipartite matching between $X$ and $Y$} is a (perfect) matching in the complete bipartite graph with parts $X$ and $Y$.
	\end{definition}
	
	\begin{definition}[Independent set]
		An \emph{independent set} of a graph is a set of vertices no two of which are adjacent. The independent set is \emph{maximal} if it is not a strict subset of another independent set, and it is \emph{maximum} if no other independent set has strictly more vertices. The abbreviation \emph{MIS} stands for a maximum independent set.
	\end{definition}
	
	\subsection{Reformulating the configuration model}\label{subsection_independent_configuration}
	
	We reformulate the problem of matching half-edges in the (bipartite) configuration model~\cite{BOLLOBAS1980,molloy1995} as an independent set problem in a graph, which we call the (bipartite) configuration space.
	
	\subsubsection{The configuration model for undirected graphs}\label{subsection_independent_simple}
	
	Let $\d=(d_1,\ldots,d_n)$ be a sequence of non-negative integers with even sum $d_1+\ldots+d_n$. Let $W_1,\ldots,W_n$ be pairwise disjoint sets with $|W_i|=d_i$ for $i=1,\ldots,n$. We refer to the elements of $W:=W_1\cup\ldots\cup W_n$ as half-edges. A perfect matching on the set $W$ naturally induces a multigraph with degree sequence $\d$ by identifying each set $ W_i $ with a vertex $w_i$ of the multigraph. We translate perfect matchings on $ W $ into MISs of the following auxiliary graph.
	
	\begin{definition}[Configuration space]\label{definition_space}
		The \emph{configuration space} for $\d$ is the graph $G_\d:=(V_\d,E_\d)$, where $V_\d$ consists of all unordered pairs of half-edges, and where two vertices in $V_\d$ are adjacent if their corresponding pairs of half-edges overlap. Note that $G_{\d}$ is the line graph of the complete graph on $W$.
	\end{definition}
	
	It follows directly from the definition that matchings on $W$ correspond exactly to independent sets of $G_\d$. The matching is perfect if and only if the corresponding independent set of $G_\d$ is maximum (or equivalently, maximal in case of $G_\d$). Thus, to sample a uniformly random perfect matching on $W$, we may instead sample a uniformly random MIS of $G_\d$.
	
	A straightforward but inefficient way to obtain simple graphs is through rejection sampling, which discards any generated graphs with self-loops or multi-edges. Then a uniform distribution of perfect matchings on $W$ induces a uniform distribution of simple graphs with degree sequence $\d$, as any such graph is induced by $d_1!\ldots d_n!$ perfect matchings on $W$.
	
	To implement rejection sampling on the configuration space, we need some notion to detect self-loops and multi-edges. For this purpose, we introduce a set of forbidden vertices $F_\d\subset V_\d$, and an equivalence relation $R_\d\subset V_\d^2$ of forbidden pairs of vertices in the configuration space. Specifically, $F_\d$ consists of those vertices that induce a self-loop, meaning they pair half-edges from the same set $W_i$, and we declare two vertices in $V_\d$ as $R_\d$-equivalent if they induce the same edge, that is, if they pair half-edges from the same sets $W_i$ and $W_j$. An MIS of $G_\d$ is then rejected if it contains a vertex from $F_\d$, or if it contains two vertices that are $R_\d$-equivalent.
	
	\subsubsection{The configuration model for bipartite graphs}\label{subsection_independent_bipartite}
	
	Let $\d=(d_1,\ldots,d_n)$ and $\d'=(d_1',\ldots,d_{n'}')$ be sequences of non-negative integers with equal sums $d_1+\ldots+d_n=d_1'+\ldots+d_{n'}'$. Let $W_1,\ldots,W_n$ and $W_1',\ldots,W_{n'}'$  be pairwise disjoint sets with $|W_i|=d_i$ for $i=1,\ldots,n$ and $|W_i'|=d_i'$ for $i=1,\ldots,n'$. We refer to the elements of $W:=W_1\cup\ldots\cup W_n$ and $W':=W_1'\cup\ldots\cup W_{n'}'$ as half-edges. A perfect bipartite matching between the sets $W$ and $W'$ naturally induces a bipartite multigraph with degree sequences $\d$ and $\d'$ on its parts by identifying each set $W_i$ with a vertex $w_i$ and each set $W_j'$ with a vertex $w_j'$ of the bipartite multigraph. We translate perfect bipartite matchings between $W$ and $W'$ into MISs of the following auxiliary graph.
	
	\begin{definition}[Bipartite configuration space]\label{definition_bispace}
		The \emph{bipartite configuration space} for $\d$ and $\d'$ is the graph $G_{\d,\d'}:=(V_{\d,\d'},E_{\d,\d'})$, where $V_{\d,\d'}$ consists of all pairs of half-edges, one from $W$ and one from $W'$, and where two vertices in $V_{\d,\d'}$ are adjacent if their corresponding pairs of half-edges overlap. Note that $G_{\d,\d'}$ is the line graph of the complete bipartite graph with parts $W$ and $W'$.
	\end{definition}
	
	Similarly to the configuration model for simple graphs, we find that bipartite matchings between $W$ and $W'$ correspond exactly to independent sets of $G_{\d,\d'}$. We find that any simple bipartite graph with degree sequences $\d$ and $\d'$ on its parts is induced by $d_1!\ldots d_n!d_1'!\ldots d_{n'}'!$ perfect bipartite matchings between $W$ and $W'$, so a uniform distribution of such graphs is induced by a uniform distribution of perfect bipartite matchings between $W$ and $W'$ through rejection sampling.
	
	To implement rejection sampling on the bipartite configuration space, since self-loops are not possible, it suffices to detect multi-edges through an equivalence relation $R_{\d,\d'}\subset V_{\d,\d'}^2$. Specifically, we declare two vertices in $V_{\d,\d'}$ as $R_{\d,\d'}$-equivalent if they induce the same edge, that is, if they pair half-edges from the same sets $W_i$ and $W_j'$. An MIS of $G_{\d,\d'}$ is then rejected if it contains two vertices that are $R_{\d,\d'}$-equivalent.
	
	\subsection{Iterative stochastic processes for sampling independent sets}
	
	We generalize the iterative random matching of half-edges in the configuration model to the iterative random greedy construction of independent sets in the configuration space. We first study a uniform iterative stochastic process, which we call the IMIS process. We show that the (bipartite) configuration space has a special property, which we call regular independent sets, that guarantees the IMIS process samples MISs uniformly at random. We then study a more sophisticated process, which we call the IMFIS process, that permits certain feasibility requirements of independent sets and a weight function to be specified. We show that the (bipartite) configuration space has another special property, which we call $2$-uniformity, and we argue that $2$-uniformity is what makes the IMFIS process well-behaved.
	
	\subsubsection{Regular independent sets and the IMIS process}\label{subsection_independent_regular}
	
	We use the following notation for the neighborhood of a vertex $v$ or a set of vertices $U$ in a graph $G=(V,E)$:
	\begin{align}
		N(v)&:=N_G(v):=\{w\in V:\{v,w\}\in E\},&N(U)&:=N_G(U):=\cup_{u\in U}N(u),\\
		\overline{N}(v)&:=\overline{N}_G(v):=N(v)\cup\{v\},&\overline{N}(U)&:=\overline{N}_G(U):=\cup_{u\in U}\overline{N}(u).
	\end{align}
	We consider the following uniform iterative stochastic process for sampling maximal independent sets in a graph.
	
	\begin{definition}[Iterative maximal independent set (IMIS) process]
		Consider a graph $G=(V,E)$. Start with the empty set $\boldsymbol{S}_0=\emptyset$. For $r=0$ onward, if $V\setminus\overline{N}(\boldsymbol{S}_r)$ is empty, let $\boldsymbol{S}_{r+1}=\boldsymbol{S}_r$. Otherwise, choose $v\in V\setminus\overline{N}(\boldsymbol{S}_r)$ uniformly at random, and let $\boldsymbol{S}_{r+1}=\boldsymbol{S}_r\cup\{v\}$. Finally, define $\boldsymbol{S}_{\infty}=\cup_{r=0}^{\infty}\boldsymbol{S}_r$.
	\end{definition}
	
	By induction, each set $\boldsymbol{S}_r$ is an independent set. Furthermore, the set $\boldsymbol{S}_r$ gains exactly one vertex at each step, as long as $V\setminus\overline{N}(\boldsymbol{S}_r)$ is non-empty, so that $|\boldsymbol{S}_r|=r$ until $\boldsymbol{S}_r$ is a maximal independent set. Once this occurs, we have $\boldsymbol{S}_{r'}=\boldsymbol{S}_r$ for all $r'>r$, and therefore $\boldsymbol{S}_{\infty}=\boldsymbol{S}_r$. Thus, the set $\boldsymbol{S}_{\infty}$ is guaranteed to be a maximal independent set.
	
	In the (bipartite) configuration space, the set $\boldsymbol{S}_{\infty}$ is in fact a uniformly random MIS. Indeed, note that any sequence $(\boldsymbol{S}_r)_{r=0}^{\infty}$ with $|\boldsymbol{S}_{\infty}|=k$ appears with probability $$\prod_{r=0}^{k-1}\frac{1}{|V\setminus \overline{N}(\boldsymbol{S}_r)|}.$$ In the (bipartite) configuration space, we can show that the sets $V\setminus\overline{N}(\boldsymbol{S}_r)$ always have the same size for a fixed value of $r$, such that all sequences are equally likely and the size of $\boldsymbol{S}_{\infty}$ is deterministic. All maximal independent sets $S$ are thus maximum, and there are $|S|!$ possible sequences $(\boldsymbol{S}_r)_{r=0}^{\infty}$ with $\boldsymbol{S}_{\infty}=S$. Crucially, the (bipartite) configuration space satisfies the following definition.
	
	\begin{definition}[Regular independent sets]
		A graph $G$ has \emph{regular independent sets} if any two independent sets $S$ and $T$ of $G$ with $|S|=|T|$ satisfy $|V\setminus\overline{N}(S)|=|V\setminus\overline{N}(T)|$, or equivalently, $|N(S)|=|N(T)|$.
	\end{definition}
	
	The previous discussion thus gives the following proposition.
	
	\begin{proposition}\label{proposition_imis}
		Let $G$ be a graph with regular independent sets. Then each possible sequence $(\boldsymbol{S}_r)_{r=0}^{\infty}$ for the IMIS process is equally likely. It follows that all maximal independent sets are maximum, and $\boldsymbol{S}_{\infty}$ is a uniformly random MIS of $G$, and since $\boldsymbol{S}_1$ can be any singular vertex, each vertex appears in the same number of MISs of $G$.
	\end{proposition}
	
	Finally, we briefly show that the (bipartite) configuration space indeed has regular independent sets.
	
	\begin{example}[Configuration space]
		Let $2m:=d_1+\ldots+d_n$ be the number of half-edges in $W$. Then any independent set $S$ of $G_\d$ of size $|S|=r$ corresponds to a matching on $W$ of size $r$, which covers $2r$ half-edges. The set $V_\d\setminus\overline{N}(S)$ therefore consists only of pairs formed from the remaining $2m-2r$ half-edges, so its size depends only on $r$.
	\end{example}
	
	\begin{example}[Bipartite configuration space]
		Let $m:=d_1+\ldots+d_n=d_1'+\ldots+d_{n'}'$ be the number of half-edges in $W$ and in $W'$. Then any independent set $S$ of $G_{\d,\d'}$ of size $|S|=r$ corresponds to a bipartite matching between $W$ and $W'$ of size $r$, which covers $r$ half-edges of both $W$ and $W'$. The set $V_{\d,\d'}\setminus\overline{N}(S)$ therefore consists only of pairs pairs formed from the remaining $m-r$ half-edges of $W$ and of $W'$, so its size  depends only on $r$.
	\end{example}
	
	\subsubsection{Feasible independent sets and the IMFIS process}\label{subsection_independent_feasible}
	
	As discussed in Section~\ref{subsection_independent_configuration}, a straightforward way to obtain a uniformly random simple (bipartite) graph with prescribed degrees is to apply rejection sampling to a uniformly random MIS of the (bipartite) configuration space. However, the probability of rejection can get arbitrarily close to $1$, making this strategy practically infeasible. Instead, we consider an adaptation of the IMIS process designed to avoid rejection. To do so, we introduce the following definition.
	
	\begin{definition}[Feasible independent set]\label{definition_feasible}
		Given a graph $G=(V,E)$ with a set $F\subset V$ of forbidden vertices and an equivalence relation $R\subset V^2$ of forbidden pairs of vertices, an independent set $S$ of $G$ is called \emph{feasible} (with respect to $F$ and $R$) if it is disjoint from $F$ and no two of its vertices are $R$-equivalent. The abbreviation \emph{FIS} stands for feasible independent set. An \emph{FMIS} (feasible maximum independent set) is an MIS that is feasible. An FIS is \emph{maximal} if it is not a strict subset of another FIS.
	\end{definition}
	
	In the configuration space $G_\d$, the FMISs with respect to $F=F_\d$ and $R=R_\d$ correspond exactly to those MISs that induce a simple graph. In the bipartite configuration space $G_{\d,\d'}$, the FMISs with respect to $F=\emptyset$ and $R=R_{\d,\d'}$ correspond exactly to those MISs that induce a simple bipartite graph. We therefore aim to sample from a uniform distribution of FMISs.
	
	For a graph $G=(V,E)$ with a set $F\subset V$ of forbidden vertices and an equivalence relation $R\subset V^2$ of forbidden pairs of vertices, we use the following notation for the equivalence class of a vertex $v$ or a set of vertices $U$:
	\begin{align}
		R(v)&:=\{w\in V:vRw\},
		\\R^*(v)&:=R(v)\setminus(\overline{N}(v)\cup F),
		\\R^*(U)&:=\cup_{u\in U}R^*(u).
	\end{align}
	Note that $R(v)$ is the usual equivalence class, but in our setting we are primarily interested in $R^*(v)$. This is because the set $R^*(v)$ consists of those vertices which cannot appear together with $v$ in an FIS \emph{due to} their $R$-equivalence. We now consider the following adaptation of the IMIS process.
	
	\begin{definition}[Iterative maximal feasible independent set (IMFIS) process]\label{definition_imfis}
		Consider a tuple $(G,F,R,\b)$, where $G=(V,E)$ is a graph with $F\subset V$ a set of forbidden vertices, $R\subset V^2$ an equivalence relation of forbidden pairs of vertices, and $\b:V\to\mathbb{R}$ a weight function.
		
		Start with the empty set $\boldsymbol{S}_0=\emptyset$. For $r=0$ onward, if $(V\setminus\overline{N}(\boldsymbol{S}_r))\setminus(F\cup R^*(\boldsymbol{S}_r))$ is empty, let $\boldsymbol{S}_{r+1}=\boldsymbol{S}_r$. Otherwise, choose $v\in(V\setminus\overline{N}(\boldsymbol{S}_r))\setminus(F\cup R^*(\boldsymbol{S}_r))$ at random with probability proportional to $e^{-\b(v)}$, and let $\boldsymbol{S}_{r+1}=\boldsymbol{S}_r\cup\{v\}$. Finally, define $\boldsymbol{S}_{\infty}=\cup_{r=0}^{\infty}\boldsymbol{S}_r$.
	\end{definition}
	
	By induction, each set $\boldsymbol{S}_r$ is an FIS. Furthermore, the set $\boldsymbol{S}_r$ gains exactly one vertex at each step, as long as $(V\setminus\overline{N}(\boldsymbol{S}_r))\setminus(F\cup R^*(\boldsymbol{S}_r))$ is non-empty, so that $|\boldsymbol{S}_r|=r$ until $\boldsymbol{S}_r$ is a maximal FIS. Once this occurs, we have $\boldsymbol{S}_{r'}=\boldsymbol{S}_r$ for all $r'>r$, and therefore $\boldsymbol{S}_{\infty}=\boldsymbol{S}_r$. Thus, the set $\boldsymbol{S}_{\infty}$ is guaranteed to be a maximal FIS.
	
	Not all maximal FISs are FMISs. Specifically, if $V\setminus\overline{N}(\boldsymbol{S}_r)$ is non-empty but contained in $F\cup R^*(\boldsymbol{S}_r)$, then $\boldsymbol{S}_r$ is a maximal FIS, but not an MIS. Hence, the process does not always lead to an FMIS $\boldsymbol{S}_{\infty}$, and we still need to perform a rejection step. Nevertheless, the probability of rejection now vanishes under certain conditions, as we show as part of our main theorem, Theorem~\ref{theorem_main}~(\ref{theorem_main_termination}).
	
	To motivate the introduction of the weight function $\b$, first consider the case when $\b(v)=0$ for all vertices $v$. In this situation, the IMFIS process has a bias towards FMISs with vertices from large $R$-equivalence classes. Choosing a vertex from a large $R$-equivalence class eliminates many other admissible vertices from future consideration. This alters the number of sequences leading to different FMISs, thereby introducing a combinatorial bias. In Corollary~\ref{corollary_uniform}, we show that this bias can be sufficiently neutralized by choosing an appropriate weight function $\b$, and that this results in an almost uniform distribution.
	
	One might also expect, by a similar reasoning, that the IMFIS process would be biased toward FMISs with vertices adjacent to many forbidden vertices. However, it turns out that both the configuration space and the bipartite configuration space possess another special property, beyond having regular independent sets, that prevents such a bias. Specifically, for every MIS $S$, we can show that every vertex outside $S$ is adjacent to exactly two vertices of $S$. The total bias contributed by forbidden vertices adjacent to an FMIS is therefore the same for every FMIS, so it has a negligible effect on the overall distribution.
	
	\begin{definition}[$2$-Uniform graph]\label{definition_uniform}
		A \emph{$2$-uniform graph} is a non-edgeless graph $G$ with regular independent sets, such that, for any MIS $S$ of $G$ and any vertex $v\notin S$, the vertex $v$ is adjacent to exactly two vertices of $S$.
	\end{definition}
	
	It turns out that the notion of a $2$-uniform graph captures exactly the property required to analyze the sampling distribution of the IMFIS process. Our main theorem on the IMFIS process is thus stated in terms of $2$-uniform graphs. The following examples show that the configuration space and the bipartite configuration space are indeed $2$-uniform if they are non-edgeless, which thus means that our main theorem, Theorem~\ref{theorem_main}, applies to both cases.
	
	\begin{example}[(Bipartite) configuration space]
		An MIS $S$ corresponds to a perfect (bipartite) matching of half-edges. A vertex $v\notin S$ thus corresponds to a pair of half-edges from two different matched pairs of half-edges, which are the two vertices of $S$ adjacent to $v$.
	\end{example}
	
	\subsection{Properties of 2-uniform graphs}\label{subsection_independent_properties}
	
	For our main theorem, it is helpful to have some basic understanding of $2$-uniform graphs. We let $\alpha(G)$ denote the independence number of any graph $G$, which is the size of any MIS. We furthermore introduce a parameter $\ell(G)$ to help quantify certain properties of any $2$-uniform graph $G$. In particular, we show that the total number of MISs of $G$ can be calculated from $\alpha(G)$ and $\ell(G)$. It turns out that the parameter $\ell(G)$ captures the difference between the configuration space and the bipartite configuration space. Finally, we classify all $2$-uniform graphs, and we use this classification to argue that the property of $2$-uniformity captures the similarity between the configuration models for undirected graphs and for bipartite graphs.
	
	\subsubsection{Parameters for $2$-uniform graphs}\label{subsection_independent_parameter}
	
	We introduce the parameter $\ell(G)$ for a $2$-uniform graph $G$, which involves discussing some other parameters as well. For $k\leq\alpha(G)$, let $d_k(G)$ denote the size of $N(S)$ for any independent set $S$ of size $|S|=k$. This is well-defined, because $G$ has regular independent sets. Note that $G$ is regular with degree $d(G):=d_1(G)$, since any singular vertex forms an independent set of size $1$. The parameter $\ell(G)$ is given by the following proposition.
	
	\begin{proposition}\label{proposition_ell}
		Consider a $2$-uniform graph $G$. Then $\alpha(G)\geq2$. We may thus define $\ell(G):=2d_1(G)-d_2(G)$. Then for any two non-adjacent vertices $v$ and $w$, we have $|N(v)\cap N(w)|=\ell(G)$. It follows that $\ell(G)\geq1$.
	\end{proposition}
	
	\begin{proof}
		Since $G$ is non-edgeless, for any MIS $S$, there exists a vertex outside $S$. Such a vertex must be adjacent to exactly two vertices of $S$, so $S$ must have at least two vertices, giving $\alpha(G)\geq2$.
		
		By the inclusion-exclusion principle, if $\{v,w\}$ is an independent set, and hence $|N(\{v,w\})| = d_2(G)$, we have $$|N(v)\cap N(w)|=|N(v)|+|N(w)|-|N(\{v,w\})|=2d_1(G)-d_2(G)=\ell(G).$$
		
		Since $G$ is non-edgeless, for any MIS $S$, there exists a vertex outside $S$. Such a vertex must be adjacent to exactly two vertices $v,w\in S$, which gives $\ell(G)=|N(v)\cap N(w)|\geq1$.
	\end{proof}
	
	Note that the value $\ell(G)$ captures the difference between the configuration space and the bipartite configuration space. We have $\ell(G_\d)=4$, since there are $2\times 2=4$ ways to pair up a half-edge from one pair with a half-edge from another disjoint pair. However, we have $\ell(G_{\d,\d'})=2$, since the restriction that pairs are between $W$ and $W'$ excludes two out of the four options. See Figure~\ref{fig:the_figure}.
	\begin{figure}
		\centering
		\begin{minipage}{0.45\textwidth}
			\centering
			\includegraphics[width=0.5\textwidth]{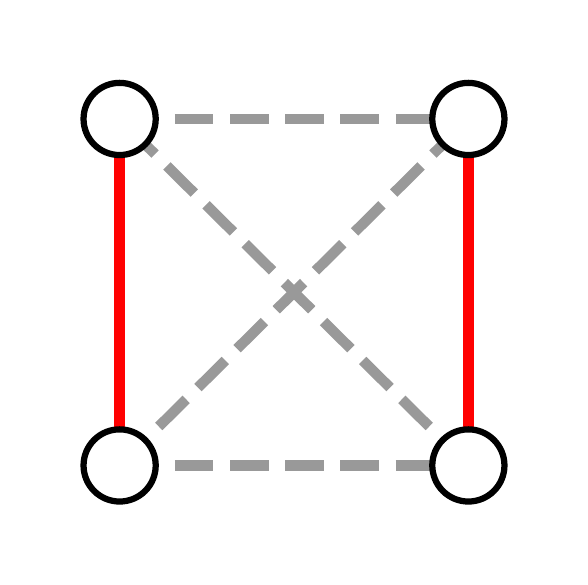}\includegraphics[width=0.5\textwidth]{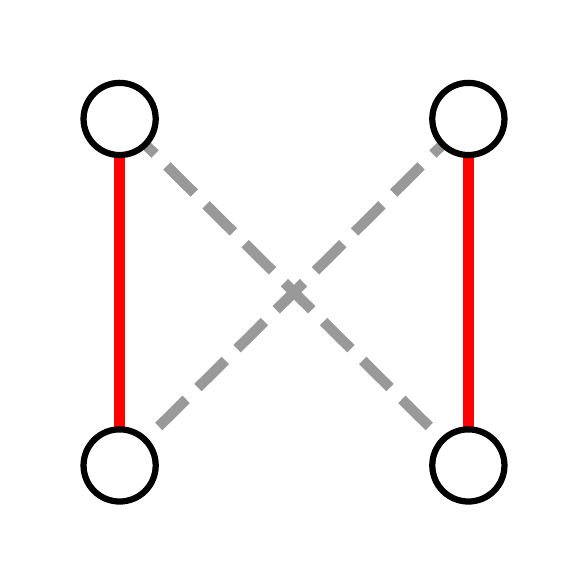}
		\end{minipage}
		\caption{The left panel shows $\ell(G_{\d})=4$, and the right panel shows $\ell(G_{\d,\d'})=2$. The red solid lines represent non-adjacent vertices $v$ and $w$ as disjoint pairs of half-edges, and the grey dashed lines represent the pairs of half-edges in $N(v)\cap N(w)$.}\label{fig:the_figure}
	\end{figure}
	
	Finally, we note that the number of vertices of $G$ and the degree $d(G)$ of any vertex can be expressed in terms of $\alpha(G)$ and $\ell(G)$.
	
	\begin{proposition}\label{proposition_n}
		Any $2$-uniform graph $G$ has $\ell(G)\binom{\alpha(G)}{2}+\alpha(G)$ vertices and $d(G)=\ell(G)(\alpha(G)-1)$.
	\end{proposition}
	
	\begin{proof}
		Consider an MIS $S$ of $G$, such that $|S|=\alpha(G)$. All vertices outside $S$ are adjacent to exactly two vertices of $S$, and any pair of vertices $v,w\in S$ has $|N(v)\cap N(w)|=\ell(G)$. The total number of vertices outside $S$ is thus $\ell(G)$ times the number of pairs of vertices in $S$. The expression for $d(G)$ furthermore follows, since any vertex in $S$ is included in $\alpha(G)-1$ pairs of vertices in $S$.
	\end{proof}
	
	\subsubsection{The hereditary nature of $2$-uniformity}\label{subsection_independent_hereditary}
	
	We give a formula for the number of MISs in a $2$-uniform graph, and we show that $\ell(G)$ is even, and thus at least $2$. We obtain these results by studying the hereditary nature of $2$-uniformity.
	
	\begin{proposition}\label{proposition_hereditary}
		Let $G=(V,E)$ be a $2$-uniform graph, and let $S$ be an independent set of size $r$ with $\alpha(G)-r\geq2$. Then the induced graph $H$ on $V\setminus\overline{N}(S)$ is another $2$-uniform graph, with $\alpha(H)=\alpha(G)-r$ and $\ell(H)=\ell(G)$.
	\end{proposition}
	
	\begin{proof}
		For any independent set $T$ of $H$, we have $$|N_H(T)|=|N_G(T\cup S)|-|N_G(S)|=d_{|T|+|S|}(G)-d_{|S|}(G),$$ since $T\cup S$ and $S$ are independent sets. Since this depends only on the size of $T$, we find that $H$ has regular independent sets.
		
		Let $T$ be an MIS of $H$. Note that $T\cup S$ is then an MIS of $G$, so $\alpha(H)=|T|=\alpha(G)-r$. Since $G$ is $2$-uniform, any vertex in $H$ outside $T$ has exactly two neighbors in $T\cup S$. By definition of $H$, neither neighbor is in $S$, so both neighbors are in $T$.
		
		Finally, since $\alpha(G)-r\geq2$, there exist non-adjacent vertices $v,w$ in $H$. Since $G$ is $2$-uniform, any vertex in $G$ outside $S\cup\{v,w\}$ is adjacent to at most two vertices of $S\cup{v,w}$. Therefore, any vertex in $N_G(v)\cap N_G(w)$ cannot be adjacent to $S$, and hence must be in $H$. We find that $|N_H(v)\cap N_H(w)|=\ell(G)$ is non-zero, so $H$ is $2$-uniform with $\ell(H)=\ell(G)$.
	\end{proof}
	
	\begin{corollary}\label{corollary_hereditary}
		Let $G=(V,E)$ be a $2$-uniform graph, and let $S$ be an independent set of size $r$. Then the induced graph $H$ on $V\setminus\overline{N}(S)$ has $\ell(G)\binom{\alpha(G)-r}{2}+\alpha(G)-r$ vertices. Furthermore, if $\alpha(G)-r\geq1$, then $H$ is regular of degree $\ell(G)(\alpha(G)-r-1)$.
	\end{corollary}
	
	\begin{proof}
		For $\alpha(G)-r\geq2$, this follows from Proposition~\ref{proposition_hereditary} and Proposition~\ref{proposition_n}. For $\alpha(G)-r=0$, we find that $S$ is a maximum independent set, so $H$ has no vertices.
		
		For $\alpha(G)-r=1$, there exists $v\in V$ such that $S\cup\{v\}$ is an MIS. Every vertex outside $S\cup\{v\}$ is thus adjacent to exactly two elements of $S\cup\{v\}$, so, in particular, at least one element of $S$. We find that $H$ consists only of vertex $v$.
	\end{proof}
	
	\begin{corollary}\label{corollary_mis_count}
		The total number of MISs in a $2$-uniform graph $G$ is given by the formula $$\frac1{\alpha(G)!}\prod_{t=1}^{\alpha(G)}\left(\ell(G)\binom{t}{2}+t\right)=\prod_{t=1}^{\alpha(G)}\left(\frac{\ell(G)}{2}(t-1)+1\right).$$ In particular, if $\ell(G)=2$, this simplifies to $\alpha(G)!$, and if $\ell(G)=4$, this simplifies to $(2\alpha(G)-1)!!=\tfrac{(2\alpha(G))!}{2^{\alpha(G)}\alpha(G)!}$.
	\end{corollary}
	
	\begin{proof}
		In the IMIS process, for $r=0,\ldots,\alpha(G)-1$, there are $|V\setminus\overline{N}(\boldsymbol{S}_{r})|=\ell(G)\binom{\alpha(G)-r}{2}+\alpha(G)-r$ possible vertices to be added to $\boldsymbol{S}_{r+1}$. Since each possible MIS $S$ has $\alpha(G)!$ possible sequences $(\boldsymbol{S}_{r})_{r=0}^{\infty}$ resulting in $\boldsymbol{S}_{\infty}=S$, the result follows by substituting $t=\alpha(G)-r$.
	\end{proof}
	
	Finally, to show that $\ell(G)$ is even, we study the intersection of neighborhoods $N(v)\cap N(w)$ for non-adjacent vertices $v,w$.
	
	\begin{proposition}\label{proposition_even}
		Let $G$ be a $2$-uniform graph. Let $v$ and $w$ be non-adjacent vertices. Then the induced graph on $(N(v)\cap N(w))\cup\{v,w\}$ is the edgewise complement of a set of disjoint pairs of adjacent vertices. It follows that $\ell(G)$ is even, and thus at least $2$.
	\end{proposition}
	
	\begin{proof}
		If $S$ is an MIS containing $v$ and $w$, then, by Corollary~\ref{corollary_hereditary}, the induced graph $H$ on $(N(v)\cap N(w))\cup\{v,w\}=V\setminus\overline{N}(S\setminus\{v,w\})$ has $\ell(G)+2$ vertices and $d(H)=\ell(G)$. Hence, the edgewise complement of $H$ is regular of degree $1$. The non-adjacencies of $H$ thus give a pairing of the vertices of $H$, so there must be an even number of such vertices.
	\end{proof}
	
	\subsubsection{Classification of $2$-uniform graphs}
	
	We give a complete classification of $2$-uniform graphs. We first give the descriptions of all $2$-uniform graphs other than the (bipartite) configuration space. To start off, we have the following infinite family.
	
	\begin{example}[Complement pairing]
		For $k\geq2$, consider the complete $k$-partite graph $K_{k\times2}$, which is the edge-wise complement of $k$ disjoint pairs of adjacent vertices, which we refer to as the $k$ parts. We find that any independent set $S$ is a subset of a single part. Then all their elements have the same neighborhood, which is all vertices outside the part. We find that $K_{k\times2}$ is $2$-uniform with $\alpha(K_{k\times2})=2$ and $\ell(K_{k\times2})=2k-2$.
	\end{example}
	
	Next, we have the following sporadic example.
	
	\begin{example}[Schl\"afli graph]
		The Schl\"afli graph was originally defined as the edge-wise complement of the intersection graph of the $27$ lines on a cubic surface. It is also the edge-wise complement of the collinearity graph of the unique generalized quadrangle $\GQ(2,4)$.
		
		An elementary construction due to~\cite{Bussemaker1992ExceptionalGW} is as follows. Let $e_1,\ldots,e_8$ denote the standard basis of $\mathbb{R}^8$, and define $c:=(\frac12,\ldots,\frac12)\in\mathbb{R}^8$. Consider the vertex set $$V:=\{e_i+e_j:i\leq6,j\geq7\}\cup\{c-e_i-e_j:i<j\leq6\},$$ and connect any two vertices by an edge if they have inner product $1$. It can be computationally verified that the resulting graph $G$ is $2$-uniform with $\alpha(G)=3$ and $\ell(G)=8$.
	\end{example}
	
	We now give the complete classification result on $2$-uniform graphs.
	
	\begin{theorem}\label{theorem_classification}
		Any $2$-uniform graph is isomorphic to either a configuration space, a bipartite configuration space, a complete $k$-partite graph $K_{k\times 2}$, or the Schl\"afli graph.
	\end{theorem}
	
	\begin{proof}
		Let $G$ be $2$-uniform. Since $G$ is non-edgeless, we have $\alpha(G)\geq2$. If $\alpha(G)=2$, then $G$ is isomorphic to $K_{(\frac12\ell(G)+1)\times2}$ by Proposition~\ref{proposition_ell}.
		
		For $\alpha(G)\geq3$, the result follows from the classification given in~\cite[Section~7, Part~A]{ZARA}, in particular, see~\cite[Remark~7.7]{ZARA}. In the notations of this work, the edge-wise complement of $G$ satisfies axiom A1 with $r=\alpha(G)$ and axiom A2 with $t=r-2$.
	\end{proof}
	
	Note that, up to isomorphism, the configuration space $G_{\d}$ is determined by the sum $d_1+\ldots+d_n$, and similarly for the bipartite configuration space. So there are three singularly parameterized infinite families of $2$-uniform graphs and one sporadic $2$-uniform graph. A $2$-uniform graph $G$ is a configuration space if $\ell(G)=4$, a bipartite configuration space if $\ell(G)=2$, a complement pairing if $\alpha(G)=2$, or the Schl\"afli graph if $\ell(G)=8$ and $\alpha(G)=3$. So the parameters $\alpha(G)$ and $\ell(G)$ completely determine any $2$-uniform graph $G$. Note that there is overlap between the infinite families where they are isomorphic.
	
	Our main theorem is stated asymptotically in terms of a $2$-uniform graph $G$ with $\ell(G)$ fixed, and with $\alpha(G)\to\infty$. It follows from the classification that this is only possible if $\ell(G)=4$, giving the configuration space, or if $\ell(G)=2$, giving the bipartite configuration space. This suggests that the property of $2$-uniformity indeed captures the similarity between the configuration space and the bipartite configuration space, and the parameter $\ell(G)$ captures their difference.

	\section{Main Results}\label{section_main}

	We analyze the sampling distribution of the IMFIS process from Definition \ref{definition_imfis} from Section~\ref{subsection_independent_feasible} on $(G,F,R,\b)$. We state our main result, which gives an asymptotic formula for the outcome distribution of the IMFIS process when $G$ is a $2$-uniform graph, as defined in Definition \ref{definition_uniform}. We also give a corollary that describes when this distribution is asymptotically uniform. We also state a secondary theorem, which gives the probability that the IMFIS process results in an FMIS $\boldsymbol{S}_{\infty}$ that does or does not contain certain specified vertices. This is particularly useful for our application to sampling hypergraphs, where we will use it to calculate the probability of rejection in Section~\ref{subsection_applications_hyper} and Section~\ref{subsection_applications_dihy}. Finally, we have a proposition on the possibility of altering the weight function $\b$ without affecting the outcome distribution of the IMFIS process asymptotically. This will be used throughout Section~\ref{section_applications} to simplify the calculations. All results in this section are proven in Section~\ref{section_proof}.
	
	Throughout this section, we assume that $G$ is a $2$-uniform graph and we write $G=(V,E)$. We also write $\ell=\ell(G)$ and $\alpha=\alpha(G)$. In order for the IMFIS process to result in an almost uniform distribution of FMISs, we need to bound the quantities
	\begin{align}
		F_{\max}&:=\max\{|\overline{N}(v)\cap F|:v\in V\},
		\\R_{\max}^{(a)}&:=\max\{|R^*(v)|:v\in V\},
		\\R_{\max}^{(b)}&:=\max\{|\overline{N}(v)\cap R^*(S)|:v\in V,S\text{ an independent set of }G\},
		\\R_{\max}^{(c)}&:=\max\{|\overline{N}(v)\cap R^*(w)|:v,w\in V\},
		\\\b_{\max}&:=\max\{|\b(v)|:v\in V\},
		\\M&:=\max\{F_{\max},R_{\max}^{(a)},R_{\max}^{(b)},(R_{\max}^{(c)})^2,\b_{\max}\alpha,1\}.
	\end{align}
	For a non-negative integer $t$, we write $n_t:=\ell\binom{t}{2}+t$. Finally, we write $P:=(\tfrac{\ell}{2}\alpha)^{-1}$ and $\b(U):=\sum_{u\in U}\b(u)$ for $U\subset V$. We can now state our main theorem.
	
	\begin{theorem}\label{theorem_main}
		Assume that $\alpha\to\infty$, with $\ell$ fixed and $M=O(\alpha^{\sfrac12}/(\log\alpha)^2)$.
		\begin{enumerate}[(i)]
			\item\label{theorem_main_equation} Consider an FMIS $S$ of $G$. Define 
			\begin{align}
				\mathscr{P}(S)&:=\exp(P|F|+\tfrac12P|R^*(S)|+P\b(V)-\b(S))\prod_{t=1}^{\alpha}\frac{t}{n_t},\\
				\mathscr{E}(\alpha,M)&:=\frac{M^2\log\alpha}{\alpha}+\frac{M(\log\alpha)^2}{\alpha}.
			\end{align}
			Then the IMFIS process samples $S$ with probability
			\begin{align}
				\mathbb{P}(\boldsymbol{S}_{\infty}=S)&=(1+O(\mathscr{E}(\alpha,M)))\mathscr{P}(S).
			\end{align}
			\item\label{theorem_main_termination} The IMFIS process samples an FMIS with probability $$\mathbb{P}(|\boldsymbol{S}_{\infty}|=\alpha)=1-O(M/\alpha).$$ In particular, for $\alpha$ sufficiently large, this probability is positive, so $G$ has at least one FMIS.
		\end{enumerate}
	\end{theorem}
	
	\begin{remark}
		\begin{enumerate}[(i)]
			\item If rejection sampling is applied to the IMFIS process to sample an FMIS $\boldsymbol{S}_{\infty}$, then it follows from Theorem~\ref{theorem_main} that the probability of rejection is $O(M/\alpha)$, and any FMIS $S$ of $G$ is sampled with probability $$\mathbb{P}(\boldsymbol{S}_{\infty}=S\,|\,|\boldsymbol{S}_{\infty}|=\alpha)=(1+O(\mathscr{E}(\alpha,M)))\mathscr{P}(S).$$
			\item The assumption that $\alpha\to\infty$ with $\ell$ fixed implies that either $\ell=4$ and $G$ is a configuration space, or $\ell=2$ and $G$ is a bipartite configuration space, as shown in Theorem~\ref{theorem_classification}.
			\item We emphasize for clarity that all the asymptotics are in terms of $\alpha\to\infty$, and this does not rely on some underlying sequence of objects that depend on $\alpha$. See Remark \ref{remark_asymptotics}.
			\item The notation $\mathscr{E}(\alpha,M)$ is consistent with the notation $\mathscr{E}(m,\Delta)$ from Equation \eqref{equation_script_e} in Section~\ref{section_results}.
			\item The value $\prod_{t=1}^{\alpha}\frac{t}{n_t}$ is the reciprocal of the total number of MISs of $G$, as shown in Corollary~\ref{corollary_mis_count}.
		\end{enumerate}
	\end{remark}
	
	The following corollary describes when the outcome distribution of the IMFIS process is asymptotically uniform.
	
	\begin{corollary}\label{corollary_uniform}
		Assume that $\b(v)=\frac12P|R^*(v)|$ for all $v\in V$. Then $\b_{\max}\alpha\leq\frac1{\ell}R_{\max}^{(a)}$, and therefore $$M=\max\{F_{\max},R_{\max}^{(a)},R_{\max}^{(b)},(R_{\max}^{(c)})^2,1\}.$$ Assume that $\alpha\to\infty$, with $\ell$ fixed and $M=O(\alpha^{\sfrac12}/(\log\alpha)^2)$. For any FMIS $S$ of $G$, we have $$\mathscr{P}(S)=\overline{\mathscr{P}}:=\exp\left(P|F|+\tfrac12P^2\sum_{v\in V}|R^*(v)|\right)\prod_{t=1}^{\alpha}\frac{t}{n_t},$$ such that $\mathbb{P}(\boldsymbol{S}_{\infty}=S\,|\,|\boldsymbol{S}_{\infty}|=\alpha)$ is asymptotically uniform. Hence, the total number of FMISs of $G$ is $(1+O(\mathscr{E}(\alpha,M)))\overline{\mathscr{P}}^{-1}$.
	\end{corollary}
	
	In Section~\ref{subsection_proof_distribution}, we will use our main theorem, Theorem~\ref{theorem_main}, to calculate the probability that the IMFIS process results in an FMIS that does or does not contain certain specified vertices. This yields the following secondary result.
	
	\begin{theorem}\label{theorem_space}
		Let $T$ be an FIS of $G$ and let $V':=V\setminus\overline{N}(T)$, $\alpha':=\alpha-|T|$, and $P':=(\tfrac{\ell}{2}\alpha')^{-1}$. Let $U\subset V'\setminus(F\cup R^*(T))$ with $U_{\max}:=\max\{|\overline{N}(v)\cap U|:v\in V'\}$, and define $M':=\max\{M,U_{\max},\tfrac{\alpha}{\alpha'}R_{\max}^{(a)}\}$. Finally, define
		\begin{align}
			\mathscr{Q}(T,U)&:=\exp\Bigg(\tfrac12P'\sum_{v\in V'}\left(P|R^*(v)|-P'|R^*(v)\cap V'|\right)\Bigg.
			\\&\quad+P|F|-P'|F\cap V'|+P\b(V)-P'\b(V')
			\\&\quad\Bigg.+\tfrac12P|R^*(T)|-P'|R^*(T)\cap V'|-\b(T)-P'|U|\Bigg)\prod_{t=\alpha'+1}^{\alpha}\frac{t}{n_t}.
		\end{align}
		Assume that $\alpha'\to\infty$, with $\ell$ fixed and $M'=O((\alpha')^{\sfrac12}/(\log\alpha')^2)$. Then the IMFIS process samples an FMIS that contains the \emph{target set} $T$ and is disjoint from the \emph{exclusion set} $U$ with probability $$\mathbb{P}(|\boldsymbol{S}_{\infty}|=\alpha\wedge T\subset\boldsymbol{S}_{\infty}\wedge U\cap\boldsymbol{S}_{\infty}=\emptyset)=(1+O(\mathscr{E}(\alpha',M')))\mathscr{Q}(T,U).$$
	\end{theorem}
	
	Note that, if rejection sampling is applied to the IMFIS process to sample an FMIS $\boldsymbol{S}_{\infty}$, then it follows from Theorem~\ref{theorem_main}~(\ref{theorem_main_termination}) that the probability of rejection is $O(M/\alpha)$, so it follows by Theorem~\ref{theorem_space} that the resulting FMIS contains $T$ and is disjoint from $U$ with probability $$\mathbb{P}(T\subset\boldsymbol{S}_{\infty}\wedge U\cap\boldsymbol{S}_{\infty}=\emptyset\,|\,|\boldsymbol{S}_{\infty}|=\alpha)=(1+O(\mathscr{E}(\alpha',M')))\mathscr{Q}(T,U).$$ The following proposition, which is also proven in Section~\ref{subsection_proof_distribution}, can be used to simplify the expression $\mathscr{Q}(T,U)$ from Theorem~\ref{theorem_space}.
	
	\begin{proposition}\label{proposition_qt}
		Consider the quantity $\mathscr{Q}(T,U)$ from Theorem~\ref{theorem_space}.
		\begin{enumerate}[(i)]
			\item\label{proposition_qt_small} We have
			\begin{align}
				\mathscr{Q}(T,U)&=\left(\frac{2}{\ell}\right)^{|T|}\frac{(\alpha')!}{\alpha!}\exp\left(-P'|U|+O\left(\frac{|T|M'}{\alpha'}\right)\right)
				\\&=P^{|T|}\exp\left(-P|U|+O\left(\frac{|T|^2+|T|M'}{\alpha'}\right)\right).
			\end{align}
			\item\label{proposition_qt_uniform} If $\b(v)=\frac12P|R^*(v)|$ for all $v\in V$, then
			\begin{align}
				\mathscr{Q}(T,U)&=\exp\Bigg(P|F|-P'|F\cap V'|-P'|R^*(T)\cap V'|-P'|U|\Bigg.
				\\&\quad\Bigg.+\tfrac12P^2\sum_{v\in V}|R^*(v)|-\tfrac12(P')^2\sum_{v\in V'}|R^*(v)\cap V'|\Bigg)\prod_{t=\alpha'+1}^{\alpha}\frac{t}{n_t}.
			\end{align}
			Therefore, a uniformly random FMIS of $G$ contains $T$ and is disjoint from $U$ with probability $(1+O(\mathscr{E}(\alpha',M')))\mathscr{Q}(T,U)$.
		\end{enumerate}
	\end{proposition}
	
	Finally, since we study the distribution of $\boldsymbol{S}_{\infty}$ only up to a factor of $1+O(\mathscr{E}(\alpha,M))$, it is possible to slightly alter the weight function $\b$ without affecting the outcome distribution of the IMFIS process asymptotically. The following proposition, which is also proven in Section~\ref{subsection_proof_distribution}, gives three conditions for such an alteration. The first condition essentially states that the weights of forbidden vertices are negligible, which is to be expected, since they are ignored by the IMFIS process. The second condition indicates that the weight $e^{-\b(v)}$ in the IMFIS process may be replaced by the weight $1-\b(v)$. Finally, the third condition is essentially an algebraic trick, which will be used throughout Section~\ref{section_applications} to simplify calculations.
	
	\begin{proposition}\label{proposition_leniency}
		Assume that $\alpha\to\infty$, with $\ell$ fixed and $M=O(\alpha^{\sfrac12}/(\log\alpha)^2)$. Consider the following three conditions on an alternative weight function $\b':V\to\mathbb{R}$ with $\b'_{\max}:=\max\{|\b'(v)|:v\in V\}=O(M/\alpha)$.
		\begin{enumerate}[(i)]
			\item\label{proposition_leniency_f} We have $\b'(v)=\b(v)$ for all $v\in V\setminus F$.
			\item\label{proposition_leniency_log} We have $\b'(v)=-\log(1-\b(v))$ for all $v\in V$.
			\item\label{proposition_leniency_b} The quantity $\b(S)-\b'(S)$ is the same for all MISs $S$ of $G$. 
		\end{enumerate}
		If $\b'$ satisfies any of these conditions, then for any FMIS $S$ of $G$, substituting $\b'$ for $\b$ in the definition of $\mathscr{P}(S)$ multiplies its value by a factor of $1+O(M^2/\alpha)$. It follows that the IMFIS process on $(G,F,R,\b')$ gives the same distribution on $\boldsymbol{S}_{\infty}$ as the IMFIS process on $(G,F,R,\b)$, up to a factor of $1+O(\mathscr{E}(\alpha,M))$.
		
		Furthermore, for any FIS $T$ of $G$ and vertex set $U$ satisfying the requirements of Theorem~\ref{theorem_space}, substituting $\b'$ for $\b$ in the definition of $\mathscr{Q}(T,U)$ multiplies its value by a factor of $1+O(\mathscr{E}(\alpha',M'))$.
	\end{proposition}

	\section[Applying the Main Results to Sampling Graphs]{\hspace{-0.35cm}Applying the Main Results to Sampling Graphs}\label{section_applications}

	We apply the results from Section~\ref{section_main} to prove the results from Section~\ref{section_results}. We show that each iterative stochastic process can be modeled as an IMFIS process with appropriately chosen $(G,F,R,\b)$. Therefore, we only need to verify the conditions of Theorem~\ref{theorem_main} and evaluate $\mathscr{P}(S)$ or $\overline{\mathscr{P}}$.
	
	\subsection{Undirected graphs: Proof of Theorem~\ref{theorem_simple}}\label{subsection_applications_simple}
	
	We consider the configuration space $G:=G_\d$ from Definition \ref{definition_space} from Section~\ref{subsection_independent_simple} equipped with the set of forbidden vertices $F_\d$ and the equivalence relation $R:=R_\d$, which were subsequently defined to cover the self-loops and multi-edges respectively. For a vertex $v\in V_\d$ connecting $W_i$ and $W_j$, we define the weight function $\b(v):=\tfrac{d_id_j}{4m}$. Finally, we let $X_\d$ contain all vertices in $V_\d$ that induce a forbidden edge from $X$, and we define $F:=F_\d\cup X_\d$.
	
	Consider the IMFIS process on $(G,F,R,\b)$, and apply rejection sampling on $\boldsymbol{S}_{\infty}$ to sample an FMIS $S$. As explained in Section~\ref{subsection_independent_simple}, if we identify the sets $W_1,\ldots,W_n$ with vertices $v_1,\ldots,v_n$ and we identify any vertex $v\in S$ connecting half-edges from $W_i$ and $W_j$ with an edge between $v_i$ and $v_j$, then this results in a simple graph with degree sequence $\d$ and with no edges from $X$. The main idea is that this process is equivalent to the iterative undirected graph process as defined in Definition \ref{definition_process_simple}. Indeed, the factor $(d_i-d_{i,r})(d_j-d_{j,r})$ accounts for the number of vertices in $V_{\d}\setminus\overline{N}(\boldsymbol{S}_r)$ that induce an edge between $v_i$ and $v_j$, and the factor $1-\frac{d_id_j}{4m}$ is equivalent to the factor $e^{-\b(v)}$ from the IMFIS process by Proposition~\ref{proposition_leniency}~(\ref{proposition_leniency_log}).
	
	Note that $\alpha(G)=m\to\infty$, and $\ell(G)=4$ is fixed, as shown in Section~\ref{subsection_independent_parameter}. Consider $v\in V_\d$ connecting half-edges $e\in W_i$ and $f\in W_j$. The following proves $M=O(\Delta)$.
	\begin{itemize}
		\item There are at most $d_{\max}$ pairs in $F_\d$ and at most $x_{\max}d_{\max}$ pairs in $X_\d$ that contain $e$. The same holds for $f$, so $F_{\max}\leq2d_{\max}+2x_{\max}d_{\max}$.
		\item There are at most $d_{\max}^2$ pairs in $V_\d$ that are $R$-equivalent to $v$, which gives $R_{\max}^{(a)}\leq d_{\max}^2$.
		\item For an independent set $S$ of $G$, at most $d_{\max}$ pairs in $S$ connect $W_i$ with some $W_k$, so $R^*(S)$ contains at most $d_{\max}^2$ pairs with $e$. The same holds for $f$, so $R_{\max}^{(b)}\leq2d_{\max}^2$.
		\item Any equivalence class $R^*(w)$ contains at most $d_{\max}$ pairs with $e$, and at most $d_{\max}$ with $f$, which gives $R_{\max}^{(c)}\leq2d_{\max}$.
		\item We have $\b(v)\leq\frac1{4m}d_{\max}^2$, so $\b_{\max}m\leq\frac14d_{\max}^2$.
	\end{itemize}
	Since $\Delta=O(m^{\sfrac12}/(\log m)^2)$ by assumption, we may thus apply Theorem~\ref{theorem_main}. We find that the probability of rejection is $O(\Delta/m)$, and any FMIS $S$ of $G$ is sampled with probability $(1+O(\mathscr{E}))\mathscr{P}(S)$.
	
	To calculate $\mathscr{P}(S)$, by Proposition~\ref{proposition_leniency}~(\ref{proposition_leniency_f},\ref{proposition_leniency_b}), we can substitute the weight function $\b'(v):=\frac12P|R^*(v)|$ for $\b(v)$. Indeed, for $v\in V_{\d}\setminus F_{\d}$ connecting $W_i$ with $W_j$, since $P=\frac1{2m}$, we get $$\b(v)-\b'(v)=\frac{d_id_j-(d_i-1)(d_j-1)}{4m}=\frac{d_i+d_j-1}{4m}.$$ Summing this over an MIS $T$ of $G$ gives $\frac1{4m}\sum_{i=1}^{n}d_i^2-\frac14$, which is independent of $T$, since any term $\frac{d_i}{4m}$ appears $d_i$ times in the sum, and the sum has $m$ terms in total. By Corollary~\ref{corollary_uniform}, it suffices to calculate $\overline{\mathscr{P}}$.
	
	Any simple graph $H$ with degree sequence $\d$ and with no forbidden edges from $X$ is realized by $\prod_{i=1}^{n}d_i!$ FMISs of $G$, due to all possible permutations of the half-edges within the sets $W_1,\ldots,W_n$. We find that $H$ is sampled with probability $(1+O(\mathscr{E}))\overline{\mathscr{P}}\prod_{i=1}^{n}d_i!$ for $$\overline{\mathscr{P}}=\exp\left(\frac1{2m}|F|+\frac1{8m^2}\sum_{v\in V_\d}|R^*(v)|\right)\prod_{t=1}^{m}\frac{t}{n_t}.$$ Since $|F_\d|=\sum_{i=1}^{n}\binom{d_i}{2}$, $|X_\d|=\sum_{\{v_i,v_j\}\in X}d_id_j$, and $\prod_{t=1}^{m}\frac{t}{n_t}=\frac1{(2m-1)!!}$ (see Corollary~\ref{corollary_mis_count}), it remains only to evaluate the term $\sum_{v\in V_\d}|R^*(v)|$. Only pairs $v\in V_\d$ connecting $W_i$ and $W_j$ with $i\neq j$ and $\{v_i,v_j\}\not\in X$ have non-empty $R^*(v)$, giving $|R^*(v)|=(d_i-1)(d_j-1)$. We get
	\begin{align}
		\sum_{v\in V_\d}|R^*(v)|
		&=\sum_{i=1}^{n}\sum_{j=i+1}^{n}d_id_j(d_i-1)(d_j-1)-\sum_{\{v_i,v_j\}\in X}d_id_j(d_i-1)(d_j-1)
		\\&=\frac12\left(\sum_{i=1}^{n}d_i(d_i-1)\right)^2+O\left(\sum_{i=1}^{n}d_id_{\max}^3x_{\max}\right).
	\end{align}
	We thus get $\frac1{8m^2}\sum_{v\in V_\d}|R^*(v)|=\frac14\lambda_\d^2+O(\Delta^2/m)$.\hfill$\square$
	
	\subsection{Bipartite graphs: Proof of Theorem~\ref{theorem_bipartite}}\label{subsection_applications_bipartite}
	
	We consider the bipartite configuration space $G:=G_{\d,\d'}$ from Definition \ref{definition_bispace} from Section~\ref{subsection_independent_bipartite} equipped with the equivalence relation $R:=R_{\d,\d'}$, which was subsequently defined to cover the multi-edges. For a vertex $v\in V_{\d,\d'}$ connecting $W_i$ and $W_j'$, we define the weight function $\b(v):=\tfrac{d_id_j'}{2m}$. Finally, we let $F$ contain all vertices in $V_{\d,\d'}$ that induce a forbidden edge from $X$.
	
	Consider the IMFIS process on $(G,F,R,\b)$, and apply rejection sampling on $\boldsymbol{S}_{\infty}$ to sample an FMIS $S$. As explained in Section~\ref{subsection_independent_bipartite}, if we identify the sets $W_1,\ldots,W_n,W_1',\ldots,W_{n'}'$ with vertices $v_1,\ldots,v_n,v_1',\ldots,v_{n'}'$ and we identify any vertex $v\in S$ connecting half-edges from $W_i$ and $W_j'$ with an edge between $v_i$ and $v_j'$, then this results in a simple bipartite graph with degree sequences $\d$ and $\d'$ on its parts and with no edges from $X$. By the same reasoning as in Section~\ref{subsection_applications_simple}, this process is equivalent to the iterative bipartite graph process.
	
	Note that $\alpha(G)=m\to\infty$, and $\ell(G)=2$ is fixed, as shown in Section~\ref{subsection_independent_parameter}. By similar reasoning as in Section~\ref{subsection_applications_simple}, we get $F_{\max}\leq2x_{\max}d_{\max}$, $R_{\max}^{(a)}\leq d_{\max}^2$, $R_{\max}^{(b)}\leq2d_{\max}^2$, $R_{\max}^{(c)}\leq2d_{\max}$, and $\b_{\max}m\leq\frac12d_{\max}^2$, which shows that $M=O(\Delta)$ holds, so we can apply Theorem~\ref{theorem_main} again. We find that the probability of rejection is $O(\Delta/m)$, and any FMIS $S$ of $G$ is sampled with probability $(1+O(\mathscr{E}))\mathscr{P}(S)$. By similar reasoning as in Section~\ref{subsection_applications_simple} again, to calculate $\mathscr{P}(S)$, we can substitute the weight function $\b'(v):=\frac12P|R^*(v)|$ for $\b(v)$, and it thus suffices to calculate $\overline{\mathscr{P}}$.
	
	Any simple bipartite graph $H$ with degree sequences $\d$ and $\d'$ on its parts and with no edges from $X$ is realized by $\prod_{i=1}^{n}d_i!\prod_{j=1}^{n'}d_j'!$ FMISs of $G$, due to all possible permutations of the half-edges within the sets $W_1,\ldots,W_n,W_1',\ldots,W_{n'}'$. We find that $H$ is sampled with probability $(1+O(\mathscr{E}))\overline{\mathscr{P}}\prod_{i=1}^{n}d_i!\prod_{j=1}^{n'}d_j'!$ for $$\overline{\mathscr{P}}=\exp\left(\frac1{m}|F|+\frac1{2m^2}\sum_{v\in V_{\d,\d'}}|R^*(v)|\right)\prod_{t=1}^{m}\frac{t}{n_t}.$$ We get $|F|=\sum_{(v_i,v_j')\in X}d_id_j'$ and $\prod_{t=1}^{m}\frac{t}{n_t}=\frac1{m!}$ (see Corollary~\ref{corollary_mis_count}), and by using similar reasoning as in Section~\ref{subsection_applications_simple} one more time, we also obtain $\frac1{2m^2}\sum_{v\in V_{\d,\d'}}|R^*(v)|=\frac12\lambda_\d\lambda_{\d'}+O(\Delta^2/m)$.\hfill$\square$
	
	\subsection{Oriented graphs: Proof of Theorem~\ref{theorem_oriented}}\label{subsection_applications_oriented}
	
	We consider the bipartite configuration space $G:=G_{\d,\d'}$ from Definition \ref{definition_bispace} from Section~\ref{subsection_independent_bipartite}, and we declare two vertices in $V_{\d,\d'}$ as $R$-equivalent if they induce an arc on the same two vertices in $V$ regardless of direction. For a vertex $v\in V_{\d,\d'}$ connecting $W_i$ and $W_j'$, we define the weight function $\b(v):=\tfrac{d_id_j'+d_i'd_j}{2m}$. Finally, we let $X_{\d,\d'}$ and $F_{\d,\d'}$ contain all vertices in $V_{\d,\d'}$ that induce a forbidden arc from $X$ or a self-loop respectively, and we define $F:=F_{\d,\d'}\cup X_{\d,\d'}$.
	
	Consider the IMFIS process on $(G,F,R,\b)$, and apply rejection sampling on $\boldsymbol{S}_{\infty}$ to sample an FMIS $S$. If we identify the sets $W_1\cup W_1',\ldots,W_n\cup W_n'$ with vertices $v_1,\ldots,v_n$, and identify any vertex $v\in S$ connecting $W_i$ and $W_j'$ with an arc from $v_i$ to $v_j$, then this results in an oriented graph with out- and in-degree sequences $\d$ and $\d'$ and with no arcs from $X$. By the same reasoning as in Section~\ref{subsection_applications_simple}, this process is equivalent to the iterative oriented graph process.
	
	Note that $\alpha(G)=m\to\infty$, and $\ell(G)=2$ is fixed, as shown in Section~\ref{subsection_independent_parameter}. By similar reasoning as in Section~\ref{subsection_applications_simple}, we get $F_{\max}\leq2d_{\max}+2x_{\max}d_{\max}$, $R_{\max}^{(a)}\leq2d_{\max}^2$, $R_{\max}^{(b)}\leq4d_{\max}^2$, $R_{\max}^{(c)}\leq2d_{\max}$, and $\b_{\max}m\leq d_{\max}^2$, which shows that $M=O(\Delta)$ holds, so we can apply Theorem~\ref{theorem_main} again. We find that the probability of rejection is $O(\Delta/m)$, and any FMIS $S$ of $G$ is sampled with probability $(1+O(\mathscr{E}))\mathscr{P}(S)$. By similar  reasoning as in Section~\ref{subsection_applications_simple} again, to calculate $\mathscr{P}(S)$, we can substitute the weight function $\b'(v):=\frac12P|R^*(v)|$ for $\b(v)$, and it thus suffices to calculate $\overline{\mathscr{P}}$.
	
	Any oriented graph $H$ with out- and in-degree sequences $\d$ and $\d'$ and with no arcs from $X$ is realized by $\prod_{i=1}^{n}d_i!d_i'!$ FMISs of $G$, due to all possible permutations of the half-edges within the sets $W_1,\ldots,W_n,W_1',\ldots,W_{n}'$. We find that $H$ is sampled with probability $(1+O(\mathscr{E}))\overline{\mathscr{P}}\prod_{i=1}^{n}d_i!d_i'!$ for $$\overline{\mathscr{P}}=\exp\left(\frac1{m}|F|+\frac1{2m^2}\sum_{v\in V_{\d,\d'}}|R^*(v)|\right)\prod_{t=1}^{m}\frac{t}{n_t}.$$ Since $|F_{\d,\d'}|=\sum_{i=1}^{n}d_id_i'$, $|X_{\d,\d'}|=\sum_{(v_i,v_j)\in X}d_id_j'$, and $\prod_{t=1}^{m}\frac{t}{n_t}=\frac1{m!}$ (see Corollary~\ref{corollary_mis_count}), it remains only to evaluate the term $\sum_{v\in V_{\d,\d'}}|R^*(v)|$. Pairs $v\in V_{\d,\d'}$ connecting $W_i$ and $W_j'$ with $i\neq j$ and $(v_i,v_j),(v_j,v_i)\not\in X$ have $|R^*(v)|=(d_i-1)(d_j'-1)+d_i'd_j$. If $(v_i,v_j)\in X$ or $(v_j,v_i)\in X$, then one or both terms vanish, resulting in a deviation of at most $2d_{\max}^2$. We get
	\begin{align}
		\sum_{v\in V_{\d,\d'}}|R^*(v)|
		&=\sum_{i=1}^{n}\sum_{j=1}^{n}d_id_j'((d_i-1)(d_j'-1)+d_i'd_j)
		\\&\quad-\sum_{i=1}^{n}d_id_i'((d_i-1)(d_i'-1)+d_i'd_i)
		\\&\quad+O\Bigg(\sum_{(v_i,v_j)\in X}(d_id_j'+d_i'd_j)d_{\max}^2\Bigg)
		\\&=\left(\sum_{i=1}^{n}d_i(d_i-1)\right)\left(\sum_{i=1}^{n}d_i'(d_i'-1)\right)+\left(\sum_{i=1}^{n}d_id_i'\right)^2
		\\&\quad+O\left(\sum_{i=1}^{n}(d_i+d_i')d_{\max}^3x_{\max}\right).
	\end{align}
	We thus get $\frac1{2m^2}\sum_{v\in V_{\d,\d'}}|R^*(v)|=\frac12\lambda_{\d,\d'}^2+\frac12\lambda_\d\lambda_{\d'}+O(\Delta^2/m)$.\hfill$\square$
	
	\subsection{Edge-colored graphs: Proof of Theorem~\ref{theorem_colored}}\label{subsection_applications_colored}
	
	For $g=1,\ldots,\chi$, we consider the configuration space $G^{(g)}:=G_{\d^{(g)}}$ from Definition \ref{definition_space} from Section~\ref{subsection_independent_configuration} equipped with the set of forbidden vertices $F_{\d^{(g)}}$ and the equivalence relation $R^{(g)}:=R_{\d^{(g)}}$, which were subsequently defined to cover the self-loops and multi-edges respectively. For a vertex $v\in V_{\d^{(g)}}$ connecting $W_i$ and $W_j$, we define the weight function $\b^{(g)}(v):=\b_{i,j}^{(g)}$. Finally, we let $X_{\d^{(g)}}$ contain all vertices in $V_{\d^{(g)}}$ that induce a forbidden edge from $X^{(g)}\cup E^{(1)}\cup\ldots\cup E^{(g-1)}$, and we define $F^{(g)}:=F_{\d^{(g)}}\cup X_{\d^{(g)}}$.
	
	Consider the IMFIS process on $(G^{(g)},F^{(g)},R^{(g)},\b^{(g)})$, and apply rejection sampling on $\boldsymbol{S}_{\infty}$ to sample an FMIS $S$. If we identify the sets $W_1,\ldots,W_n$ with vertices $v_1,\ldots,v_n$ and we identify any vertex $v\in S$ connecting half-edges from $W_i$ and $W_j$ with an edge between $v_i$ and $v_j$, then this results in a simple graph with degree sequence $\d$ and with no edges from $X^{(g)}\cup E^{(1)}\cup\ldots\cup E^{(g-1)}$. By the same reasoning as in Section~\ref{subsection_applications_simple}, this process is equivalent to iteration $g$ of the iterative edge-colored graph process.
	
	For convenience, write $G=G^{(g)}$, $F=F^{(g)}$, $R=R^{(g)}$, and $\b=\b^{(g)}$. Note that $\alpha(G)=m^{(g)}\geq m_{\min}\to\infty$, and $\ell(G)=4$ is fixed, as shown in Section~\ref{subsection_independent_simple}. Using similar reasoning as in Section~\ref{subsection_applications_simple}, we get $$F_{\max}\leq2d_{\max}^{(g)}(1+x_{\max}^{(g)}+d_{\max}^{(1)}+\ldots+d_{\max}^{(g-1)}),$$ $R_{\max}^{(a)}\leq(d_{\max}^{(g)})^2$, $R_{\max}^{(b)}\leq2(d_{\max}^{(g)})^2$, $R_{\max}^{(c)}\leq2d_{\max}^{(g)}$, and $$\b_{\max}m^{(g)}\leq\tfrac14(d_{\max}^{(g)})^2+\tfrac12m^{(g)}\sum_{h=g+1}^{\chi}\frac{(d_{\max}^{(h)})^2}{m^{(h)}},$$ which shows that $M=O(\Delta^{(g)})$ holds, so we can apply Theorem~\ref{theorem_main} again. We find that the probability of rejection in iteration $g$ is $O(\Delta^{(g)}/m^{(g)})$, and any FMIS $S$ of $G$ is sampled with probability $(1+O(\mathscr{E}^{(g)}))\mathscr{P}(S)$, where $\mathscr{E}^{(g)}:=\mathscr{E}(m^{(g)},\Delta^{(g)})$.
	
	To calculate $\mathscr{P}(S)$, for $v\in V_{\d^{(g)}}$ connecting half-edges from $W_i$ and $W_j$, consider the weight function defined by $\b'(v):=0$ if $i=j$ and $$\b'(v):=\b_{i,j}^{(g)}+\tfrac12P|R^*(v)|-\frac{d_i^{(g)}d_j^{(g)}}{4m^{(g)}}$$ otherwise, which replaces the first term of $\b_{i,j}^{(g)}$ with $\frac12P|R^*(v)|$. By similar reasoning as in Section~\ref{subsection_applications_simple} again, we can substitute $\b'$ for $\b$.
	
	Any simple graph $H^{(g)}=(V,E^{(g)})$ with degree sequence $\d^{(g)}$ and with no edges from $X^{(g)}\cup E^{(1)}\cup\ldots\cup E^{(g-1)}$ is realized by $\prod_{i=1}^{n}d_i^{(g)}!$ FMISs of $G$, due to all possible permutations of the half-edges within the sets $W_1,\ldots,W_n$. We find that $H^{(g)}$ is sampled with probability $(1+O(\mathscr{E}^{(g)}))\mathscr{P}(E^{(g)})\prod_{i=1}^{n}d_i^{(g)}!$ for
	\begin{align}
		\mathscr{P}(E^{(g)})&:=\exp\bigg(\frac1{2m^{(g)}}|F|+\frac1{2m^{(g)}}\b'(V_{\d^{(g)}})\bigg.
		\\&\quad\bigg.-\sum_{h=g+1}^{\chi}\sum_{\{v_i,v_j\}\in E^{(g)}\setminus X^{(h)}}\frac{d_i^{(h)}d_j^{(h)}}{2m^{(h)}}\bigg)\prod_{t=1}^{m^{(g)}}\frac{t}{n_t}.
	\end{align}
	To simplify this expression, first note that $|F|=|F_{\d^{(g)}}|+|X_{\d^{(g)}}|$, $|F_{\d^{(g)}}|=\sum_{i=1}^{n}\tbinom{d_i^{(g)}}{2}$, and $\prod_{t=1}^{m^{(g)}}\tfrac{t}{n_t}=\tfrac1{(2m^{(g)}-1)!!}$ (see Corollary~\ref{corollary_mis_count}). Next, note that we can rewrite $X^{(g)}\cup E^{(1)}\cup\ldots\cup E^{(g-1)}$ as the disjoint union of sets $X^{(g)},E^{(1)}\setminus X^{(g)},\ldots,E^{(g-1)}\setminus X^{(g)}$, so we get $$|X_{\d^{(g)}}|=\sum_{\{v_i,v_j\}\in X^{(g)}}d_i^{(g)}d_j^{(g)}+\sum_{h=1}^{g-1}\sum_{\{v_i,v_j\}\in E^{(h)}\setminus X^{(g)}}d_i^{(g)}d_j^{(g)}.$$ Finally, it remains to evaluate
	\begin{align}
		\frac{\b'(V_{\d^{(g)}})}{2m^{(g)}}&=\sum_{i=1}^{n}\sum_{j=i+1}^{n}\frac{d_i^{(g)}d_j^{(g)}}{2m^{(g)}}\left(\frac{(d_i^{(g)}-1)(d_j^{(g)}-1)}{4m^{(g)}}+\sum_{h=g+1}^{\chi}\frac{d_i^{(h)}d_j^{(h)}}{2m^{(h)}}\right)
		\\&\quad-\sum_{\{v_i,v_j\}\in X^{(g)}}\frac{d_i^{(g)}d_j^{(g)}(d_i^{(g)}-1)(d_j^{(g)}-1)}{8(m^{(g)})^2}
		\\&\quad-\sum_{h=1}^{g-1}\sum_{\{v_i,v_j\}\in E^{(h)}\setminus X^{(g)}}\frac{d_i^{(g)}d_j^{(g)}(d_i^{(g)}-1)(d_j^{(g)}-1)}{8(m^{(g)})^2}
		\\&\quad-\sum_{h=g+1}^{\chi}\sum_{\{v_i,v_j\}\in X^{(h)}}\frac{d_i^{(g)}d_j^{(g)}d_i^{(h)}d_j^{(h)}}{4m^{(g)}m^{(h)}}.
	\end{align}
	Here, the second and third term correct for the case $|R^*(v)|=0$, and the fourth term corrects for the excluded terms in the definition of $\b_{i,j}^{(g)}$. We get
	\begin{align}
		\frac{\b'(V_{\d^{(g)}})}{2m^{(g)}}&=\tfrac14\lambda_{\d^{(g)}}^2-\sum_{i=1}^{n}\left(\frac{d_i^{(g)}(d_i^{(g)}-1)}{4m^{(g)}}\right)^2
		\\&\quad+\sum_{h=g+1}^{\chi}\left(\tfrac12\lambda_{\d^{(g)},\d^{(h)}}^2-\sum_{i=1}^{n}\frac{(d_i^{(g)})^2(d_i^{(h)})^2}{8m^{(g)}m^{(h)}}\right)
		\\&\quad+O\left(\sum_{h=g}^{\chi}\sum_{i=1}^{n}\frac{x_{\max}^{(h)}d_i^{(g)}d_{\max}^{(g)}(d_{\max}^{(h)})^2}{m^{(g)}m^{(h)}}+\sum_{h=1}^{g-1}\sum_{i=1}^{n}\frac{d_{\max}^{(h)}d_i^{(g)}(d_{\max}^{(g)})^3}{(m^{(g)})^2}\right)
		\\&=\tfrac14\lambda_{\d^{(g)}}^2+\tfrac12\sum_{h=g+1}^{\chi}\lambda_{\d^{(g)},\d^{(h)}}^2
		\\&\quad+O\left(\sum_{h=g}^{\chi}\frac{x_{\max}^{(h)}d_{\max}^{(g)}(d_{\max}^{(h)})^2}{m^{(h)}}+\sum_{h=1}^{g-1}\frac{d_{\max}^{(h)}(d_{\max}^{(g)})^3}{m^{(g)}}\right).
	\end{align}
	We conclude that
	\begin{align}
		\mathscr{P}(E^{(g)})&=\exp\Bigg(\tfrac12\lambda_{\d^{(g)}}+\tfrac14\lambda_{\d^{(g)}}^2+\tfrac12\sum_{h=g+1}^{\chi}\lambda_{\d^{(g)},\d^{(h)}}^2+\mu_{\d^{(g)}}^{X^{(g)}}\Bigg.
		\\&\quad+\sum_{h=1}^{g-1}\sum_{\{v_i,v_j\}\in E^{(h)}\setminus X^{(g)}}\frac{d_i^{(g)}d_j^{(g)}}{2m^{(g)}}-\sum_{h=g+1}^{\chi}\sum_{\{v_i,v_j\}\in E^{(g)}\setminus X^{(h)}}\frac{d_i^{(h)}d_j^{(h)}}{2m^{(h)}}
		\\&\quad\Bigg.+O\Bigg(\sum_{h=g}^{\chi}\frac{x_{\max}^{(h)}d_{\max}^{(g)}(d_{\max}^{(h)})^2}{m^{(h)}}+\sum_{h=1}^{g-1}\frac{d_{\max}^{(h)}(d_{\max}^{(g)})^3}{m^{(g)}}\Bigg)\Bigg)\frac1{(2m^{(g)}-1)!!}.
	\end{align}
	
	Finally, consider any sequence of edge-disjoint simple graphs $(V,E^{(g)})$ for $g=1,\ldots,\chi$ with degree sequence $\d^{(g)}$ and with no edges from $X^{(g)}$. By the chain rule, this sequence is sampled with probability $$\exp\left(O\left(\sum_{g=1}^{\chi}\mathscr{E}^{(g)}\right)\right)\prod_{g=1}^{\chi}\left(\mathscr{P}(E^{(g)})\prod_{i=1}^{n}d_i^{(g)}!\right).$$ Finally, we capitalize off of our choice of weight function and notice that the double sums in our expression for $\mathscr{P}(E^{(g)})$ all cancel each other through this product, leaving us with
	\begin{align}
		&\exp\Bigg(\sum_{g=1}^{\chi}\Bigg(\tfrac12\lambda_{\d^{(g)}}+\tfrac14\lambda_{\d^{(g)}}^2+\tfrac12\sum_{h=g+1}^{\chi}\lambda_{\d^{(g)},\d^{(h)}}^2+\mu_{\d^{(g)}}^{X^{(g)}}\Bigg)\Bigg.
		\\&+O\Bigg(\sum_{g=1}^{\chi}\frac{(d_{\max}^{(g)})^3}{m^{(g)}}\sum_{h=1}^{g-1}d_{\max}^{(h)}+\sum_{h=1}^{\chi}\frac{x_{\max}^{(h)}(d_{\max}^{(h)})^2}{m^{(h)}}\sum_{g=1}^{h}d_{\max}^{(g)}+\mathscr{E}\Bigg)\Bigg)
		\\&\cdot\prod_{g=1}^{\chi}\left(\frac1{(2m^{(g)}-1)!!}\prod_{i=1}^{n}d_i^{(g)}!\right).
	\end{align}
	Since
	\begin{align}
		\sum_{g=1}^{\chi}\frac{(d_{\max}^{(g)})^3}{m^{(g)}}\sum_{h=1}^{g-1}d_{\max}^{(h)}&=O\left(\sum_{g=1}^{\chi}\frac{(\Delta^{(g)})^2}{m^{(g)}}\right)=O(\mathscr{E}),
		\\\sum_{h=1}^{\chi}\frac{x_{\max}^{(h)}(d_{\max}^{(h)})^2}{m^{(h)}}\sum_{g=1}^{h}d_{\max}^{(g)}&=O\left(\sum_{h=1}^{\chi}\frac{(\Delta^{(h)})^2}{m^{(h)}}\right)=O(\mathscr{E}),
	\end{align}
	the result follows.\hfill$\square$
	
	\subsection{Edge-colored bipartite graphs: Proof of Theorem~\ref{theorem_bipcol}}\label{subsection_applications_bipcol}
	
	For $g=1,\ldots,\chi$, we consider the bipartite configuration space $G^{(g)}:=G_{\d^{(g)},\d'^{(g)}}$ from Definition \ref{definition_bispace} from Section~\ref{subsection_independent_configuration} equipped with the equivalence relation $R^{(g)}:=R_{\d^{(g)},\d'^{(g)}}$, which was subsequently defined to cover the multi-edges. For a vertex $v\in V_{\d^{(g)},\d'^{(g)}}$ connecting $W_i$ and $W_j'$, we define the weight function $\b^{(g)}(v):=\b_{i,j}^{(g)}$. Finally, we let $F^{(g)}$ contain all vertices in $V_{\d^{(g)},\d'^{(g)}}$ that induce an edge from $X^{(g)}\cup E^{(1)}\cup\ldots\cup E^{(g-1)}$.
	
	Consider the IMFIS process on $(G^{(g)},F^{(g)},R^{(g)},\b^{(g)})$, and apply rejection sampling on $\boldsymbol{S}_{\infty}$ to sample an FMIS $S$. Identifying the sets $W_1,\ldots,W_n,W_1',\ldots,W_{n'}'$ with vertices $v_1,\ldots,v_n,v_1',\ldots,v_{n'}'$ and identifying any vertex $v\in\boldsymbol{S}_{\infty}$ connecting half-edges from $W_i$ and $W_j'$ with an edge between $v_i$ and $v_j'$, results in a simple bipartite graph with degree sequences $\d$ and $\d'$ on its parts and with no edges from $X^{(g)}\cup E^{(1)}\cup\ldots\cup E^{(g-1)}$. By the same reasoning as in Section~\ref{subsection_applications_simple}, this process is equivalent to iteration $g$ of the iterative edge-colored bipartite graph process.
	
	For convenience, write $G=G^{(g)}$, $R=R^{(g)}$, $F=F^{(g)}$, and $\b=\b^{(g)}$. Note that $\alpha(G)=m^{(g)}\geq m_{\min}\to\infty$, and $\ell(G)=2$ is fixed, as shown in Section~\ref{subsection_independent_parameter}. Using similar reasoning as in Section~\ref{subsection_applications_simple}, we get $$F_{\max}\leq2d_{\max}(x_{\max}^{(g)}+d_{\max}^{(1)}+\ldots+d_{\max}^{(g-1)}),$$ $R_{\max}^{(a)}\leq(d_{\max}^{(g)})^2$, $R_{\max}^{(b)}\leq2(d_{\max}^{(g)})^2$, $R_{\max}^{(c)}\leq2d_{\max}^{(g)}$, and $$\b_{\max}m^{(g)}\leq\tfrac12(d_{\max}^{(g)})^2+m^{(g)}\sum_{h=g+1}^{\chi}\frac{(d_{\max}^{(h)})^2}{m^{(h)}},$$ which shows that $M=O(\Delta^{(g)})$ holds, so we can apply Theorem~\ref{theorem_main} again. We find that the probability of rejection in iteration $g$ is $O(\Delta^{(g)}/m^{(g)})$, and any FMIS $S$ of $G$ is sampled with probability $(1+O(\mathscr{E}^{(g)}))\mathscr{P}(S)$, where $\mathscr{E}^{(g)}:=\mathscr{E}(m^{(g)},\Delta^{(g)})$.
	
	To calculate $\mathscr{P}(S)$, for $v\in V_{\d^{(g)},\d'^{(g)}}$ connecting half-edges from $W_i$ and $W_j'$, consider the weight function defined by $\b'(v):=0$ if $i=j$ and $$\b'(v):=\b(v)+\tfrac12P|R^*(v)|-\frac{d_i^{(g)}d_j'^{(g)}}{2m^{(g)}}$$ otherwise, which replaces the first term of $\b_{i,j}^{(g)}$ with $\frac12P|R^*(v)|$. By similar reasoning as in Section~\ref{subsection_applications_simple} again, we can substitute $\b'$ for $\b$.
	
	Any simple bipartite graph $H^{(g)}=(V,E^{(g)})$ with degree sequences $\d^{(g)}$ and $\d'^{(g)}$ on its parts and with no edges from $X^{(g)}\cup E^{(1)}\cup\ldots\cup E^{(g-1)}$ is realized by $\prod_{i=1}^{n}d_i^{(g)}!\prod_{j=1}^{n'}d_j'^{(g)}!$ FMISs of $G$, due to all possible permutations of the half-edges within the sets $W_1,\ldots,W_n,W_1',\ldots,W_{n'}'$. We find that $H^{(g)}$ is sampled with probability $(1+O(\mathscr{E}^{(g)}))\mathscr{P}(E^{(g)})\prod_{i=1}^{n}d_i^{(g)}!\prod_{j=1}^{n'}d_j'^{(g)}!$ for
	\begin{align}
		\mathscr{P}(E^{(g)})&:=\exp\bigg(\frac1{m^{(g)}}|F|+\frac1{m^{(g)}}\b'(V_{\d^{(g)},\d'^{(g)}})\bigg.
		\\&\quad\bigg.-\sum_{h=g+1}^{\chi}\sum_{(v_i,v_j')\in E^{(g)}\setminus X^{(h)}}\frac{d_i^{(h)}d_j'^{(h)}}{m^{(h)}}\bigg)\prod_{t=1}^{m^{(g)}}\frac{t}{n_t}.
	\end{align}
	Following the same reasoning as in Section~\ref{subsection_applications_colored}, this expression simplifies to
	\begin{align}
		\mathscr{P}(E^{(g)})&=\exp\Bigg(\tfrac12\lambda_{\d^{(g)}}\lambda_{\d'^{(g)}}+\sum_{h=g+1}^{\chi}\lambda_{\d^{(g)},\d^{(h)}}\lambda_{\d'^{(g)},\d'^{(h)}}+\mu_{\d^{(g)},\d'^{(g)}}^{X^{(g)}}\Bigg.
		\\&\quad+\sum_{h=1}^{g-1}\sum_{(v_i,v_j')\in E^{(h)}\setminus X^{(g)}}\frac{d_i^{(g)}d_j'^{(g)}}{m^{(g)}}-\sum_{h=g+1}^{\chi}\sum_{(v_i,v_j')\in E^{(g)}\setminus X^{(h)}}\frac{d_i^{(h)}d_j'^{(h)}}{m^{(h)}}
		\\&\quad\Bigg.+O\Bigg(\sum_{h=g}^{\chi}\frac{x_{\max}^{(h)}d_{\max}^{(g)}(d_{\max}^{(h)})^2}{m^{(h)}}+\sum_{h=1}^{g-1}\frac{d_{\max}^{(h)}(d_{\max}^{(g)})^3}{m^{(g)}}\Bigg)\Bigg)\frac1{m^{(g)}!}.
	\end{align}
	Finally, following the same reasoning as in Section~\ref{subsection_applications_colored} again, the result follows.\hfill$\square$
	
	\subsection{Hypergraphs: Proof of Theorem~\ref{theorem_hyper}}\label{subsection_applications_hyper}
	
	By Theorem~\ref{theorem_bipartite}, the probability of rejection of the iterative bipartite graph process is $O(d_{\max}^2/s)$, and any simple bipartite graph $(V\cup V',\widetilde{E})$ with degree sequences $\d$ and $\d'$ on its parts is sampled with probability $(1+O(\mathscr{E}))\widetilde{\mathscr{P}}$, where $$\widetilde{\mathscr{P}}:=\exp\left(\tfrac12\lambda_{\d}\lambda_{\d'}\right)\frac1{s!}\prod_{i=1}^{n}d_i!\prod_{k=2}^{d_{\max}}k!^{c_k}.$$
	
	Consider the IMFIS process on $(G,F,R,\b)$ defined in Section~\ref{subsection_applications_bipartite} for the proof of Theorem~\ref{theorem_bipartite}, which was shown to be equivalent to the iterative bipartite graph process. Consider some forbidden hyperedge $e\in X$ and any $j=1,\ldots,m$ with $k:=|e|=d_j'$. Then there are $k!\prod_{v_i\in e}d_i$ possible FISs $T$ of $G$ that induce $k$ edges connecting $v_j'$ to all vertices in $e$. By Theorem~\ref{theorem_space} and Proposition~\ref{proposition_qt}~(\ref{proposition_qt_small}), for any such FIS $T$, the probability that the IMFIS process samples an FMIS containing $T$ is $(1+O(\mathscr{E}))s^{-k}$. Indeed, note that $|T|=k\leq d_{\max}$ gives $kd_{\max}^2/s=O(\mathscr{E})$. It follows that the expected number of vertices in $V'$ with a neighborhood inducing a forbidden hyperedge from $X$ is $(1+O(\mathscr{E}))\rho_X$, so the probability that such a vertex exists is $O(\rho_X)$.
	
	Next, consider any $1\leq j_1<j_2\leq m$ with $k:=d'_{j_1}=d'_{j_2}$ and any $1\leq i_1<\ldots<i_k\leq n$. Then there are $$k!^2\prod_{l=1}^kd_{i_l}(d_{i_l}-1)$$ possible FISs $T$ of $G$ that induce $2k$ edges connecting both vertices $v_{j_1}'$ and $v_{j_2}'$ to all vertices $v_{i_1},\ldots,v_{i_k}$. By Theorem~\ref{theorem_space} and Proposition~\ref{proposition_qt}~(\ref{proposition_qt_small}) again, for any such FIS $T$, the probability that the IMFIS process samples an FMIS containing $T$ is $(1+O(\mathscr{E}))s^{-2k}$. It follows that the expected number of pairs of vertices in $V'$ with identical neighborhoods inducing the same hyperedge of cardinality $k$ is
	\begin{align}
		&(1+O(\mathscr{E}))\binom{c_k}{2}\frac{k!^2}{s^{2k}}\sum_{1\leq i_1<\ldots<i_k\leq n}\prod_{l=1}^{k}d_{i_l}(d_{i_l}-1)
		\\&\leq(1+O(\mathscr{E}))\binom{c_k}{2}\frac{k!^2}{s^{2k}}\cdot\frac1{k!}\left(\sum_{i=1}^{n}d_{i}(d_{i}-1)\right)^k
		\\&\leq(1+O(\mathscr{E}))\rho_k.
	\end{align}
	The probability that there exist two vertices in $V'$ with identical neighborhoods is thus $O(\rho_2+\rho_3+\ldots+\rho_{d_{\max}})$. Since $\lambda_{\d}\leq d_{\max}$, we can simplify
	\begin{align}
		\sum_{k=4}^{d_{\max}}\rho_k&\leq4!\left(\frac{d_{\max}}{s}\right)^4\cdot\sum_{k=4}^{d_{\max}}\frac{k!}{4!}\left(\frac{d_{\max}}{s}\right)^{k-4}\binom{c_k}{2}
		\\&=O\left(\left(\frac{d_{\max}}{s}\right)^4\cdot\sum_{k=4}^{d_{\max}}\binom{c_k}{2}\right)
		\\&=O\left(d_{\max}^4s^{-4}\binom{m}{2}\right)=O(d_{\max}^4/s^2)=O(1/s).
	\end{align}
	The second line uses $k!/4!\leq d_{\max}^{k-4}$ and $d_{\max}^2/s=o(1)$, and the third line uses $c_4+\ldots+c_{d_{\max}}\leq m\leq s$. We conclude that identical hyperedges occur with probability $O(\rho_2+\rho_3+1/s)$.
	
	We conclude that the probability of rejection after the iterative bipartite graph process is $O(\rho_2+\rho_3+\rho_X+1/s)$. Note that any simple hypergraph $H$ with degree sequence $\d$ and with hyperedges not in $X$ with cardinalities matching $\d'$ is realized by $\prod_{k=2}^{d_{\max}}c_k!$ possible simple bipartite graphs with degree sequences $\d$ and $\d'$ on its parts, due to all possible permutations of the vertices in $V'$ with the same degree. It follows that $H$ is sampled with probability $(1+O(\mathscr{E}))\widetilde{\mathscr{P}}\prod_{k=2}^{d_{\max}}c_k!$. Since $\rho_3=O(s^2(d_{\max}/s)^3)=O(d_{\max}^3/s)=O(\mathscr{E})$, the result follows.\hfill$\square$
	
	\subsection{Directed hypergraphs: Proof of Theorem~\ref{theorem_dihy}}\label{subsection_applications_dihy}
	
	By Theorem~\ref{theorem_bipcol}, the probability of rejection of the iterative edge-colored bipartite graph process is $O(d_{\max}^2/s)$ in both iterations, and any simple edge-disjoint bipartite graphs $(V\cup V',\widetilde{E}^{(i)})$ with degree sequences $\d^{(i)}$ and $\d'^{(i)}$ on its parts for $i\in\{1,2\}$ are sampled with probability $(1+O(\mathscr{E}))\widetilde{\mathscr{P}}$, where
	\begin{align}
		\widetilde{\mathscr{P}}&:=\exp\left(\tfrac12\lambda_{\d^{(1)}}\lambda_{\d'^{(1)}}+\tfrac12\lambda_{\d^{(2)}}\lambda_{\d'^{(2)}}+\lambda_{\d^{(1)},\d^{(2)}}\lambda_{\d'^{(1)},\d'^{(2)}}\right)
		\\&\quad\cdot\frac1{s^{(1)}!s^{(2)}!}\prod_{i=1}^{n}d_i^{(1)}!d_i^{(2)}!\prod_{k^{(1)}=1}^{d_{\max}}\prod_{k^{(2)}=1}^{d_{\max}}(k^{(1)}!k^{(2)}!)^{c_{k^{(1)},k^{(2)}}}.
	\end{align}
	
	Consider the IMFIS processes on $(G^{(i)},F^{(i)},R^{(i)},\b^{(i)})$ for $i\in\{1,2\}$ defined in Section~\ref{subsection_applications_bipcol} for the proof of Theorem~\ref{theorem_bipcol}, which were shown to be equivalent to iterations $i=1$ and $i=2$ of the iterative edge-colored bipartite graph process. Consider some forbidden hyperarc $(e^{(1)},e^{(2)})\in X$ and any $j=1,\ldots,m$ with $(k^{(1)},k^{(2)}):=|e|=(d_j'^{(1)},d_j'^{(2)})$. Then there are $$k^{(1)}!k^{(2)}!\prod_{v_{i^{(1)}}\in e^{(1)}}d_{i^{(1)}}^{(1)}\prod_{v_{i^{(2)}}\in e^{(2)}}d_{i^{(2)}}^{(2)}$$ possible pairs of FISs $T^{(1)}$ and $T^{(2)}$ of $G^{(1)}$ and $G^{(2)}$ that induce $k^{(1)}$ edges connecting $v_j'$ to all vertices in $e^{(1)}$ and $k^{(2)}$ edges connecting $v_j'$ to all vertices in $e^{(2)}$ respectively. By Theorem~\ref{theorem_space} and Proposition~\ref{proposition_qt}~(\ref{proposition_qt_small}), for any such pairs of FISs, the probability that the IMFIS processes for $i=1$ and $i=2$ sample FMISs containing $T^{(1)}$ and $T^{(2)}$ respectively is $$(1+O(\mathscr{E}))(s^{(1)})^{-k^{(1)}}(s^{(2)})^{-k^{(2)}}.$$ Indeed, note that $|T^{(i)}|=k^{(i)}$ gives $(k^{(1)}+k^{(2)})d_{\max}^2/s=O(\mathscr{E})$. It follows that the expected number of vertices in $V'$ with a pair of neighborhoods inducing a forbidden hyperarc from $X$ is $(1+O(\mathscr{E}))\rho_X$, so the probability that such a vertex exists is $O(\rho_X)$.
	
	By similar reasoning as in Section~\ref{subsection_applications_hyper}, the expected number of pairs of vertices in $V'$ with identical pairs of neighborhoods inducing the same hyperarc of order $(k^{(1)},k^{(2)})$ is $O(\rho_{k^{(1)},k^{(2)}})$. The probability that there exist two vertices in $V'$ with identical pairs of neighborhoods is thus $$O\left(\sum_{k^{(1)}=1}^{d_{\max}}\sum_{k^{(2)}=1}^{d_{\max}}\rho_{k^{(1)},k^{(2)}}\right).$$ By similar reasoning as in Section~\ref{subsection_applications_hyper}, we can simplify the sum over $k^{(1)}+k^{(2)}\geq4$ to $O(1/s)$, leaving a probability of $O(\rho_{1,1}+\rho_{1,2}+\rho_{2,1}+1/s)$ that identical hyperarcs occur.
	
	We conclude that the probability of rejection after the iterative edge-colored bipartite graph process is $O(\rho_{1,1}+\rho_{1,2}+\rho_{2,1}+\rho_X+1/s)$. Note that any simple directed hypergraph $H$ with out- and in-degree sequences $\d^{(1)}$ and $\d^{(2)}$ and with hyperarcs not in $X$ with orders $(d_1'^{(1)},d_1'^{(2)}),\ldots,(d_m'^{(1)},d_m'^{(2)})$ is realized by $\prod_{k^{(1)}=1}^{d_{\max}}\prod_{k^{(2)}=1}^{d_{\max}}c_{k^{(1)},k^{(2)}}!$ possible simple edge-disjoint bipartite graphs $(V\cup V',\widetilde{E}^{(i)})$ with degree sequences $\d^{(i)}$ and $\d'^{(i)}$ on its parts for $i\in\{1,2\}$, due to all possible permutations of the vertices in $V'$ with the same pair of degrees. It follows that $H$ is sampled with probability $$(1+O(\mathscr{E}))\widetilde{\mathscr{P}}\prod_{k^{(1)}=1}^{d_{\max}}\prod_{k^{(2)}=1}^{d_{\max}}c_{k^{(1)},k^{(2)}}!.$$ Since $\rho_{1,2},\rho_{2,1}=O(s^2(d_{\max}/s)^3)=O(d_{\max}^3/s)=O(\mathscr{E})$, the result follows.\hfill$\square$

	\section{Proofs of Main Results}\label{section_proof}

	In this section, we prove Theorem~\ref{theorem_main}, Theorem~\ref{theorem_space}, Proposition~\ref{proposition_qt}, and Proposition~\ref{proposition_leniency}. First, in Section~\ref{subsection_proof_distribution}, we briefly show how Theorem~\ref{theorem_space} follows from Theorem~\ref{theorem_main}, and we prove Proposition~\ref{proposition_qt} and Proposition~\ref{proposition_leniency}. Then, in Section~\ref{subsection_proof_framework}--\ref{subsection_proof_termination}, we prove Theorem~\ref{theorem_main}.
	
	\subsection{Distribution}\label{subsection_proof_distribution}
	
	We use Theorem~\ref{theorem_main} to calculate the probability that the IMFIS process does or does not end up using certain specified vertices. The main idea is to compare with a modified IMFIS process that guarantees the specified vertex requirements.
	
	\begin{proof}[Proof of Theorem~\ref{theorem_space}]
		We need to calculate
		\begin{equation}\label{eq:sum_S_inf=s}
			\mathbb{P}(|\boldsymbol{S}_{\infty}|=\alpha\wedge T\subset\boldsymbol{S}_{\infty}\wedge U\cap\boldsymbol{S}_{\infty}=\emptyset)=\sum_{S\in\mathcal{G}(T,U)}\mathbb{P}(\boldsymbol{S}_{\infty}=S),
		\end{equation}
		where $\mathcal{G}(T,U)$ is the set of all FMISs of $G$ that contain the target set $T$ and are disjoint from the exclusion set $U$.
		
		We now consider the induced subgraph $G'$ of $G$ on the set $V'$ of vertices that are compatible with $T$ in an independent set. First, we rewrite equation \eqref{eq:sum_S_inf=s} as
		\begin{equation}\label{eq:sum_S_inf=t_s}
			\mathbb{P}(|\boldsymbol{S}_{\infty}|=\alpha\wedge T\subset\boldsymbol{S}_{\infty}\wedge U\cap\boldsymbol{S}_{\infty}=\emptyset)=\sum_{S'\in\mathcal{G}'(T,U)}\mathbb{P}(\boldsymbol{S}_{\infty}=T\cup S'),
		\end{equation}
		where $\mathcal{G}'(T,U)$ is the set of all $S'\subset V'\setminus U$ for which $T\cup S'$ is an FMIS of $G$. Let furthermore $R'$ be $R$ restricted to $V'$, and let $F':=(F\cap V')\cup(R^*(T)\cap V')\cup U$. Then $\mathcal{G}'(T,U)$ is the set of all FMISs of $G'$ with respect to $F'$ and $R'$.
		
		Now, we choose a modified weight function $\b':V'\to\mathbb{R}$, $$\b'(v):=\b(v)+\tfrac12P'|(R')^*(v)|-\tfrac12P|R^*(v)|$$ on $G'$ in such a way that the IMFIS process on $(G',F',R',\b')$ behaves like the original IMFIS process conditioned on containing $T$. We write $\boldsymbol{S}_{\infty}'$ to distinguish this IMFIS process from the one on $(G,F,R,\b)$.
		
		The main idea is to apply Theorem~\ref{theorem_main} to the IMFIS process on $(G',F',R',\b')$. Supposing the assumptions are met, by Theorem~\ref{theorem_main}~(\ref{theorem_main_termination}), we have $$\PP(|\boldsymbol{S}_{\infty}'| = \alpha') = 1-O(M'/\alpha').$$ This can be rewritten as $$\sum_{S'\in\mathcal{G}'(T,U)}\mathbb{P}(\boldsymbol{S}_{\infty}=T\cup S')\frac{\mathbb{P}(\boldsymbol{S}_{\infty}'=S')}{\mathbb{P}(\boldsymbol{S}_{\infty}=T\cup S')} = 1-O(M'/\alpha').$$ The result then follows from \eqref{eq:sum_S_inf=t_s} if we show that Theorem~\ref{theorem_main}~(\ref{theorem_main_equation}) gives
		\begin{equation}\label{eq:partial_asymptotic_prob}
			\frac{\mathbb{P}(\boldsymbol{S}_{\infty}'=S')}{\mathbb{P}(\boldsymbol{S}_{\infty}=T\cup S')}=(1+O(\mathscr{E}(\alpha',M')))\mathscr{Q}(T,U)^{-1}
		\end{equation}
		for any $S'\in\mathcal{G}'(T,U)$. In what follows, we thus verify the assumptions of Theorem~\ref{theorem_main} and we derive \eqref{eq:partial_asymptotic_prob}.
		
		First, by Proposition~\ref{proposition_hereditary}, $G'$ is indeed $2$-uniform with $\alpha(G')=\alpha'\to\infty$, and with $\ell(G')=\ell$ fixed. Next, note that $(R')_{\max}^{(a)}\leq R_{\max}^{(a)}\leq M'$, $(R')_{\max}^{(b)}\leq R_{\max}^{(b)}\leq M'$, and $((R')_{\max}^{(c)})^2\leq(R_{\max}^{(c)})^2\leq M'$. For any $v\in V'$, we have $$|\overline{N}_{G'}(v)\cap F'|\leq|\overline{N}_G(v)\cap F|+|\overline{N}_G(v)\cap R^*(T)|+|\overline{N}_G(v)\cap U|,$$ so $F'_{\max}\leq F_{\max}+R_{\max}^{(b)}+U_{\max}\leq3M'$. Finally, we have $$\b'_{\max}\alpha\leq \b_{\max}\alpha+\tfrac12P'\alpha(R')_{\max}^{(a)}+\tfrac12P\alpha R_{\max}^{(a)}\leq2M',$$ since $\ell\geq2$ gives $P\leq\alpha^{-1}$ and $P'\leq(\alpha')^{-1}$. We conclude that $$\max\{F'_{\max},(R')_{\max}^{(a)},(R')_{\max}^{(b)},((R')_{\max}^{(c)})^2,\b'_{\max}\alpha,1\}\leq3M'.$$
		
		We may thus apply Theorem~\ref{theorem_main} to the IMFIS process on $(G',F',R',\b')$. It remains to derive \eqref{eq:partial_asymptotic_prob} by evaluating
		\begin{align}
			&\frac{\mathbb{P}(\boldsymbol{S}_{\infty}'=S')}{\mathbb{P}(\boldsymbol{S}_{\infty}=T\cup S')}=\frac{(1+O(\mathscr{E}(\alpha',M')))\mathscr{P}'(S')}{(1+O(\mathscr{E}(\alpha,M)))\mathscr{P}(T\cup S')}
		\end{align} 
		for $S'\in\mathcal{G}'(T,U)$, where
		\begin{align}
			&\mathscr{P}'(S')=\exp(P'|F'|+\tfrac12P'|(R')^*(S')|+P'\b'(V')-\b'(S'))\prod_{t=1}^{\alpha'}\frac{t}{n_t},
			\\&\mathscr{P}(T\cup S')=\exp(P|F|+\tfrac12P|R^*(T\cup S')|+P\b(V)-\b(T\cup S'))\prod_{t=1}^{\alpha}\frac{t}{n_t},
		\end{align}
		Comparing to \eqref{eq:partial_asymptotic_prob}, it suffices to show $\mathscr{P}(T\cup S')/\mathscr{P}'(S')=\mathscr{Q}(T,U)$. We have
		\begin{align}
			\frac{\mathscr{P}(T\cup S')}{\mathscr{P}'(S')}&=\exp\bigg(\tfrac12P|R^*(S')|-\tfrac12P'|(R')^*(S')|-\b(S')+\b'(S')\bigg.
			\\&\quad+P|F|-P'|F'|+P\b(V)-P'\b'(V')
			\\&\quad\bigg.+\tfrac12P|R^*(T)|-\b(T)\bigg)\prod_{t=\alpha'+1}^{\alpha}\frac{t}{n_t}.
		\end{align}
		Here, all terms in the exponent that depend on $S'$ are grouped in the first line, and by definition of $\b'$, they sum to zero, since $S'$ is feasible. Since $|F'|=|F\cap V'|+|R^*(T)\cap V'|+|U|$, it remains to evaluate $$P'\b'(V')=P'\b(V')+\tfrac12P'\sum_{v\in V'}(P'|(R')^*(v)|-P|R^*(v)|).$$ The result follows, since $(R')^*(v)=R^*(v)\cap V'$ for $v\in V'$.
	\end{proof}
	
	To prove Proposition~\ref{proposition_qt} and Proposition~\ref{proposition_leniency}, we use the following bounds.
	
	\begin{lemma}\label{lemma_terms}
		Consider the IMFIS process on $(G,F,R,\b)$ with $G=(V,E)$ $2$-uniform. Assume $\ell$ is fixed, and let $\alpha\to\infty$. Then
		\begin{enumerate}[(i)]
			\item\label{lemma_terms_n} $\begin{aligned}[t]
				|V|=n_{\alpha}=(1+O(\alpha^{-1}))\tfrac{\ell}{2}\alpha^2,
			\end{aligned}$
			\item\label{lemma_terms_f} $\begin{aligned}[t]
				|F|\leq\alpha F_{\max}\leq\alpha M,
			\end{aligned}$
			\item\label{lemma_terms_rs} $\begin{aligned}[t]
				|R^*(U)|\leq|U|R_{\max}^{(a)}\leq|U|M,\ \forall U\subset V.
			\end{aligned}$
		\end{enumerate}
	\end{lemma}
	
	\begin{proof}
		(\ref{lemma_terms_n}) By Proposition~\ref{proposition_n}, we have $|V|=n_{\alpha}=\ell\binom{\alpha}{2}+\alpha=\frac{\ell}{2}\alpha^2+(1-\frac{\ell}{2})\alpha$.
		
		(\ref{lemma_terms_f}) Let $S$ be an MIS of $G$. Let $C$ be the number of pairs $(v,s)\in F\times S$ with $v\in\overline{N}(s)$. By definition, we have $C\leq|S|F_{\max}$, and since $G$ is $2$-uniform, every vertex $v\in F$ is either in $S$ or adjacent to two vertices in $S$, which gives $C\geq|F|$.
		
		(\ref{lemma_terms_rs}) We have $|R^*(U)|\leq\sum_{v\in U}|R^*(v)|$.
	\end{proof}
	
	\begin{proof}[Proof of Proposition~\ref{proposition_qt}]
		(\ref{proposition_qt_small}) We first prove the second equality. Using $|T|=\alpha-\alpha'$, we find that $$\frac{\alpha^{|T|}(\alpha')!}{\alpha!}=\prod_{t=\alpha'+1}^{\alpha}\frac{\alpha}{t}=\prod_{t=\alpha'+1}^{\alpha}\left(1+\frac{\alpha-t}{t}\right)$$ has logarithm of order $O(|T|^2/\alpha')$, since $0\leq\alpha-t<|T|$ for $t=\alpha'+1,\ldots,\alpha$. It thus suffices to prove $P'|U|=P|U|+O(|T|M'/\alpha')$. Note that $$|P-P'|=\frac{|\alpha-\alpha'|}{\frac{\ell}{2}\alpha\alpha'}=O\left(\frac{|T|}{\alpha\alpha'}\right).$$ By the same argument as in the proof of Lemma~\ref{lemma_terms}~(\ref{lemma_terms_f}), we have $|U|\leq\alpha U_{\max}\leq\alpha M'$, so we get $$P'|U|=P|U|-(P-P')|U|=P|U|+O(|T|M'/\alpha').$$
		
		We now prove the first equality, so we need to study $\mathscr{Q}(T,U)$. We have $$\prod_{t=\alpha'+1}^{\alpha}\frac{t}{n_t}=\prod_{t=\alpha'+1}^{\alpha}\frac{1}{\frac{\ell}{2}t+1-\frac{\ell}{2}}=\prod_{t=\alpha'+1}^{\alpha}\frac{1}{\frac{\ell}{2}t}\left(1+\frac{1-\frac{\ell}{2}}{\frac{\ell}{2}t}\right)^{-1}.$$ Using $\log(1+x)=O(x)$ for $x=o(1)$ and $|T|=\alpha-\alpha'$, we get $$\prod_{t=\alpha'+1}^{\alpha}\frac{t}{n_t}=\left(\frac{2}{\ell}\right)^{|T|}\frac{(\alpha')!}{\alpha!}\exp\left(O\left(\frac{|T|}{\alpha'}\right)\right).$$ It thus suffices to prove that the exponent in the definition of $\mathscr{Q}(T,U)$ is $-P'|U|+O(|T|M'/\alpha')$.
		
		Note that $M\leq M'$ and $\alpha'\leq\alpha$, so $|T|M/\alpha\leq|T|M'/\alpha'$. We have
		\begin{align}
			\tfrac12P|R^*(T)|&\leq\tfrac12P|T|R_{\max}^{(a)}=O(|T|M/\alpha),
			\\P'|R^*(T)\cap V'|&\leq P'|T|R_{\max}^{(a)}=O(|T|M/\alpha'),\\|\b(T)|&\leq|T|\b_{\max}=O(|T|M/\alpha).
		\end{align}
		Since $R^*(T)$ is disjoint from $F$, Lemma~\ref{lemma_terms}~(\ref{lemma_terms_f}) gives
		\begin{align}
			|P|F|-P'|F\cap V'||
			&=|(P-P')|F|+P'|F\cap\overline{N}(T)||
			\\&\leq|P-P'|\alpha F_{\max}+P'|T|F_{\max}=O(|T|M/\alpha').
		\end{align}
		Next, by Proposition~\ref{proposition_n}, we have $d(G)=O(\alpha)$, giving $|\overline{N}(T)|=O(\alpha|T|)$, and $|V|=O(\alpha^2)$. We get
		\begin{align}
			|P\b(V)-P'\b(V')|
			&=|(P-P')\b(V)+P'\b(\overline{N}(T))|
			\\&\leq|P-P'||V|\b_{\max}+P'|\overline{N}(T)|\b_{\max}=O(|T|M/\alpha').
		\end{align}
		Finally, we have
		\begin{align}
			&\sum_{v\in V'}|P|R^*(v)|-P'|R^*(v)\cap V'||
			\\&=\sum_{v\in V'}|(P-P')|R^*(v)|+P'|R^*(v)\cap\overline{N}(T)||
			\\&\leq|P-P'||V|R_{\max}^{(a)}+P'\sum_{v\in V}|R^*(v)\cap\overline{N}(T)|.
		\end{align}
		This is $O(|T|\alpha R_{\max}^{(a)}/\alpha')$, since $$\sum_{v\in V}|R^*(v)\cap\overline{N}(T)|=\sum_{w\in\overline{N}(T)}|R^*(w)|\leq|\overline{N}(T)|R_{\max}^{(a)}.$$ Including the factor $\frac12P'=O(1/\alpha')$, we find that this term is also $O(|T|M'/\alpha')$.
		
		(\ref{proposition_qt_uniform}) This follows by plugging $\b(v)=\frac12P|R^*(v)|$ into the definition of $\mathscr{Q}(T,U)$, and by Corollary~\ref{corollary_uniform}.
	\end{proof}
	
	\begin{proof}[Proof of Proposition~\ref{proposition_leniency}]
		Let $S$ be an FMIS of $G$. We first show for each class of weight functions $\b'$ that substituting $\b'$ for $\b$ in the definition of $\mathscr{P}(S)$ multiplies its value by a factor of $1+O(M^2/\alpha)$.
		
		(\ref{proposition_leniency_f}) This follows since $S$ is disjoint from $F$ and $|P\b(F)-P\b'(F)|\leq P|F|(\b_{\max}+\b'_{\max})=O(M^2/\alpha)$ by Lemma~\ref{lemma_terms}~(\ref{lemma_terms_f}).
		
		(\ref{proposition_leniency_log}) We use $\log(1-x)=-x+O(x^2)$ for $x=o(1)$. Since $\b_{\max}=o(1)$, we get $|\b(v)-\b'(v)|=O(\b_{\max}^2)$ for any $v\in V$. It follows that $\b'_{\max}=(1+O(\b_{\max}))\b_{\max}$, $\b(S)-\b'(S)=O(M^2/\alpha)$, and $P\b(V)-P\b'(V)=O(M^2/\alpha)$ by Lemma~\ref{lemma_terms}~(\ref{lemma_terms_n}).
		
		(\ref{proposition_leniency_b}) Since $G$ has regular independent sets, from the IMIS process, by Proposition~\ref{proposition_imis}, every vertex is in the same number of MISs. It follows that $P(\b(V)-\b'(V))$ is the arithmetic mean of $\b(T)-\b'(T)$ over all possible MISs $T$. We thus get $\b(S)-\b'(S)=P(\b(V)-\b'(V))$, so the value of $\mathscr{P}(S)$ actually stays the same.
		
		For the final part of the statement, consider a weight function $\b'$ from any of the three classes. Let $\boldsymbol{S}_{\infty}$ and $\boldsymbol{S}'_{\infty}$ denote the results from the IMFIS process on $(G,F,R,\b)$ and $(G,F,R,\b')$ respectively. By Theorem~\ref{theorem_main}, for any FMIS $S$ of $G$, we get $$\frac{\mathbb{P}(\boldsymbol{S}'_{\infty}=S)}{\mathbb{P}(\boldsymbol{S}_{\infty}=S)}=\frac{(1+O(\mathscr{E}(\alpha,M)))\mathscr{P}(S)}{(1+O(\mathscr{E}(\alpha,M)))\mathscr{P}(S)}=1+O(\mathscr{E}(\alpha,M)).$$ Consider an FIS $T$ of $G$ and a vertex set $U$ satisfying the requirements of Theorem~\ref{theorem_space}. Let $\mathcal{G}(T,U)$ be the set of all FMISs of $G$ that contain $T$ and are disjoint from $U$, such that we get
		\begin{align}
			&\mathbb{P}(|\boldsymbol{S}'_{\infty}|=\alpha\wedge T\subset\boldsymbol{S}'_{\infty}\wedge U\cap\boldsymbol{S}_{\infty}'=\emptyset)
			\\&=\sum_{S\in\mathcal{G}(T,U)}\mathbb{P}(\boldsymbol{S}'_{\infty}=S)
			\\&=(1+O(\mathscr{E}(\alpha,M)))\sum_{S\in\mathcal{G}(T,U)}\mathbb{P}(\boldsymbol{S}_{\infty}=S)
			\\&=(1+O(\mathscr{E}(\alpha,M)))\mathbb{P}(|\boldsymbol{S}_{\infty}|=\alpha\wedge T\subset\boldsymbol{S}_{\infty}\wedge U\cap\boldsymbol{S}_{\infty}=\emptyset).
		\end{align}
		By Theorem~\ref{theorem_space}, the right hand side approximates $\mathscr{Q}(T,U)$ up to a factor of $1+O(\mathscr{E}(\alpha',M'))$, and the left hand side does the same, but with $\b'$ substituted for $\b$ in the definition of $\mathscr{Q}(T,U)$. It follows that substituting $\b'$ for $\b$ in the definition of $\mathscr{Q}(T,U)$ multiplies its value by a factor of $1+O(\mathscr{E}(\alpha',M'))$.
	\end{proof}
	
	\subsection{Framework}\label{subsection_proof_framework}
	
	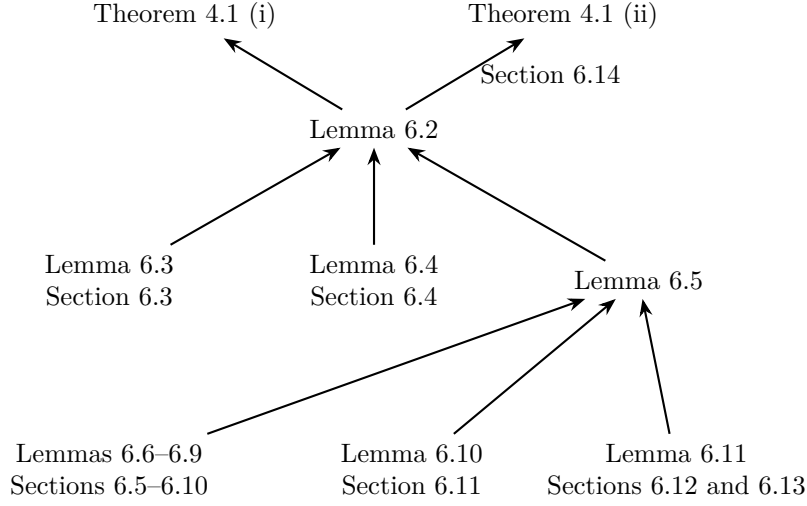
\begin{figure}[htbp]
		\centering
		\begin{tikzpicture}[
			>=Stealth,
			every node/.style={font=\normalsize, align=center},
			thick
			]
			% --- Top level ---
			%\node (thm) at (2.5, 6)      {Theorem~\ref{theorem_main}};
			
			% --- Second level ---
			\node (i)   at (0, 4.5)      {Theorem~\ref{theorem_main}~(\ref{theorem_main_equation})};
			\node (ii)  at (5, 4.5)      {Theorem~\ref{theorem_main}~(\ref{theorem_main_termination})};
			
			% --- Third level ---
			\node (lem2) at (2.5, 3)     {Lemma~\ref{lemma_generalization}};
			
			% --- Fourth level ---
			\node (lem3) at (-1, 1)      {Lemma~\ref{lemma_expectation} \\ Section~\ref{subsection_proof_simplified}};
			\node (lem4) at (2.5, 1)     {Lemma~\ref{lemma_product} \\ Section~\ref{subsection_proof_product}};
			\node (lem5) at (6, 1)       {Lemma~\ref{lemma_concentration}};
			
			% --- Fifth level ---
			\node (lem68)  at (-1, -1.5) {Lemmas~\ref{lemma_bar_e}--\ref{lemma_sigma} \\ Sections~\ref{subsection_proof_concentration}--\ref{subsection_proof_psi1}};
			\node (lem9)   at (3, -1.5)  {Lemma~\ref{lemma_remaining_p} \\ Section~\ref{subsection_proof_remaining}};
			\node (lem10)  at (6.5, -1.5) {Lemma~\ref{lemma_remaining_e} \\ Sections~\ref{subsection_proof_further} and \ref{subsection_proof_rpe}};
			
			% --- Arrows: top level ---
			%\draw[->] (i) -- (thm);
			%\draw[->] (ii) -- (thm);
			
			% --- Arrows: second level ---
			\draw[->] (lem2) -- (i);
			\draw[->] (lem2) -- (ii) node[midway, right, xshift=2pt] {Section~\ref{subsection_proof_termination}};
			
			% --- Arrows: third level ---
			\draw[->] (lem3) -- (lem2);
			\draw[->] (lem4) -- (lem2);
			\draw[->] (lem5) -- (lem2);
			
			% --- Arrows: fourth level (all converge into Lemma~6.5) ---
			\draw[->] (lem68) -- (lem5);
			\draw[->] (lem9)  -- (lem5);
			\draw[->] (lem10) -- (lem5);
			
		\end{tikzpicture}
		\caption{
			Schematic representation of the proof structure for Theorem~\ref{theorem_main}. Arrows indicate logical dependencies between components. Anything labeled with section numbers is shown in those sections. The rest is established here, in Section~\ref{subsection_proof_framework}.}
		\label{fig:proof-framework}
	\end{figure}
	We show that Theorem~\ref{theorem_main} is implied by a series of lemmas, the proofs of which are spread across Sections \ref{subsection_proof_simplified}-\ref{subsection_proof_rpe}.
	Figure~\ref{fig:proof-framework} gives a schematic representation of the proof structure.
	%We give a framework for the proof of Theorem~\ref{theorem_main}, breaking it down into a series of lemmas that we prove throughout the rest of this section. See Figure~\ref{fig:proof-framework} for a schematic representation.

	We thus start analysing the IMFIS process on $(G,F,R,\b)$, with $G=(V,E)$ a $2$-uniform graph with $\alpha\to\infty$, $\ell$ fixed, and $M=O(\alpha^{\sfrac12}/(\log\alpha)^2)$. Write $\mathscr{E}:=\mathscr{E}(\alpha,M)$. In order to prove Theorem~\ref{theorem_main}~(\ref{theorem_main_termination}), we need to study the event $|\boldsymbol{S}_{\infty}|<\alpha$, which requires that $(V\setminus\overline{N}(\boldsymbol{S}_r))\setminus(F\cup R^*(\boldsymbol{S}_r))$ is empty for some $r<\alpha$. We thus consider the following generalization of Theorem~\ref{theorem_main}~(\ref{theorem_main_equation}).
	
	\begin{lemma}\label{lemma_generalization}
		Let $S$ be an FIS of $G$ of size $\alpha-k$ for some $k=O(M)$. Let $K:=|(V\setminus\overline{N}(S))\cap(F\cup R^*(S))|$ be the number of vertices $v\in V$ for which $S\cup\{v\}$ is still an independent set, but not an FIS. Finally, define $$\mathscr{P}_k(S):=\exp(P|F|+\tfrac12P|R^*(S)|+P\b(V)-\b(S))\prod_{t=k+1}^{\alpha}\frac{t-k}{n_t-K}.$$
		\begin{enumerate}[(i)]
			\item\label{lemma_generalization_zero} If $k=0$, then $\mathbb{P}(\boldsymbol{S}_{\alpha-k}=S)=(1+O(\mathscr{E}))\mathscr{P}_k(S)$.
			\item\label{lemma_generalization_small} If $k\geq1$ with $k=O(1)$, then $\mathbb{P}(\boldsymbol{S}_{\alpha-k}=S)\leq(1+o(1))\mathscr{P}_k(S)$.
			\item\label{lemma_generalization_large} If $k\geq1$, then $\mathbb{P}(\boldsymbol{S}_{\alpha-k}=S)\leq\alpha^{o(k)}\mathscr{P}_k(S)$.
		\end{enumerate}
	\end{lemma}
	
	Note that Theorem~\ref{theorem_main}~(\ref{theorem_main_equation}) is indeed implied by Lemma~\ref{lemma_generalization}~(\ref{lemma_generalization_zero}). In Section~\ref{subsection_proof_termination}, we use some additional arguments to prove that Theorem~\ref{theorem_main}~(\ref{theorem_main_termination}) also follows from Lemma~\ref{lemma_generalization}. Until then, we focus on the proof of Lemma~\ref{lemma_generalization}, throughout which we assume that $S$ is an FIS of $G$ of size $\alpha-k$ for some $k=O(M)$. We also assume without loss of generality that $\b$ is non-negative. Indeed, we may consider substituting $\b+\b_{\max}$ for $\b$, which does not change the IMFIS process, since all weights are scaled by the same factor, and which changes $M$ by a factor of $O(1)$.
	
	We first show that Lemma~\ref{lemma_generalization} is implied by three more subordinate lemmas. This involves considering all possible sequences $(\boldsymbol{S}_r)_{r=0}^{\alpha-k}$ of the IMFIS process with $\boldsymbol{S}_{\alpha-k}=S$, and averaging their probabilities. By taking this average, any set $\boldsymbol{S}_r$ for $r=0,\ldots,\alpha-k$ is a uniformly random subset of $S$ of size $r$, which we then compare to a simplified model, where each element of $S$ is included with probability $r/|S|$.
	
	Let $\mathcal{P}$ be the set of all possible permutations $\pi=(s_0,\ldots,s_{\alpha-k-1})$ of the elements of $S$. Let $S_{\pi}$ denote the event that $\boldsymbol{S}_{r+1}=\boldsymbol{S}_r\cup\{s_r\}$ for all $r=0,\ldots,\alpha-k-1$, so that $$\mathbb{P}(\boldsymbol{S}_{\alpha-k}=S)=\sum_{\pi\in\mathcal{P}}\mathbb{P}(S_{\pi}).$$ Let $S_r(\pi):=\{s_0,\ldots,s_{r-1}\}$ be the value of $\boldsymbol{S}_r$ in the event $S_{\pi}$. Let $\widetilde{V}_r(\pi):=V\setminus\overline{N}(S_r(\pi))$ be the set of vertices $v\in V$ for which $S_r(\pi)\cup\{v\}$ is still an independent set. Note that $|\widetilde{V}_r(\pi)|=n_{\alpha-r}$ by Corollary~\ref{corollary_hereditary}. Finally, let $V_r(\pi):=\widetilde{V}_r(\pi)\setminus(F\cup R^*(S_r(\pi)))$ be the set of vertices $v\in V$ for which $S_r(\pi)\cup\{v\}$ is still an FIS. By repeated conditioning, the IMFIS process gives
	\begin{align}
		\mathbb{P}(S_{\pi})
		&=\prod_{r=0}^{\alpha-k-1}\mathbb{P}(\boldsymbol{S}_{r+1}=S_{r+1}(\pi)|\boldsymbol{S}_r=S_r(\pi))
		\\&=\prod_{r=0}^{\alpha-k-1}\frac{e^{-\b(s_r)}}{\sum_{v\in V_r(\pi)}e^{-\b(v)}}.
	\end{align}
	Note that $\b_{\max}=O(M/\alpha)=o(1)$, so since $e^{-x}=1-x+O(x^2)$ for $x=o(1)$, for $\Psi_r(\pi):=|\widetilde{V}_r(\pi)\setminus V_r(\pi)|+\b(V_r(\pi))$, the sum in the denominator is $$|V_r(\pi)|-\b(V_r(\pi))+O(|V_r(\pi)|\b_{\max}^2)=(1+O(M^2/\alpha^2))(n_{\alpha-r}-\Psi_r(\pi)),$$ because $n_{\alpha-r}-\Psi_r(\pi)=|V_r(\pi)|-\b(V_r(\pi))=\Omega(|V_r(\pi)|)$. Since $\log(1+x)=O(x)$ for $x=o(1)$, we get $$\mathbb{P}(S_{\pi})=\exp(-\b(S)+O(M^2/\alpha))\prod_{r=0}^{\alpha-k-1}\frac{1}{n_{\alpha-r}-\Psi_r(\pi)}.$$
	
	Now, suppose that $\pi$ is a uniformly random permutation from $\mathcal{P}$, such that
	\begin{align}
		\mathbb{P}(\boldsymbol{S}_{\alpha-k}=S)&=(\alpha-k)!\mathbb{E}(\mathbb{P}(S_{\pi}|\pi))
		\\&=(\alpha-k)!\exp(-\b(S)+O(M^2/\alpha))\mathbb{E}\left(\prod_{r=0}^{\alpha-k-1}\frac{1}{n_{\alpha-r}-\Psi_r(\pi)}\right).
	\end{align}
	Note that the objects $S_r(\pi)$, $\widetilde{V}_r(\pi)$, $V_r(\pi)$, and $\Psi_r(\pi)$ now correspond to random variables $S_r$, $\widetilde{V}_r$, $V_r$, and $\Psi_r$. Indeed, for any fixed value of $r$, the random variable $S_r$ is a uniformly random subset of $S$ of size $r$, and the random variables $\widetilde{V}_r$, $V_r$ and $\Psi_r$ are derived from $S_r$. We refer to this construction as the \emph{true model}.

	To define the simplified model, we now approximate the distribution on $S_r$ in the true model by including each element of $S$ independently with probability $p=p_r:=r/(\alpha-k)$. For general $p\in[0,1]$, let $S_p$ be the resulting random subset of $S$. From $S_p$, we can derive the random variables $\widetilde{V}_p$, $V_p$, and $\Psi_p$, using the same definitions as in the true model. We refer to this construction as the \emph{simplified model}, and the subscript $p$ is used to distinguish it from the true model. Furthermore, we write $q:=1-p$ and $q_r:=1-p_r$.
	
	For $p\in[0,1]$, let $\psi_p:=\mathbb{E}(\Psi_p)$. We get $$\prod_{r=0}^{\alpha-k-1}\frac1{n_{\alpha-r}-\Psi_r}=\prod_{r=0}^{\alpha-k-1}\frac{n_{\alpha-r}-\psi_{p_r}}{n_{\alpha-r}-\Psi_r}\prod_{r=0}^{\alpha-k-1}\frac{n_{\alpha-r}-K}{n_{\alpha-r}-\psi_{p_r}}\prod_{t=k+1}^{\alpha}\frac1{n_t-K},$$ such that the following three lemmas suffice to prove Lemma~\ref{lemma_generalization}.
	
	\begin{lemma}\label{lemma_expectation}
		For $p\in[0,1]$, we have the expected value $$\psi_p=K+q^2(|F|+p|R^*(S)|+\b(V))+O(qM^2+(n_k-K)M/\alpha).$$
	\end{lemma}
	
	\begin{lemma}\label{lemma_product}
		We have the product estimate $$\prod_{r=0}^{\alpha-k-1}\frac{n_{\alpha-r}-K}{n_{\alpha-r}-\psi_{p_r}}=\exp(P|F|+\tfrac12P|R^*(S)|+P\b(V)+O(M^2\log\alpha/\alpha)).$$
	\end{lemma}
	
	\begin{lemma}\label{lemma_concentration}
		Consider the random variable $$\Pi:=\prod_{r=0}^{\alpha-k-1}\frac{n_{\alpha-r}-\psi_{p_r}}{n_{\alpha-r}-\Psi_r}.$$
		\begin{enumerate}[(i)]
			\item\label{lemma_concentration_zero} If $k=0$, then $\mathbb{E}(\Pi)=1+O(\mathscr{E})$.
			\item\label{lemma_concentration_small} If $k\geq1$ with $k=O(1)$, then $\mathbb{E}(\Pi)\leq1+o(1)$.
			\item\label{lemma_concentration_large} If $k\geq1$, then $\mathbb{E}(\Pi)\leq\alpha^{o(k)}$.
		\end{enumerate}
	\end{lemma}
	
	While Lemmas \ref{lemma_expectation} and \ref{lemma_product} are proven in Sections \ref{subsection_proof_simplified} and \ref{subsection_proof_product} respectively, the proof of Lemma~\ref{lemma_concentration} is more involved. In the rest of this section, we show that Lemma~\ref{lemma_concentration} is implied by six more subordinate lemmas.
	
	We consider the logarithm
	\begin{align}
		\log\Pi=\sum_{r=0}^{\alpha-k-1}\log(1+\widetilde{Q}_r),&&\widetilde{Q}_r:=\frac{\Psi_r-\psi_{p_r}}{n_{\alpha-r}-\Psi_r}.
	\end{align}
	Since $\log(1+x)\leq(x)^+:=\max\{x,0\}$ for $1+x>0$, this motivates studying $\log\Pi$ in terms of the sums
	\begin{align}
		\Sigma&:=\sum_{r=0}^{\alpha-k-1}Q_r,&Q_r&:=\frac{|\Psi_r-\psi_{p_r}|}{n_{\alpha-r}-n_k},
		\\\hat{\Sigma}&:=\sum_{r=0}^{\alpha-k-1}\hat{Q}_r,&\hat{Q}_r&:=\frac{(\Psi_r-\psi_{p_r})^+}{n_{\alpha-r}-n_k}.
	\end{align}
	In order to compare $\log\Pi$ with $\Sigma$ and $\hat{\Sigma}$, for a parameter $0<\delta<1$, we consider the event $$\mathcal{P}^{\delta}:=\{n_{\alpha-r}-\Psi_r\geq\delta(n_{\alpha-r}-n_k),\ \forall r=0,\ldots,\alpha-k-1\}\subset\mathcal{P}.$$ In $\mathcal{P}^{\delta}$, we find that $|\widetilde{Q}_r|\leq\frac1{\delta}Q_r$ and $(\widetilde{Q}_r)^+\leq\frac1{\delta}\hat{Q}_r$. We thus have
	\begin{equation}\label{equation_pdelta}
		\Pi\leq e^{\frac1{\delta}\hat{\Sigma}}\leq e^{\frac1{\delta}\Sigma}\ \text{in}\ \mathcal{P}^{\delta}.
	\end{equation}
	
	For the case $k=0$, we need to be more precise. We define
	\begin{align}
		\overline{\Sigma}:=\sum_{r=0}^{\alpha-1}\overline{Q}_r,&&\overline{Q}_r:=\frac{\Psi_r-\psi_{p_r}}{n_{\alpha-r}}.
	\end{align}
	The idea is to approximate $\Pi$ with $1+\overline{\Sigma}$, so we write $$\mathbb{E}(\Pi)=1+\mathbb{E}(\overline{\Sigma})+\mathbb{E}(\Pi-(1+\overline{\Sigma})).$$
	
	For an event $\mathcal{Q}\subset\mathcal{P}$, let $1_{\mathcal{Q}}$ denote its indicator function. We show that Lemma~\ref{lemma_concentration} follows from the following six lemmas.
	
	\begin{lemma}\label{lemma_bar_e}
		If $k=0$, we have $\mathbb{E}(\overline{\Sigma})=O(M\log\alpha/\alpha)$.
	\end{lemma}
	
	\begin{lemma}\label{lemma_bar_bound}
		If $k=0$, we have $|1+\overline{\Sigma}|=O(M\log\alpha)$.
	\end{lemma}
	
	\begin{lemma}\label{lemma_delta}
		If $k=0$, then for any $\Omega(1)\leq\delta\leq1-\Omega(1)$, we have $\mathbb{E}((\Pi-(1+\overline{\Sigma}))1_{\mathcal{P}^{\delta}})=O(\mathscr{E})$.
	\end{lemma}
	
	\begin{lemma}\label{lemma_sigma}
		Let $C=\Omega(1)$ be a parameter.
		\begin{enumerate}[(i)]
			\item\label{lemma_sigma_zero} If $k=0$ and $$C\leq\left(8\left(1-\tfrac{\log M}{\log\alpha}\right)+\Omega(1)\right)^{-1}\tfrac{\log\alpha}{\log\log\alpha},$$ then for any event $\mathcal{Q}\subset\mathcal{P}^{\delta}$ with $\delta:=\tfrac1{C}$, we have $$\mathbb{E}(\Pi1_{\mathcal{Q}})\leq(1+o(1))\mathbb{P}(\mathcal{Q})+O(\mathscr{E}).$$
			\item\label{lemma_sigma_large} If $k\geq1$ and $$C\leq\left(8\left(1-\tfrac{\log M}{\log\alpha}\right)+2+\Omega(1)\right)^{-1}\tfrac{\log\alpha}{\log\log\alpha},$$ then $\mathbb{E}(e^{C\hat{\Sigma}})\leq e^{o(k)}$.
		\end{enumerate}
	\end{lemma}
	
	\begin{lemma}\label{lemma_remaining_p}
		For $0<\delta\leq1-\Omega(1)$, if $k=0$, we have $\mathbb{P}(\mathcal{P}\setminus\mathcal{P}^{\delta})=O(M/\alpha)$.
	\end{lemma}
	
	\begin{lemma}\label{lemma_remaining_e}
		Let $0<\delta<1$ with $\delta=O(\tfrac{\log\log\alpha}{\log\alpha})$, such that there exist $\mu_1,\mu_2>0$ with $\mu_2=\Omega(1)$, $\mu_1+\mu_2=1$, and
		\begin{align}
			\delta&\geq\tfrac1{\mu_1}\left(8\left(1-\tfrac{\log M}{\log\alpha}\right)+2+\Omega(1)\right)\tfrac{\log\log\alpha}{\log\alpha},
			\\\delta\left(1-\tfrac{\log M}{\log\alpha}\right)&\leq4\left(\tfrac12-\tfrac{\log M}{\log\alpha}\right)-(\tfrac{1}{\mu_2}+\Omega(1))\tfrac{\log\log\alpha}{\log\alpha}.
		\end{align}
		Then we have $$\mathbb{E}(\Pi1_{\mathcal{P}\setminus\mathcal{P}^{\delta}})\leq\tfrac{M}{\alpha}\exp\left(O\left(\tfrac{k\log\alpha}{\log\log\alpha}\right)\right).$$
	\end{lemma}
	
	\begin{proof}[Proof of Lemma~\ref{lemma_concentration}]
		Let $\delta_1:=\frac12$ and $$\delta_2:=\tfrac{5}{3}\left(8\left(1-\tfrac{\log M}{\log\alpha}\right)+2+\tfrac15\right)\tfrac{\log\log\alpha}{\log\alpha},$$ and assume $\alpha$ is large enough, such that $\delta_2\leq\delta_1$. Note that $C=\tfrac1{\delta_2}$ fulfills the requirements of Lemma~\ref{lemma_sigma}~(\ref{lemma_sigma_zero},\ref{lemma_sigma_large}), and $\delta_1$ fulfills the requirement of Lemma~\ref{lemma_remaining_p}.
		
		Consider $\mu_1:=\tfrac35$ and $\mu_2:=\tfrac25$, such that $$\delta_2\geq\tfrac1{\mu_1}\left(8\left(1-\tfrac{\log M}{\log\alpha}\right)+2+\Omega(1)\right)\tfrac{\log\log\alpha}{\log\alpha}.$$ We furthermore have
		\begin{align}
			\tfrac12-\tfrac{\log M}{\log\alpha}
			&\geq\tfrac{2\log\log\alpha-O(1)}{\log\alpha},
			\\\tfrac{\log\log\alpha}{\log\alpha}
			&\leq(\tfrac12+o(1))\left(\tfrac12-\tfrac{\log M}{\log\alpha}\right),
			\\\delta_2
			&=\tfrac{5}{3}(4+2+\tfrac15)\tfrac{\log\log\alpha}{\log\alpha}+\tfrac{40}{3}\tfrac{\log\log\alpha}{\log\alpha}\left(\tfrac12-\tfrac{\log M}{\log\alpha}\right)
			\\&=\tfrac{31}{3}\tfrac{\log\log\alpha}{\log\alpha}+o\left(\tfrac12-\tfrac{\log M}{\log\alpha}\right),
			\\\delta_2\left(1-\tfrac{\log M}{\log\alpha}\right)
			&=\tfrac{31}{3}\left(\tfrac12+\left(\tfrac12-\tfrac{\log M}{\log\alpha}\right)\right)\tfrac{\log\log\alpha}{\log\alpha}+o\left(\tfrac12-\tfrac{\log M}{\log\alpha}\right)
			\\&=\tfrac{31}{6}\tfrac{\log\log\alpha}{\log\alpha}+o\left(\tfrac12-\tfrac{\log M}{\log\alpha}\right)
			\\&=(8-\tfrac16)\tfrac{\log\log\alpha}{\log\alpha}-(\tfrac1{\mu_2}+\tfrac16)\tfrac{\log\log\alpha}{\log\alpha}+o\left(\tfrac12-\tfrac{\log M}{\log\alpha}\right)
			\\&\leq(4-\tfrac1{12}+o(1))\left(\tfrac12-\tfrac{\log M}{\log\alpha}\right)-(\tfrac1{\mu_2}+\tfrac16)\tfrac{\log\log\alpha}{\log\alpha}.
		\end{align}
		We find that $\delta_2$ fulfills the requirements of Lemma~\ref{lemma_remaining_e}.
		
		(\ref{lemma_concentration_zero}) Using $\mathcal{P}^{\delta_1}\subset\mathcal{P}^{\delta_2}$, we split up
		\begin{align}
			\mathbb{E}(\Pi)
			&=1+\mathbb{E}(\overline{\Sigma})+\mathbb{E}(\Pi-(1+\overline{\Sigma}))
			\\&=1+\mathbb{E}(\overline{\Sigma})+\mathbb{E}(\Pi1_{\mathcal{P}\setminus\mathcal{P}^{\delta_2}})+\mathbb{E}(\Pi1_{\mathcal{P}^{\delta_2}\setminus\mathcal{P}^{\delta_1}})
			\\&\quad-\mathbb{E}((1+\overline{\Sigma})1_{\mathcal{P}\setminus\mathcal{P}^{\delta_1}})+\mathbb{E}((\Pi-(1+\overline{\Sigma}))1_{\mathcal{P}^{\delta_1}}).
		\end{align}
		Then Lemma~\ref{lemma_bar_e} gives $\mathbb{E}(\overline{\Sigma})=O(M\log\alpha/\alpha)$, Lemma~\ref{lemma_delta} gives $\mathbb{E}((\Pi-(1+\overline{\Sigma}))1_{\mathcal{P}^{\delta_1}})=O(\mathscr{E})$, and Lemma~\ref{lemma_bar_bound} and Lemma~\ref{lemma_remaining_p} give  $$\mathbb{E}((1+\overline{\Sigma})1_{\mathcal{P}\setminus\mathcal{P}^{\delta_1}})=O(M\log\alpha\cdot\mathbb{P}(\mathcal{P}\setminus\mathcal{P}^{\delta_1}))=O(M^2\log\alpha/\alpha).$$ Next, Lemma~\ref{lemma_sigma}~(\ref{lemma_sigma_zero}) and Lemma~\ref{lemma_remaining_p} give $$\mathbb{E}(\Pi1_{\mathcal{P}^{\delta_2}\setminus\mathcal{P}^{\delta_1}})\leq(1+o(1))\mathbb{P}(\mathcal{P}\setminus\mathcal{P}^{\delta_1})+O(\mathscr{E})=O(\mathscr{E}).$$
		Finally, by Lemma~\ref{lemma_remaining_e}, we have $\mathbb{E}(\Pi1_{\mathcal{P}\setminus\mathcal{P}^{\delta_2}})=O(M/\alpha)$.
		
		(\ref{lemma_concentration_small},\ref{lemma_concentration_large}) By \eqref{equation_pdelta}, Lemma~\ref{lemma_sigma}~(\ref{lemma_sigma_large}), and Lemma~\ref{lemma_remaining_e}, we have
		\begin{align}
			\mathbb{E}(\Pi)&=\mathbb{E}(\Pi1_{\mathcal{P}^{\delta_2}})+\mathbb{E}(\Pi1_{\mathcal{P}\setminus\mathcal{P}^{\delta_2}})
			\\&\leq\mathbb{E}(e^{\frac1{\delta_2}\hat{\Sigma}})+\mathbb{E}(\Pi1_{\mathcal{P}\setminus\mathcal{P}^{\delta_2}})
			\\&\leq e^{o(k)}+\alpha^{o(k)-\Omega(1)},
		\end{align}
		so the results follow.
	\end{proof}
	
	Throughout Sections~\ref{subsection_proof_concentration}--\ref{subsection_proof_psi1}, we prove Lemmas~\ref{lemma_bar_e}--\ref{lemma_sigma}, in Section~\ref{subsection_proof_remaining}, we prove Lemma~\ref{lemma_remaining_p}, and in Sections~\ref{subsection_proof_further} and \ref{subsection_proof_rpe}, we prove Lemma~\ref{lemma_remaining_e}. Together, these conclude the proof of Lemma~\ref{lemma_concentration}, Lemma~\ref{lemma_generalization}, and therefore, the proof of Theorem~\ref{theorem_main}~(\ref{theorem_main_equation}).
	
	\subsection{Simplified model}\label{subsection_proof_simplified}
	
	We prove Lemma~\ref{lemma_expectation}. We thus need to determine the expected value $\mathbb{E}(\Psi_p)$ for $p\in[0,1]$. We split up $\Psi_p$ into all of its relevant components. We write $\Delta_p:=|\widetilde{V}_p\setminus V_p|$ and $\Lambda_p:=\b(V_p)$, such that $\Psi_p=\Delta_p+\Lambda_p$. For a set of vertices $U\subset V$, we write $\Delta_p(U):=|\widetilde{V}_p\cap U|$ and $\Lambda_p(U):=\b(\widetilde{V}_p\cap U)$, such that
	\begin{align}
		\Delta_p&=\Delta_p(F)+\Delta_p(R^*(S_p)),
		\\\Lambda_p&=\Lambda_p(V\setminus S)+\Lambda_p(S)-\Lambda_p(F)-\Lambda_p(R^*(S_p)).
	\end{align}
	Note that this relies on $R^*(S)$ being disjoint from $F$ by definition. If $k=0$, then $S$ is an MIS, so each vertex outside $S$ is adjacent to two elements of $S$ by $2$-uniformity. However, if $k>0$, this is no longer true, and we have to consider the vertices outside $S$ adjacent to a different number of elements of $S$ separately. We thus write $V^{(i)}$, $F^{(i)}$, $R^{(i)}$, and $R_p^{(i)}$ for the sets of vertices from $V\setminus S$, $F$, $R^*(S)$, and $R^*(S_p)$ respectively with exactly $i$ neighbors in $S$.
	
	\begin{lemma}\label{lemma_vi}
		The set $V\setminus S$ is the disjoint union of $V^{(2)}$, $V^{(1)}$, and $V^{(0)}$.
	\end{lemma}
	
	\begin{proof}
		Let $\overline{S}$ be a maximal independent set that contains $S$, such that $\overline{S}$ is an MIS by Proposition \ref{proposition_imis}. Since $G$ is $2$-uniform, any vertex in $V\setminus S$ is either in $\overline{S}\setminus S$ and thus in $V^{(0)}$, or in $V\setminus\overline{S}$ and thus adjacent to exactly two elements of $\overline{S}$, so at most two elements of $S$.
	\end{proof}
	
	Since $F$, $R^*(S)$ and $R^*(S_p)$ are subsets of $V\setminus S$, we get
	\begin{align}
		\Delta_p(F)&=\Delta_p(F^{(2)})+\Delta_p(F^{(1)})+\Delta_p(F^{(0)}),
		\\\Delta_p(R^*(S_p))&=\Delta_p(R^{(2)}_p)+\Delta_p(R^{(1)}_p)+\Delta_p(R^{(0)}_p),
		\\\Lambda_p(V\setminus S)&=\Lambda_p(V^{(2)})+\Lambda_p(V^{(1)})+\Lambda_p(V^{(0)}),
		\\\Lambda_p(F)&=\Lambda_p(F^{(2)})+\Lambda_p(F^{(1)})+\Lambda_p(F^{(0)}),
		\\\Lambda_p(R^*(S_p))&=\Lambda_p(R^{(2)}_p)+\Lambda_p(R^{(1)}_p)+\Lambda_p(R^{(0)}_p).
	\end{align}
	We calculate $\mathbb{E}(\Psi_p)$ by adding up the expected value of each relevant component.
	
	\begin{lemma}\label{lemma_expect}
		For $p\in[0,1]$ and $i\in\{2,1,0\}$, we have
		\begin{enumerate}[(i)]
			\item\label{lemma_expect_df} $\begin{aligned}[t]
				\mathbb{E}(\Delta_p(F^{(i)}))=q^i|F^{(i)}|,
			\end{aligned}$
			\item\label{lemma_expect_dr} $\begin{aligned}[t]
				\mathbb{E}(\Delta_p(R^{(i)}_p))=pq^i|R^{(i)}|,
			\end{aligned}$
			\item\label{lemma_expect_lv} $\begin{aligned}[t]
				\mathbb{E}(\Lambda_p(V^{(i)}))=q^i\b(V^{(i)}),
			\end{aligned}$
			\item\label{lemma_expect_ls} $\begin{aligned}[t]
				\mathbb{E}(\Lambda_p(S))=q\b(S),
			\end{aligned}$
			\item\label{lemma_expect_lf} $\begin{aligned}[t]
				\mathbb{E}(\Lambda_p(F^{(i)}))=q^i\b(F^{(i)}),
			\end{aligned}$
			\item\label{lemma_expect_lr} $\begin{aligned}[t]
				\mathbb{E}(\Lambda_p(R^{(i)}_p))=pq^i\b(R^{(i)}).
			\end{aligned}$
		\end{enumerate}
	\end{lemma}
	
	\begin{proof}
		(\ref{lemma_expect_df},\ref{lemma_expect_lv},\ref{lemma_expect_lf}) These follow, since any vertex in $F^{(i)}$ or $V^{(i)}$ is in $\widetilde{V}_p$ exactly if all $i$ neighbors in $S$ are not in $S_p$.
		
		(\ref{lemma_expect_dr},\ref{lemma_expect_lr}) Consider a vertex $v\in R^{(i)}$, such that $v\in R^*(s)$ for some $s\in S$. Then $v\in\widetilde{V}_p$ holds exactly if all $i$ neighbors in $S$ are not in $S_p$, and $v\in R^*(S_p)$ holds exactly if $s\in S_p$. The results follow, since these events are independent, because $s$ is not a neighbor of $v$ by definition of $R^*(s)$.
		
		(\ref{lemma_expect_ls}) This follows, since any vertex $s\in S$ is in $\widetilde{V}_p$ exactly if $s\not\in S_p$.
	\end{proof}
	
	In order to prove Lemma~\ref{lemma_expectation} from this, we need to bound some of these quantities. We first briefly show some fundamental properties of the value $n_t$ for integer $t\geq0$.
	
	\begin{lemma}\label{lemma_nr}
		For integers $t,t'\geq0$, we have
		\begin{enumerate}[(i)]
			\item\label{lemma_nr_geq} $\begin{aligned}[t]
				n_t\geq t^2,
			\end{aligned}$
			\item\label{lemma_nr_leq} $\begin{aligned}[t]
				n_t\leq\tfrac{\ell}{2}t^2,
			\end{aligned}$
			\item\label{lemma_nr_sum} $\begin{aligned}[t]
				n_{t+t'}=n_t+n_{t'}+\ell tt'.
			\end{aligned}$
		\end{enumerate}
	\end{lemma}
	
	\begin{proof}
		(\ref{lemma_nr_geq}) Since $\ell\geq2$, we have $n_t=\ell\binom{t}{2}+t\geq2\binom{t}{2}+t=t^2$.
		
		(\ref{lemma_nr_leq}) Since $\ell\geq2$, we have $n_t=\tfrac{\ell}{2}t^2+(1-\tfrac{\ell}{2})t\leq\tfrac{\ell}{2}t^2$.
		
		(\ref{lemma_nr_sum}) This can be verified algebraically from the definition.
	\end{proof}
	
	\begin{remark}
		There is also a combinatorial interpretation of Lemma~\ref{lemma_nr}~(\ref{lemma_nr_sum}). Consider a $2$-uniform graph $H$ with $\ell(H)=\ell$ and with an MIS $A\cup B$ with $|A|=t$ and $|B|=t'$. Then $V(H)$ is the disjoint union of $V(H)\setminus\overline{N}(A)$, $V(H)\setminus\overline{N}(B)$, and $\overline{N}(A)\cap\overline{N}(B)$. The equation then follows from Proposition~\ref{proposition_n}, Corollary~\ref{corollary_hereditary}, and Proposition~\ref{proposition_ell}.
	\end{remark}
	
	\begin{lemma}\label{lemma_bounds1}
		We have the bounds
		\begin{enumerate}[(i)]
			\item\label{lemma_bounds1_f2} $\begin{aligned}[t]
				|F^{(2)}|\leq|F|=O(\alpha M),
			\end{aligned}$
			\item\label{lemma_bounds1_r2} $\begin{aligned}[t]
				|R^{(2)}|\leq|R^*(S)|=O(\alpha M),
			\end{aligned}$
			\item\label{lemma_bounds1_f1} $\begin{aligned}[t]
				|F^{(1)}|\leq kF_{\max}=O(M^2),
			\end{aligned}$
			\item\label{lemma_bounds1_r1} $\begin{aligned}[t]
				|R^{(1)}|\leq kR^{(b)}_{\max}=O(M^2),
			\end{aligned}$
			\item\label{lemma_bounds1_v1} $\begin{aligned}[t]
				|V^{(1)}|=\ell k(\alpha-k)=O(\alpha M),
			\end{aligned}$
			\item\label{lemma_bounds1_v0} $\begin{aligned}[t]
				K=|F^{(0)}|+|R^{(0)}|\leq|V^{(0)}|=n_k=O(k^2)=O(M^2).
			\end{aligned}$
		\end{enumerate}
	\end{lemma}
	
	\begin{proof}
		Let $\overline{S}$ be a maximal independent set that contains $S$, such that $\overline{S}$ is an MIS by Proposition \ref{proposition_imis}. Since $G$ is $2$-uniform, any vertex in $V^{(1)}$ neighbors exactly one vertex in $S$ and one vertex in $\overline{S}\setminus S$.
		
		(\ref{lemma_bounds1_f2},\ref{lemma_bounds1_r2}) These follow from Lemma~\ref{lemma_terms}~(\ref{lemma_terms_f},\ref{lemma_terms_rs}), because $F^{(2)}\subset F$ and $R^{(2)}\subset R^*(S)$.
		
		(\ref{lemma_bounds1_f1},\ref{lemma_bounds1_r1}) Every element of $F^{(1)}$ or $R^{(1)}$ neighbors exactly one vertex in $\overline{S}\setminus S$, so $|F^{(1)}|\leq F_{\max}|\overline{S}\setminus S|$ and $|R^{(1)}|\leq R^{(b)}_{\max}|\overline{S}\setminus S|$.
		
		(\ref{lemma_bounds1_v1}) Every element of $V^{(1)}$ neighbors exactly one vertex in $S$ and one vertex in $\overline{S}\setminus S$. The result thus follows from Proposition~\ref{proposition_ell}.
		
		(\ref{lemma_bounds1_v0}) We have $V^{(0)}=V\setminus\overline{N}(S)$, which has size $n_k$ by Corollary~\ref{corollary_hereditary}, which is $O(k^2)$ by Lemma~\ref{lemma_nr}~(\ref{lemma_nr_leq}). We also find that $F^{(0)}$ and $R^{(0)}$ are disjoint subsets of $V^{(0)}$ with combined cardinality $K$ by definition.
	\end{proof}
	
	We can now prove Lemma~\ref{lemma_expectation}.
	
	\begin{proof}[Proof of Lemma~\ref{lemma_expectation}]
		We combine Lemma~\ref{lemma_expect} and Lemma~\ref{lemma_bounds1}, and we use $p,q\leq1$, $M\geq1$, and $\b_{\max}\leq M/\alpha$. We get
		\begin{align}
			&\mathbb{E}(\Delta_p(F^{(2)}))=q^2(|F|-|F^{(0)}|-|F^{(1)}|)=q^2|F|+O(qM^2),
			\\&\mathbb{E}(\Delta_p(R^{(2)}_p))=pq^2(|R^*(S)|-|R^{(0)}|-|R^{(1)}|)=pq^2|R^*(S)|+O(qM^2),
			\\&\mathbb{E}(\Lambda_p(V^{(2)}))=q^2(\b(V)-\b(S)-\b(V^{(1)})-\b(V^{(0)}))=q^2\b(V)+O(qM^2),
			\\&\mathbb{E}(\Delta_p(F^{(1)})+\Delta_p(R^{(1)}_p))=q(|F^{(1)}|+p|R^{(1)}|)=O(qM^2),
			\\&\mathbb{E}(\Lambda_p(V^{(1)})+\Lambda_p(S)-\Lambda_p(F^{(2)})-\Lambda_p(R^{(2)}_p)-\Lambda_p(F^{(1)})-\Lambda_p(R^{(1)}_p))
			\\&=q(\b(V^{(1)})+\b(S)-q\b(F^{(2)})-pq\b(R^{(2)})-\b(F^{(1)})-p\b(R^{(1)}))
			\\&=O(qM^2),
			\\&\mathbb{E}(\Delta_p(F^{(0)})+\Delta_p(R^{(0)}_p))=|F^{(0)}|+p|R^{(0)}|=K-q|R^{(0)}|=K+O(qM^2),
			\\&\mathbb{E}(\Lambda_p(V^{(0)})-\Lambda_p(F^{(0)})-\Lambda_p(R^{(0)}_p))=\b(V^{(0)})-\b(F^{(0)})-p\b(R^{(0)})
			\\&=\b(V^{(0)}\setminus(F^{(0)}\cup R^{(0)}))+q\b(R^{(0)})=O((n_k-K)M/\alpha+qM^2).
		\end{align}
		The result follows.
	\end{proof}
	
	\subsection{Product estimate}\label{subsection_proof_product}
	
	We prove Lemma~\ref{lemma_product}. We estimate the product
	\begin{align}
		\widetilde\Pi&=\prod_{r=0}^{\alpha-k-1}\frac{n_{\alpha-r}-K}{n_{\alpha-r}-\psi_{p_r}}
		\\&=\prod_{r=0}^{\alpha-k-1}\left(1-\frac{\psi_{p_r}-K}{n_{\alpha-r}-K}\right)^{-1}
		\\&=\exp\left(\sum_{r=0}^{\alpha-k-1}-\log\left(1-\frac{\psi_{p_r}-K}{n_{\alpha-r}-K}\right)\right).
	\end{align}
	The idea is to now use the approximation $-\log(1-x)=x+O(x^2)$ for $x=o(1)$. However, this requires some more bounds.
	
	\begin{lemma}\label{lemma_bounds2}
		For $r=0,\ldots,\alpha-k-1$, we have the bounds
		\begin{enumerate}[(i)]
			\item\label{lemma_bounds2_nrnk} $\begin{aligned}[t]
				n_{\alpha-r}-K\geq n_{\alpha-r}-n_k\geq(\alpha-r)(\alpha-k-r),
			\end{aligned}$
			\item\label{lemma_bounds2_nrk} $\begin{aligned}[t]
				\frac1{n_{\alpha-r}-K}=\left(1+O\left(\frac{k+1}{\alpha-r}\right)\right)\tfrac2{\ell}(\alpha-r)^{-2},\ \text{if}\ \frac{k+1}{\alpha-r}=o(1),
			\end{aligned}$
			\item\label{lemma_bounds2_pr} $\begin{aligned}[t]
				p_r=(1+O(\tfrac{k}{\alpha}))\tfrac{r}{\alpha},
			\end{aligned}$
			\item\label{lemma_bounds2_qr} $\begin{aligned}[t]
				q_r=(1+O(\tfrac{k}{\alpha}))\tfrac{\alpha-k-r}{\alpha}=(1+O(\tfrac{k}{\alpha-r}))\tfrac{\alpha-r}{\alpha},
			\end{aligned}$
			\item\label{lemma_bounds2_psipr} $\begin{aligned}[t]
				\frac{\psi_{p_r}-K}{n_{\alpha-r}-K}=O\left(\frac{M}{\alpha}+\frac{M^2}{(\alpha-r)\alpha}\right)=O(M^2/\alpha).
			\end{aligned}$
		\end{enumerate}
	\end{lemma}
	
	\begin{proof}
		(\ref{lemma_bounds2_nrnk}) The first inequality follows by Lemma~\ref{lemma_bounds1}~(\ref{lemma_bounds1_v0}).  By Lemma~\ref{lemma_nr}~(\ref{lemma_nr_sum},\ref{lemma_nr_geq}), we have $$n_{\alpha-r}-n_k=n_{\alpha-k-r}+\ell k(\alpha-k-r)\geq(\alpha-k-r)^2+k(\alpha-k-r).$$
		
		(\ref{lemma_bounds2_nrk}) By Lemma~\ref{lemma_bounds1}~(\ref{lemma_bounds1_v0}), we have
		\begin{align}
			n_{\alpha-r}-K
			&=\left(1+\frac{\frac{2}{\ell}-1}{\alpha-r}+\frac{K}{\frac{\ell}{2}(\alpha-r)^2}\right)\tfrac{\ell}{2}(\alpha-r)^2
			\\&=\left(1+O\left(\frac{k+1}{\alpha-r}\right)\right)\tfrac{\ell}{2}(\alpha-r)^2,
		\end{align}
		in particular, since $K/(\alpha-r)^2=O((k/(\alpha-r))^2)=O(k/(\alpha-r))$.
		
		(\ref{lemma_bounds2_pr}) We have $$p_r=\tfrac{r}{\alpha-k}=(1+\tfrac{k}{\alpha-k})\tfrac{r}{\alpha},$$ so the result follows.
		
		(\ref{lemma_bounds2_qr}) We have
		\begin{align}
			q_r&=\tfrac{\alpha-k-r}{\alpha-k}
			\\&=(1+\tfrac{k}{\alpha-k})\tfrac{\alpha-k-r}{\alpha}
			\\&=(1-\tfrac{kr}{(\alpha-k)(\alpha-r)})\tfrac{\alpha-r}{\alpha},
		\end{align}
		so the result follows.
		
		(\ref{lemma_bounds2_psipr}) By Lemma~\ref{lemma_expectation} and Lemma~\ref{lemma_terms}, we have $$\psi_{p_r}-K=O(q_r^2\alpha M+q_rM^2+(n_k-K)M/\alpha).$$ By (\ref{lemma_bounds2_qr}), we have $q_r=O((\alpha-k-r)/\alpha)$, so (\ref{lemma_bounds2_nrnk}) gives $$\frac{\psi_{p_r}-K}{n_{\alpha-r}-K}=O\left(\frac{(\alpha-k-r)M+M^2}{(\alpha-r)\alpha}+\frac{(n_k-K)M/\alpha}{n_{\alpha-r}-K}\right).$$ The result follows, since $\alpha-k-r\leq\alpha-r$ and $n_k-K\leq n_{\alpha-r}-K$.
	\end{proof}
	
	Note that $$\sum_{r=0}^{\alpha-k-1}\left(\frac{M}{\alpha}+\frac{M^2}{(\alpha-r)\alpha}\right)^2=O(M^2/\alpha).$$ By Lemma~\ref{lemma_expectation} and Lemma~\ref{lemma_bounds2}~(\ref{lemma_bounds2_nrnk},\ref{lemma_bounds2_psipr}), we thus get
	\begin{align}
		\log\widetilde\Pi&=(|F|+\b(V))\sum_{r=0}^{\alpha-k-1}\frac{q_r^2}{n_{\alpha-r}-K}+|R^*(S)|\sum_{r=0}^{\alpha-k-1}\frac{q_r^2p_r}{n_{\alpha-r}-K}
		\\&\quad+O\left(\sum_{r=0}^{\alpha-k-1}\frac{q_rM^2+(n_k-K)M/\alpha}{n_{\alpha-r}-n_k}\right)+O(M^2/\alpha).
	\end{align}
	Since $n_k-K\leq n_k=O(k^2)=O((k+1)M)$ by Lemma~\ref{lemma_nr}~(\ref{lemma_nr_leq}), by Lemma~\ref{lemma_terms}, we find that Lemma~\ref{lemma_product} follows from the following lemma.
	
	\begin{lemma}\label{lemma_sums}
		We have the sum estimates
		\begin{enumerate}[(i)]
			\item\label{lemma_sums_1} $\begin{aligned}[t]
				\sum_{r=0}^{\alpha-k-1}\frac{q_r^2}{n_{\alpha-r}-K}=P+O(M\log\alpha/\alpha^2),
			\end{aligned}$
			\item\label{lemma_sums_2} $\begin{aligned}[t]
				\sum_{r=0}^{\alpha-k-1}\frac{q_r^2p_r}{n_{\alpha-r}-K}=\frac12P+O(M\log\alpha/\alpha^2),
			\end{aligned}$
			\item\label{lemma_sums_3} $\begin{aligned}[t]
				\sum_{r=0}^{\alpha-k-1}\frac{q_r}{n_{\alpha-r}-n_k}=O(\log\alpha/\alpha),
			\end{aligned}$
			\item\label{lemma_sums_4} $\begin{aligned}[t]
				\sum_{r=0}^{\alpha-k-1}\frac1{n_{\alpha-r}-n_k}=O(\tfrac1{k+1}(\log(k+1)+1))=O(1).
			\end{aligned}$
		\end{enumerate}
	\end{lemma}
	
	\begin{proof}[Proof of Lemma~\ref{lemma_sums}]
		(\ref{lemma_sums_1}) For $k+1=o(\alpha-r)$, Lemma~\ref{lemma_bounds2}~(\ref{lemma_bounds2_nrk},\ref{lemma_bounds2_qr}) gives $$\frac{q_r^2}{n_{\alpha-r}-K}=\left(1+O\left(\frac{k+1}{\alpha-r}\right)\right)\tfrac{2}{\ell}\alpha^{-2}=\tfrac{P}{\alpha}+O\left(\frac{M}{(\alpha-r)\alpha^2}\right).$$ For the sum, we get
		\begin{align}
			&\sum_{r=0}^{\alpha-\lfloor M\log\alpha\rfloor-1}\frac{q_r^2}{n_{\alpha-r}-K}
			\\&=\tfrac{P}{\alpha}(\alpha+O(M\log\alpha))+O\left(\sum_{r=0}^{\alpha-\lfloor M\log\alpha\rfloor-1}\frac{M}{(\alpha-r)\alpha^2}\right)
			\\&=P+O(M\log\alpha/\alpha^2).
		\end{align}
		For the rest of the sum, since $\alpha-k-r\leq\alpha-r$, Lemma~\ref{lemma_bounds2}~(\ref{lemma_bounds2_nrnk},\ref{lemma_bounds2_qr}) gives $$\sum_{r=\alpha-\lfloor M\log\alpha\rfloor}^{\alpha-k-1}\frac{q_r^2}{n_{\alpha-r}-K}=O\left(\sum_{r=\alpha-\lfloor M\log\alpha\rfloor}^{\alpha-k-1}\frac{\alpha-k-r}{\alpha^2(\alpha-r)}\right)=O(M\log\alpha/\alpha^2).$$
		
		(\ref{lemma_sums_2}) For $k+1=o(\alpha-r)$, Lemma~\ref{lemma_bounds2}~(\ref{lemma_bounds2_nrk},\ref{lemma_bounds2_pr},\ref{lemma_bounds2_qr}) gives $$\frac{q_r^2p_r}{n_{\alpha-r}-K}=\left(1+O\left(\frac{k+1}{\alpha-r}\right)\right)\tfrac{2}{\ell}r\alpha^{-3}=\tfrac{P}{\alpha^2}r+O\left(\frac{M}{(\alpha-r)\alpha^2}\right).$$ For the sum, we get
		\begin{align}
			&\sum_{r=0}^{\alpha-\lfloor M\log\alpha\rfloor-1}\frac{q_r^2p_r}{n_{\alpha-r}-K}
			\\&=\tfrac{P}{\alpha^2}(\tfrac12\alpha^2+O(\alpha M\log\alpha))+O\left(\sum_{r=0}^{\alpha-\lfloor M\log\alpha\rfloor-1}\frac{M}{(\alpha-r)\alpha^2}\right)
			\\&=\tfrac12P+O(M\log\alpha/\alpha^2).
		\end{align}
		For the rest of the sum, we can reuse the estimate from (\ref{lemma_sums_1}), since $p_r\leq1$.
		
		(\ref{lemma_sums_3}) Lemma~\ref{lemma_bounds2}~(\ref{lemma_bounds2_nrnk},\ref{lemma_bounds2_qr}) gives $$\sum_{r=0}^{\alpha-k-1}\frac{q_r}{n_{\alpha-r}-n_k}=O\left(\sum_{r=0}^{\alpha-k-1}\frac1{\alpha(\alpha-r)}\right)=O(\log\alpha/\alpha).$$
		
		(\ref{lemma_sums_4}) By Lemma~\ref{lemma_bounds2}~(\ref{lemma_bounds2_nrnk}), the substitution $t=\alpha-k-r$ gives $$\sum_{r=0}^{\alpha-k-1}\frac1{n_{\alpha-r}-n_k}\leq\sum_{t=1}^{\infty}\frac1{t(t+k)}.$$ If $k=0$, then this evaluates to $\tfrac{\pi^2}{6}=O(1)$. Otherwise, this evaluates to $\tfrac1{k}H_k$ for $H_k:=\sum_{t=1}^{k}\tfrac1{t}\leq1+\log k$ the $k$'th harmonic number.
	\end{proof}
	
	\subsection{Sigma concentration}\label{subsection_proof_concentration}
	
	We show that Lemmas~\ref{lemma_delta} and \ref{lemma_sigma} are implied by a series of lemmas, the proofs of which are spread across Sections~\ref{subsection_proof_vu}--\ref{subsection_proof_psi1}. Figure~\ref{fig:proof-framework-2} gives a schematic representation of the proof structure.
	
	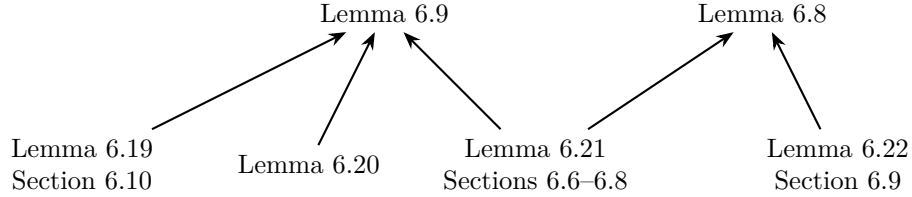
\begin{figure}[htbp]
		\centering
		\begin{tikzpicture}[
			>=Stealth,
			every node/.style={font=\normalsize, align=center},
			thick
			]
			% --- Top level ---
			\node (lem8) at (0, 3)     {Lemma~\ref{lemma_sigma}};
			\node (lem7) at (5, 3)     {Lemma~\ref{lemma_delta}};
			
			% --- Middle level ---
			\node (lem17) at (-4, 1)   {Lemma~\ref{lemma_sigma1hat} \\ Section~\ref{subsection_proof_psi1}};
			\node (lem18) at (-1, 1)   {Lemma~\ref{lemma_sigmaihat}};
			\node (lem19) at (2, 1)    {Lemma~\ref{lemma_sigma2} \\ Sections~\ref{subsection_proof_vu}--\ref{subsection_proof_partition}};
			\node (lem20) at (6, 1)    {Lemma~\ref{lemma_error} \\ Section~\ref{subsection_proof_error}};
			
			% --- Arrows into Lemma~6.9 ---
			\draw[->] (lem17) -- (lem8);
			\draw[->] (lem18) -- (lem8);
			\draw[->] (lem19) -- (lem8);
			
			% --- Arrows into Lemma~6.8 ---
			\draw[->] (lem19) -- (lem7);
			\draw[->] (lem20) -- (lem7);
			
		\end{tikzpicture}
		\caption{Proof structure for Lemmas~\ref{lemma_delta} and~\ref{lemma_sigma}. Arrows indicate logical dependencies between components. Anything labeled with section numbers is shown in those sections. The rest is established here, in Section~\ref{subsection_proof_concentration}. Furthermore, note that Section~\ref{subsection_proof_stronger} also proves Lemma~\ref{lemma_bar_bound}, and Section~\ref{subsection_proof_error} also proves Lemma~\ref{lemma_bar_e}.}
		\label{fig:proof-framework-2}
	\end{figure}
	
	We split up $\Psi_r$ into all of its relevant components, in the same way that $\Psi_p$ was split up in Section~\ref{subsection_proof_simplified}. We define $\Delta_r:=|\widetilde{V}_r\setminus V_r|$ and $\Lambda_r:=\b(V_r)$, and for a set of vertices $U\subset V$, we write $\Delta_r(U):=|\widetilde{V}_r\cap U|$ and $\Lambda_r(U):=\b(\widetilde{V}_r\cap U)$. Finally, we write $R_r^{(i)}$ for the set of vertices from $R^*(S_r)$ with exactly $i$ neighbors in $S$. We split up
	\begin{align}
		\Psi_r^{(2)}&:=\Delta_r(F^{(2)})-\Lambda_r(F^{(2)})+\Delta_r(R_r^{(2)})-\Lambda_r(R_r^{(2)})+\Lambda_r(V^{(2)})+\Lambda_r(S),
		\\\Psi_r^{(1)}&:=\Delta_r(F^{(1)})-\Lambda_r(F^{(1)})+\Delta_r(R_r^{(1)})-\Lambda_r(R_r^{(1)}),
		\\\Psi_r^{(a)}&:=\Lambda_r(V^{(1)}),
		\\\Psi_r^{(b)}&:=\Delta_r(F^{(0)})-\Lambda_r(F^{(0)})+\Lambda_r(V^{(0)}),
		\\\Psi_r^{(c)}&:=\Delta_r(R_r^{(0)})-\Lambda_r(R_r^{(0)}),
	\end{align}
	such that $\Psi_r=\Psi_r^{(2)}+\Psi_r^{(1)}+\Psi_r^{(a)}+\Psi_r^{(b)}+\Psi_r^{(c)}$. For $i\in\{2,1,a,b,c\}$, we similarly define $\Psi_p^{(i)}$ for $p\in[0,1]$, and we define $\psi_p^{(i)}:=\mathbb{E}(\Psi_p^{(i)})$, and
	\begin{align}
		\Sigma^{(2)}:=\sum_{r=0}^{\alpha-k-1}Q_r^{(2)},&&Q_r^{(2)}:=\frac{|\Psi_r^{(2)}-\psi_{p_r}^{(2)}|}{n_{\alpha-r}-n_k},
		\\\hat{\Sigma}^{(i)}:=\sum_{r=0}^{\alpha-k-1}\hat{Q}_r^{(i)},&&\hat{Q}_r^{(i)}:=\frac{(\Psi_r^{(i)}-\psi_{p_r}^{(i)})^+}{n_{\alpha-r}-n_k},
	\end{align}
	such that $\hat{\Sigma}\leq\hat{\Sigma}^{(2)}+\hat{\Sigma}^{(1)}+\hat{\Sigma}^{(a)}+\hat{\Sigma}^{(b)}+\hat{\Sigma}^{(c)}$. We bound $\hat{\Sigma}^{(1)}$, $\hat{\Sigma}^{(a)}$, $\hat{\Sigma}^{(b)}$, and $\hat{\Sigma}^{(c)}$ with the following two lemmas.
	
	\begin{lemma}\label{lemma_sigma1hat}
		Let $C=\Omega(1)$ be a parameter with $C\leq(\frac12-\Omega(1))\frac{\log\alpha}{\log\log\alpha}$. If $k\geq1$, then $\mathbb{E}(e^{C\hat{\Sigma}^{(1)}})\leq e^{o(k)}$.
	\end{lemma}
	
	\begin{lemma}\label{lemma_sigmaihat}
		We have
		\begin{enumerate}[(i)]
			\item\label{lemma_sigmaihat_a} $\hat{\Sigma}^{(a)}=O(M^2\log\alpha/\alpha)$,
			\item\label{lemma_sigmaihat_b} $\hat{\Sigma}^{(b)}=0$,
			\item\label{lemma_sigmaihat_c} $\hat{\Sigma}^{(c)}=O(M^2\log\alpha/\alpha)$.
		\end{enumerate}
	\end{lemma}
	
	It remains to bound $\hat{\Sigma}^{(2)}$, and study the case $k=0$. If $k=0$, then $V\setminus S=V^{(2)}$ by $2$-uniformity, so we get $\Psi_r=\Psi_r^{(2)}$, which gives $Q_r=Q_r^{(2)}$ and $\Sigma=\Sigma^{(2)}\geq\hat{\Sigma}=\hat{\Sigma}^{(2)}$. We start with the following lemma.
	
	\begin{lemma}\label{lemma_sigma2}
		Let $C=\Omega(1)$ be a parameter with $$C\leq\left(8\left(1-\tfrac{\log M}{\log\alpha}\right)+\Omega(1)\right)^{-1}\tfrac{\log\alpha}{\log\log\alpha}.$$ Then there exists an event $\mathcal{R}\subset\mathcal{P}$, such that
		\begin{enumerate}[(i)]
			\item\label{lemma_sigma2_pi} $\begin{aligned}[t]
				\mathbb{E}(\Pi1_{\mathcal{P}^{\delta}\setminus\mathcal{R}})=O(\mathscr{E})\ \text{and}\ \mathbb{E}(|1+\overline{\Sigma}|1_{\mathcal{P}^{\delta}\setminus\mathcal{R}})=O(\mathscr{E})\ \text{for}\ \delta:=\tfrac1{C}\ \text{if}\ k=0,
			\end{aligned}$
			\item\label{lemma_sigma2_exp} $\begin{aligned}[t]
				\mathbb{E}(e^{C\Sigma^{(2)}}1_{\mathcal{P}\setminus\mathcal{R}})=o(1),
			\end{aligned}$
			\item\label{lemma_sigma2_out} $\begin{aligned}[t]
				C\Sigma^{(2)}=o(1)\ \text{in}\ \mathcal{R}.
			\end{aligned}$
		\end{enumerate}
	\end{lemma}
	
	With the previous lemma, for the case $k=0$, the idea is that $\log\Pi\leq\tfrac1{\delta}\Sigma=o(1)$ in $\mathcal{P}^{\delta}\cap\mathcal{R}$ by \eqref{equation_pdelta}, which gives the approximation $$\Pi=\exp(\log\Pi)=1+\log\Pi+O((\log\Pi)^2).$$ Since, furthermore, we have $|\widetilde{Q}_r|\leq\tfrac1{\delta}Q_r\leq\tfrac1{\delta}\Sigma=o(1)$ for $r=0,\ldots,\alpha-1$, we also have the approximation $$\log\Pi=\sum_{r=0}^{\alpha-1}\log(1+\widetilde{Q}_r)=\overline{\Sigma}+\sum_{r=0}^{\alpha-1}(\widetilde{Q}_r-\overline{Q}_r)+O\left(\sum_{r=0}^{\alpha-1}\widetilde{Q}_r^2\right).$$ We are thus interested in the following lemma.
	
	\begin{lemma}\label{lemma_error}
		Assume $k=0$, let $\Omega(1)\leq\delta\leq1-\Omega(1)$, and let $\mathcal{R}\subset\mathcal{P}$ be an event, such that $\tfrac1{\delta}\Sigma=o(1)$ in $\mathcal{P}^{\delta}\cap\mathcal{R}$. Then
		\begin{enumerate}[(i)]
			\item\label{lemma_error_diff} $\begin{aligned}[t]
				\mathbb{E}\left(1_{\mathcal{P}^{\delta}}\sum_{r=0}^{\alpha-1}|\widetilde{Q}_r-\overline{Q}_r|\right)=O(M/\alpha),
			\end{aligned}$
			\item\label{lemma_error_sum} $\begin{aligned}[t]
				\mathbb{E}\left(1_{\mathcal{P}^{\delta}}\sum_{r=0}^{\alpha-1}\widetilde{Q}_r^2\right)=O(M/\alpha),
			\end{aligned}$
			\item\label{lemma_error_log} $\begin{aligned}[t]
				\mathbb{E}(1_{\mathcal{P}^{\delta}\cap\mathcal{R}}(\log\Pi)^2)=O(\mathscr{E}).
			\end{aligned}$
		\end{enumerate}
		It follows that $\mathbb{E}(1_{\mathcal{P}^{\delta}\cap\mathcal{R}}(\Pi-(1+\overline{\Sigma})))=O(\mathscr{E})$.
	\end{lemma}
	
	We find that Lemma~\ref{lemma_delta} follows from Lemma~\ref{lemma_sigma2} and Lemma~\ref{lemma_error}, with the former handling $\mathcal{P}^{\delta}\setminus\mathcal{R}$ by simply estimating $|\Pi-(1+\overline{\Sigma})|\leq\Pi+|1+\overline{\Sigma}|$, and with the latter handling $\mathcal{P}^{\delta}\cap\mathcal{R}$ more carefully. We can also prove Lemma~\ref{lemma_sigma}.
	
	\begin{proof}[Proof of Lemma~\ref{lemma_sigma}]
		(\ref{lemma_sigma_zero}) Let $\mathcal{R}\subset\mathcal{P}$ be as in Lemma~\ref{lemma_sigma2}. We split up $$\mathbb{E}(\Pi1_{\mathcal{Q}})=\mathbb{E}(\Pi1_{\mathcal{Q}\setminus\mathcal{R}})+\mathbb{E}(\Pi1_{\mathcal{Q}\cap\mathcal{R}}).$$ By Lemma~\ref{lemma_sigma2}~(\ref{lemma_sigma2_pi},\ref{lemma_sigma2_out}), since $\mathcal{Q}\subset\mathcal{P}^{\delta}$, combined with \eqref{equation_pdelta}, we get
		\begin{align}
			\mathbb{E}(\Pi1_{\mathcal{Q}\setminus\mathcal{R}})&\leq\mathbb{E}(\Pi1_{\mathcal{P}^{\delta}\setminus\mathcal{R}})=O(\mathscr{E}),
			\\\mathbb{E}(\Pi1_{\mathcal{Q}\cap\mathcal{R}})&\leq\mathbb{E}(e^{C\Sigma}1_{\mathcal{Q}\cap\mathcal{R}})=(1+o(1))\mathbb{P}(\mathcal{Q}\cap\mathcal{R})\leq(1+o(1))\mathbb{P}(\mathcal{Q}).
		\end{align}
		
		(\ref{lemma_sigma_large}) By Lemma~\ref{lemma_sigmaihat}, we have $C\hat{\Sigma}^{(i)}=o(1)$ for $i\in\{a,b,c\}$, so $$e^{C\hat{\Sigma}}\leq e^{C\hat{\Sigma}^{(2)}+C\hat{\Sigma}^{(1)}+o(1)}.$$ Assume that $\alpha$ is large enough, such that $$C\leq\left(8\left(1-\tfrac{\log M}{\log\alpha}\right)+2+2\varepsilon\right)^{-1}\tfrac{\log\alpha}{\log\log\alpha}$$ for some fixed $\varepsilon>0$. Let
		\begin{align}
			\widetilde{\mu}^{(2)}&:=8\left(1-\tfrac{\log M}{\log\alpha}\right)+\varepsilon,
			\\\widetilde{\mu}^{(1)}&:=2+\varepsilon,
		\end{align}
		such that $(\widetilde{\mu}^{(2)}+\widetilde{\mu}^{(1)})C\leq\tfrac{\log\alpha}{\log\log\alpha}$. Let $\mu^{(i)}:=\widetilde{\mu}^{(i)}/(\widetilde{\mu}^{(2)}+\widetilde{\mu}^{(1)})$ for $i\in\{1,2\}$, such that $\mu^{(2)},\mu^{(1)}>0$ and $\mu^{(2)}+\mu^{(1)}=1$. By convexity of the exponential function, we get
		\begin{align}
			e^{C\hat{\Sigma}^{(2)}+C\hat{\Sigma}^{(1)}}
			&=\exp\left(\mu^{(2)}\tfrac{C}{\mu^{(2)}}\hat{\Sigma}^{(2)}+\mu^{(1)}\tfrac{C}{\mu^{(1)}}\hat{\Sigma}^{(1)}\right)
			\\&\leq\mu^{(2)}\exp\left(\tfrac{C}{\mu^{(2)}}\hat{\Sigma}^{(2)}\right)+\mu^{(1)}\exp\left(\tfrac{C}{\mu^{(1)}}\hat{\Sigma}^{(1)}\right)
			\\&\leq\mu^{(2)}\exp\left((\widetilde{\mu}^{(2)})^{-1}\tfrac{\log\alpha}{\log\log\alpha}\hat{\Sigma}^{(2)}\right)
			\\&\quad+\mu^{(1)}\exp\left((\widetilde{\mu}^{(1)})^{-1}\tfrac{\log\alpha}{\log\log\alpha}\hat{\Sigma}^{(1)}\right).
		\end{align}
		By Lemma~\ref{lemma_sigma2}~(\ref{lemma_sigma2_exp},\ref{lemma_sigma2_out}) and Lemma~\ref{lemma_sigma1hat}, for $C^{(i)}:=(\widetilde{\mu}^{(i)})^{-1}\tfrac{\log\alpha}{\log\log\alpha}$, there exists an event $\mathcal{R}\subset\mathcal{P}$, such that
		\begin{align}
			\mathbb{E}(e^{C^{(2)}\Sigma^{(2)}})&=\mathbb{E}(e^{C^{(2)}\Sigma^{(2)}}1_{\mathcal{P}\setminus\mathcal{R}})+\mathbb{E}(e^{C^{(2)}\Sigma^{(2)}}1_{\mathcal{R}})\leq1+o(1),
			\\\mathbb{E}(e^{C^{(1)}\hat{\Sigma}^{(1)}})&\leq e^{o(k)}.
		\end{align}
		The result follows, because $\hat{\Sigma}^{(2)}\leq\Sigma^{(2)}$ and $\mu^{(1)}+\mu^{(2)}=1$.
	\end{proof}
	
	We can prove Lemma~\ref{lemma_sigmaihat} immediately.
	
	\begin{proof}[Proof of Lemma~\ref{lemma_sigmaihat}]
		(\ref{lemma_sigmaihat_a}) Let $\overline{S}$ be a maximal independent set that contains $S$, such that $\overline{S}$ is an MIS by Proposition \ref{proposition_imis}. Since $G$ is $2$-uniform, every vertex in $\widetilde{V}_r\cap V^{(1)}$ is adjacent to one vertex in $S\setminus S_r$ and one vertex in $\overline{S}\setminus S$. By Proposition~\ref{proposition_ell}, we get $|\widetilde{V}_r\cap V^{(1)}|=\ell k(\alpha-k-r)\leq\ell q_rk\alpha$, and thus $\Psi_r^{(a)}=O(q_rM^2)$. Since $(\Psi_r^{(a)}-\psi_{p_r}^{(a)})^+\leq\Psi_r^{(a)}$, the result follows by Lemma~\ref{lemma_sums}~(\ref{lemma_sums_3}).
		
		(\ref{lemma_sigmaihat_b}) Note that $\Psi_r^{(b)}=\Psi_p^{(b)}=|F^{(0)}|+\b(V^{(0)}\setminus F^{(0)})$ is deterministic and independent of $r$ or $p$, and thus equal to $\psi_p^{(b)}$.
		
		(\ref{lemma_sigmaihat_c}) Since $\b_{\max}=O(M/\alpha)=o(1)$, for $\alpha$ large enough such that $1-\b_{\max}\geq0$, we have $\Psi_r^{(c)}=\sum_{v\in S_r}(1-\b(v))\leq\sum_{v\in S}(1-\b(v))=|R^{(0)}|-\b(R^{(0)})$. Since Lemma~\ref{lemma_expect}~(\ref{lemma_expect_dr},\ref{lemma_expect_lr}) furthermore gives $\psi_{p_r}^{(c)}=p_r(|R^{(0)}|-\b(R^{(0)}))$, by Lemma~\ref{lemma_bounds1}~(\ref{lemma_bounds1_v0}), we get $\Psi_r^{(c)}-\psi_{p_r}^{(c)}\leq q_r|R^{(0)}|=O(q_rM^2)$. The result follows by Lemma~\ref{lemma_sums}~(\ref{lemma_sums_3}).
	\end{proof}
	
	In Sections~\ref{subsection_proof_vu}--\ref{subsection_proof_partition}, we prove Lemma~\ref{lemma_sigma2}, in Section~\ref{subsection_proof_error}, we prove Lemma~\ref{lemma_error}, and in Section~\ref{subsection_proof_psi1}, we prove Lemma~\ref{lemma_sigma1hat}. Together, these conclude the proofs of Lemmas~\ref{lemma_delta} and \ref{lemma_sigma}. Furthermore, Section~\ref{subsection_proof_stronger} also proves Lemma~\ref{lemma_bar_bound}, and Section~\ref{subsection_proof_error} also proves Lemma~\ref{lemma_bar_e}.
	
	\subsection{Vu's concentration inequality}\label{subsection_proof_vu}
	
	We analyze concentration of $\Psi_r^{(2)}$. We use Vu's concentration inequality to analyze concentration of $\Psi_{p_r}^{(2)}$. The following version follows from Theorem~4.2 in~\cite{Vu}, which generalizes an earlier result by Kim and Vu~\cite{KimVu}.
	
	\begin{theorem}[Vu's concentration inequality]
		Let $t_1,\ldots,t_n$ be independent random variables in $[0,1]$. Let $Y(t_1,\ldots,t_n)$ be a multivariate polynomial of degree $d$ with coefficients in $[0,U]$ for some $U>0$ and with no variables raised to a power higher than $1$.
		
		For any multiset of variables $A$, let $\partial_A$ denote the partial derivative with respect to the variables in $A$. For $j\leq d$, define $\mathbb{E}_j(Y):=\max_{|A|=j}\mathbb{E}(\partial_AY)$. Let $\lambda>1$ be a parameter, and let $\mathcal{E}_0>\ldots>\mathcal{E}_d=1$ be a sequence satisfying $\mathcal{E}_j/\mathcal{E}_{j+1}\geq\lambda+4j\log n$ and $U\mathcal{E}_j\geq\mathbb{E}_j(Y)$.
		
		Finally, define recursively $c_1=1$, $c_j=2j^{\sfrac12}(c_{j-1}+1)$, and $c'_1=2$, $c'_j=2(c'_{j-1}+1)$. Then $$\mathbb{P}\left(|Y-\mathbb{E}(Y)|\geq Uc_d\sqrt{\lambda\mathcal{E}_0\mathcal{E}_1}\right)\leq c'_de^{-\lambda/4}.$$
	\end{theorem}
	
	\begin{remark}
		If $U=1$, then this version follows, since $\max_{t\in\Omega,|A|\geq d}\partial_AY(t)\leq1$ for $\Omega=[0,1]^n$, and since $\mathcal{E}_0>\ldots>\mathcal{E}_d$ implies $\mathcal{E}_j\geq\max_{|A|\geq j}\mathbb{E}(\partial_AY)$. Note that the former relies on the lack of variables raised to a power higher than $1$. If $U\neq1$, then we can simply consider $Y/U$.
	\end{remark}
	
	Because of the independence requirement on the random variables, we will apply Vu's concentration inequality to the simplified model rather than the true model. To translate concentration of $\Psi_{p_r}^{(2)}$ to concentration of $\Psi_r^{(2)}$, we can condition the simplified model on the event that $|S_{p_r}|=r$.
	
	\begin{lemma}\label{lemma_reduction}
		For $r=0,\ldots,\alpha-k-1$, we have $\mathbb{P}(|S_{p_r}|=r)=\Omega(\alpha^{-\sfrac12})$.
	\end{lemma}
	
	\begin{proof}
		Since $|S_{p_r}|$ follows a binomial distribution, we have
		\begin{align}
			\mathbb{P}(|S_{p_r}|=r)&=p_r^r(1-p_r)^{\alpha-k-r}\tbinom{\alpha-k}{r}=f_{\alpha-k}(r),\ \text{for}
			\\f_n(r)&:=\tfrac{r^r(n-r)^{n-r}}{n^n}\cdot\tfrac{n!}{r!(n-r)!}.
		\end{align}
		For fixed $n$, the function $f_n$ is symmetric around $\tfrac12n$ and attains its minimum at $\lfloor\tfrac12n\rfloor$. To see this, we calculate $$\frac{f(n,r+1)}{f(n,r)}=\frac{(1+1/r)^r}{(1+1/(n-r-1))^{n-r}}.$$ Since $x\mapsto(1+\tfrac1x)^x$ is monotonically increasing for $x>0$, if $r\leq n-r-1$, then we get $$(1+\tfrac1r)^r\leq(1+\tfrac1{n-r-1})^{n-r-1}\leq(1+\tfrac1{n-r-1})^{n-r},$$ so $f(n,r+1)\leq f(n,r)$. It follows that $$\mathbb{P}(|S_{p_r}|=r)\geq f_{\alpha-k}(\lfloor\tfrac12(\alpha-k)\rfloor)\sim\sqrt{\tfrac2{\pi(\alpha-k)}},$$ using Stirling's approximation.
	\end{proof}
	
	\begin{lemma}\label{lemma_vu}
		Let $\lambda\geq4\log\alpha$. For $p\in[0,1]$, we define the functions
		\begin{align}
			\beta^{(a)}_p(\lambda)&:=\ell c_2\sqrt{\lambda(qM+2\lambda)(q^2\alpha M+3\lambda^2)},
			\\\beta^{(b)}_p(\lambda)&:=3\ell c_3\sqrt{\lambda(q^2M+7\lambda^2)(q^3\alpha M+8\lambda^3)},
			\\\beta_p(\lambda)&:=4\beta^{(a)}_p(\lambda)+\min\{\beta^{(b)}_p(\lambda),2q^2\alpha M\}.
		\end{align}
		Then the events
		\begin{enumerate}[(i)]
			\item\label{lemma_vu_f} $\begin{aligned}[t]
				|\Delta_p(F^{(2)})-\Lambda_p(F^{(2)})-\mathbb{E}(\Delta_p(F^{(2)})-\Lambda_p(F^{(2)}))|\geq\beta^{(a)}_p(\lambda),
			\end{aligned}$
			\item\label{lemma_vu_r} $\begin{aligned}[t]
				&|\Delta_p(R_p^{(2)})-\Lambda_p(R_p^{(2)})-\mathbb{E}(\Delta_p(R_p^{(2)})-\Lambda_p(R_p^{(2)}))|\\&\geq\beta^{(a)}_p(\lambda)+\min\{\beta^{(b)}_p(\lambda),2q^2\alpha M\},
			\end{aligned}$
			\item\label{lemma_vu_v} $\begin{aligned}[t]
				|\Lambda_p(V^{(2)})-\mathbb{E}(\Lambda_p(V^{(2)}))|\geq\beta^{(a)}_p(\lambda),
			\end{aligned}$
			\item\label{lemma_vu_s} $\begin{aligned}[t]
				|\Lambda_p(S)-\mathbb{E}(\Lambda_p(S))|\geq\beta^{(a)}_p(\lambda)
			\end{aligned}$
		\end{enumerate}
		have probability $O(e^{-\frac14\lambda})$ each. For $r=0,\ldots,\alpha-k-1$, it follows that
		\begin{align}
			\mathbb{P}(|\Psi_p^{(2)}-\psi_p^{(2)}|\geq\beta_p(\lambda))&=O(e^{-\frac14\lambda}),
			\\\mathbb{P}(|\Psi_r^{(2)}-\psi_{p_r}^{(2)}|\geq\beta_{p_r}(\lambda))&=O(\alpha^{\sfrac12}e^{-\frac14\lambda}).
		\end{align}
	\end{lemma}
	
	\begin{proof}
		First, note that the given concentration inequality on $\Psi_p^{(2)}$ follows from (\ref{lemma_vu_f},\ref{lemma_vu_r},\ref{lemma_vu_v},\ref{lemma_vu_s}) by definition of $\beta_p(\lambda)$. By Lemma~\ref{lemma_reduction}, we get
		\begin{align}
			\mathbb{P}(|\Psi_{p_r}^{(2)}-\psi_{p_r}^{(2)}|\geq\beta_{p_r}(\lambda))&\geq\mathbb{P}(|S_{p_r}|=r)\mathbb{P}(|\Psi_{p_r}^{(2)}-\psi_{p_r}^{(2)}|\geq\beta_{p_r}(\lambda)\,|\,|S_{p_r}|=r)
			\\&=\Omega\left(\alpha^{-\sfrac12}\mathbb{P}(|\Psi_r^{(2)}-\psi_{p_r}^{(2)}|\geq\beta_{p_r}(\lambda))\right),
		\end{align}
		so the given concentration formula for $\Psi_r^{(2)}$ follows as well.
		
		To prove (\ref{lemma_vu_f},\ref{lemma_vu_r},\ref{lemma_vu_v},\ref{lemma_vu_s}), consider for every $s\in S$ the random variable $x_s$ that is $0$ if $s\in S_p$ and $1$ otherwise, such that they are independent Bernoulli distributed random variables with mean $q$. We can write the quantities $\Delta_p(U)$ and $\Lambda_p(U)$ for $U\subset V$ as multivariate polynomials in terms of these $\alpha-k$ random variables, which allows us to apply Vu's concentration inequality.
		
		Note that $\lambda\geq4\log\alpha$ implies, for $\alpha$ sufficiently large, that $\lambda>1$, and that $\mathcal{E}_j\geq(j+1)\lambda\mathcal{E}_{j+1}$ suffices for $\mathcal{E}_0>\ldots>\mathcal{E}_d$ and $\mathcal{E}_j/\mathcal{E}_{j+1}\geq\lambda+4j\log(\alpha-k)$ in Vu's concentration inequality.
		
		(\ref{lemma_vu_f}) Since $S$ is feasible, it is disjoint from $F^{(2)}$. We find that $Y:=\Delta_p(F^{(2)})-\Lambda_p(F^{(2)})$ is the multivariate polynomial $$Y=\sum_{\substack{v\in F^{(2)}\\N(v)\cap S=\{s,t\}}}(1-\b(v))x_sx_t.$$ Note that the coefficients are in $[0,U]$ for $U=\ell$, since any pair $\{s,t\}$ appears at most $|N(s)\cap N(t)|=\ell$ times. We get the values
		\begin{align}
			\mathbb{E}_2(Y)&\leq\ell,&\mathcal{E}_2^{(\ref{lemma_vu_f})}&=1,
			\\\mathbb{E}_1(Y)&\leq qM,&\mathcal{E}_1^{(\ref{lemma_vu_f})}&=qM+2\lambda,
			\\\mathbb{E}_0(Y)&\leq q^2\alpha M,&\mathcal{E}_0^{(\ref{lemma_vu_f})}&=q^2\alpha M+3\lambda^2.
		\end{align}
		Indeed, we have $\mathbb{E}_0(Y)=\mathbb{E}(Y)\leq q^2|F|\leq q^2\alpha M$ by Lemma~\ref{lemma_terms}~(\ref{lemma_terms_f}). Furthermore, for any $s\in S$, we have $\mathbb{E}(\partial_{x_s}Y)\leq q|N(s)\cap F|\leq qF_{\max}\leq qM$, and $\mathbb{E}_2(Y)$ is the largest coefficient.
		
		For $\alpha$ large enough, we also have $\ell\mathcal{E}_j^{(\ref{lemma_vu_f})}\geq\mathbb{E}_j(Y)$ and $\mathcal{E}_j^{(\ref{lemma_vu_f})}\geq(j+1)\lambda\mathcal{E}_{j+1}^{(\ref{lemma_vu_f})}$. In particular, because $\mathcal{E}_0^{(\ref{lemma_vu_f})}-\lambda\mathcal{E}_1^{(\ref{lemma_vu_f})}\geq(qM-\lambda)^2\geq0$, for $\alpha$ large enough, such that $\alpha\geq M$. We find that Vu's concentration inequality applies, and the result follows.
		
		(\ref{lemma_vu_r}) Since $S$ is feasible, no two elements are $R$-equivalent. It follows that all sets $R^*(s)$ for $s\in S$ are disjoint from each other and disjoint from $S$. We find that $Y:=\Delta_p(R_p^{(2)})-\Lambda_p(R_p^{(2)})$ is the multivariate polynomial
		\begin{align}
			Y&=\sum_{s\in S}(1-x_s)\sum_{\substack{v\in R^*(s)\cap V^{(2)}\\N(v)\cap S=\{t,r\}}}(1-\b(v))x_tx_r
			\\&=\sum_{\substack{v\in R^{(2)}\\N(v)\cap S=\{t,r\}}}(1-\b(v))x_tx_r-\sum_{\substack{s\in S\\v\in R^*(s)\cap V^{(2)}\\N(v)\cap S=\{t,r\}}}(1-\b(v))x_sx_tx_r=:Y_2-Y_3.
		\end{align}
		We find that Vu's concentration inequality applies on $Y_2$ with the sequence $\mathcal{E}^{(\ref{lemma_vu_f})}_j$, by the same reasoning as in (\ref{lemma_vu_f}).
		
		For $Y_3$, note that the coefficients are in $[0,U]$ for $U=3\ell$, since any triple $\{s,t,r\}$ appears at most $|N(s)\cap N(t)|+|N(s)\cap N(r)|+|N(t)\cap N(r)|=3\ell$ times. We get the values
		\begin{align}
			\mathbb{E}_3(Y_3)&\leq3\ell,&\mathcal{E}_3^{(\ref{lemma_vu_r})}&=1,
			\\\mathbb{E}_2(Y_3)&\leq q(\ell+2M^{\sfrac12}),&\mathcal{E}_2^{(\ref{lemma_vu_r})}&=qM^{\sfrac12}+3\lambda,
			\\\mathbb{E}_1(Y_3)&\leq2q^2M,&\mathcal{E}_1^{(\ref{lemma_vu_r})}&=q^2M+7\lambda^2,
			\\\mathbb{E}_0(Y_3)&\leq q^3\alpha M,&\mathcal{E}_0^{(\ref{lemma_vu_r})}&=q^3\alpha M+8\lambda^3.
		\end{align}
		Indeed, we have $\mathbb{E}_0(Y_3)=\mathbb{E}(Y_3)\leq q^3|R^*(S)|\leq q^3\alpha M$ by Lemma~\ref{lemma_terms}~(\ref{lemma_terms_rs}). Furthermore, for any $s\in S$, we have $$\mathbb{E}(\partial_{x_s}Y_3)\leq q^2(|R^*(s)|+|N(s)\cap R^*(S)|)\leq q^2(R_{\max}^{(a)}+R_{\max}^{(b)}),$$ for any $s,t\in S$ we have
		\begin{align}
			\mathbb{E}(\partial_{x_s,x_t}Y_3)
			&\leq q(|N(s)\cap N(t)\cap R^*(S)|+|N(s)\cap R^*(t)|+|N(t)\cap R^*(s)|)
			\\&\leq q(\ell+2R_{\max}^{(c)}),
		\end{align}
		and $\mathbb{E}_3(Y_3)$ is the largest coefficient. Note that this is the only place in the entire proof of Theorem~\ref{theorem_main} where we need $M\geq(R_{\max}^{(c)})^2$.
		
		For $\alpha$ large enough, we also have $3\ell\mathcal{E}_j^{(\ref{lemma_vu_r})}\geq\mathbb{E}_j(Y_3)$ and $\mathcal{E}_j^{(\ref{lemma_vu_r})}\geq(j+1)\lambda\mathcal{E}_{j+1}^{(\ref{lemma_vu_r})}$. In particular, because $\mathcal{E}_1^{(\ref{lemma_vu_r})}-2\lambda\mathcal{E}_2^{(\ref{lemma_vu_r})}=(qM^{\sfrac12}-\lambda)^2\geq0$ and $\mathcal{E}_0^{(\ref{lemma_vu_r})}-\lambda\mathcal{E}_1^{(\ref{lemma_vu_r})}\geq(q^2M-\lambda)^2\geq0$, for $\alpha$ large enough, such that $\alpha\geq M\geq qM$ and $\lambda^3\geq\lambda^2$. We find that Vu's concentration inequality applies on $Y_3$ with the sequence $\mathcal{E}^{(\ref{lemma_vu_r})}_j$.
		
		Together with $\mathbb{E}(Y_2)\leq q^2|R^*(S)|\leq q^2\alpha M$ and using $Y_2\geq Y\geq0$, the result follows from $$|Y-\mathbb{E}(Y)|\leq Y+\mathbb{E}(Y)\leq Y_2+\mathbb{E}(Y_2)\leq |Y_2-\mathbb{E}(Y_2)|+2\mathbb{E}(Y_2).$$
		
		(\ref{lemma_vu_v}) We find that $Y:=\Lambda_p(V^{(2)})$ is the multivariate polynomial $$Y=\sum_{\substack{v\in V^{(2)}\\N(v)\cap S=\{s,t\}}}\b(v)x_sx_t.$$ Note that the coefficients are in $[0,U]$ for $U=\ell\b_{\max}$, since any pair $\{s,t\}$ appears at most $|N(s)\cap N(t)|=\ell$ times. For $\alpha$ large enough, we get the values
		\begin{align}
			\mathbb{E}_2(Y)&\leq\ell\b_{\max},&\mathcal{E}_2^{(\ref{lemma_vu_v})}&=1,
			\\\mathbb{E}_1(Y)&\leq q\alpha\ell\b_{\max},&\mathcal{E}_1^{(\ref{lemma_vu_v})}&=q\alpha+2\lambda,
			\\\mathbb{E}_0(Y)&\leq q^2\alpha^2\ell\b_{\max},&\mathcal{E}_0^{(\ref{lemma_vu_v})}&=q^2\alpha^2+3\lambda^2.
		\end{align}
		Indeed, we have $\mathbb{E}_0(Y)=\mathbb{E}(Y)\leq q^2\b_{\max}|V\setminus S|$ with $$|V\setminus S|=\ell\tbinom{\alpha}{2}+k\leq\tfrac12\alpha^2\ell+M\leq\alpha^2\ell$$ for $\alpha$ large enough, such that $M\leq\frac12\alpha^2\ell$, by Proposition~\ref{proposition_n}. Furthermore, for any $s\in S$, we have $\mathbb{E}(\partial_{x_s}Y)\leq q\b_{\max}|N(s)|$, with $|N(s)|=d(G)\leq\ell\alpha$ by Proposition~\ref{proposition_n}, and $\mathbb{E}_2(Y)$ is the largest coefficient.
		
		We also have $\ell\b_{\max}\mathcal{E}_j^{(\ref{lemma_vu_v})}\geq\mathbb{E}_j(Y)$ and $\mathcal{E}_j^{(\ref{lemma_vu_v})}\geq(j+1)\lambda\mathcal{E}_{j+1}^{(\ref{lemma_vu_v})}$. In particular, because $\mathcal{E}_0^{(\ref{lemma_vu_v})}-\lambda\mathcal{E}_1^{(\ref{lemma_vu_v})}\geq(q\alpha-\lambda)^2\geq0$. We find that Vu's concentration inequality applies, and the result follows, since $\b_{\max}\mathcal{E}_j^{(\ref{lemma_vu_v})}\leq\mathcal{E}_j^{(\ref{lemma_vu_f})}$ for $\alpha$ large enough, such that $\b_{\max}\leq M/\alpha\leq1$.
		
		(\ref{lemma_vu_s}) We find that $Y:=\Lambda_p(S)$ is the multivariate polynomial $$Y=\sum_{s\in S}\b(s)x_s.$$ Note that the coefficients are in $[0,U]$ for $U=\b_{\max}$. We get the values
		\begin{align}
			\mathbb{E}_1(Y)&\leq\b_{\max},&\mathcal{E}_1^{(\ref{lemma_vu_s})}&=1,
			\\\mathbb{E}_0(Y)&\leq q\alpha\b_{\max},&\mathcal{E}_0^{(\ref{lemma_vu_s})}&=q\alpha+\lambda.
		\end{align}
		Indeed, we have $\mathbb{E}_0(Y)=\mathbb{E}(Y)\leq q\b_{\max}|S|$, and $\mathbb{E}_1(Y)$ is the largest coefficient. We also have $\b_{\max}\mathcal{E}_j^{(\ref{lemma_vu_s})}\geq\mathbb{E}_j(Y)$ and $\mathcal{E}_0^{(\ref{lemma_vu_s})}\geq\lambda\mathcal{E}_1^{(\ref{lemma_vu_s})}$. We find that Vu's concentration inequality applies, and the result follows, since $c_1\leq c_2$ and $\b_{\max}\mathcal{E}_j^{(\ref{lemma_vu_s})}\leq\ell\mathcal{E}_j^{(\ref{lemma_vu_f})}$ for $\alpha$ large enough, such that $\b_{\max}\leq M/\alpha\leq1$ and $\lambda\geq1$. In particular, because $\mathcal{E}_0^{(\ref{lemma_vu_f})}-\b_{\max}\mathcal{E}_0^{(\ref{lemma_vu_s})}\geq(qM-\lambda)^2\geq0$.
	\end{proof}
	
	Lemma \ref{lemma_vu} shows that, with probability $1-O(\alpha^{\sfrac32}e^{-\frac14\lambda})$, we have $$Q_r^{(2)}\leq\frac{\beta_{p_r}(\lambda)}{n_{\alpha-r}-n_k}$$ for all $r=0,\ldots,\alpha-k-1$. This lets us bound $\Sigma^{(2)}$ with the following lemma.
	
	\begin{lemma}\label{lemma_sum}
		For $\lambda\geq1$, we have
		\begin{align}
			\sum_{\alpha-k-r\geq\lambda\log\alpha}^{\alpha-k-r\leq\alpha-k}\frac{\beta_{p_r}(\lambda)}{n_{\alpha-r}-n_k}=O\left(\frac{\lambda}{\log\alpha}\right)&&\text{if}\ \lambda=O\left(\frac{\alpha}{M(\log\alpha)^2}\right).
		\end{align}
		Here, the notation $\sum_{\alpha-k-r\geq a}^{\alpha-k-r\leq b}$ represents the sum over all integers $r$ with $a\leq\alpha-k-r\leq b$.
	\end{lemma}
	
	\begin{proof}
		By Lemma~\ref{lemma_bounds2}~(\ref{lemma_bounds2_nrnk}), we have $n_{\alpha-r}-n_k\geq(\alpha-k-r)^2$, so
		\begin{align}
			&\sum_{\alpha-k-r\geq\lambda\log\alpha}^{\alpha-k-r\leq\alpha-k}\frac{\beta_{p_r}(\lambda)}{n_{\alpha-r}-n_k}
			\\&\leq4\sum_{\alpha-k-r\geq\lambda\log\alpha}^{\alpha-k-r\leq\alpha-k}\frac{\beta_{p_r}^{(a)}(\lambda)}{(\alpha-k-r)^2}+\sum_{\alpha-k-r\geq\lambda^2\log\alpha}^{\alpha-k-r\leq\alpha-k}\frac{\beta_{p_r}^{(b)}(\lambda)}{(\alpha-k-r)^2}
			\\&\quad+\sum_{\alpha-k-r\geq\lambda\log\alpha}^{\alpha-k-r\leq\lambda^2\log\alpha}\frac{2q_r^2\alpha M}{(\alpha-k-r)^2}.
		\end{align}
		For the last sum, note that $2q_r^2\alpha M/(\alpha-k-r)^2=2\alpha M/(\alpha-k)^2=O(M/\alpha)$, so the sum is $O(\lambda^2\log\alpha\cdot M/\alpha)=O(\lambda/\log\alpha)$. For the first two sums, note that subadditivity of the square root gives
		\begin{align}
			\beta_p^{(a)}(\lambda)&\leq\ell c_2\left(\sqrt{\lambda q^3\alpha M^2}+\sqrt{2\lambda^2q^2\alpha M}+\sqrt{3\lambda^3qM}+\sqrt{6\lambda^4}\right),
			\\\beta_p^{(b)}(\lambda)&\leq3\ell c_3\left(\sqrt{\lambda q^5\alpha M^2}+\sqrt{7\lambda^3q^3\alpha M}+\sqrt{8\lambda^4q^2M}+\sqrt{56\lambda^6}\right)
		\end{align}
		for $p\in[0,1]$. Using Lemma~\ref{lemma_bounds2}~(\ref{lemma_bounds2_qr}), we get
		\begin{align}
			&\sum_{\alpha-k-r\geq\lambda\log\alpha}^{\alpha-k-r\leq\alpha-k}\frac{(\lambda q_r^3\alpha M^2)^{\sfrac12}}{(\alpha-k-r)^2}=\sum_{\alpha-k-r\geq\lambda\log\alpha}^{\alpha-k-r\leq\alpha-k}\frac{O(\lambda^{\sfrac12}M/\alpha)}{(\alpha-k-r)^{\sfrac12}}=O\left(\frac{\lambda^{\sfrac12}M}{\alpha^{\sfrac12}}\right),
			\\&\sum_{\alpha-k-r\geq\lambda\log\alpha}^{\alpha-k-r\leq\alpha-k}\frac{(\lambda^2q_r^2\alpha M)^{\sfrac12}}{(\alpha-k-r)^2}=\sum_{\alpha-k-r\geq\lambda\log\alpha}^{\alpha-k-r\leq\alpha-k}\frac{O(\lambda M^{\sfrac12}/\alpha^{\sfrac12})}{\alpha-k-r}=O\left(\frac{\lambda M^{\sfrac12}\log\alpha}{\alpha^{\sfrac12}}\right),
			\\&\sum_{\alpha-k-r\geq\lambda^2\log\alpha}^{\alpha-k-r\leq\alpha-k}\frac{(\lambda^3q_r^3\alpha M)^{\sfrac12}}{(\alpha-k-r)^2}=\sum_{\alpha-k-r\geq\lambda^2\log\alpha}^{\alpha-k-r\leq\alpha-k}\frac{O(\lambda^{\sfrac32}M^{\sfrac12}/\alpha)}{(\alpha-k-r)^{\sfrac12}}=O\left(\frac{\lambda^{\sfrac32}M^{\sfrac12}}{\alpha^{\sfrac12}}\right),
			\\&\sum_{\alpha-k-r\geq\lambda\log\alpha}^{\alpha-k-r\leq\alpha-k}\frac{(\lambda^4q_rM)^{\sfrac12}}{(\alpha-k-r)^2}=\sum_{\alpha-k-r\geq\lambda\log\alpha}^{\alpha-k-r\leq\alpha-k}\frac{O(\lambda^2M^{\sfrac12}/\alpha^{\sfrac12})}{(\alpha-k-r)^{\sfrac32}}=O\left(\frac{\lambda^{\sfrac32}M^{\sfrac12}}{(\alpha\log\alpha)^{\sfrac12}}\right),
			\\&\sum_{\alpha-k-r\geq\lambda\log\alpha}^{\alpha-k-r\leq\alpha-k}\frac{\lambda^2}{(\alpha-k-r)^2}=O\left(\frac{\lambda}{\log\alpha}\right),
			\\&\sum_{\alpha-k-r\geq\lambda^2\log\alpha}^{\alpha-k-r\leq\alpha-k}\frac{\lambda^3}{(\alpha-k-r)^2}=O\left(\frac{\lambda}{\log\alpha}\right).
		\end{align}
		Finally, using $\lambda\geq1$ and $M=O(\alpha^{\sfrac12}/(\log\alpha)^2)$, we get
		\begin{align}
			&\frac{\lambda^{\sfrac12}M}{\alpha^{\sfrac12}}\leq\frac{\lambda M}{\alpha^{\sfrac12}}=O\left(\frac{\lambda}{(\log\alpha)^2}\right)=O\left(\frac{\lambda}{\log\alpha}\right),
			\\&\frac{\lambda M^{\sfrac12}\log\alpha}{\alpha^{\sfrac12}}=O\left(\frac{\lambda}{\alpha^{\sfrac14}}\right)=O\left(\frac{\lambda}{\log\alpha}\right),
		\end{align}
		and using $\lambda=O(\alpha/(M(\log\alpha)^2))$ as well, we get
		\begin{equation}
			\frac{\lambda^{\sfrac32}M^{\sfrac12}}{(\alpha\log\alpha)^{\sfrac12}}=O\left(\frac{\lambda^{\sfrac32}M^{\sfrac12}}{\alpha^{\sfrac12}}\right)=O\left(\frac{\lambda}{\log\alpha}\right),
		\end{equation}
		so the result follows.
	\end{proof}
	
	\subsection{Specialized concentration}\label{subsection_proof_stronger}
	
	Since Lemma~\ref{lemma_sum} only sums over $\alpha-k-r\geq\lambda\log\alpha$, we show a stronger concentration than Lemma~\ref{lemma_vu} for $\alpha-k-r<\lambda\log\alpha$ to finish the bound on $\Sigma^{(2)}$, and we prove Lemma~\ref{lemma_bar_bound}.
	
	\begin{lemma}\label{lemma_stronger}
		For $r=0,\ldots,\alpha-k-1$, write $t:=\alpha-k-r$, and let $\rho_r\geq t^2\exp(-O(\log\log\alpha))$. If $t\leq\alpha\exp(-\omega(\log\log\alpha))$, then
		\begin{enumerate}[(i)]
			\item\label{lemma_stronger_delta} $\begin{aligned}[t]
				\mathbb{P}(\Delta_r^{(2)}\geq\rho_r)\leq t^{\sfrac32}\exp\left(\tfrac{2\rho_r}{\ell t}(\log\tfrac{M}{\alpha}+\log\tfrac{t^2}{\rho_r}+O(1))\right),
			\end{aligned}$
			\item\label{lemma_stronger_psi} $\begin{aligned}[t]
				\mathbb{P}(|\Psi_r^{(2)}-\psi_{p_r}^{(2)}|\geq\rho_r)\leq t^{\sfrac32}\exp\left(\tfrac{\rho_r}{\ell t}(-\log\alpha+O(\log\log\alpha))\right).
			\end{aligned}$
		\end{enumerate}
	\end{lemma}
	
	We introduce the quantities $\Delta_r^{(2)}:=\Delta_r(F^{(2)})+\Delta_r(R_r^{(2)})$ and $\Lambda_r^{(2)}:=\Lambda_r(V^{(2)})+\Lambda_r(S)-\Lambda_r(F^{(2)})-\Lambda_r(R_r^{(2)})$, such that $\Psi_r^{(2)}=\Delta_r^{(2)}+\Lambda_r^{(2)}$. We use the following lemma.
	
	\begin{lemma}\label{lemma_psi2}
		For $r=0,\ldots,\alpha-k-1$, we have
		\begin{enumerate}[(i)]
			\item\label{lemma_psi2_lambda} $\Lambda_r^{(2)}=O(M(\alpha-k-r)^2/\alpha)$,
			\item\label{lemma_psi2_psi} $\psi_{p_r}^{(2)}=O(M(\alpha-k-r)^2/\alpha)$.
		\end{enumerate}
	\end{lemma}
	
	\begin{proof}
		(\ref{lemma_psi2_lambda}) Note that $\Lambda_r^{(2)}\leq\b(X)$ for $X:=(V^{(2)}\cup S)\cap\widetilde{V}_r$. Let $\overline{S}$ be a maximal independent set that contains $S$, such that $\overline{S}$ is an MIS by Proposition \ref{proposition_imis}. We get $V^{(2)}\cup S=V\setminus\overline{N}(\overline{S}\setminus S)$ by $2$-uniformity of $G$. By Corollary~\ref{corollary_hereditary}, the set $X=V\setminus\overline{N}((\overline{S}\setminus S)\cup S_r)$ has size $n_{\alpha-k-r}$, so the result follows by Lemma~\ref{lemma_nr}~(\ref{lemma_nr_leq}).
		
		(\ref{lemma_psi2_psi}) This follows from Lemma~\ref{lemma_expect}, Lemma~\ref{lemma_bounds1}~(\ref{lemma_bounds1_f2},\ref{lemma_bounds1_r2}), and Lemma~\ref{lemma_bounds2}~(\ref{lemma_bounds2_qr}).
	\end{proof}
	
	To bound $\Delta_r^{(2)}$, we say that two distinct vertices $s_1,s_2\in S$ are \emph{interfering} if their intersection of neighborhoods $N(s_1)\cap N(s_2)$ contains a vertex from $F\cup R^*(S)$.
	
	\begin{lemma}\label{lemma_delta2}
		The quantity $\Delta_r^{(2)}$ is at most $\ell$ times the number of interfering pairs of vertices in $S\setminus S_r$. It follows that there is some vertex $s\in S\setminus S_r$ that interferes with at least $\frac{2}{\ell}\Delta_r^{(2)}/(\alpha-k-r)$ other vertices from $S\setminus S_r$.
	\end{lemma}
	
	\begin{proof}
		Because $G$ is $2$-uniform, all intersections of neighborhood $N(s_1)\cap N(s_2)$ are disjoint, so this follows from Proposition~\ref{proposition_ell}.
	\end{proof}
	
	Finally, we use the following lemma to analyze interference.
	
	\begin{lemma}\label{lemma_subset}
		Consider integers $b\geq c\geq i\geq1$ and $a\geq0$. Let $A\subset B$ be sets of sizes $a$ and $b$ respectively. Then the probability that a uniformly random subset of $B$ of size $c$ overlaps $A$ in exactly $i$ elements is bounded from above by $$\exp\left(\tfrac12\log(2\pi c)+\tfrac{1}{12c}+\tfrac{c^2}{b-c}\right)\left(\tfrac{eca}{ib}\right)^i.$$
	\end{lemma}
	
	\begin{proof}
		There are $\binom{b}{c}$ possible subsets, $\binom{a}{i}\binom{b-a}{c-i}$ of which overlap $A$ in exactly $i$ elements, so the probability in question is $$\tbinom{b}{c}^{-1}\tbinom{a}{i}\tbinom{b-a}{c-i}\leq\tbinom{b}{c}^{-1}\tbinom{a}{i}\tbinom{b}{c-i}.$$ We use the general bounds $\frac{(n-k)^k}{k!}\leq\binom{n}{k}\leq(\frac{en}{k})^k$ to get the upper bound $$(\tfrac{a}{i})^i(\tfrac{b}{c-i})^{c-i}(\tfrac{e}{b-c})^cc!.$$ We use the bound $n!\leq\sqrt{2\pi n}(\frac{n}{e})^n\exp(\frac{1}{12n})$ for $n\geq1$, which is shown in~\cite{Stirling}, to get the upper bound
		\begin{align}
			&\exp\left(\tfrac12\log(2\pi c)+\tfrac1{12c}\right)(\tfrac{a}{i})^i(\tfrac{b}{c-i})^{c-i}(\tfrac{c}{b-c})^c
			\\&=\exp\left(\tfrac12\log(2\pi c)+\tfrac1{12c}\right)(\tfrac{ca}{ib})^i(\tfrac{c}{c-i})^{c-i}(\tfrac{b}{b-c})^c.
		\end{align}
		Using $1+x\leq e^x$, the result follows from $\frac{c}{c-i}=1+\frac{i}{c-i}$ and $\frac{b}{b-c}=1+\frac{c}{b-c}$. 
	\end{proof}
	
	\begin{proof}[Proof of Lemma~\ref{lemma_stronger}]
		(\ref{lemma_stronger_delta}) If $\Delta_r^{(2)}\geq\rho_r$, then, by Lemma~\ref{lemma_delta2}, there exists $s\in S\setminus S_r$ interfering with at least $$\iota_r:=\tfrac{2}{\ell}\Delta_r^{(2)}/t\geq\tfrac{2\rho_r}{\ell t}$$ other vertices in $S\setminus S_r$. If $t=1$, note that there are no pairs of vertices in $S\setminus S_r$, so we get $\mathbb{P}(\Delta_r^{(2)}\geq\rho_r)=0$. We may thus assume $t\geq2$.
		
		For fixed $s\in S$, consider conditioning on $s\in S\setminus S_r$. Then $S\setminus S_r\setminus\{s\}$ is a uniformly random subset of $S\setminus\{s\}$ of size $t-1$. Since $s$ interferes with at most $F_{\max}+R_{\max}^{(b)}\leq2M$ other vertices in $S$, by Lemma~\ref{lemma_subset}, the probability that $S\setminus S_r\setminus\{s\}$ contains exactly $i\geq1$ of them is at most $$\exp\left(\tfrac12\log(2\pi(t-1))+\tfrac1{12(t-1)}+\tfrac{(t-1)^2}{\alpha-k-t}\right)\left(\tfrac{2eM(t-1)}{i(\alpha-k-1)}\right)^i.$$ Since there are $\alpha-k$ options for $s\in S$, each with $\mathbb{P}(s\in S\setminus S_r)=t/(\alpha-k)$, we get $$\mathbb{P}(|\Psi_r^{(2)}-\psi_{p_r}^{(2)}|\geq\rho_r)\leq t\sum_{i\geq\iota_r}\exp\left(\tfrac12\log t+O\left(\tfrac{t^2}{\alpha-k-t}+1\right)\right)\left(\tfrac{2eMt}{i(\alpha-k-1)}\right)^i.$$ For $i\geq\iota_r$, we have $\tfrac{t}{i}\leq\tfrac{\ell t^2}{2\rho_r}$, and thus $$\tfrac{2eMt}{i(\alpha-k-1)}\leq\exp\left(\log\tfrac{M}{\alpha}+\log\tfrac{t^2}{\rho_r}+O(1)\right)=o(1).$$ Since $\sum_{i\geq n}x^i=(1+o(1))x^n$ for $x=o(1)$, we get $$\mathbb{P}(|\Psi_r^{(2)}-\psi_{p_r}^{(2)}|\geq\rho_r)\leq t^{\sfrac32}\exp\left(\lceil\iota_r\rceil(\log\tfrac{M}{\alpha}+\log\tfrac{t^2}{\rho_r}+O(1))+O\left(\tfrac{t^2}{\alpha-k-t}\right)\right).$$ Finally, since $\lceil\iota_r\rceil\geq\tfrac{2\rho_r}{\ell t}$, and since the given bounds on $\rho_r$ and $t$ give $$\tfrac{2\rho_r}{\ell t}\geq t\exp(-O(\log\log\alpha))=\Omega\left(\tfrac{t^2}{\alpha-k-t}\right),$$ the result follows.
		
		(\ref{lemma_stronger_psi}) By Lemma~\ref{lemma_psi2}, if $|\Psi_r^{(2)}-\psi_{p_r}^{(2)}|\geq\rho_r$, we have $$\Delta_r^{(2)}\geq|\Psi_r^{(2)}-\psi_{p_r}^{(2)}|-\Lambda_r^{(2)}-\psi_{p_r}^{(2)}\geq\rho_r-O(Mt^2/\alpha),$$ so since $\log\tfrac{M}{\alpha}\leq-\tfrac12\log\alpha+O(\log\log\alpha)$ and $$\tfrac{\rho_r\log\log\alpha}{\ell t}\geq t\exp(-O(\log\log\alpha))=\Omega(Mt\log\alpha/\alpha),$$ the result follows by (\ref{lemma_stronger_delta}).
	\end{proof}
	
	For $\lambda\geq1$, we study the event
	\begin{align}
		A(\lambda)&:=\{|\Psi_r^{(2)}-\psi_{p_r}^{(2)}|\leq T_r(\lambda),\ \forall r=0,\ldots,\alpha-k-1\}\subset\mathcal{P},
		\\T_r(\lambda)&:=\begin{cases}
			\beta_{p_r}(\lambda)&\alpha-k-r\geq\lambda\log\alpha
			\\\tfrac{\ell(\alpha-k-r)\lambda}{10\log\alpha}&\alpha-k-r<\lambda\log\alpha
		\end{cases}.
	\end{align}
	We find the following.
	
	\begin{lemma}\label{lemma_pai}
		For $\lambda\geq(10-o(1))\log\alpha$ with $\lambda\leq\alpha^{\sfrac12}\exp(O(\log\log\alpha))$, we have $$\mathbb{P}(\mathcal{P}\setminus A(\lambda))\leq\exp(-\tfrac1{10}\lambda+o(\lambda)).$$
	\end{lemma}
	
	\begin{proof}
		By Lemma~\ref{lemma_vu}, the probability that $|\Psi_r^{(2)}-\psi_{p_r}^{(2)}|\geq T_r(\lambda)$ for some $r=0,\ldots,\alpha-k-1$ with $\alpha-k-r\geq\lambda\log\alpha$ is $O(\alpha^{\sfrac32}e^{-\frac14\lambda})$. By Lemma~\ref{lemma_stronger}~(\ref{lemma_stronger_psi}), the probability that $|\Psi_r^{(2)}-\psi_{p_r}^{(2)}|\geq T_r(\lambda)$ for some $r=0,\ldots,\alpha-k-1$ with $\alpha-k-r<\lambda\log\alpha$ is at most $$(\lambda\log\alpha)^{\sfrac52}\exp\left(\tfrac{\lambda}{10\log\alpha}(-\log\alpha+O(\log\log\alpha))\right).$$ The result follows, in particular, since $\tfrac32\log\alpha\leq\tfrac3{20}\lambda+o(\lambda)$, and since $\lambda\geq(10-o(1))\log\alpha\to\infty$ gives $\log\log\alpha=O(\log\lambda)=o(\lambda)$.
	\end{proof}
	
	This finally lets us fully bound $\Sigma^{(2)}$ with the following lemma.
	
	\begin{lemma}\label{lemma_sai}
		For $\lambda\geq1$, in $A(\lambda)$, we have
		\begin{align}
			\Sigma^{(2)}\leq\tfrac{2\lambda\log\log\alpha}{5\log\alpha}+O\left(\tfrac{\lambda}{\log\alpha}\right)&&\text{if}\ \lambda=O\left(\tfrac{\alpha}{M(\log\alpha)^2}\right).
		\end{align}
	\end{lemma}
	
	\begin{proof}
		Note that $0\leq\Psi_r^{(2)}\leq|V^{(2)}|=n_{\alpha-k-r}$ for $\alpha$ large enough, since $\Psi_r^{(2)}$ sums either $\b(v)$ or $1$ over all $v\in V^{(2)}$. By the same logic, we have $0\leq\Psi_p^{(2)}\leq n_{\alpha-k-r}$ and thus $0\leq\psi_p^{(2)}\leq n_{\alpha-k-r}$ for $p\in[0,1]$, so we have $|\Psi_r^{(2)}-\psi_{p_r}^{(2)}|\leq n_{\alpha-k-r}$. By Lemma~\ref{lemma_nr}~(\ref{lemma_nr_sum}), we have $n_{\alpha-r}-n_k\geq n_{\alpha-k-r}$, and thus $Q_r^{(2)}\leq1$. By Lemma~\ref{lemma_sum}, in $A(\lambda)$, we get $$\Sigma^{(2)}\leq\sum_{\alpha-k-r\geq1}^{\alpha-k-r\leq\lambda/\log\alpha}1+\sum_{\alpha-k-r\geq\lambda/\log\alpha}^{\alpha-k-r\leq\lambda\log\alpha}\frac{\ell(\alpha-k-r)\lambda}{10\log\alpha\cdot n_{\alpha-k-r}}+O\left(\frac{\lambda}{\log\alpha}\right).$$ To bound the second sum, note that
		\begin{align}
			n_{\alpha-k-r}
			&=\tfrac{\ell}{2}(\alpha-k-r)^2+(1-\tfrac{\ell}{2})(\alpha-k-r)
			\\&=(1-(1-\tfrac{2}{\ell})\tfrac{1}{\alpha-k-r})\tfrac{\ell}{2}(\alpha-k-r)^2.
		\end{align}
		For $y:=(1-\tfrac{2}{\ell})\tfrac{1}{\alpha-k-r}\leq1-\tfrac{2}{\ell}$, note that $\tfrac{1}{1-y}=1+\tfrac{y}{1-y}\leq1+\tfrac{\ell}{2}y$, so $$\frac{\ell(\alpha-k-r)\lambda}{10\log\alpha\cdot n_{\alpha-k-r}}\leq\frac{\lambda}{5\log\alpha}\left(\frac1{\alpha-k-r}+O\left(\frac{1}{(\alpha-k-r)^2}\right)\right).$$ Finally, since
		\begin{align}
			&\sum_{\alpha-k-r\geq\lambda/\log\alpha}^{\alpha-k-r\leq\lambda\log\alpha}\frac1{\alpha-k-r}\leq\int_{\lambda/\log\alpha}^{\lambda\log\alpha}\frac1x\,\mathrm{d}x+1=2\log\log\alpha+1,
			\\&\sum_{\alpha-k-r\geq\lambda/\log\alpha}^{\alpha-k-r\leq\lambda\log\alpha}\frac1{(\alpha-k-r)^2}\leq\sum_{\alpha-k-r\geq1}\frac{1}{(\alpha-k-r)^2}=\frac{\pi^2}{6},
		\end{align}
		the result follows.
	\end{proof}
	
	We also have the following universal bound on $\Sigma^{(2)}$, which we use to prove Lemma~\ref{lemma_bar_bound}.
	
	\begin{lemma}\label{lemma_universal}
		We have $\Sigma^{(2)}=O(M\log\alpha)$.
	\end{lemma}
	
	\begin{proof}
		For $r=0,\ldots\alpha-k-1$, since $\Delta_r^{(2)}\leq(\alpha-k-r)(F_{\max}+R_{\max}^{(b)})$, by Lemma~\ref{lemma_psi2}, we have $$|\Psi_r^{(2)}-\psi_{p_r}^{(2)}|\leq\Delta_r^{(2)}+\Lambda_r^{(2)}+\psi_{p_r}^{(2)}=O(M(\alpha-k-r)).$$ By Lemma~\ref{lemma_bounds2}~(\ref{lemma_bounds2_nrnk}), we get $Q_r^{(2)}=O(M/(\alpha-r))$, so the result follows.
	\end{proof}
	
	\begin{proof}[Proof of Lemma~\ref{lemma_bar_bound}]
		If $k=0$, then $\Sigma=\Sigma^{(2)}$ and $|\overline{Q}_r|=Q_r$ for $r=0,\ldots,\alpha-1$, so we have $|1+\overline{\Sigma}|\leq1+\Sigma=O(M\log\alpha)$ by Lemma~\ref{lemma_universal}.
	\end{proof}
	
	\subsection{Partitioning}\label{subsection_proof_partition}
	
	We prove Lemma~\ref{lemma_sigma2}. We have $$(1+\varepsilon)C\leq\left(8\left(1-\tfrac{\log M}{\log\alpha}\right)+\Omega(1)\right)^{-1}\tfrac{\log\alpha}{\log\log\alpha}$$ for some fixed $\varepsilon>0$. For $i\geq0$, define $$\lambda_i:=20(1+\varepsilon)^i\left(1-\tfrac{\log M}{\log\alpha}\right)\log\alpha,$$ and let $N\geq0$ be minimal such that $\lambda_N\geq\alpha^{\sfrac12}$.
	
	\begin{lemma}\label{lemma_an}
		We have $$\mathbb{E}(e^{C\Sigma^{(2)}}1_{\mathcal{P}\setminus A(\lambda_0)})=O(M^2/\alpha).$$
	\end{lemma}
	
	\begin{proof}
		Note that $\lambda_0\geq(10-o(1))\log\alpha$. By Lemma~\ref{lemma_pai} and Lemma~\ref{lemma_universal}, since $CM\log\alpha=o(\lambda_N)$, we have $$\mathbb{E}(e^{C\Sigma^{(2)}}1_{\mathcal{P}\setminus A(\lambda_N)})\leq\exp\left(-\tfrac1{10}\lambda_N+o(\lambda_N)\right)\leq\exp\left(-\Omega(\alpha^{\sfrac12})\right)=O(\alpha^{-1}).$$ Next, for $i=0,\ldots,N-1$, by Lemma~\ref{lemma_pai} and Lemma~\ref{lemma_sai}, we get $$\mathbb{E}(e^{C\Sigma^{(2)}}1_{A(\lambda_{i+1})\setminus A(\lambda_i)})\leq\exp\left(-\tfrac1{10}\lambda_i+\tfrac{2(1+\varepsilon)C\lambda_i\log\log\alpha}{5\log\alpha}+o(\lambda_i)\right).$$ We have
		\begin{align}
			&-\tfrac1{10}\lambda_i+\tfrac{2(1+\varepsilon)C\lambda_i\log\log\alpha}{5\log\alpha}
			\\&\leq\left(-\tfrac1{10}+\tfrac1{20}\left(1-\tfrac{\log M}{\log\alpha}\right)^{-1}-\Omega(1)\right)\lambda_i
			\\&\leq\left(-\tfrac1{10}+\tfrac1{20}\left(1-\tfrac{\log M}{\log\alpha}\right)^{-1}\right)20\left(1-\tfrac{\log M}{\log\alpha}\right)\log\alpha-\Omega(\lambda_i)
			\\&=2\log M-\log\alpha-\Omega(\lambda_i).
		\end{align}
		The second inequality uses that $M<\alpha^{\sfrac12}$ for $\alpha$ large enough, such that $$\left(-\tfrac1{10}+\tfrac1{20}\left(1-\tfrac{\log M}{\log\alpha}\right)^{-1}\right)<0,$$ and $\lambda_i\geq\lambda_0$. We get $$\mathbb{E}(e^{C\Sigma^{(2)}}1_{A(\lambda_{i+1})\setminus A(\lambda_i)})\leq(M^2/\alpha)e^{-\Omega(\lambda_i)}\leq(M^2/\alpha)e^{-\Omega((1+i\varepsilon)\log\alpha)},$$ so summing over $i=0,\ldots,N-1$ gives $$\mathbb{E}(e^{C\Sigma^{(2)}}1_{A(\lambda_N)\setminus A(\lambda_0)})=O(M^2/\alpha).$$ Indeed, for each permutation $\pi\in A(\lambda_N)\setminus A(\lambda_0)$, there is a largest $i=0,\ldots,N-1$ such that $\pi\not\in A(\lambda_i)$, which gives $\pi\in A(\lambda_{i+1})\setminus A(\lambda_i)$. The result follows.
	\end{proof}
	
	It remains to bound $\Sigma^{(2)}$ in $A(\lambda_0)$.
	
	\begin{lemma}\label{lemma_suma0}
		In $A(\lambda_0)$, for $t\geq\lambda_0^2\log\alpha$, we have $$\sum_{\alpha-k-r\geq t}^{\alpha-k-r\leq\alpha-k}Q_r^{(2)}=O\left(\frac{M(\log\alpha)^{\sfrac12}}{\alpha^{\sfrac12}}+\frac{M^{\sfrac12}(\log\alpha)^2}{\alpha^{\sfrac12}}+\frac{(\log\alpha)^3}{t}\right).$$
	\end{lemma}
	
	\begin{proof}
		Following the proof of Lemma~\ref{lemma_sum}, using $\lambda_0=O(\log\alpha)$, we find that the inequalities
		\begin{align}
			&\sum_{\alpha-k-r\geq\lambda_0^2\log\alpha}^{\alpha-k-r\leq\alpha-k}\frac{(\lambda_0q_r^3\alpha M^2)^{\sfrac12}}{(\alpha-k-r)^2}=O\left(\frac{M(\log\alpha)^{\sfrac12}}{\alpha^{\sfrac12}}\right),
			\\&\sum_{\alpha-k-r\geq\lambda_0^2\log\alpha}^{\alpha-k-r\leq\alpha-k}\frac{(\lambda_0^2q_r^2\alpha M)^{\sfrac12}}{(\alpha-k-r)^2}=O\left(\frac{M^{\sfrac12}(\log\alpha)^2}{\alpha^{\sfrac12}}\right),
			\\&\sum_{\alpha-k-r\geq\lambda_0^2\log\alpha}^{\alpha-k-r\leq\alpha-k}\frac{(\lambda_0^3q_r^3\alpha M)^{\sfrac12}}{(\alpha-k-r)^2}=O\left(\frac{M^{\sfrac12}(\log\alpha)^{\sfrac32}}{\alpha^{\sfrac12}}\right),
			\\&\sum_{\alpha-k-r\geq\lambda_0^2\log\alpha}^{\alpha-k-r\leq\alpha-k}\frac{(\lambda_0^4q_rM)^{\sfrac12}}{(\alpha-k-r)^2}=O\left(\frac{M^{\sfrac12}(\log\alpha)^{\sfrac12}}{\alpha^{\sfrac12}}\right),
			\\&\sum_{\alpha-k-r\geq t}^{\alpha-k-r\leq\alpha-k}\frac{\lambda_0^3}{(\alpha-k-r)^2}=O\left(\frac{(\log\alpha)^3}t\right)
		\end{align}
		suffice to prove the result.
	\end{proof}
	
	For $t\geq1$ and $i\in\{0,1\}$, we thus study the event $$B_i(t):=\{\Delta_r^{(2)}\leq i\ell,\ \forall r\geq\alpha-k-t\}\subset\mathcal{P}.$$ Let $t_0:=\lfloor\log\alpha\rfloor$ and $t_1:=t_0^4$. We get the following.
	
	\begin{lemma}\label{lemma_sumab}
		Let $\delta:=\tfrac1{C}$.
		\begin{enumerate}[(i)]
			\item\label{lemma_sumab_ab1} In $A(\lambda_0)\cap B_1(t_1)$, we have $\Sigma^{(2)}=O(1)$.
			\item\label{lemma_sumab_ab0} In $A(\lambda_0)\cap B_1(t_1)\cap B_0(t_0)$, we have $\Sigma^{(2)}=O(1/\log\alpha)$.
			\item\label{lemma_sumab_pi} In $\mathcal{P}^{\delta}\cap A(\lambda_0)\cap B_1(t_1)$, if $k=0$, we have $\Pi=O(1)$.
		\end{enumerate}
	\end{lemma}
	
	\begin{proof}
		(\ref{lemma_sumab_ab1},\ref{lemma_sumab_ab0}) Since $|\Psi_r^{(2)}-\psi_{p_r}^{(2)}|\leq\Delta_r^{(2)}+\Lambda_r^{(2)}+\psi_{p_r}^{(2)}$, if $\Delta_r^{(2)}\leq i\ell$, then we have $Q_r^{(2)}=O(M/\alpha+i/(\alpha-k-r)^2)$ by Lemma~\ref{lemma_bounds2}~(\ref{lemma_bounds2_nrnk}) and Lemma~\ref{lemma_psi2}. The results follow by Lemma~\ref{lemma_suma0}.
		
		(\ref{lemma_sumab_pi}) We have $$\log\Pi=\sum_{r=0}^{\alpha-1}\log(1+\widetilde{Q}_r)\leq\sum_{r=0}^{\alpha-1}|\widetilde{Q}_r|.$$ In $\mathcal{P}^{\delta}$, we have $|\widetilde{Q}_r|\leq CQ_r$ for all $r=0,\ldots,\alpha-1$, so Lemma~\ref{lemma_suma0} gives $$\sum_{\alpha-r\geq t_1}^{\alpha-r\leq\alpha}|\widetilde{Q}_r|=O\left(\tfrac{C}{\log\alpha}\right)=o(1)$$ in $\mathcal{P}^{\delta}\cap A(\lambda_0)$. For $\alpha-r<t_1$, in $B_1(t_1)$, we have
		\begin{align}
			|\Psi_r-\psi_{p_r}|&\leq\Delta_r+\Lambda_r+\psi_{p_r}\leq\ell+O((\alpha-r)^2M/\alpha)=O(1),
			\\\Psi_r&=\Delta_r+\Lambda_r\leq\ell+O((\alpha-r)^2M/\alpha)=O(1)
		\end{align}
		by Lemma~\ref{lemma_psi2}. Since $n_{\alpha-r}\geq(\alpha-r)^2$ by Lemma~\ref{lemma_nr}~(\ref{lemma_nr_geq}), for some constant $T$ and for $\alpha$ large enough, we get $n_{\alpha-r}-\Psi_r=\Omega((\alpha-r)^2)$ for $\alpha-r\geq T$. It follows that $$\sum_{\alpha-r\geq T}^{\alpha-r\leq t_1-1}|\widetilde{Q}_r|=O\left(\sum_{\alpha-r\geq T}^{\alpha-r\leq t_1-1}\frac{1}{(\alpha-r)^2}\right)=O(1).$$ For $\alpha$ large enough, such that $\b_{\max}\leq\tfrac12$, we furthermore have $$n_{\alpha-r}-\Psi_r=\sum_{v\in V_r}(1-\b(v))\geq\tfrac12|V_r|.$$ Since $|V_r|\geq|S\setminus S_r|=\alpha-r$, we finally get $$\sum_{\alpha-r\geq1}^{\alpha-r\leq T-1}|\widetilde{Q}_r|=O\left(\sum_{\alpha-r\geq 0}^{\alpha-r\leq T-1}\frac{1}{\alpha-r}\right)=O(1).$$ The result follows.
	\end{proof}
	
	\begin{lemma}\label{lemma_pb}
		For $t\geq1$ and $i\in\{0,1\}$, we have $\mathbb{P}(\mathcal{P}\setminus B_i(t))=O((t^2M/\alpha)^{i+1})$.
	\end{lemma}
	
	\begin{proof}
		Let $r:=\alpha-k-t$. In $\mathcal{P}\setminus B_i(t)$, we have $\Delta_{r'}^{(2)}\geq i\ell+1$ for some $r'\geq\alpha-k-t$, so by Lemma~\ref{lemma_delta2}, there are at least $i+1$ pairs of interfering vertices in $S\setminus S_{r'}$, which are thus also in $S\setminus S_r$. By Markov's inequality, we find that $\mathbb{P}(\mathcal{P}\setminus B_0(t))$ is at most the expected number of pairs of interfering vertices in $S\setminus S_r$, and $\mathbb{P}(\mathcal{P}\setminus B_1(t))$ is at most the expected number of pairs of distinct pairs of interfering vertices in $S\setminus S_r$.
		
		Since any vertex $s\in S$ interferes with at most $F_{\max}+R_{\max}^{(b)}\leq2M$ other vertices in $S$, there are at most $\alpha M$ interfering pairs of vertices in $S$. Since $S\setminus S_r$ is a uniformly random subset of $S$ of size $\alpha-k-r$, each such pair is in $S\setminus S_r$ with probability $$\tbinom{\alpha-k-2}r/\tbinom{\alpha-k}r=\tfrac{(\alpha-k-r)(\alpha-k-r-1)}{(\alpha-k)(\alpha-k-1)}=O(t^2/\alpha^2).$$ The bound on $\mathbb{P}(\mathcal{P}\setminus B_0(t))$ follows.
		
		We also find that there are at most $\binom{\alpha M}{2}$ pairs of distinct pairs of interfering vertices in $S$. Each such pair is in $S\setminus S_r$ with probability $$\tbinom{\alpha-k-4}r/\tbinom{\alpha-k}r=O(t^4/\alpha^4)$$ if they do not overlap. For any vertex $s\in S$, at most $\binom{2M}{2}$ such pairs overlap in $s$, so we have at most $\alpha\binom{2M}{2}$ overlapping pairs in total. Each such pair is in $S\setminus S_r$ with probability $$\tbinom{\alpha-k-3}r/\tbinom{\alpha-k}r=O(t^3/\alpha^3).$$ The bound on $\mathbb{P}(\mathcal{P}\setminus B_1(t))$ follows.
	\end{proof}
	
	\begin{lemma}\label{lemma_b1}
		We have $$\mathbb{E}(e^{C\Sigma^{(2)}}1_{A(\lambda_0)\setminus B_1(t_1)})=O(M^2/\alpha).$$
	\end{lemma}
	
	\begin{proof}
		Note that $$C\leq(1-\Omega(1))\tfrac1{8}\left(1-\tfrac{\log M}{\log\alpha}\right)^{-1}\tfrac{\log\alpha}{\log\log\alpha}.$$ By Lemma~\ref{lemma_sai} and Lemma~\ref{lemma_pb}, we get
		\begin{align}
			&\mathbb{E}(e^{C\Sigma^{(2)}}1_{A(\lambda_0)\setminus B_1(t_1)})
			\\&\leq(t_1^4M^2/\alpha^2)\exp\left(\tfrac{2C\lambda_0\log\log\alpha}{5\log\alpha}+O\left(\tfrac{C\lambda_0}{\log\alpha}\right)\right)
			\\&\leq(M^2/\alpha^2)\exp\left((1-\Omega(1))\log\alpha\right)=O(M^2/\alpha),
		\end{align}
		in particular, because $\log(t_1)=o(\log\alpha)$ and $C\lambda_0/\log\alpha=o(\log\alpha)$.
	\end{proof}
	
	We can now finally prove Lemma~\ref{lemma_sigma2}.
	
	\begin{proof}[Proof of Lemma~\ref{lemma_sigma2}]
		Let $\mathcal{R}:=A(\lambda_0)\cap B_1(t_1)\cap B_0(t_0)$.
		
		(\ref{lemma_sigma2_pi}) We have $|1+\overline{\Sigma}|\leq1+\Sigma\leq e^{\Sigma}$, since $|\overline{Q}_r|\leq Q_r$, and in $\mathcal{P}^{\delta}$, we have $\Pi\leq e^{C\Sigma}$ by \eqref{equation_pdelta}. By Lemma~\ref{lemma_an} and Lemma~\ref{lemma_b1}, we get $$\mathbb{E}((\Pi+|1+\overline{\Sigma}|)1_{\mathcal{P}^{\delta}\setminus(A(\lambda_0)\cap B_1(t_1))})\leq2\mathbb{E}(e^{C\Sigma}1_{\mathcal{P}\setminus(A(\lambda_0)\cap B_1(t_1))})=O(M^2/\alpha).$$ In $\mathcal{P}^{\delta}\cap A(\lambda_0)\cap B_1(t_1)$, we furthermore have $\Pi=O(1)$ and $|1+\overline{\Sigma}|\leq e^{\Sigma}=O(1)$ by Lemma~\ref{lemma_sumab}~(\ref{lemma_sumab_pi},\ref{lemma_sumab_ab1}). By Lemma~\ref{lemma_pb}, we get $$\mathbb{E}((\Pi+|1+\overline{\Sigma}|)1_{\mathcal{P}^{\delta}\cap A(\lambda_0)\cap B_1(t_1)\setminus B_0(t_0)})=O(\mathbb{P}(\mathcal{P}\setminus B_0(t_0)))=O(t_0^2M/\alpha).$$ The result follows.
		
		(\ref{lemma_sigma2_exp}) By Lemma~\ref{lemma_an} and Lemma~\ref{lemma_b1}, we have $$\mathbb{E}(e^{C\Sigma^{(2)}}1_{\mathcal{P}\setminus(A(\lambda_0)\cap B_1(t_1))})=O(M^2/\alpha).$$ By Lemma~\ref{lemma_sumab}~(\ref{lemma_sumab_ab1}) and Lemma~\ref{lemma_pb}, since $C=o(\log\alpha)$, we furthermore have $$\mathbb{E}(e^{C\Sigma^{(2)}}1_{A(\lambda_0)\cap B_1(t_1)\setminus B_0(t_0)})\leq t_0^2M\alpha^{-1+o(1)}=o(1).$$ The result follows.
		
		(\ref{lemma_sigma2_out}) This follows by Lemma~\ref{lemma_sumab}~(\ref{lemma_sumab_ab0}).
	\end{proof}
	
	\subsection{Error terms}\label{subsection_proof_error}
	
	We prove Lemma~\ref{lemma_bar_e} and Lemma~\ref{lemma_error}. We may thus assume $k=0$. We mainly need to study the random variables $\Psi_r-\psi_{p_r}$ for $r=0,\ldots,\alpha-1$. We split up
	\begin{align}
		Y_r^{(a)}&:=\Delta_r(F\cup R^*(S))-\Lambda_r(F\cup R^*(S))
		\\&-\mathbb{E}(\Delta_{p_r}(F\cup R^*(S))-\Lambda_{p_r}(F\cup R^*(S))),
		\\Y_r^{(b)}&:=\Delta_r(R^*(S\setminus S_r))-\Lambda_r(R^*(S\setminus S_r))
		\\&-\mathbb{E}(\Delta_{p_r}(R^*(S\setminus S_{p_r}))-\Lambda_{p_r}(R^*(S\setminus S_{p_r}))),
		\\Y_r^{(c)}&:=\Lambda_r(V\setminus S)-\mathbb{E}(\Lambda_{p_r}(V\setminus S)),
		\\Y_r^{(d)}&:=\Lambda_r(S)-\mathbb{E}(\Lambda_{p_r}(S)),
	\end{align}
	such that $\Psi_r-\psi_{p_r}=Y_r^{(a)}-Y_r^{(b)}+Y_r^{(c)}+Y_r^{(d)}$. For every $s\in S$, we consider the random variable $x_s^r$ that is $0$ if $s\in S_r$ and $1$ otherwise, and the random variable $\widetilde{x}_s^r$ that is $0$ if $s\in S_{p_r}$ and $1$ otherwise. We get
	\begin{align}
		Y_r^{(a)}&=\sum_{\substack{v\in F\cup R^*(S)\\N(v)\cap S=\{s_1,s_2\}}}(1-\b(v))(x_{s_1}^rx_{s_2}^r-\mathbb{E}(\widetilde{x}_{s_1}^r\widetilde{x}_{s_2}^r)),
		\\Y_r^{(b)}&=\sum_{\substack{s_1\in S\\v\in R^*(s_1)\\N(v)\cap S=\{s_2,s_3\}}}(1-\b(v))(x_{s_1}^rx_{s_2}^rx_{s_3}^r-\mathbb{E}(\widetilde{x}_{s_1}^r\widetilde{x}_{s_2}^r\widetilde{x}_{s_3}^r)),
		\\Y_r^{(c)}&=\sum_{\substack{v\in V\setminus S\\N(v)\cap S=\{s_1,s_2\}}}\b(v)(x_{s_1}^rx_{s_2}^r-\mathbb{E}(\widetilde{x}_{s_1}^r\widetilde{x}_{s_2}^r)),
		\\Y_r^{(d)}&=\sum_{s\in S}\b(s)(x_s^r-\mathbb{E}(\widetilde{x}_s^r)).
	\end{align}
	To bound the expected values of these quantities, we use the following lemma.
	
	\begin{lemma}\label{lemma_x}
		For $r=0,\ldots,\alpha-1$ and distinct vertices $s_1,\ldots,s_d\in S$ for some constant $d\geq1$, we have
		\begin{enumerate}[(i)]
			\item\label{lemma_x_simple} $\begin{aligned}[t]
				\mathbb{E}(\widetilde{x}_{s_1}^r\ldots\widetilde{x}_{s_d}^r)=q_r^d,
			\end{aligned}$
			\item\label{lemma_x_true} $\begin{aligned}[t]
				\mathbb{E}(x_{s_1}^r\ldots x_{s_d}^r)=\tbinom{\alpha-d}{\alpha-r-d}/\tbinom{\alpha}{\alpha-r}=q_r^d+O(q_r^{d-1}/\alpha).
			\end{aligned}$
		\end{enumerate}
	\end{lemma}
	
	\begin{proof}
		(\ref{lemma_x_simple}) This follows from the definition of the simplified model, since the variables $\widetilde{x}_{s_1}^r,\ldots,\widetilde{x}_{s_d}^r$ are independent with expected value $q_r$.
		
		(\ref{lemma_x_true}) The product $x_{s_1}^r\ldots x_{s_d}^r$ is $1$ if $s_1,\ldots,s_d\in S\setminus S_r$ and $0$ otherwise. Since $S\setminus S_r$ is a uniformly random subset of $S$ of size $\alpha-r$, we get
		\begin{align}
			\mathbb{E}(x_{s_1}^r\ldots x_{s_d}^r)
			&=\tbinom{\alpha-d}{\alpha-r-d}/\tbinom{\alpha}{\alpha-r}
			\\&=\tfrac{(\alpha-r)!/(\alpha-r-d)!}{\alpha!/(\alpha-d)!}
			\\&=\tfrac{\alpha-r}{\alpha}\cdot\tfrac{\alpha-r-1}{\alpha-1}\cdot\ldots\cdot\tfrac{\alpha-r-d+1}{\alpha-d+1}.
		\end{align}
		For $i=0,\ldots,d-1$, we have $$\tfrac{\alpha-r-i}{\alpha-i}=\tfrac{\alpha-r}{\alpha}\left(1+\tfrac{i}{\alpha-i}\right)\left(1-\tfrac{i}{\alpha-r}\right)=q_r\left(1+O(\tfrac1{\alpha-r})\right).$$ We thus get $\mathbb{E}(x_{s_1}^r\ldots x_{s_d}^r)=q_r^d(1+O(\tfrac1{\alpha-r}))=q_r^d+O(q_r^{d-1}/\alpha)$.
	\end{proof}
	
	We can now bound $\Psi_r-\psi_{p_r}$.
	
	\begin{lemma}\label{lemma_psipsi}
		For $r=0,\ldots,\alpha-1$, we have
		\begin{enumerate}[(i)]
			\item\label{lemma_psipsi_1} $\begin{aligned}[t]
				\mathbb{E}(\Psi_r-\psi_{p_r})=O(q_rM),
			\end{aligned}$
			\item\label{lemma_psipsi_2} $\begin{aligned}[t]
				\mathbb{E}((\Psi_r-\psi_{p_r})^2)=O(q_r^3\alpha M^2+q_r^2\alpha M),
			\end{aligned}$
			\item\label{lemma_psipsi_abs} $\begin{aligned}[t]
				\mathbb{E}(|\Psi_r-\psi_{p_r}|)=O(\sqrt{q_r^3\alpha M^2}+\sqrt{q_r^2\alpha M}).
			\end{aligned}$
		\end{enumerate}
	\end{lemma}
	
	\begin{proof}
		(\ref{lemma_psipsi_1}) We have $\mathbb{E}(\Psi_r-\psi_{p_r})=\mathbb{E}(Y_r^{(a)})-\mathbb{E}(Y_r^{(b)})+\mathbb{E}(Y_r^{(c)})+\mathbb{E}(Y_r^{(d)})$. By Lemma~\ref{lemma_x} and Lemma~\ref{lemma_terms}, we get
		\begin{align}
			\mathbb{E}(Y_r^{(a)})&=O(|F\cup R^*(S)|q_r/\alpha)=O(q_rM),
			\\\mathbb{E}(Y_r^{(b)})&=O(|R^*(S)|q_r^2/\alpha)=O(q_r^2M),
			\\\mathbb{E}(Y_r^{(c)})&=O(|V\setminus S|\b_{\max}q_r/\alpha)=O(q_rM),
			\\\mathbb{E}(Y_r^{(d)})&=O(|S|\b_{\max}/\alpha)=O(M/\alpha).
		\end{align}
		Since $q_r\geq1/\alpha$, the result follows.
		
		(\ref{lemma_psipsi_2}) By convexity of the function $x\mapsto x^2$, we have
		\begin{align}
			(\Psi_r-\psi_{p_r})^2&=(\tfrac14(4Y_r^{(a)}-4Y_r^{(b)}+4Y_r^{(c)}+4Y_r^{(d)}))^2
			\\&\leq\tfrac14((4Y_r^{(a)})^2+(4Y_r^{(b)})^2+(4Y_r^{(c)})^2+(4Y_r^{(d)})^2)
			\\&=4(Y_r^{(a)})^2+4(Y_r^{(b)})^2+4(Y_r^{(c)})^2+4(Y_r^{(d)})^2.
		\end{align}
		By Lemma~\ref{lemma_x}, we thus need to study, for two sequences of distinct vertices $s_1,\ldots,s_d\in S$ and $s'_1,\ldots,s'_d\in S$ for some constant $d\geq1$, how the expected value
		\begin{align}
			&\mathbb{E}((x_{s_1}^r\ldots x_{s_d}^r-q_r^d)(x_{s'_1}^r\ldots x_{s'_d}^r-q_r^d))
			\\&=\mathbb{E}(x_{s_1}^r\ldots x_{s_d}^rx_{s'_1}^r\ldots x_{s'_d}^r)-\mathbb{E}(x_{s_1}^r\ldots x_{s_d}^r)\mathbb{E}(x_{s'_1}^r\ldots x_{s'_d}^r)
			\\&+\mathbb{E}(x_{s_1}^r\ldots x_{s_d}^r-q_r^d)\mathbb{E}(x_{s'_1}^r\ldots x_{s'_d}^r-q_r^d)
			\\&=\mathbb{E}(x_{s_1}^r\ldots x_{s_d}^rx_{s'_1}^r\ldots x_{s'_d}^r)-q_r^{2d}+O(q_r^{2d-1}/\alpha)
		\end{align}
		depends on the overlap $i:=|\{s_1,\ldots,s_d\}\cap\{s'_1,\ldots,s'_d\}|$. If $i=0$, we get $$\mathbb{E}((x_{s_1}^r\ldots x_{s_d}^r-q_r^d)(x_{s'_1}^r\ldots x_{s'_d}^r-q_r^d))=O(q_r^{2d-1}/\alpha),$$ and otherwise, we get $$\mathbb{E}((x_{s_1}^r\ldots x_{s_d}^r-q_r^d)(x_{s'_1}^r\ldots x_{s'_d}^r-q_r^d))=O(q_r^{2d-i}).$$
		
		For $Y_r^{(a)}$, the expected value $\mathbb{E}((Y_r^{(a)})^2)$ is bounded by $$\sum_{\substack{v\in F\cup R^*(S)\\N(v)\cap S=\{s_1,s_2\}}}\sum_{\substack{v'\in F\cup R^*(S)\\N(v')\cap S=\{s_3,s_4\}}}|\mathbb{E}((x_{s_1}^rx_{s_2}^r-q_r^2)(x_{s_3}^rx_{s_4}^r-q_r^2))|.$$ We use that $|F\cup R^*(S)|=O(\alpha M)$ by Lemma~\ref{lemma_terms}~(\ref{lemma_terms_f},\ref{lemma_terms_rs}).
		\begin{itemize}
			\item Since any pair $\{s_1,s_2\}$ appears at most $|N(s_1)\cap N(s_2)|=\ell$ times in $Y_r^{(a)}$, there are at most $\ell|F\cup R^*(S)|=O(\alpha M)$ terms of order $O(q_r^2)$.
			\item Any vertex $s\in S$ appears at most $F_{\max}+R_{\max}^{(b)}$ times in $Y_r^{(a)}$, so there are $O(M^2\alpha)$ remaining terms of order $O(q_r^3)$.
			\item There are at most $|F\cup R^*(S)|^2=O(M^2\alpha^2)$ remaining terms, which are of order $O(q_r^3/\alpha)$.
		\end{itemize}
		We conclude that $$\mathbb{E}((Y_r^{(a)})^2)=O(q_r^3M^2\alpha+q_r^2\alpha M).$$
		
		For $Y_r^{(b)}$, the expected value $\mathbb{E}((Y_r^{(b)})^2)$ is bounded by $$\sum_{\substack{s_1\in S\\v\in R^*(s_1)\\N(v)\cap S=\{s_2,s_3\}}}\sum_{\substack{s_4\in S\\v'\in R^*(s_4)\\N(v')\cap S=\{s_5,s_6\}}}|\mathbb{E}((x_{s_1}^rx_{s_2}^rx_{s_3}^r-q_r^3)(x_{s_4}^rx_{s_5}^rx_{s_6}^r-q_r^3))|.$$ We use that $|R^*(S)|=O(\alpha M)$ by Lemma~\ref{lemma_terms}~(\ref{lemma_terms_rs}).
		\begin{itemize}
			\item Since any triple $\{s_1,s_2,s_3\}$ appears at most $$|N(s_1)\cap N(s_2)|+|N(s_1)\cap N(s_3)|+|N(s_2)\cap N(s_3)|=3\ell$$ times in $Y_r^{(b)}$, there are at most $3\ell|R^*(S)|=O(\alpha M)$ terms of order $O(q_r^3)$.
			\item Any two distinct vertices $s,t\in S$ appear together at most $$|N(s)\cap N(t)\cap R^*(S)|+|N(s)\cap R^*(t)|+|N(t)\cap R^*(s)|\leq\ell+2R_{\max}^{(c)}$$ times in $Y_r^{(b)}$, so there are at most $(\ell+2R_{\max}^{(c)})|R^*(S)|=O(M^{\sfrac32}\alpha)$ remaining terms of order $O(q_r^4)$.
			\item Any vertex $s\in S$ appears at most $$|R^*(s)|+|N(s)\cap R^*(S)|\leq R_{\max}^{(a)}+R_{\max}^{(b)}$$ times in $Y_r^{(b)}$, so there are $O(M^2\alpha)$ remaining terms of order $O(q_r^5)$.
			\item There are at most $|R^*(S)|^2=O(M^2\alpha^2)$ remaining terms, which are of order $O(q_r^5/\alpha)$.
		\end{itemize}
		Since $q_r\leq1\leq M$, we conclude that $$\mathbb{E}((Y_r^{(b)})^2)=O(q_r^3M^2\alpha).$$
		
		For $Y_r^{(c)}$, the expected value $\mathbb{E}((Y_r^{(c)})^2)$ is bounded by $$\sum_{\substack{v\in V\setminus S\\N(v)\cap S=\{s_1,s_2\}}}\sum_{\substack{v'\in V\setminus S\\N(v')\cap S=\{s_3,s_4\}}}\b_{\max}^2|\mathbb{E}((x_{s_1}^rx_{s_2}^r-q_r^2)(x_{s_3}^rx_{s_4}^r-q_r^2))|.$$ We use that $|V|=O(\alpha^2)$ by Lemma~\ref{lemma_terms}~(\ref{lemma_terms_n}).
		\begin{itemize}
			\item Since any pair $\{s_1,s_2\}$ appears at most $|N(s_1)\cap N(s_2)|=\ell$ times in $Y_r^{(c)}$, there are at most $\ell|V|=O(\alpha^2)$ terms of order $O(q_r^2M^2/\alpha^2)$.
			\item Any vertex $s\in S$ appears at most $d(G)=\ell(\alpha-1)=O(\alpha)$ times in $Y_r^{(c)}$ by Proposition~\ref{proposition_n}, so there are $O(\alpha|V|)=O(\alpha^3)$ remaining terms of order $O(q_r^3M^2/\alpha^2)$.
			\item There are at most $|V|^2=O(\alpha^4)$ remaining terms, which are of order $O(q_r^3M^2/\alpha^3)$.
		\end{itemize}
		Since $q_r\alpha=\alpha-r\geq1$, we conclude that $$\mathbb{E}((Y_r^{(c)})^2)=O(q_r^3M^2\alpha).$$
		
		Finally, for $Y_r^{(d)}$, the expected value $\mathbb{E}((Y_r^{(d)})^2)$ is bounded by $$\sum_{s_1\in S}\sum_{s_2\in S}\b_{\max}^2|\mathbb{E}((x_{s_1}^r-q_r)(x_{s_2}^r-q_r))|.$$
		\begin{itemize}
			\item Since any vertex $s_1\in S$ appears only once in $Y_r^{(d)}$, there are $\alpha$ terms of order $O(q_rM^2/\alpha^2)$.
			\item There are at most $\alpha^2$ remaining terms, which are of order $O(q_rM^2/\alpha^3)$.
		\end{itemize}
		Since $q_r\alpha=\alpha-r\geq1$, we conclude that $$\mathbb{E}((Y_r^{(d)})^2)=O(q_rM^2/\alpha)=O(q_r^3M^2\alpha).$$ The result follows.
		
		(\ref{lemma_psipsi_abs}) By Jensen's inequality, we have $(\mathbb{E}(|\Psi_r-\psi_{p_r}|))^2\leq\mathbb{E}((\Psi_r-\psi_{p_r})^2)$, so the result follows by (\ref{lemma_psipsi_2}) and subadditivity of the square root.
	\end{proof}
	
	We can now prove Lemma~\ref{lemma_bar_e} and Lemma~\ref{lemma_error}.
	
	\begin{proof}[Proof of Lemma~\ref{lemma_bar_e}]
		By Lemma~\ref{lemma_psipsi}~(\ref{lemma_psipsi_1}), we have $\mathbb{E}(\overline{Q}_r)=O(q_rM/n_{\alpha-r})$ for $r=0,\ldots,\alpha-1$, so the result follows by Lemma~\ref{lemma_sums}~(\ref{lemma_sums_3}).
	\end{proof}
	
	\begin{proof}[Proof of Lemma~\ref{lemma_error}]
		(\ref{lemma_error_diff}) In $\mathcal{P}^{\delta}$, for $r=0,\ldots,\alpha-1$, we have
		\begin{align}
			|\widetilde{Q}_r-\overline{Q}_r|
			&=\left|\frac{\Psi_r-\psi_{p_r}}{n_{\alpha-r}-\Psi_r}-\frac{\Psi_r-\psi_{p_r}}{n_{\alpha-r}}\right|
			\\&=\frac{|\Psi_r-\psi_{p_r}|\Psi_r}{n_{\alpha-r}(n_{\alpha-r}-\Psi_r)}
			\\&\leq\frac{|\Psi_r-\psi_{p_r}|(\psi_{p_r}+|\Psi_r-\psi_{p_r}|)}{\delta n_{\alpha-r}^2}
			\\&=O\left(\frac{(\Psi_r-\psi_{p_r})^2+\psi_{p_r}|\Psi_r-\psi_{p_r}|}{n_{\alpha-r}^2}\right).
		\end{align}
		By Lemma~\ref{lemma_psipsi}~(\ref{lemma_psipsi_2},\ref{lemma_psipsi_abs}), Lemma~\ref{lemma_psi2}~(\ref{lemma_psi2_psi}), and Lemma~\ref{lemma_nr}~(\ref{lemma_nr_geq}), we get
		\begin{align}
			\mathbb{E}(|\widetilde{Q}_r-\overline{Q}_r|)
			&=O\left(\frac{q_r^3\alpha M^2}{(\alpha-r)^4}+\frac{q_r^2\alpha M}{(\alpha-r)^4}+\frac{M\sqrt{q_r^3\alpha M^2}}{\alpha(\alpha-r)^2}+\frac{M\sqrt{q_r^2\alpha M}}{\alpha(\alpha-r)^2}\right)
			\\&=O\left(\frac{M^2/\alpha^2}{\alpha-r}+\frac{M/\alpha}{(\alpha-r)^2}+\frac{M^2/\alpha^2}{(\alpha-r)^{\sfrac12}}+\frac{M^{\sfrac32}/\alpha^{\sfrac32}}{\alpha-r}\right).
		\end{align}
		Summing over $r=0,\ldots,\alpha-1$, the result follows, since $M=O(\alpha^{\sfrac12}/(\log\alpha)^2)$.
		
		(\ref{lemma_error_sum}) By Lemma~\ref{lemma_psipsi}~(\ref{lemma_psipsi_2}) and Lemma~\ref{lemma_nr}~(\ref{lemma_nr_geq}), in $\mathcal{P}^{\delta}$, we have
		\begin{align}
			\mathbb{E}(\widetilde{Q}_r^2)
			&\leq\frac{\mathbb{E}((\Psi_r-\psi_{p_r})^2)}{\delta^2n_{\alpha-r}^2}
			\\&=O\left(\frac{q_r^3\alpha M^2}{(\alpha-r)^4}+\frac{q_r^2\alpha M}{(\alpha-r)^4}\right)
			\\&=O\left(\frac{M^2/\alpha^2}{\alpha-r}+\frac{M/\alpha}{(\alpha-r)^2}\right).
		\end{align}
		Summing over $r=0,\ldots,\alpha-1$, the result follows, since $M=O(\alpha^{\sfrac12}/(\log\alpha)^2)$.
		
		(\ref{lemma_error_log}) We apply Titu's lemma, a variant of Cauchy-Schwartz, to the sequences $a_r:=\log(1+\widetilde{Q}_r)$ and $b_r:=\tfrac1{\alpha-r}$ for $r=0,\ldots,\alpha-1$. We get
		\begin{align}
			(\log\Pi)^2&=\left(\sum_{r=0}^{\alpha-1}a_r\right)^2\leq\left(\sum_{r=0}^{\alpha-1}\frac{a_r^2}{b_r}\right)\left(\sum_{r=0}^{\alpha-1}b_r\right)
			\\&=\left(\sum_{r=0}^{\alpha-1}(\alpha-r)(\log(1+\widetilde{Q}_r))^2\right)\left(\sum_{r=0}^{\alpha-1}\frac1{\alpha-r}\right).
		\end{align}
		In $\mathcal{P}^{\delta}\cap\mathcal{R}$, for $r=0,\ldots,\alpha-1$, we have $|\widetilde{Q}_r|\leq\tfrac1{\delta}Q_r\leq\tfrac1{\delta}\Sigma=o(1)$, so $\log(1+\widetilde{Q}_r)=O(\widetilde{Q}_r)$. Reusing the estimate on $\mathbb{E}(\widetilde{Q}_r^2)$ from (\ref{lemma_error_sum}) thus gives
		\begin{align}
			\mathbb{E}((\log\Pi)^2)
			&=O\left(\log\alpha\sum_{r=0}^{\alpha-1}(\alpha-r)\mathbb{E}(\widetilde{Q}_r^2)\right)
			\\&=O\left(\log\alpha\sum_{r=0}^{\alpha-1}\left(\frac{M^2}{\alpha^2}+\frac{M/\alpha}{\alpha-r}\right)\right)
		\end{align}
		The result follows by the definition of $\mathscr{E}$.
	\end{proof}
	
	\subsection{Singularly adjacent terms}\label{subsection_proof_psi1}
	
	We prove Lemma~\ref{lemma_sigma1hat}. We have $(\Psi_r^{(1)}-\psi_{p_r}^{(1)})^+\leq\Delta_r^{(1)}:=\Delta_r(F^{(1)})+\Delta_r(R_r^{(1)})$, so $\hat{Q}_r^{(1)}\leq\Delta_r^{(1)}/(n_{\alpha-r}-n_k)$. For $s\in S$, define $f(s):=|N(s)\cap(F^{(1)}\cup R^{(1)})|$, such that $\Delta_r^{(1)}\leq\sum_{s\in S\setminus S_r}f(s)$. Recall that $S\setminus S_r=\{s_r,\ldots,s_{\alpha-k-1}\}$ with $\pi=(s_0,\ldots,s_{\alpha-k-1})\in\mathcal{P}$ uniformly random. We get $$\hat{\Sigma}^{(1)}\leq\sum_{r=0}^{\alpha-k-1}\frac{\Delta_r^{(1)}}{n_{\alpha-r}-n_k}\leq\sum_{r=0}^{\alpha-k-1}\sum_{i=r}^{\alpha-k-1}\frac{f(s_i)}{n_{\alpha-r}-n_k}.$$ Let $\theta(t):=\sum_{r=0}^{\alpha-k-t}\frac1{n_{\alpha-r}-n_k}$, such that switching the order of summation and substituting $t=\alpha-k-i$ gives $$\hat{\Sigma}^{(1)}\leq\sum_{i=0}^{\alpha-k-1}f(s_i)\theta(\alpha-k-i)=\sum_{t=1}^{\alpha-k}f(s_{\alpha-k-t})\theta(t).$$ We bound $f(s)$ and $\theta(t)$.
	
	\begin{lemma}\label{lemma_fs}
		We have
		\begin{enumerate}[(i)]
			\item\label{lemma_fs_one} $f(s)\leq\ell k$ for all $s\in S$,
			\item\label{lemma_fs_sum} $\sum_{s\in S}f(s)\leq2kM$.
		\end{enumerate}
	\end{lemma}
	
	\begin{proof}
		(\ref{lemma_fs_one}) Let $\overline{S}$ be an MIS that contains $S$. Since $G$ is $2$-uniform, any element of $N(s)\cap V^{(1)}$ is adjacent to exactly one vertex in $\overline{S}\setminus S$. There are $k$ such vertices, each having $\ell$ neighbors in $N(s)$, so $f(s)\leq|N(s)\cap V^{(1)}|=\ell k$.
		
		(\ref{lemma_fs_sum}) Since every vertex in $F^{(1)}\cup R^{(1)}$ is adjacent to exactly one vertex in $S$, we have $\sum_{s\in S}f(s)=|F^{(1)}\cup R^{(1)}|$. The result follows by Lemma~\ref{lemma_bounds1}~(\ref{lemma_bounds1_f1},\ref{lemma_bounds1_r1}).
	\end{proof}
	
	\begin{lemma}\label{lemma_theta}
		For $t=1,\ldots,\alpha-k$, we have $\theta(t)\leq\frac{1}{\ell k}(\frac1t+\log(1+\frac{\ell k}{t}))\leq\frac2t$.
	\end{lemma}
	
	\begin{proof}
		Substituting $i=\alpha-k-r$ gives $$\theta(t)=\sum_{i=t}^{\alpha-k}\frac1{n_{k+i}-n_k}\leq\sum_{i=t}^{\infty}\frac1{i^2+\ell ki}=\frac1{\ell k}\sum_{i=t}^{\infty}\left(\frac1{i}-\frac1{i+\ell k}\right),$$ since Lemma~\ref{lemma_nr}~(\ref{lemma_nr_sum},\ref{lemma_nr_geq}) give $n_{k+i}-n_k=n_i+\ell ki\geq i^2+\ell ki$. Telescoping gives $$\theta(t)\leq\frac1{\ell k}\sum_{i=t}^{t+\ell k-1}\frac1{i}\leq\frac1{\ell k}\left(\frac1{t}+\int_t^{t+\ell k}\frac1x\,\mathrm{d}x\right)=\frac1{\ell k}\left(\frac1{t}+\log\left(1+\frac{\ell k}{t}\right)\right).$$ The result follows, since $\log(1+x)\leq x$ and $\ell k\geq1$.
	\end{proof}
	
	Let $\Sigma^{(l)}:=\sum_{t=1}^{\ell k}f(s_{\alpha-k-t})\theta(t)$ and $\Sigma^{(u)}:=\sum_{t=\ell k+1}^{\alpha-k}f(s_{\alpha-k-t})\theta(t)$, such that $\hat{\Sigma}^{(1)}\leq\Sigma^{(l)}+\Sigma^{(u)}$. Consider the random variable $T:=\{s_{\alpha-k-t}:t=\ell k+1,\ldots,\alpha-k\}$. We show that Lemma~\ref{lemma_sigma1hat} follows from the following two lemmas.
	
	\begin{lemma}\label{lemma_sigmal}
		Let $C=\Omega(1)$ be a parameter with $C\leq(\frac12-\Omega(1))\frac{\log\alpha}{\log\log\alpha}$. Then $$\mathbb{E}(e^{C\Sigma^{(l)}})\leq e^{o(k)}.$$
	\end{lemma}
	
	\begin{lemma}\label{lemma_sigmau}
		Let $C=\Omega(1)$ be a parameter with $C=o(\log\alpha)$. Then $$\mathbb{E}(e^{C\Sigma^{(u)}}|T)\leq e^{o(k)}.$$
	\end{lemma}
	
	\begin{proof}[Proof of Lemma~\ref{lemma_sigma1hat}]
		Conditioning on $T$, we find that $(s_{\alpha-k-t})_{t=1,\ldots,\ell k}$ and $(s_{\alpha-k-t})_{t=\ell k+1,\ldots,\alpha-k}$ are independent uniformly random permutations of $S\setminus T$ and $T$ respectively, so $\Sigma^{(l)}$ and $\Sigma^{(u)}$ are independent. By the law of total expectation, we get
		\begin{align}
			\mathbb{E}(e^{C\hat{\Sigma}^{(1)}})
			&=\mathbb{E}(\mathbb{E}(e^{C\hat{\Sigma}^{(1)}}|T))
			\\&\leq\mathbb{E}(\mathbb{E}(e^{C\Sigma^{(l)}}|T)\mathbb{E}(e^{C\Sigma^{(u)}}|T))
			\\&\leq e^{o(k)}\mathbb{E}(\mathbb{E}(e^{C\Sigma^{(l)}}|T))
			\\&=e^{o(k)}\mathbb{E}(e^{C\Sigma^{(l)}})\leq e^{o(k)}
		\end{align}
		by Lemma~\ref{lemma_sigmal} and Lemma~\ref{lemma_sigmau}.
	\end{proof}
	
	It remains only to prove Lemma~\ref{lemma_sigmal} and Lemma~\ref{lemma_sigmau}. Both rely on Theorem~9.A.23 from~\cite{BookSO} on stochastic orders.
	
	\begin{theorem}[Stochastic orders]\label{theorem_so}
		Let $X_1,\ldots,X_n$ be negatively associated random variables, and let $Y_1,\ldots,Y_n$ be independent random variables, such that $X_i$ and $Y_i$ have the same distribution for all $i$. Then for any supermodular function $\phi:\mathbb{R}^n\to\mathbb{R}$, we have $\mathbb{E}(\phi(X_1,\ldots,X_n))\leq\mathbb{E}(\phi(Y_1,\ldots,Y_n))$.
	\end{theorem}
	
	\begin{remark}
		Any uniform sample without replacement from a multiset of real numbers is negatively associated.
	\end{remark}
	
	\begin{remark}
		A $C^2$ function $f:\mathbb{R}^n\to\mathbb{R}$ is supermodular if and only if $\tfrac{\partial^2}{\partial x_i\partial x_j}f\geq0$ for all $i\neq j$.
	\end{remark}
	
	\begin{proof}[Proof of Lemma~\ref{lemma_sigmal}]
		Note that $X_t:=f(s_{\alpha-k-t})$ for $t=1,\ldots,\ell k$ forms a uniform sample without replacement from the multiset $\Omega:=\{f(s):s\in S\}$. We have $e^{C\Sigma^{(l)}}=\phi(X_1,\ldots,X_{\ell k})$, for $$\phi(x_1,\ldots,x_{\ell k}):=\exp\left(C\sum_{t=1}^{\ell k}\theta(t)x_t\right),$$ which is supermodular, since $\tfrac{\partial^2}{\partial x_i\partial x_j}\phi=C^2\theta(i)\theta(j)\phi\geq0$. For independent uniform samples $Y_1,\ldots,Y_{\ell k}$ from $\Omega$, by Theorem~\ref{theorem_so}, we have $$\mathbb{E}(e^{C\Sigma^{(l)}})\leq\mathbb{E}(\phi(Y_1,\ldots,Y_{\ell k}))=\prod_{t=1}^{\ell k}\left(\frac1{|S|}\sum_{s\in S}\exp(C\theta(t)f(s))\right).$$
		
		For fixed $t$, consider maximizing the sum $\Sigma_t(\widetilde{f}):=\sum_{s\in S}\exp(C\theta(t)\widetilde{f}(s))$ over all functions $\widetilde{f}$ satisfying the constraints $0\leq\widetilde{f}(s)\leq\ell k$ for all $s\in S$, and $\sum_{s\in S}\widetilde{f}(s)\leq 2kM$. Recall that $f$ satisfies these constraints by Lemma~\ref{lemma_fs}.
		
		Whenever there are two vertices $s_1,s_2\in S$ with $0<\widetilde{f}(s_1)\leq\widetilde{f}(s_2)<\ell k$, if we increase $\widetilde{f}(s_2)$ and decrease $\widetilde{f}(s_1)$ by the same amount, then $\widetilde{f}$ still satisfies its constraints, while $\Sigma_t(\widetilde{f})$ increases. It follows that $\Sigma_t(\widetilde{f})$ can only be maximized if there is at most one vertex $s$ with $0<\widetilde{f}(s)<\ell k$. We get $$\Sigma_t(f)\leq|S|+\lceil\tfrac{2}{\ell}M\rceil(\exp(C\theta(t)\ell k)-1).$$ By Lemma~\ref{lemma_theta}, we have $\theta(t)\ell k\leq1+\log(\frac{2\ell k}{t})=\log(\frac{2e\ell k}{t})$ for $t\leq\ell k$, so we get
		\begin{align}
			\mathbb{E}(e^{C\Sigma^{(l)}})
			&\leq\prod_{t=1}^{\ell k}\left(1+\frac{\lceil\frac{2}{\ell}M\rceil}{|S|}\exp(C\theta(t)\ell k)\right)
			\\&\leq\prod_{t=1}^{\ell k}(1+\alpha^{-\sfrac12}(\tfrac{2e\ell k}{t})^{C}),
		\end{align}
		for $\alpha$ large enough such that $\lceil\frac{2}{\ell}M\rceil/|S|\leq\alpha^{-\sfrac12}$. Let $t^*:=2e\ell k\cdot\alpha^{-\frac1{2C}}$, such that $1+x\leq e^x$ gives $$\mathbb{E}(e^{C\Sigma^{(l)}})\leq\prod_{t=1}^{\ell k}(1+(\tfrac{t^*}{t})^{C})\leq\prod_{t=1}^{\lfloor t^*\rfloor}(2(\tfrac{t^*}{t})^{C})\prod_{t=\lfloor t^*\rfloor+1}^{\ell k}\exp((\tfrac{t^*}{t})^{C})=:\Pi^{(l)}\Pi^{(u)}.$$ It thus suffices to show $\log\Pi^{(l)}=o(k)$ and $\log\Pi^{(u)}=o(k)$. Since both values are monotone with $C$, we may assume without loss of generality that $C\geq2$.
		
		For $\log\Pi^{(l)}$, we get an empty product if $t^*<1$, so we may assume $t^*\geq1$. We have $$\log\Pi^{(l)}=\lfloor t^*\rfloor(\log2+C\log t^*)-C\log(\lfloor t^*\rfloor!).$$ Since $x\mapsto\log x$ is monotonically increasing for $x\geq1$, we have
		\begin{align}
			\log(\lfloor t^*\rfloor!)
			&\geq\int_1^{\lfloor t^*\rfloor}\log x\,\mathrm{d}x
			\\&\geq\int_1^{t^*}\log x\,\mathrm{d}x-(t^*-\lfloor t^*\rfloor)\log t^*
			\\&=\lfloor t^*\rfloor\log t^*-t^*+1.
		\end{align}
		Since $C\geq\log2$, it follows that $$\log\Pi^{(l)}\leq(C+\log2)t^*\leq k\exp(\log C-\tfrac1{2C}\log\alpha+O(1)).$$ Since $\frac1{2C}\geq(1+\Omega(1))\frac{\log\log\alpha}{\log\alpha}$ and $\log C\leq\log\log\alpha+O(1)$, we get $$\log\Pi^{(l)}\leq k\exp(-\Omega(\log\log\alpha))=o(k).$$
		
		For $\log\Pi^{(u)}$, we have
		\begin{align}
			\log\Pi^{(u)}
			&=(t^*)^C\sum_{t=\lfloor t^*\rfloor+1}^{\ell k}t^{-C}
			\\&\leq\left(\frac{t^*}{\lfloor t^*\rfloor+1}\right)^{C}+(t^*)^C\int_{\lfloor t^*\rfloor+1}^{\ell k}x^{-C}\,\mathrm{d}x,
			\\\int_{\lfloor t^*\rfloor+1}^{\ell k}x^{-C}\,\mathrm{d}x
			&\leq\int_{t^*}^{\infty}x^{-C}\,\mathrm{d}x=\frac{(t^*)^{1-C}}{C-1}.
		\end{align}
		If $t^*\geq1$, then $(\frac{t^*}{\lfloor t^*\rfloor+1})^C\leq1\leq t^*$. Otherwise, $(\frac{t^*}{\lfloor t^*\rfloor+1})^C=(t^*)^C\leq t^*$, because $C\geq2$. It follows that $\log\Pi^{(u)}\leq(1+\tfrac1{C-1})t^*\leq Ct^*$, which we already showed to be $o(k)$.
	\end{proof}
	
	\begin{proof}[Proof of Lemma~\ref{lemma_sigmau}]
		For $s=s_{\alpha-k-t}\in T$, let $X_s:=\theta(t)$, such that $X_s$ for $s\in T$ forms a uniform sample without replacement from the set $\Omega:=\{\theta(t):t\in I\}$ with $I:=\{\ell k+1,\ldots,\alpha-k\}$. We have $e^{C\Sigma^{(u)}}=\phi((X_s)_{s\in T})$, for $$\phi((x_s)_{s\in T}):=\exp\left(C\sum_{s\in T}f(s)x_s\right),$$ which is supermodular, since $\frac{\partial^2}{\partial x_{s_1}\partial x_{s_2}}\phi=C^2f(s_1)f(s_2)\phi\geq0$. For independent uniform samples $Y_s$ for $s\in T$ from $\Omega$, by Theorem~\ref{theorem_so}, we have $$\mathbb{E}(e^{C\Sigma^{(u)}})\leq\mathbb{E}(\phi((Y_s)_{s\in T}))=\prod_{s\in T}\left(\frac1{|I|}\sum_{t\in I}\exp(Cf(s)\theta(t))\right).$$ By Lemma~\ref{lemma_theta}, we have $Cf(s)\theta(t)\leq2Cf(s)/t$. We use that $e^x-1=O(x)$ for $x=o(1)$ and $C=o(\log\alpha)$ to get
		\begin{align}
			&\frac1{|I|}\sum_{t\in I}\exp(Cf(s)\theta(t))
			\\&\leq1+\frac1{|I|}\sum_{t\in I}\left(\exp\left(\frac{2Cf(s)}{t}\right)-1\right)
			\\&\leq1+\frac1{|I|}\sum_{\substack{t\in I\\t\leq f(s)\log\alpha}}\exp\left(\frac{2Cf(s)}{t}\right)+\frac1{|I|}\sum_{\substack{t\in I\\t>f(s)\log\alpha}}O\left(\frac{2Cf(s)}{t}\right)
			\\&\leq1+\frac{1}{|I|}e^{2C}f(s)\log\alpha+O\left(\frac1{|I|}Cf(s)\log\alpha\right),
		\end{align}
		because $f(s)\leq\ell k\leq t$ for $t\in I$ by Lemma~\ref{lemma_fs}~(\ref{lemma_fs_one}), and $\sum_{t=1}^{\alpha}\frac1t=O(\log\alpha)$. Since $C=o(\log\alpha)$, we have $e^{2C}+O(C)=O(\alpha^{\sfrac12})$, so since $|I|=\Omega(\alpha)$, we get $$\mathbb{E}(e^{C\Sigma^{(u)}})\leq\prod_{s\in T}\left(1+O\left(\frac{f(s)\log\alpha}{\alpha^{\sfrac12}}\right)\right).$$ Since $f(s)\leq\ell k=O(M)$, we can use $\log(1+x)=O(x)$ for $x=o(1)$ to get $$\mathbb{E}(e^{C\Sigma^{(u)}})\leq\exp\left(O\left(\frac{\log\alpha}{\alpha^{\sfrac12}}\sum_{s\in T}f(s)\right)\right).$$ The result follows by Lemma~\ref{lemma_fs}~(\ref{lemma_fs_sum}).
	\end{proof}
	
	This finishes the proof of Lemma~\ref{lemma_sigma}.
	
	\subsection{Extremal probability}\label{subsection_proof_remaining}
	
	We prove Lemma~\ref{lemma_remaining_p}. For $i\in\{2,1,0\}$, we define $\Delta_r^{(i)}:=\Delta_r(F^{(i)})+\Delta_r(R_r^{(i)})$, such that $\Delta_r=\Delta_r^{(2)}+\Delta_r^{(1)}+\Delta_r^{(0)}$. Note that this coincides with the definitions given in Section~\ref{subsection_proof_stronger} and Section~\ref{subsection_proof_psi1}. We can thus bound $\Delta_r^{(2)}$ again with Lemma~\ref{lemma_delta2}, and bound $\Delta_r^{(1)}$ by $\sum_{s\in S\setminus S_r}f(s)$ again, for $f$ as defined in Section~\ref{subsection_proof_psi1}. We can furthermore bound $\Delta_r^{(0)}$ by $K$, since $R_r^{(0)}\subset R^{(0)}$.
	
	\begin{lemma}\label{lemma_boundpt}
		Let $0<\delta<1$ and $r=0,\ldots,\alpha-k-1$ with $t:=\alpha-k-r$. If $n_{\alpha-r}-\Psi_r<\delta(n_{\alpha-r}-K)$ and $M^2/\alpha=o(\delta)$, then $$\Delta_r^{(2)}+\Delta_r^{(1)}\geq(1-(1+o(1))\delta)(n_t+\ell tk).$$
	\end{lemma}
	
	\begin{proof}
		Since $\Delta_r^{(2)}+\Delta_r^{(1)}=\Psi_r-\Delta_r^{(0)}-\Lambda_r$ and $\Delta_r^{(0)}\leq K$, we have
		\begin{align}
			&\Delta_r^{(2)}+\Delta_r^{(1)}>(1-\delta)(n_{\alpha-r}-K)-\Lambda_r,
			\\&\Lambda_r\leq\b(\widetilde{V}_r)\leq\b_{\max}n_{\alpha-r}\leq\tfrac{M}{\alpha}\left(1+\tfrac{K}{n_{\alpha-r}-K}\right)(n_{\alpha-r}-K).
		\end{align}
		Note that $n_{\alpha-r}-K\geq n_t+\ell tk\geq k$ by Lemma~\ref{lemma_bounds2}~(\ref{lemma_bounds2_nrnk}) and Lemma~\ref{lemma_nr}~(\ref{lemma_nr_sum}), so $K/(n_{\alpha-r}-K)=O(k)=O(M)$ by Lemma~\ref{lemma_bounds1}~(\ref{lemma_bounds1_v0}). We get $$\Delta_r^{(2)}+\Delta_r^{(1)}\geq(1-\delta-O(M^2/\alpha))(n_{\alpha-r}-K),$$ so the result follows.
	\end{proof}
	
	To bound the probability of $\mathcal{P}\setminus\mathcal{P}^{\delta}$, it thus suffices to bound $\Delta_r^{(2)}$ and $\Delta_r^{(1)}$.
	
	\begin{lemma}\label{lemma_pd0}
		For $r=0,\ldots,\alpha-k-1$ with $t:=\alpha-k-r$, we have
		\begin{enumerate}[(i)]
			\item\label{lemma_pd0_d2} $\begin{aligned}[t]
				\mathbb{P}(\Delta_r^{(2)}\geq1)=O(t^2M/\alpha),
			\end{aligned}$
			\item\label{lemma_pd0_d1} $\begin{aligned}[t]
				\mathbb{P}(\Delta_r^{(1)}\geq1)=O(tkM/\alpha).
			\end{aligned}$
		\end{enumerate}
	\end{lemma}
	
	\begin{proof}
		(\ref{lemma_pd0_d2}) This follows from Lemma~\ref{lemma_pb}.
		
		(\ref{lemma_pd0_d1}) Recall that $S\setminus S_r=\{s_r,\ldots,s_{\alpha-k-1}\}$ with $\pi=(s_0,\ldots,s_{\alpha-k-1})\in\mathcal{P}$ uniformly random. We use Lemma~\ref{lemma_fs}. By Markov's inequality, we have $$\mathbb{P}(\Delta_r^{(1)}\geq1)\leq\mathbb{E}\left(\sum_{s\in S\setminus S_r}f(s)\right)=\sum_{i=r}^{\alpha-k-1}\mathbb{E}(f(s_i)).$$ By Lemma~\ref{lemma_fs}~(\ref{lemma_fs_sum}), we have $\mathbb{E}(f(s_i))\leq2kM/(\alpha-k)$ for any $i$. The result follows.
	\end{proof}
	
	\begin{lemma}\label{lemma_pd}
		Let $0<\delta<1$ and $r=0,\ldots,\alpha-k-1$ with $t:=\alpha-k-r$.
		\begin{enumerate}[(i)]
			\item\label{lemma_pd_2small} If $\delta\leq1-\Omega(1)$ and $t\leq\alpha\exp(-\omega(\log\log\alpha))$, we have $$\mathbb{P}(\Delta_r^{(2)}\geq(1-\delta)n_t)\leq\exp\left((1-\delta)t\log\tfrac{M}{\alpha}+O(t+\log\alpha)\right).$$
			\item\label{lemma_pd_2large} If $\delta=o(1)$, $t=o(\alpha)$, and $\delta t\to\infty$, we have $$\mathbb{P}(\Delta_r^{(2)}\geq(1-\delta)n_t)\leq\exp\left((1-(\tfrac12+o(1))\delta)t\log\tfrac{M}{\alpha}+O(t)\right).$$
			\item\label{lemma_pd_1} If $\delta\leq1-\Omega(1)$ and $t=o(\alpha)$, we have $$\mathbb{P}(\Delta_r^{(1)}\geq(1-\delta)\ell tk)\leq\exp\left((1-(1+o(1))\delta)t\log\tfrac{M}{\alpha}+o(t\log\log\alpha)\right).$$
		\end{enumerate}
	\end{lemma}
	
	\begin{proof}
		(\ref{lemma_pd_2small}) This follows by Lemma~\ref{lemma_stronger}~(\ref{lemma_stronger_delta}), since $n_t=\tfrac{\ell}{2}t^2+(1-\tfrac{\ell}{2})t$.
		
		(\ref{lemma_pd_2large}) Consider if $\Delta_r^{(2)}\geq(1-\delta)n_t$. Let $\widetilde{u}:=\lceil\tfrac{2}{\ell}(1-\delta)n_t/t\rceil$. By Lemma~\ref{lemma_delta2}, there exists a vertex $s\in S\setminus S_r$ that interferes with some set $\widetilde{U}\subset S\setminus S_r\setminus\{s\}$ of $\widetilde{u}$ other vertices in $S\setminus S_r$. Let $U:=\widetilde{U}\cup\{s\}$ with $|U|=u:=\widetilde{u}+1$, and let $\overline{U}:=S\setminus S_r\setminus U$ with $|\overline{U}|=\overline{u}:=t-u$.
		
		By Lemma~\ref{lemma_delta2}, there are at least $\tfrac1{\delta}\Delta_r^{(2)}\geq\tfrac1{\ell}(1-\delta)n_t$ interfering pairs of vertices in $S\setminus S_r$. So there are at least $$\tfrac1{\ell}(1-\delta)n_t-\tbinom{u}{2}-\tbinom{\overline{u}}{2}$$ pairs of interfering vertices between $U$ and $\overline{U}$. Each vertex in $\overline{U}$ interferes with at most $|U|\leq|S\setminus S_r|=t$ vertices in $U$. For $y:=\exp(\tfrac1{\delta})$, let $\widetilde{z}$ be the number of vertices in $\overline{U}$ that interfere with at least $\tfrac{t}{y}$ vertices in $U$. Then there are at most $\widetilde{z}\cdot t+(\overline{u}-\widetilde{z})\cdot\tfrac{t}{y}$ pairs of interfering vertices between $U$ and $\overline{U}$, so
		\begin{align}
			&\left(1-\tfrac1{y}\right)\widetilde{z}t\geq\tfrac1{\ell}(1-\delta)n_t-\tbinom{u}{2}-\tbinom{\overline{u}}{2}-\tfrac{t}{y}\overline{u},
			\\&\widetilde{z}\geq z:=\left\lceil\tfrac{y}{y-1}\tfrac1{t}\left(\tfrac1{\ell}(1-\delta)n_t-\tbinom{u}{2}-\tbinom{\overline{u}}{2}-\tfrac{t}{y}\overline{u}\right)\right\rceil.
		\end{align}
		
		For a vertex $s\in S$, we define the random variable $X_s$ as follows. If $s\in S_r$, then $X_s=0$. If $s\in S\setminus S_r$, then $X_s$ is the number of ways you can choose a set $\widetilde{U}\subset S\setminus S_r$ of $\widetilde{u}$ other vertices in $S\setminus S_r$ that each interfere with $s$, and a set $Z\subset\overline{U}$ of $z$ yet other vertices in $S\setminus S_r$ that each interfere with at least $\tfrac{t}{y}$ vertices in $U$. Here $U$ and $\overline{U}$ are defined as above. Then we can conclude that $$\mathbb{P}(\Delta_r^{(2)}\geq(1-\delta)n_t)\leq\mathbb{P}\left(\sum_{s\in S}X_s\geq1\right)\leq\mathbb{E}\left(\sum_{s\in S}X_s\right)=\sum_{s\in S}\mathbb{E}(X_s),$$ by Markov's inequality.
		
		We calculate $\mathbb{E}(X_s)$ by summing over all possible sets $\widetilde{U}\subset S\setminus\{s\}$ of $\widetilde{u}$ other vertices in $S$ that each interfere with $s$, and all possible sets $Z\subset S\setminus U$ of $z$ yet other vertices in $S$ that each interfere with at least $\tfrac{t}{y}$ vertices in $U$, the probability that $U\cup Z\subset S\setminus S_r$. Again, $U$ is defined as above. Since $S\setminus S_r$ is a uniformly random subset of $S$ of size $t$, and since $|U\cup Z|=u+z$, the probability in question is $\tbinom{\alpha-k-u-z}{t-u-z}/\tbinom{\alpha-k}{t}$. It remains to calculate how many possibilities there are for the sets $\widetilde{U}$ and $Z$.
		
		Since $s$ interferes with at most $F_{\max}+R_{\max}^{(b)}\leq2M$ other vertices in $S$, there are at most $\tbinom{2M}{\tilde{u}}$ possibilities for the set $\widetilde{U}$. For each such possibility, each vertex in $U$ interferes with at most $2M$ other vertices in $S\setminus U$, so there are at most $2Muy/t$ vertices in $S\setminus U$ that interfere with at least $\tfrac{t}{y}$ vertices in $U$. Without loss of generality, we have $2Muy/t\leq 2My$, since $u>t$ would imply the impossible $|U|>|S\setminus S_r|$ and thus $X_s=0$. We find at most $\tbinom{2My}{z}$ possibilities for the set $Z$. We conclude that $$\mathbb{P}(\Delta_r^{(2)}\geq(1-\delta)n_t)\leq\alpha\tbinom{2M}{\tilde{u}}\tbinom{2My}{z}\tbinom{\alpha-k-u-z}{t-u-z}/\tbinom{\alpha-k}{t}.$$ With the same combinatorial bounds as in the proof of Lemma~\ref{lemma_subset}, and $\alpha-k-u-z\leq\alpha$ and $u=\widetilde{u}+1$, we get
		\begin{align}
			\mathbb{P}(\Delta_r^{(2)}\geq(1-\delta)n_t)&\leq\alpha(\tfrac{2eM}{\tilde{u}})^{\tilde{u}}(\tfrac{2eMy}{z})^z(\tfrac{e\alpha}{t-u-z})^{t-u-z}\tfrac{t!}{(\alpha-k-t)^t}
			\\&\leq\tfrac{t}{e}\sqrt{2\pi t}\exp(\tfrac1{12t})(\tfrac{\alpha}{\alpha-k-t})^t(\tfrac{2Mt}{\alpha\tilde{u}})^{\tilde{u}}(\tfrac{2Mty}{\alpha z})^z(\tfrac{t}{t-u-z})^{t-u-z}.
		\end{align}
		To bound this, we need to approximate $\widetilde{u}$ and $z$.
		
		Since $n_t=\tfrac{\ell}{2}t^2+O(t)$, we have $\widetilde{u}=(1-\delta)t+O(1)$ and thus $u=(1-\delta)t+O(1)$ and $\overline{u}=\delta t+O(1)$. We get $\tbinom{u}{2}=\tfrac12(1-\delta)^2t^2+O(t)$ and $\tbinom{\overline{u}}{2}=\tfrac12\delta^2t^2+O(\delta t)$, so $$\tfrac1{\ell}(1-\delta)n_t-\tbinom{u}{2}-\tbinom{\overline{u}}{2}=\tfrac12\delta(1-2\delta)t^2+O(t)=(\tfrac12+o(1))\delta t^2.$$ Since $\overline{u}=O(\delta t)$ and $y\to\infty$, we also have $\tfrac{t}{y}\overline{u}=o(\delta t^2)$. Since, $\tfrac{y}{y-1}=1+\tfrac1{y-1}=1+o(1)$, we get $z=(\tfrac12+o(1))\delta t$.
		
		To finish the bound on $\mathbb{P}(\Delta_r^{(2)}\geq(1-\delta)n_t)$, first note that $\log t\leq t$ gives $$\tfrac{t}{e}\sqrt{2\pi t}\exp(\tfrac1{12 t})\leq\exp(O(t)).$$ Second, applying $1+x\leq e^x$ to $\tfrac{\alpha}{\alpha-k-t}=1+\tfrac{k+t}{\alpha-k-t}=1+o(1)$ and $\tfrac{t}{t-u-z}=1+\tfrac{u+z}{t-u-z}$ gives $$(\tfrac{\alpha}{\alpha-k-t})^t(\tfrac{t}{t-u-z})^{t-u-z}\leq\exp(o(t)+O(u+z))\leq\exp(O(t)).$$ Third, since $x^{1/x}\leq e^{1/e}$ for all $x>0$, we have $(t/\widetilde{u})^{\tilde{u}/t}=O(1)$ and $(t/z)^{z/t}=O(1)$, so $$(\tfrac{t}{\tilde{u}})^{\tilde{u}}(\tfrac{t}{z})^z\leq\exp(O(t)).$$ We are thus left with $$\mathbb{P}(\Delta_r^{(2)}\geq(1-\delta)n_t)\leq\exp\left((\widetilde{u}+z)\log\tfrac{2M}{\alpha}+z\log y+O(t)\right).$$ The result follows, since $\widetilde{u}+z=(1-(\tfrac12+o(1))\delta)t=O(t)$ and $z\log y=O(t)$.
		
		(\ref{lemma_pd_1}) We use Lemma~\ref{lemma_fs}. For $y:=\log\log\alpha$, consider if less than $(1-\tfrac{y}{y-1}\delta)t$ vertices $s\in S\setminus S_r$ have $f(s)\geq\tfrac1y\ell k$. Then Lemma~\ref{lemma_fs}~(\ref{lemma_fs_one}) gives $$\Delta_r^{(1)}\leq\sum_{s\in S\setminus S_r}f(s)<(1-\tfrac{y}{y-1}\delta)t\cdot\ell k+\tfrac{y}{y-1}\delta t\cdot\tfrac1y\ell k=(1-\delta)\ell tk.$$ By contraposition, $\Delta_r^{(1)}\geq(1-\delta)\ell tk$ implies that at least $(1-\tfrac{y}{y-1}\delta)t$ vertices $s\in S\setminus S_r$ have $f(s)\geq\tfrac1y\ell k$.
		
		By Lemma~\ref{lemma_fs}~(\ref{lemma_fs_sum}), at most $\tfrac{2}{\ell}My$ of the $\alpha-k$ vertices $s\in S$ have $f(s)\geq\tfrac1y\ell k$. Since $S\setminus S_r$ is a uniformly random subset of $S$ of size $t$, by Lemma~\ref{lemma_subset}, there are exactly $i$ vertices $s\in S\setminus S_r$ with $f(s)\geq\tfrac1y\ell k$ with probability at most $$\exp\left(\tfrac12\log(2\pi t)+\tfrac1{12t}+\tfrac{t^2}{\alpha-k-t}\right)\left(\tfrac{\frac{2}{\ell}etMy}{i(\alpha-k)}\right)^i.$$ For $i\geq(1-\tfrac{y}{y-1}\delta)t$, since $t=o(\alpha)$, this is bounded by $$\exp\left(i\log\tfrac{My}{\alpha}+O(i)\right),$$ in particular, because $\tfrac{y}{y-1}=1+\tfrac1{y-1}=1+o(1)$ and $0<\delta\leq1-\Omega(1)$ give $i=\Omega(t)$. Since $My/\alpha=o(1)$, summing over all $i\geq(1-\tfrac{y}{y-1}\delta)t$ yields $$\mathbb{P}(\Delta_r^{(1)}\geq(1-\delta)\ell tk)\leq\exp\left(\left(1-\tfrac{y}{y-1}\delta\right)t\log\tfrac{My}{\alpha}+O(t)\right).$$ The result follows, since $t\log y=o(t\log\log\alpha)$.
	\end{proof}
	
	For $0<\delta<1$, we split up the event $\mathcal{P}\setminus\mathcal{P}^{\delta}$. For $t=1,2,\ldots$, let $$\mathcal{P}_t^{\delta}:=\{n_{\alpha-r}-\Psi_r\geq\delta(n_{\alpha-r}-K),\ \forall r=0,\ldots,\alpha-k-t\}.$$ Note that $\mathcal{P}_t^{\delta}=\mathcal{P}$ for $t>\alpha-k$, so let $T^{\delta}$ be the smallest value of $t$ for which $\mathcal{P}_t^{\delta}=\mathcal{P}$. Also note that $\mathcal{P}_1^{\delta}\subset\mathcal{P}^{\delta}$ by Lemma~\ref{lemma_bounds1}~(\ref{lemma_bounds1_v0}). Since $\mathcal{P}_1^{\delta}\subset\mathcal{P}_2^{\delta}\subset\ldots\subset\mathcal{P}_{T^{\delta}}^{\delta}=\mathcal{P}$, we can split up $$\mathcal{P}\setminus\mathcal{P}^{\delta}\subset(\mathcal{P}_2^{\delta}\setminus\mathcal{P}_1^{\delta})\cup(\mathcal{P}_3^{\delta}\setminus\mathcal{P}_2^{\delta})\cup\ldots\cup(\mathcal{P}_{T^{\delta}}^{\delta}\setminus\mathcal{P}_{T^{\delta}-1}^{\delta}).$$ Then $\mathcal{P}_{t+1}^{\delta}\setminus\mathcal{P}_t^{\delta}$ is the event that $r=\alpha-k-t$ is the smallest value of $r$ for which $n_{\alpha-r}-\Psi_r<\delta(n_{\alpha-r}-K)$. For $t\leq\alpha$, we furthermore define the event $$\mathcal{Q}_t^{\delta}:=\{n_{\alpha-r}-\Psi_r<\delta(n_{\alpha-r}-K)\ \text{for}\ r:=\alpha-k-t\},$$ such that $\mathcal{P}_{t+1}^{\delta}\setminus\mathcal{P}_t^{\delta}=\mathcal{P}_{t+1}^{\delta}\cap\mathcal{Q}_t^{\delta}$.
	
	\begin{lemma}\label{lemma_t}
		For $0<\delta<1-\Omega(1)$ with $M^2/\alpha=o(\delta)$, we have $$T^{\delta}\leq2M/(1-(1+o(1))\delta)-k=O(M).$$
	\end{lemma}
	
	\begin{proof}
		For $r=0,\ldots,\alpha-k-1$, consider if $n_{\alpha-r}-\Psi_r<\delta(n_{\alpha-r}-K)$. For $t:=\alpha-k-r$, by Lemma~\ref{lemma_boundpt}, we have $$\Delta_r^{(2)}+\Delta_r^{(1)}\geq(1-(1+o(1))\delta)(n_t+\ell tk).$$ Note that Lemma~\ref{lemma_nr}~(\ref{lemma_nr_geq}) gives $n_t+\ell tk\geq t^2+tk=t(\alpha-r)$. Since $$\Delta_r^{(2)}+\Delta_r^{(1)}\leq|N(S\setminus S_r)\cap(F\cup R^*(S))|\leq t(F_{\max}+R_{\max}^{(b)})\leq 2tM,$$ we get $\alpha-r\leq 2M/(1-(1+o(1))\delta)$. The result follows.
	\end{proof}
	
	We can now prove Lemma~\ref{lemma_remaining_p}.
	
	\begin{proof}[Proof of Lemma~\ref{lemma_remaining_p}]
		Since $\mathbb{P}(\mathcal{P}\setminus\mathcal{P}^{\delta})$ is monotonous with respect to $\delta$, we may assume without loss of generality that $M^2/\alpha=o(\delta)$. Consider some $t=1,\ldots,T^{\delta}-1$ with $r:=\alpha-t$. By Lemma~\ref{lemma_boundpt}, in $\mathcal{Q}_t^{\delta}$, we have $\Delta_r^{(2)}=\Delta_r\geq(1-(1+o(1))\delta)n_t$, because $k=0$.
		
		By Lemma~\ref{lemma_pd0}~(\ref{lemma_pd0_d2}), we get $\mathbb{P}(\mathcal{Q}_t^{\delta})=O(t^2M/\alpha)$. Since $t\leq T^{\delta}-1=O(M)$ by Lemma~\ref{lemma_t}, by Lemma~\ref{lemma_pd}~(\ref{lemma_pd_2small}), we also get $\mathbb{P}(\mathcal{Q}_t^{\delta})\leq e^{O(\log\alpha)-\Omega(t\log\alpha)}$. For $\alpha$ large enough, we thus have $\mathbb{P}(\mathcal{Q}_t^{\delta})\leq\alpha^{C-\varepsilon t}$ for some fixed constants $C,\varepsilon>0$.
		
		Summing over $t=1,\ldots,t^*:=\lfloor\tfrac{C+1}{\varepsilon}\rfloor$, we get $$\mathbb{P}(\mathcal{Q}_1^{\delta}\cup\ldots\cup\mathcal{Q}_{t^*}^{\delta})=O((t^*)^3M/\alpha)=O(M/\alpha).$$ Summing over $t=t^*+1,\ldots,T^{\delta}-1$, we get $$\mathbb{P}(\mathcal{Q}_{t^*+1}^{\delta}\cup\ldots\cup\mathcal{Q}_{T^{\delta}-1}^{\delta})\leq(1+o(1))\alpha^{-1}=O(\alpha^{-1}),$$ since $\sum_{i\geq n}x^i=(1+o(1))x^n$ for $x=o(1)$. The result follows.
	\end{proof}
	
	\subsection{Special term}\label{subsection_proof_further}
	
	We prove the following lemma, which is crucial for the proof of Lemma~\ref{lemma_remaining_e}.
	
	\begin{lemma}\label{lemma_depth}
		Let $0<\delta<1$ with $\delta=o(1)$ and $\tfrac1{\delta}=o(\log\alpha)$. Let $0\leq\widetilde{k}\leq k-1$ and $\widetilde{K}\leq n_{\tilde{k}}$ with $K\leq n_k-\delta(n_k-\widetilde{K})$. Consider the event $$\widetilde{\mathcal{P}}^{\delta}:=\{n_{\alpha-r}-\Psi_r\geq\delta(n_{\alpha-r}-\widetilde{K}),\ \forall r=0,\ldots,\alpha-k-1\}\subset\mathcal{P}.$$ Then for $C=O(1)$, we have
		\begin{align}
			&\mathbb{E}\left(1_{\widetilde{\mathcal{P}}^{\delta}}\exp\left(C\sum_{r=0}^{\alpha-k-1}\frac{(K-\widetilde{K})^+}{n_{\alpha-r}-\Psi_r}\right)\right)
			\\&\leq\exp\left((\tfrac{C}{2}+o(1))k\log\log\alpha+O(\tfrac{1}{\delta}\widetilde{k})\right).
		\end{align}
	\end{lemma}
	
	To prove this, for some $0<\delta'<1$, we split up $$\mathcal{P}=\mathcal{P}_1^{\delta'}\cup(\mathcal{P}_2^{\delta'}\cap\mathcal{Q}_1^{\delta'})\cup(\mathcal{P}_3^{\delta'}\cap\mathcal{Q}_2^{\delta'})\cup\ldots\cup(\mathcal{P}_{T^{\delta'}}^{\delta'}\cap\mathcal{Q}_{T^{\delta'}-1}^{\delta'}).$$ In $\mathcal{P}_t^{\delta'}$, if $K\leq n_k-\delta(n_k-\widetilde{K})$, then $n_{\alpha-r}-\Psi_r\geq\delta'(n_{\alpha-r}-n_k+\delta(n_k-\widetilde{K}))$.
	
	\begin{lemma}\label{lemma_special}
		Let $0<\delta<1$ with $\delta=o(1)$ and $\tfrac1{\delta}=O(\log\alpha)$. Let $0\leq\widetilde{k}\leq k-1$ and $\widetilde{K}\leq n_{\tilde{k}}$. Then $$\sum_{r=0}^{\alpha-k-1}\frac{(K-\widetilde{K})^+}{n_{\alpha-r}-n_k+\delta(n_k-\widetilde{K})}\leq (\tfrac12+o(1))k\log\log\alpha+\tfrac1{\delta}\widetilde{k}.$$
	\end{lemma}
	
	\begin{proof}
		By Lemma~\ref{lemma_bounds1}~(\ref{lemma_bounds1_v0}), substituting $i=\alpha-k-r$ gives $$\widetilde{\Sigma}^{(1)}:=\sum_{r=0}^{\alpha-k-1}\frac{(K-\widetilde{K})^+}{n_{\alpha-r}-n_k+\delta(n_k-\widetilde{K})}\leq\sum_{i=1}^{\alpha}\frac{(n_k-\widetilde{K})^+}{n_{k+i}-n_k+\delta(n_k-\widetilde{K})}.$$ Note that $\widetilde{K}\leq n_{\tilde{k}}<n_k$, so each term is at most $\tfrac1{\delta}$, which gives
		\begin{align}
			\widetilde{\Sigma}^{(1)}
			&\leq\tfrac1{\delta}\widetilde{k}+\sum_{i=\tilde{k}+1}^{\alpha}\frac{n_k}{n_{k+i}-n_k+\delta(n_k-\widetilde{K})}
			\\&\leq\tfrac1{\delta}\widetilde{k}+n_k\sum_{i=\tilde{k}+1}^{\alpha}\frac{1}{n_{k+i}-n_k+\delta(n_k-n_{\tilde{k}})}.
		\end{align}
		By Lemma~\ref{lemma_nr}~(\ref{lemma_nr_sum}), we get
		\begin{align}
			\widetilde{\Sigma}^{(2)}
			&:=\sum_{i=\tilde{k}+1}^{\alpha}\frac{1}{n_{k+i}-n_k+\delta(n_k-n_{\tilde{k}})}
			\\&=\sum_{i=\tilde{k}+1}^{\alpha}\frac{1}{n_i+\ell ki+\delta(n_k-n_{\tilde{k}})}
			\\&=\sum_{i=\tilde{k}+1}^{\alpha}\frac{1}{ai^2+bi+c}
		\end{align}
		for $a:=\tfrac{\ell}{2}$, $b:=\ell k+1-\tfrac{\ell}{2}$, and $c:=\delta(n_k-n_{\tilde{k}})$. Note that $a,b,c>0$, since $k\geq\widetilde{k}+1\geq1$. By Lemma~\ref{lemma_nr}~(\ref{lemma_nr_leq}), we find that $c\leq\delta k^2=o(b^2)$, so $D:=b^2-4ac=(1+o(1))b^2$ is also positive for $\alpha$ large enough.
		
		For $\xi^{\pm}:=\tfrac1{2a}(b\pm\sqrt{D})$, we have $ai^2+bi+c=a(i+\xi^-)(i+\xi^+)$, which gives $$\widetilde{\Sigma}^{(2)}=\sum_{i=\tilde{k}+1}^{\alpha}\frac1{\sqrt{D}}\left(\frac1{i+\xi^-}-\frac1{i+\xi^+}\right).$$ We use the bounds
		\begin{align}
			\sum_{i=\tilde{k}+1}^{\alpha}\frac1{i+\xi^-}
			&\leq\int_{\tilde{k}+\xi^-}^{\alpha+\xi^-}\frac1x\,\mathrm{d}x=\log\left(\frac{\alpha+\xi^-}{\tilde{k}+\xi^-}\right),
			\\\sum_{i=\tilde{k}+1}^{\alpha}\frac1{i+\xi^+}
			&\geq\int_{\tilde{k}+\xi^++1}^{\alpha+\xi^++1}\frac1x\,\mathrm{d}x=\log\left(\frac{\alpha+\xi^++1}{\tilde{k}+\xi^++1}\right),
			\\\widetilde{\Sigma}^{(2)}&\leq\frac1{\sqrt{D}}\log\left(\frac{\tilde{k}+\xi^++1}{\tilde{k}+\xi^-}\right),
		\end{align}
		where the last inequality uses $\alpha+\xi^-\leq\alpha+\xi^++1$.
		
		Note that $D=b^2(1-\tfrac{4ac}{b^2})\leq b^2(1-\tfrac{2ac}{b^2})^2\leq b^2$, so $\xi^+\leq\tfrac{b}{a}$ and $\xi^-\geq\tfrac{c}{b}$. If $\widetilde{k}\leq k/\log\alpha$, such that $c=\Omega(\delta k^2)$ and $b(b+a)=O(k^2)$, we thus get $$\tfrac{\xi^++1}{\xi^-}\leq\tfrac{b(b+a)}{ac}=O(\tfrac{1}{\delta})=O(\log\alpha).$$ Otherwise, if $\widetilde{k}\geq k/\log\alpha$, we get $$\tfrac{\xi^++1}{\tilde{k}}\leq\tfrac{b+a}{a\tilde{k}}=O(\tfrac{k}{\tilde{k}})=O(\log\alpha).$$ Either way, since $\sqrt{D}=(1+o(1))b$, we find that
		\begin{align}
			\frac{\tilde{k}+\xi^++1}{\tilde{k}+\xi^-}
			&\leq1+\frac{\xi^++1}{\max\{\tilde{k},\xi^-\}}=O(\log\alpha),
			\\\widetilde{\Sigma}^{(2)}
			&\leq\tfrac1{\sqrt{D}}\log(O(\log\alpha))\leq(1+o(1))\tfrac1{b}\log\log\alpha,
			\\\widetilde{\Sigma}^{(1)}
			&\leq\tfrac1{\delta}\widetilde{k}+(1+o(1))\tfrac{n_k}{b}\log\log\alpha.
		\end{align}
		Since $n_k=(\tfrac{\ell}{2}k+1-\tfrac{\ell}2)k\leq\tfrac12bk$, the result follows.
	\end{proof}
	
	We can now prove Lemma~\ref{lemma_depth}.
	
	\begin{proof}[Proof of Lemma~\ref{lemma_depth}]
		Fix some $0<\tau<1$, and let $\delta':=\tfrac1{1+\tau}$. We use
		\begin{align}
			\mathbb{E}(1_{\widetilde{\mathcal{P}}^{\delta}}e^{C\widetilde{\Sigma}})
			&\leq\mathbb{E}(1_{\widetilde{\mathcal{P}}^{\delta}\cap\mathcal{P}_1^{\delta'}}e^{C\widetilde{\Sigma}})+\sum_{t=1}^{T^{\delta'}-1}\mathbb{E}(1_{\widetilde{\mathcal{P}}^{\delta}\cap\mathcal{P}_{t+1}^{\delta'}\cap\mathcal{Q}_t^{\delta'}}e^{C\widetilde{\Sigma}}),
			\\\widetilde{\Sigma}
			&:=\sum_{r=0}^{\alpha-k-1}\frac{(K-\widetilde{K})^+}{n_{\alpha-r}-\Psi_r}.
		\end{align}
		
		For $t=1,\ldots,T^{\delta'}$, in $\widetilde{\mathcal{P}}^{\delta}\cap\mathcal{P}_t^{\delta'}$, we have $$\widetilde{\Sigma}\leq\tfrac1{\delta'}\sum_{r=0}^{\alpha-k-t}\frac{(K-\widetilde{K})^+}{n_{\alpha-r}-n_k+\delta(n_k-\widetilde{K})}+\tfrac1{\delta}\sum_{r=\alpha-k-t+1}^{\alpha-k-1}\frac{(K-\widetilde{K})^+}{n_{\alpha-r}-\widetilde{K}}.$$ Since $K\leq n_k\leq n_{\alpha-r}$ by Lemma~\ref{lemma_bounds1}~(\ref{lemma_bounds1_v0}), by Lemma~\ref{lemma_special}, we get
		\begin{align}
			\widetilde{\Sigma}
			&\leq\tfrac1{\delta'}\left((\tfrac12+o(1))k\log\log\alpha+\tfrac1{\delta}\widetilde{k}\right)+\tfrac1{\delta}(t-1)
			\\&\leq(1+\tau)(\tfrac12+o(1))k\log\log\alpha+O(\tfrac1{\delta}\widetilde{k})+o((t-1)\log\alpha).
		\end{align}
		In particular, this gives $$\mathbb{E}(1_{\widetilde{\mathcal{P}}^{\delta}\cap\mathcal{P}_1^{\delta'}}e^{C\tilde{\Sigma}})\leq\exp\left((1+\tau)(\tfrac{C}{2}+o(1))k\log\log\alpha+O(\tfrac1{\delta}\widetilde{k})\right).$$
		
		For $r=0,\ldots,\alpha-k-1$ with $t:=\alpha-k-r\leq T^{\delta'}-1$, in $\mathcal{Q}_t^{\delta'}$, we have $\Delta_r^{(2)}+\Delta_r^{(1)}\geq(1-(1+o(1))\delta')(n_t+\ell tk)$ by Lemma~\ref{lemma_boundpt}. It follows that either $\Delta_r^{(2)}\geq(1-(1+o(1))\delta')n_t$ or $\Delta_r^{(1)}\geq(1-(1+o(1))\delta')\ell tk$. By Lemma~\ref{lemma_pd}~(\ref{lemma_pd_2small},\ref{lemma_pd_1}), since $\delta'\leq1-\Omega(1)$, we get $$\mathbb{P}(\mathcal{Q}_t^{\delta'})\leq\exp(O(\log\alpha)-\Omega(t\log\alpha)).$$ It follows that
		\begin{align}
			&\mathbb{E}(1_{\widetilde{\mathcal{P}}^{\delta}\cap\mathcal{P}_{t+1}^{\delta'}\cap\mathcal{Q}_t^{\delta'}}e^{C\tilde{\Sigma}})
			\\&\leq\exp\left((1+\tau)(\tfrac{C}{2}+o(1))k\log\log\alpha+O(\tfrac1{\delta}\widetilde{k}+\log\alpha)-\Omega(t\log\alpha)\right).
		\end{align}
		For $\alpha$ large enough, we thus have
		\begin{align}
			&\mathbb{E}(1_{\widetilde{\mathcal{P}}^{\delta}\cap\mathcal{P}_{t+1}^{\delta'}\cap\mathcal{Q}_t^{\delta'}}e^{C\tilde{\Sigma}})
			\\&\leq\alpha^{C-\varepsilon t}\exp\left((1+\tau)(\tfrac{C}{2}+o(1))k\log\log\alpha+O(\tfrac1{\delta}\widetilde{k})\right)
		\end{align}
		for some fixed constants $C,\varepsilon>0$. For $t^*:=\lfloor\tfrac{C}{\varepsilon}\rfloor$, we get
		\begin{align}
			&\sum_{t=t^*+1}^{T^{\delta'}-1}\mathbb{E}(1_{\widetilde{\mathcal{P}}^{\delta}\cap\mathcal{P}_{t+1}^{\delta'}\cap\mathcal{Q}_t^{\delta'}}e^{C\tilde{\Sigma}})
			\\&\leq(1+o(1))\exp\left((1+\tau)(\tfrac{C}{2}+o(1))k\log\log\alpha+O(\tfrac1{\delta}\widetilde{k})\right),
		\end{align}
		since $\sum_{i\geq n}x^i=(1+o(1))x^n$ for $x=o(1)$.
		
		Finally, for $t\leq t^*=O(1)$, by Lemma~\ref{lemma_pd0}, we get $$\mathbb{P}(\mathcal{Q}_t^{\delta'})=O((k+1)M/\alpha)=o(1).$$ It follows that
		\begin{align}
			&\mathbb{E}(1_{\widetilde{\mathcal{P}}^{\delta}\cap\mathcal{P}_{t+1}^{\delta'}\cap\mathcal{Q}_t^{\delta'}}e^{C\tilde{\Sigma}})
			\\&\leq\exp\left((1+\tau)(\tfrac{C}{2}+o(1))k\log\log\alpha+O(\tfrac1{\delta}\widetilde{k})\right),
			\\&\sum_{t=1}^{t^*}\mathbb{E}(1_{\widetilde{\mathcal{P}}^{\delta}\cap\mathcal{P}_{t+1}^{\delta'}\cap\mathcal{Q}_t^{\delta'}}e^{C\tilde{\Sigma}})
			\\&\leq\exp\left((1+\tau)(\tfrac{C}{2}+o(1))k\log\log\alpha+O(\tfrac1{\delta}\widetilde{k})\right).
		\end{align}
		
		Collecting results, we have $$\mathbb{E}(1_{\widetilde{\mathcal{P}}^{\delta}}e^{C\tilde{\Sigma}})\leq\exp\left((1+\tau)(\tfrac{C}{2}+o(1))k\log\log\alpha+O(\tfrac1{\delta}\widetilde{k})\right).$$ Since this holds for any $\tau>0$, the result follows.
	\end{proof}
	
	\subsection{Extremal expectation}\label{subsection_proof_rpe}
	
	We prove Lemma~\ref{lemma_remaining_e}. For a given $r=0,\ldots,\alpha-k-1$ with $t:=\alpha-k-r<T^{\delta}$, we thus need to bound $\mathbb{E}(\Pi1_{\mathcal{P}_{t+1}^{\delta}\setminus\mathcal{P}_t^{\delta}})$. We consider $$\Pi_t:=\prod_{r'=0}^{\alpha-k-t-1}\frac{n_{\alpha-r'}-\psi_{p_{r'}}}{n_{\alpha-r'}-\Psi_{r'}}.$$ The main idea is to condition on the random variable $S':=S_r$. Then $\mathbb{E}(\Pi_t|S')$ is related to $\mathbb{P}(\boldsymbol{S}_{\alpha-k-t}=S')$, similar to how $\mathbb{E}(\Pi)$ is related to $\mathbb{P}(\boldsymbol{S}_{\alpha-k}=S)$.
	
	\begin{lemma}\label{lemma_recycle}
		Let $0<\delta<1$ with $\delta=O(\tfrac{\log\log\alpha}{\log\alpha})$, and let $\mu_1,\mu_2>0$ with $\mu_2=\Omega(1)$, $\mu_1+\mu_2=1$, and
		\begin{equation}\label{equation_recycle}
			\delta\geq\tfrac1{\mu_1}\left(8\left(1-\tfrac{\log M}{\log\alpha}\right)+2+\Omega(1)\right)\tfrac{\log\log\alpha}{\log\alpha}.
		\end{equation}
		Let $r=0,\ldots,\alpha-k-1$ with $t:=\alpha-k-r=O(M)$. Then $$\mathbb{E}(\Pi_t1_{\mathcal{P}_t^{\delta}}|S_r)\leq\exp\left((\tfrac{1}{2\mu_2}+o(1))t\log\log\alpha+O(\tfrac1{\delta}k)\right).$$
	\end{lemma}
	
	\begin{proof}
		Consider the FIS $S':=S_r$ of size $\alpha-k'$ with $k':=k+t=O(M)$, and let $K':=|(V\setminus\overline{N}(S'))\cap(F\cup R^*(S'))|=\Delta_r$. Since this fits the premise of Lemma~\ref{lemma_generalization}, we can follow the framework from Section~\ref{subsection_proof_framework} to reach the same conclusions.
		
		Let $\mathcal{P}'$ be the set of all possible permutations $\pi'=(s_0',\ldots,s_{\alpha-k'-1}')$ of the elements of $S'$, and consider if $\pi'$ is a uniformly random permutation from $\mathcal{P}'$. For $r'=0,\ldots,\alpha-k'-1$, we get the random variable $S_{r'}':=S_{r'}'(\pi'):=\{s_0',\ldots,s_{r'-1}'\}$. We can define $\widetilde{V}_{r'}'$, $V_{r'}'$, and $\Psi_{r'}'$ again, as done in Section~\ref{subsection_proof_framework}.
		
		Next, for $p\in[0,1]$, let $S_p'$ contain each element of $S'$ independently with probability $p$. Again, similarly define $\widetilde{V}_p'$, $V_p'$, and $\Psi_p'$, and define $\psi_p':=\mathbb{E}(\Psi_p')$. Furthermore, for $r'=0,\ldots,\alpha-k'-1$, define $p_{r'}':=r'/(\alpha-k')$, and write $q:=1-p$ and $q_{r'}':=1-p_{r'}'$.
		
		We can apply any Lemma~of which the proof is already finished to this new context. Firstly, Lemma~\ref{lemma_expectation} gives $$\psi_p'=K'+q^2(|F|+p|R^*(S')|+\b(V))+O(qM^2+(n_{k'}-K')M/\alpha).$$ Secondly, using \eqref{equation_recycle}, Lemma~\ref{lemma_sigma}~(\ref{lemma_sigma_large}) gives $$\mathbb{E}\left(\exp\left(\tfrac{1}{\mu_1\delta}\sum_{r'=0}^{\alpha-k'-1}\frac{(\Psi_{r'}'-\psi_{p_{r'}'}')^+}{n_{\alpha-r'}-n_{k'}}\right)\right)\leq e^{o(k')}.$$ Thirdly, and finally, if $K'\leq n_{k'}-\delta(n_{k'}-K)$, then for $$\widetilde{\mathcal{P}}'^{\delta}:=\{n_{\alpha-r'}-\Psi_{r'}'\geq\delta(n_{\alpha-r'}-K),\ \forall r'=0,\ldots,\alpha-k'-1\},$$ Lemma~\ref{lemma_depth} gives
		\begin{align}
			&\mathbb{E}\left(1_{\widetilde{\mathcal{P}}'^{\delta}}\exp\left(\tfrac1{\mu_2}\sum_{r'=0}^{\alpha-k'-1}\frac{(K'-K)^+}{n_{\alpha-r'}-\Psi_{r'}'}\right)\right)
			\\&\leq\exp\left((\tfrac{1}{2\mu_2}+o(1))k'\log\log\alpha+O(\tfrac1{\delta}k)\right).
		\end{align}
		
		The crucial observation is that the permutation $\overline{\pi}:=(s_0,\ldots,s_{\alpha-k'-1})$ is already a uniformly random permutation from $\mathcal{P}'$, if we condition on $S_r$. We may thus assume that $\pi'=\overline{\pi}$. For all $r'=0,\ldots,\alpha-k'-1$, this gives $S_{r'}'=S_{r'}$, and thus $\Psi_{r'}'=\Psi_{r'}$. We even get equality of the events $\mathcal{P}_{t+1}^{\delta}=\widetilde{\mathcal{P}}'^{\delta}$.
		
		Reinterpreting the results, we find that Lemma~\ref{lemma_sigma}~(\ref{lemma_sigma_large}) now gives $$\mathbb{E}\Bigg(\exp\Bigg(\tfrac{1}{\mu_1\delta}\sum_{r'=0}^{\alpha-k'-1}\frac{(\Psi_{r'}-\psi_{p_{r'}'}')^+}{n_{\alpha-r'}-n_{k'}}\Bigg)\Bigg|S_r\Bigg)\leq e^{o(k+t)}.$$ Furthermore, in $\mathcal{P}_t^{\delta}$, for $r'=0,\ldots,\alpha-k'-1$ we have $$n_{\alpha-r'}-\Psi_{r'}\geq\delta(n_{\alpha-r'}-K)\geq\delta(n_{\alpha-r'}-n_k)\geq\delta(n_{\alpha-r'}-n_{k'})$$ by Lemma~\ref{lemma_bounds1}~(\ref{lemma_bounds1_v0}), so it follows that $$\mathbb{E}\Bigg(1_{\mathcal{P}_t^{\delta}}\exp\Bigg(\tfrac{1}{\mu_1}\sum_{r'=0}^{\alpha-k'-1}\frac{(\Psi_{r'}-\psi_{p_{r'}'}')^+}{n_{\alpha-r'}-\Psi_{r'}}\Bigg)\Bigg|S_r\Bigg)\leq e^{o(k+t)}.$$ Next, in $\mathcal{P}_t^{\delta}$, note that we have $n_{k'}-\Psi_r\geq\delta(n_{k'}-K)$, and thus $K'=\Delta_r\leq\Psi_r\leq n_{k'}-\delta(n_{k'}-K)$. Since $\mathcal{P}_t^{\delta}\subset\mathcal{P}_{t+1}^{\delta}=\widetilde{\mathcal{P}}'^{\delta}$, Lemma~\ref{lemma_depth} thus now gives
		\begin{align}
			&\mathbb{E}\Bigg(1_{\mathcal{P}_t^{\delta}}\exp\Bigg(\tfrac1{\mu_2}\sum_{r'=0}^{\alpha-k'-1}\frac{(K'-K)^+}{n_{\alpha-r'}-\Psi_{r'}}\Bigg)\Bigg|S_r\Bigg)
			\\&\leq\exp\left((\tfrac{1}{2\mu_2}+o(1))(k+t)\log\log\alpha+O(\tfrac1{\delta}k)\right)
			\\&\leq\exp\left((\tfrac{1}{2\mu_2}+o(1))t\log\log\alpha+O(\tfrac1{\delta}k)\right),
		\end{align}
		since $\log\log\alpha=O(\tfrac{1}{\delta})$.
		
		In order to bound $\mathbb{E}(\Pi_t1_{\mathcal{P}_t^{\delta}}|S_r)$ from this, we use
		\begin{align}
			\log\Pi_t
			&\leq\sum_{r'=0}^{\alpha-k'-1}(\widetilde{Q}_{r'})^+=\sum_{r'=0}^{\alpha-k'-1}\frac{(\Psi_{r'}-\psi_{p_{r'}})^+}{n_{\alpha-r'}-\Psi_{r'}}
			\\&\leq\sum_{r'=0}^{\alpha-k'-1}\frac{(\Psi_{r'}-\psi_{p_{r'}'}')^++(\psi_{p_{r'}'}'-\psi_{p_{r'}}+K-K')^++(K'-K)^+}{n_{\alpha-r'}-\Psi_{r'}}.
		\end{align}
		For $r'=0,\ldots,\alpha-k'-1$, we use Lemma~\ref{lemma_expectation} to bound
		\begin{align}
			&(\psi_{p_{r'}'}'-\psi_{p_{r'}}+K-K')^+
			\\&\leq|q_{r'}'^2-q_{r'}^2|(|F|+\b(V))+|q_{r'}'^2p_{r'}'|R^*(S_r)|-q_{r'}^2p_{r'}|R^*(S)||
			\\&+O((q_{r'}'+q_{r'})M^2+(n_k-K+n_{k'}-K')M/\alpha).
		\end{align}
		Note that $k\leq k'$ gives $n_k-K+n_{k'}-K'\leq2n_{k'}\leq\ell k'^2$ by Lemma~\ref{lemma_nr}~(\ref{lemma_nr_leq}). We furthermore get $p_{r'}\leq p_{r'}'$, so $q_{r'}'\leq q_{r'}$, and thus $q_{r'}'+q_{r'}\leq2q_{r'}$. Next, we have
		\begin{align}
			&|q_{r'}'^2p_{r'}'|R^*(S_r)|-q_{r'}^2p_{r'}|R^*(S)||
			\\&\leq p_{r'}'|q_{r'}'^2-q_{r'}^2||R^*(S)|+q_{r'}'^2p_{r'}'|R^*(S\setminus S_r)|+q_{r'}^2|p_{r'}'-p_{r'}||R^*(S)|
			\\&\leq|q_{r'}'^2-q_{r'}^2||R^*(S)|+q_{r'}|R^*(S\setminus S_r)|+q_{r'}|p_{r'}'-p_{r'}||R^*(S)|,
		\end{align}
		so we need to bound
		\begin{align}
			q_{r'}^2-q_{r'}'^2
			&=\tfrac{(\alpha-k-r')^2}{(\alpha-k)^2}-\tfrac{(\alpha-k'-r')^2}{(\alpha-k')^2}
			\\&=\tfrac{(\alpha-k-r')^2(\alpha-k')^2-(\alpha-k'-r')^2(\alpha-k)^2}{(\alpha-k)^2(\alpha-k')^2}
			\\&=\tfrac{2r't(\alpha-k)(\alpha-k'-r')+r'^2t^2}{(\alpha-k)^2(\alpha-k')^2}=O(q_{r'}'M/\alpha+t^2/\alpha^2),
			\\p_{r'}'-p_{r'}
			&=\tfrac{r't}{(\alpha-k')(\alpha-k)}=O(M/\alpha).
		\end{align}
		We further use Lemma~\ref{lemma_terms} to conclude $$(\psi_{p_{r'}'}'-\psi_{p_{r'}}+K-K')^+=O(q_{r'}M^2+k'^2M/\alpha).$$ In $\mathcal{P}_t^{\delta}$, by Lemma~\ref{lemma_sums}~(\ref{lemma_sums_3},\ref{lemma_sums_4}), we thus get
		\begin{align}
			&\sum_{r'=0}^{\alpha-k'-1}\frac{(\psi_{p_{r'}'}'-\psi_{p_{r'}}+K-K')^+}{n_{\alpha-r'}-\Psi_{r'}}
			\\&=O\left(\tfrac{1}{\delta}\sum_{r'=0}^{\alpha-k-1}\frac{q_{r'}M^2}{n_{\alpha-r'}-n_k}+\tfrac{1}{\delta}\sum_{r'=0}^{\alpha-k'-1}\frac{k'^2M/\alpha}{n_{\alpha-r'}-n_{k'}}\right)
			\\&=O(\tfrac{1}{\delta}M^2\log\alpha/\alpha)=o(1).
		\end{align}
		
		By convexity of the exponential function, we get
		\begin{align}
			&\mathbb{E}(\Pi_t1_{\mathcal{P}_t^{\delta}}|S_r)
			\\&\leq\mathbb{E}\Bigg(1_{\mathcal{P}_t^{\delta}}\exp\Bigg(\sum_{r'=0}^{\alpha-k'-1}\frac{(\Psi_{r'}-\psi_{p_{r'}'}')^++(K'-K)^+}{n_{\alpha-r'}-\Psi_{r'}}+o(1)\Bigg)\Bigg|S_r\Bigg)
			\\&\leq\mu_1e^{o(1)}\mathbb{E}\Bigg(1_{\mathcal{P}_t^{\delta}}\exp\Bigg(\tfrac1{\mu_1}\sum_{r'=0}^{\alpha-k'-1}\frac{(\Psi_{r'}-\psi_{p_{r'}'}')^+}{n_{\alpha-r'}-\Psi_{r'}}\Bigg)\Bigg|S_r\Bigg)
			\\&+\mu_2e^{o(1)}\mathbb{E}\Bigg(1_{\mathcal{P}_t^{\delta}}\exp\Bigg(\tfrac1{\mu_2}\sum_{r'=0}^{\alpha-k'-1}\frac{(K'-K)^+}{n_{\alpha-r'}-\Psi_{r'}}\Bigg)\Bigg|S_r\Bigg).
		\end{align}
		The result follows, because $\mu_1+\mu_2=1$.
	\end{proof}
	
	To apply Lemma~\ref{lemma_recycle}, we need to bound $\Pi$ in terms of $\Pi_t$.
	
	\begin{lemma}\label{lemma_pit}
		For $r=0,\ldots,\alpha-k-1$ with $t:=\alpha-k-r$, we have $$\Pi\leq\Pi_t\exp(t\log(k+t)+O(t)).$$
	\end{lemma}
	
	\begin{proof}
		For $r'=\alpha-k-t,\ldots,\alpha-k-1$ with $t':=\alpha-k-r'\leq t$, by Lemma~\ref{lemma_expectation} and Lemma~\ref{lemma_nr}~(\ref{lemma_nr_sum},\ref{lemma_nr_leq}), we have
		\begin{align}
			n_{\alpha-r'}-\psi_{p_{r'}}
			&=n_{t'+k}-n_k+n_k-K+O(q_{r'}^2\alpha M+q_{r'}M^2+(n_k-K)M/\alpha)
			\\&\leq\tfrac{\ell}{2}t'^2+\ell t'k+O(t'^2M/\alpha+t'M^2/\alpha+(n_k-K))
			\\&=O(t'(k+t)+(n_k-K)).
		\end{align}
		Furthermore, for $\alpha$ large enough, such that $\b_{\max}\leq\tfrac12$, we have $$n_{\alpha-r'}-\Psi_{r'}=\sum_{v\in V_{r'}}(1-\b(v))\geq\tfrac12|V_{r'}|.$$ Since $|V_{r'}|\geq|S\setminus S_{r'}|+|V^{(0)}\setminus (F^{(0)}\cup R^{(0)})|=t'+(n_k-K)$, we get $$\frac{n_{\alpha-r'}-\psi_{p_{r'}}}{n_{\alpha-r'}-\Psi_{r'}}=O(k+t).$$ The result follows by expanding $\log(\Pi/\Pi_t)$.
	\end{proof}
	
	Combining these two lemmas, we get the following bound on $\mathbb{E}(\Pi1_{\mathcal{P}_{t+1}^{\delta}\setminus\mathcal{P}_t^{\delta}})$.
	
	\begin{lemma}\label{lemma_pipt}
		Let $0<\delta<1$ with $\delta=O(\tfrac{\log\log\alpha}{\log\alpha})$, and let $\mu_1,\mu_2>0$ with $\mu_2=\Omega(1)$, $\mu_1+\mu_2=1$, and $$\delta\geq\tfrac1{\mu_1}\left(8\left(1-\tfrac{\log M}{\log\alpha}\right)+2+\Omega(1)\right)\tfrac{\log\log\alpha}{\log\alpha}.$$ Let $r=0,\ldots,\alpha-k-1$ with $t:=\alpha-k-r=O(M)$. Then $$\mathbb{E}(\Pi1_{\mathcal{P}_{t+1}^{\delta}\cap\mathcal{Q}_t^{\delta}})\leq\exp\left(t\log t+(\tfrac1{2\mu_2}+o(1))t\log\log\alpha+O(\tfrac1{\delta}k)\right)\mathbb{P}(\mathcal{Q}_t^{\delta}).$$
	\end{lemma}
	
	\begin{proof}
		By the law of total expectation, since $\mathcal{Q}_t^{\delta}$ is $S_r$-measurable, we have
		\begin{align}
			\mathbb{E}(\Pi1_{\mathcal{P}_{t+1}^{\delta}\cap\mathcal{Q}_t^{\delta}})
			&=\mathbb{E}(\mathbb{E}(\Pi1_{\mathcal{P}_{t+1}^{\delta}\cap\mathcal{Q}_t^{\delta}}|S_r))
			\\&=\mathbb{E}(1_{\mathcal{Q}_t^{\delta}}\mathbb{E}(\Pi1_{\mathcal{P}_{t+1}^{\delta}}|S_r))
			\\&=\mathbb{E}(1_{\mathcal{Q}_t^{\delta}}\mathbb{E}(\mathbb{E}(\Pi1_{\mathcal{P}_{t+1}^{\delta}}|S_r,S_{r-1})|S_r)).
		\end{align}
		Let $r':=r-1$ and $t':=\alpha-k-r'=t+1$. By Lemma~\ref{lemma_pit}, we have $$\mathbb{E}(\Pi1_{\mathcal{P}_{t+1}^{\delta}}|S_r,S_{r-1})\leq\exp(t'\log(k+t')+O(t'))\mathbb{E}(\Pi_{t'}1_{\mathcal{P}_{t'}^{\delta}}|S_{r'+1},S_{r'}).$$ Conditioned on $S_{r'}$, we note that $\Pi_{t'}$ and $1_{\mathcal{P}_{t'}^{\delta}}$ are independent of $S_{r'+1}$, so by Lemma~\ref{lemma_recycle}, we get
		\begin{align}
			&\mathbb{E}(\Pi_{t'}1_{\mathcal{P}_{t'}^{\delta}}|S_{r'+1},S_{r'})
			=\mathbb{E}(\Pi_{t'}1_{\mathcal{P}_{t'}^{\delta}}|S_{r'})
			\\&\leq\exp\left((\tfrac1{2\mu_2}+o(1))t'\log\log\alpha+O(\tfrac1{\delta}k)\right).
		\end{align}
		Combining results, since $t'=o(t'\log\log\alpha)$, we get
		\begin{align}
			&\mathbb{E}(\Pi1_{\mathcal{P}_{t+1}^{\delta}\cap\mathcal{Q}_t^{\delta}})
			\\&\leq\exp\left(t'\log(k+t')+(\tfrac1{2\mu_2}+o(1))t'\log\log\alpha+O(\tfrac1{\delta}k)\right)\mathbb{E}(1_{\mathcal{Q}_t^{\delta}})
			\\&\leq\exp\left(t\log t+(\tfrac1{2\mu_2}+o(1))t\log\log\alpha+O(\tfrac1{\delta}k)\right)\mathbb{P}(\mathcal{Q}_t^{\delta}).
		\end{align}
		The second inequality uses $\log(k+t')\leq k+t'\leq O(\tfrac1{\delta}k)+o(t'\log\log\alpha)$ and $t\log(\tfrac{k+t'}{t})\leq k+1\leq O(\tfrac1{\delta}k)+o(t'\log\log\alpha)$.
	\end{proof}
	
	We can now finally prove Lemma~\ref{lemma_remaining_e}.
	
	\begin{proof}[Proof of Lemma~\ref{lemma_remaining_e}]
		We use $$\mathbb{E}(\Pi1_{\mathcal{P}\setminus\mathcal{P}^{\delta}})\leq\sum_{t=1}^{T^{\delta}-1}\mathbb{E}(\Pi1_{\mathcal{P}_{t+1}^{\delta}\cap\mathcal{Q}_t^{\delta}}).$$ Consider some $t=1,\ldots,T^{\delta}-1$ with $r:=\alpha-k-t$. By Lemma~\ref{lemma_boundpt}, in $\mathcal{Q}_t^{\delta}$, we have $\Delta_r^{(2)}+\Delta_r^{(1)}\geq(1-(1+o(1))\delta)(n_t+\ell tk)$. It follows that either $\Delta_r^{(2)}\geq(1-(1+o(1))\delta)n_t$ or $\Delta_r^{(1)}\geq(1-(1+o(1))\delta)\ell tk$.
		
		For $t\leq\log\alpha$, by Lemma~\ref{lemma_pd}~(\ref{lemma_pd_2small},\ref{lemma_pd_1}), we get $$\mathbb{P}(\mathcal{Q}_t^{\delta})\leq\exp\left(O(\log\alpha)-\Omega(t\log\alpha)\right).$$ By Lemma~\ref{lemma_pipt}, this gives $$\mathbb{E}(\Pi1_{\mathcal{P}_{t+1}^{\delta}\cap\mathcal{Q}_t^{\delta}})\leq\exp\left(O(\tfrac1{\delta}k+\log\alpha)-\Omega(t\log\alpha)\right).$$ For $\alpha$ large enough, we thus have $$\mathbb{E}(\Pi1_{\mathcal{P}_{t+1}^{\delta}\cap\mathcal{Q}_t^{\delta}})\leq\alpha^{C-\varepsilon t}e^{O(\frac1{\delta}k)}$$ for some fixed constants $C,\varepsilon>0$.
		
		First, let $t^{(1)}:=\lfloor\tfrac{C+1}{\varepsilon}\rfloor$ and $t^{(2)}:=\lfloor\log\alpha\rfloor$, and assume that $\alpha$ is large enough, such that $t^{(1)}<t^{(2)}$. We get $$\sum_{t=t^{(1)}+1}^{t^{(2)}}\mathbb{E}(\Pi1_{\mathcal{P}_{t+1}^{\delta}\cap\mathcal{Q}_t^{\delta}})\leq(1+o(1))\alpha^{C-\varepsilon(t^{(1)}+1)}e^{O(\frac1{\delta}k)}\leq\tfrac{1}{\alpha}e^{O(\frac1{\delta}k)},$$ since $\sum_{i\geq n}x^i=(1+o(1))x^n$ for $x=o(1)$, and $C-\varepsilon(t^{(1)}+1)<-1$.
		
		Second, consider $1\leq t\leq t^{(1)}=O(1)$. Note that $\Delta_r^{(2)}\leq|\widetilde{V}_r\cap V^{(2)}|=|\widetilde{V}_r\setminus S|=n_t-t$. In $\mathcal{Q}_t^{\delta}$, we get
		\begin{align}
			\Delta_r^{(1)}
			&\geq(1-(1+o(1))\delta)(n_t+\ell tk)-\Delta_r^{(2)}
			\\&\geq(1-(1+o(1))\delta)\ell tk+t-(1+o(1))\delta n_t
			\\&\geq(1-(1+o(1))\delta)\ell tk,
		\end{align}
		for $\alpha$ large enough, since $(1+o(1))\delta n_t=o(1)$. By Lemma~\ref{lemma_pd0}~(\ref{lemma_pd0_d1}), we get $$\mathbb{P}(\mathcal{Q}_t^{\delta})=O(kM/\alpha).$$ By Lemma~\ref{lemma_pipt}, and using case analysis on $k=0$ and $k\geq1$, this gives $$\mathbb{E}(\Pi1_{\mathcal{P}_{t+1}^{\delta}\cap\mathcal{Q}_t^{\delta}})\leq\tfrac{kM}{\alpha}\exp\left(O(\tfrac1{\delta}k+\log\log\alpha)\right)\leq\tfrac{M}{\alpha}e^{O(\frac1{\delta}k)}.$$ Indeed, if $k=0$, then the expectation evaluates to $0$, and otherwise, we have $\log k\leq k=O(\tfrac1{\delta}k)$ and $\log\log\alpha=O(\tfrac1{\delta}k)$, since $k\geq1$. It follows that $$\sum_{t=1}^{t^{(1)}}\mathbb{E}(\Pi1_{\mathcal{P}_{t+1}^{\delta}\cap\mathcal{Q}_t^{\delta}})\leq\tfrac{M}{\alpha}e^{O(\frac1{\delta}k)}.$$
		
		Third, let $t^{(3)}:=\min\{\lfloor k(\log\alpha)^{\sfrac12}\rfloor,T^{\delta}-1\}$, and consider if $t^{(2)}<t^{(3)}$. Note that we get $k=\Omega((\log\alpha)^{\sfrac12})$, such that $\log\alpha=O(\tfrac1{\delta}k)$. For $t=t^{(2)}+1,\ldots,t^{(3)}$, we furthermore have $t\log\log\alpha=O(\tfrac1{\delta}k)$.
		
		By Lemma~\ref{lemma_pd}~(\ref{lemma_pd_2small},\ref{lemma_pd_1}), we get $$\mathbb{P}(\mathcal{Q}_t^{\delta})\leq\exp\left(t\log\tfrac{M}{\alpha}+O(\tfrac1{\delta}k)\right).$$ By Lemma~\ref{lemma_pipt}, this gives
		\begin{align}
			\mathbb{E}(\Pi1_{\mathcal{P}_{t+1}^{\delta}\cap\mathcal{Q}_t^{\delta}})
			&\leq\exp\left(t\log\tfrac{tM}{\alpha}+O(\tfrac1{\delta}k)\right)\leq(\tfrac{T^{\delta}M}{\alpha})^te^{O(\frac1{\delta}k)},
			\\\sum_{t=t^{(2)}+1}^{t^{(3)}}\mathbb{E}(\Pi1_{\mathcal{P}_{t+1}^{\delta}\cap\mathcal{Q}_t^{\delta}})
			&\leq(1+o(1))(\tfrac{T^{\delta}M}{\alpha})^{t^{(2)}}e^{O(\frac1{\delta}k)}\leq\tfrac{1}{\alpha}e^{O(\frac1{\delta}k)},
		\end{align}
		because $T^{\delta}M/\alpha=O(M^2/\alpha)=o(1)$ by Lemma~\ref{lemma_t}.
		
		Fourth and finally, we consider $t=\max\{t^{(2)},t^{(3)}\}+1,\ldots,T^{\delta}-1$, such that $\delta t\to\infty$ and $k=o(t)$. Note that Lemma~\ref{lemma_fs}~(\ref{lemma_fs_one}) gives $\Delta_r^{(1)}\leq\ell tk$. In $\mathcal{Q}_t^{\delta}$, we get
		\begin{align}
			\Delta_r^{(2)}
			&\geq(1-(1+o(1))\delta)(n_t+\ell tk)-\Delta_r^{(1)}
			\\&\geq(1-(1+o(1))\delta)n_t-(1+o(1))\delta\ell tk
			\\&=(1-(1+o(1))\delta)n_t,
		\end{align}
		since $(1+o(1))\delta\ell tk=o(\delta t^2)=o(\delta n_t)$ by Lemma~\ref{lemma_nr}~(\ref{lemma_nr_geq}). By Lemma~\ref{lemma_pd}~(\ref{lemma_pd_2large}), we get $$\mathbb{P}(\mathcal{Q}_t^{\delta})\leq\exp\left((1-(\tfrac12+o(1))\delta)t\log\tfrac{M}{\alpha}+O(t)\right).$$ By Lemma~\ref{lemma_pipt}, this gives
		\begin{align}
			&\mathbb{E}(\Pi1_{\mathcal{P}_{t+1}^{\delta}\cap\mathcal{Q}_t^{\delta}})
			\\&\leq\exp\left(t\left(\log\tfrac{tM}{\alpha}+\tfrac1{2\mu_2}\log\log\alpha-\tfrac12\delta\log\tfrac{M}{\alpha}+o(\log\log\alpha)\right)+O(t+\tfrac1{\delta}k)\right)
			\\&\leq\exp\left(O(\tfrac1{\delta}k)-\Omega(t\log\log\alpha)\right),
		\end{align}
		where the second inequality uses
		\begin{align}
			-\tfrac12\delta\log\tfrac{M}{\alpha}
			&=\tfrac12\log\alpha\cdot\delta\left(1-\tfrac{\log M}{\log\alpha}\right)
			\\&\leq\log\alpha\left(2\left(\tfrac12-\tfrac{\log M}{\log\alpha}\right)-(\tfrac1{2\mu_2}+\Omega(1))\tfrac{\log\log\alpha}{\log\alpha}\right)
			\\&=-\log\tfrac{M^2}{\alpha}-\tfrac1{2\mu_2}\log\log\alpha-\Omega(\log\log\alpha).
		\end{align}
		For $\alpha$ large enough, we thus have $$\mathbb{E}(\Pi1_{\mathcal{P}_{t+1}^{\delta}\cap\mathcal{Q}_t^{\delta}})\leq(\log\alpha)^{-\varepsilon't}e^{O(\frac1{\delta}k)}$$ for some fixed constant $\varepsilon'>0$. We get
		\begin{align}
			\sum_{t=\max\{t^{(2)},t^{(3)}\}+1}^{T^{\delta}-1}\mathbb{E}(\Pi1_{\mathcal{P}_{t+1}^{\delta}\cap\mathcal{Q}_t^{\delta}})
			&\leq(1+o(1))(\log\alpha)^{-\varepsilon'(t^{(2)}+1)}e^{O(\frac1{\delta}k)}
			\\&\leq\exp\left(O(\tfrac1{\delta}k)-\Omega(\log\alpha\log\log\alpha)\right)
			\\&\leq\exp\left(O(\tfrac1{\delta}k)-\log\alpha\right)
		\end{align}
		for $\alpha$ large enough. The result follows.
	\end{proof}
	
	This finishes the proof of Lemma~\ref{lemma_concentration}, Lemma~\ref{lemma_generalization}, and Theorem~\ref{theorem_main}~(\ref{theorem_main_equation}).
	
	\subsection{Rejection}\label{subsection_proof_termination}
	
	We use Lemma~\ref{lemma_generalization} to prove Theorem~\ref{theorem_main}~(\ref{theorem_main_termination}). In order for $|\boldsymbol{S}_{\infty}|=\alpha$ to fail, we need $|\boldsymbol{S}_{\infty}|=\alpha-k$ for some integer $k\geq1$. We use a double counting argument, based on a concept of relatedness inspired by~\cite{MCKAY2003273}, to compare $\mathbb{P}(|\boldsymbol{S}_{\infty}|=\alpha-k)$ to $\mathbb{P}(|\boldsymbol{S}_{\infty}|=\alpha)$, which is at most $1$. We show that Theorem~\ref{theorem_main}~(\ref{theorem_main_termination}) follows from the following two lemmas.
	
	\begin{lemma}\label{lemma_size}
		For any maximal FIS $\widetilde{S}$ with $|\widetilde{S}|=\alpha-k$, we have $k\leq M+1$.
	\end{lemma}
	
	\begin{lemma}\label{lemma_stuck}
		Let $k\geq1$ be an integer.
		\begin{enumerate}[(i)]
			\item\label{lemma_stuck_small} If $k=O(1)$, then $\mathbb{P}(|\boldsymbol{S}_{\infty}|=\alpha-k)=O((M/\alpha)^k)$.
			\item\label{lemma_stuck_large} If $k=O(M)$, then $\mathbb{P}(|\boldsymbol{S}_{\infty}|=\alpha-k)\leq\alpha^{o(k)}(M/\alpha)^k$.
		\end{enumerate}
	\end{lemma}
	
	\begin{proof}[Proof of Theorem~\ref{theorem_main}~(\ref{theorem_main_termination})]
		By Lemma~\ref{lemma_stuck}, for $k\leq2$, we have $\mathbb{P}(|\boldsymbol{S}_{\infty}|=\alpha-k)=O(M/\alpha)$, and for $3\leq k\leq M+1$, we have $$\mathbb{P}(|\boldsymbol{S}_{\infty}|=\alpha-k)\leq\alpha^{-\frac13k}$$ for $\alpha$ large enough, since $\alpha^{o(1)}M/\alpha\leq\alpha^{-\sfrac12+o(1)}$. Since $\boldsymbol{S}_{\infty}$ is a maximal FIS, by Lemma~\ref{lemma_size}, we get $$\mathbb{P}(|\boldsymbol{S}_{\infty}|<\alpha)=\sum_{k=1}^{M+1}\mathbb{P}(|\boldsymbol{S}_{\infty}|=\alpha-k)\leq\sum_{k=3}^{M+1}\alpha^{-\frac13k}+O(M/\alpha).$$ The result follows, since $\sum_{i\geq n}x^i=(1+o(1))x^n$ for $x=o(1)$.
	\end{proof}
	
	We can prove Lemma~\ref{lemma_size} immediately.
	
	\begin{proof}[Proof of Lemma~\ref{lemma_size}]
		Since $\widetilde{S}$ is maximal, we know that $V\setminus\overline{N}(\widetilde{S})$ is contained in $F\cup R^*(\widetilde{S})$. By Proposition~\ref{proposition_hereditary}, the induced graph $H$ on $V\setminus\overline{N}(\widetilde{S})$ is $2$-uniform with $\alpha(H)=k$ and $\ell(H)=\ell$, so by Proposition~\ref{proposition_n}, we have $d(H)=\ell(k-1)$. Each vertex in $V\setminus\overline{N}(\widetilde{S})$ thus has $\ell(k-1)$ neighbors in $V\setminus\overline{N}(\widetilde{S})$. Since these are all in $F\cup R^*(\widetilde{S})$, there can be at most $F_{\max}+R_{\max}^{(b)}\leq2M$ such vertices. Therefore, $\ell(k-1)\leq2M$, and the result follows because $\ell\geq2$.
	\end{proof}
	
	To prove Lemma~\ref{lemma_stuck}, for an integer $k\geq1$, let $\mathcal{G}_k$ be the set of all maximal FISs $\widetilde{S}$ of $G$ of size $|\widetilde{S}|=\alpha-k$. Note that maximality means that $V\setminus\overline{N}(\widetilde{S})$ is contained in $F\cup R^*(\widetilde{S})$. Furthermore, let $\mathcal{G}$ be the set of all FMISs of $G$. We get
	\begin{align}
		\sum_{\widetilde{S}\in\mathcal{G}_k}\mathbb{P}(\boldsymbol{S}_{\infty}=\widetilde{S})
		&=\mathbb{P}(|\boldsymbol{S}_{\infty}|=\alpha-k),
		\\\sum_{S\in\mathcal{G}}\mathbb{P}(\boldsymbol{S}_{\infty}=S)
		&=\mathbb{P}(|\boldsymbol{S}_{\infty}|=\alpha)\leq1.
	\end{align}
	To relate these two sums, we introduce a notion of relatedness between sets $S\in\mathcal{G}$ and $\widetilde{S}\in\mathcal{G}_k$. In essence, we consider $S$ and $\widetilde{S}$ as related if you can convert $S$ into $\widetilde{S}$ by a series of swaps of a specific kind. Specifically, for a chordless 6-cycle $xuyvzw$, we consider swapping $x,y,z$ for $u,v$, leaving $w$ out to make sure that $\widetilde{S}$ can be smaller in size compared to $S$.
	
	\begin{definition}
		We say that two sets $S\in\mathcal{G}$ and $\widetilde{S}\in\mathcal{G}_k$ are \emph{related} if there exists a sequence of pairwise disjoint chordless 6-cycles $(x_iu_iy_iv_iz_iw_i)_{i=1,\ldots,k}$ in $G$, such that $S\setminus\widetilde{S}=\{x_i,y_i,z_i:i=1,\ldots,k\}$, $\widetilde{S}\setminus S=\{u_i,v_i:i=1,\ldots,k\}$, and $w_i\not\in R^*(\{u_i,v_i\})$ for $i=1,\ldots,k$. We refer to such a sequence as a \emph{cycle sequence} relating $S$ to $\widetilde{S}$.
	\end{definition}
	
	Here $\{x_i,y_i,z_i:i=1,\ldots,k\}$ is shorthand notation for $\{x_i:i=1,\ldots,k\}\cup\{y_i:i=1,\ldots,k\}\cup\{z_i:i=1,\ldots,k\}$. We show that Lemma~\ref{lemma_stuck} follows from the following three lemmas.
	
	\begin{lemma}\label{lemma_ratio}
		Let $k\geq1$ be an integer. Consider related sets $S\in\mathcal{G}$ and $\widetilde{S}\in\mathcal{G}_k$.
		\begin{enumerate}[(i)]
			\item\label{lemma_ratio_small} If $k=O(1)$, then $\mathbb{P}(\boldsymbol{S}_{\infty}=\widetilde{S})/\mathbb{P}(\boldsymbol{S}_{\infty}=S)=O((k^2/\alpha)^k)$.
			\item\label{lemma_ratio_large} If $k=O(M)$, then $\mathbb{P}(\boldsymbol{S}_{\infty}=\widetilde{S})/\mathbb{P}(\boldsymbol{S}_{\infty}=S)\leq\alpha^{o(k)}(k^2/\alpha)^k$.
		\end{enumerate}
	\end{lemma}
	
	\begin{lemma}\label{lemma_related_good}
		Let $k\geq1$ be an integer with $k=O(M)$. Then any set $\widetilde{S}\in\mathcal{G}_k$ is related to at least $e^{-O(k)}(\alpha^2k)^k$ sets $S\in\mathcal{G}$.
	\end{lemma}
	
	\begin{lemma}\label{lemma_related_bad}
		Let $k\geq1$ be an integer with $k=O(M)$. Then any set $S\in\mathcal{G}$ is related to at most $e^{O(k)}(\alpha^2M/k)^k$ sets $\widetilde{S}\in\mathcal{G}_k$.
	\end{lemma}
	
	\begin{proof}[Proof of Lemma~\ref{lemma_stuck}]
		Let $\mathcal{R}$ be the set of all related pairs $(S,\widetilde{S})\in\mathcal{G}\times\mathcal{G}_k$. We define
		\begin{align}
			\Sigma&:=\sum_{(S,\widetilde{S})\in\mathcal{R}}\mathbb{P}(\boldsymbol{S}_{\infty}=S),
			\\\widetilde{\Sigma}&:=\sum_{(S,\widetilde{S})\in\mathcal{R}}\mathbb{P}(\boldsymbol{S}_{\infty}=\widetilde{S}).
		\end{align}
		By Lemma~\ref{lemma_ratio}, we get
		\begin{align}
			\widetilde{\Sigma}/\Sigma&=O((k^2/\alpha)^k),&&\mbox{if}\ k=O(1),
			\\\widetilde{\Sigma}/\Sigma&\leq\alpha^{o(k)}(k^2/\alpha)^k,&&\mbox{if}\ k=O(M).
		\end{align}
		By Lemma~\ref{lemma_related_good} and Lemma~\ref{lemma_related_bad}, we furthermore get
		\begin{align}
			&\widetilde{\Sigma}\geq e^{-O(k)}(\alpha^2k)^k\sum_{\widetilde{S}\in\mathcal{G}_k}\mathbb{P}(\boldsymbol{S}_{\infty}=\widetilde{S})\geq e^{-O(k)}(\alpha^2k)^k\mathbb{P}(|\boldsymbol{S}_{\infty}|=\alpha-k),
			\\&\Sigma\leq e^{O(k)}(\alpha^2M/k)^k\sum_{S\in\mathcal{G}}\mathbb{P}(\boldsymbol{S}_{\infty}=S)\leq e^{O(k)}(\alpha^2M/k)^k,\ \text{so}
			\\&\mathbb{P}(|\boldsymbol{S}_{\infty}|=\alpha-k)\leq e^{O(k)}(M/k^2)^{k}\widetilde{\Sigma}/\Sigma.
		\end{align}
		The result follows, since $e^{O(k)}=O(1)$ if $k=O(1)$ and $e^{O(k)}=\alpha^{o(k)}$ if $k=O(M)$.
	\end{proof}
	
	The proof of Lemma~\ref{lemma_ratio} uses Lemma~\ref{lemma_generalization}.
	
	\begin{proof}[Proof of Lemma~\ref{lemma_ratio}]
		By Lemma~\ref{lemma_generalization}, we have $$\mathbb{P}(\boldsymbol{S}_{\infty}=\widetilde{S})/\mathbb{P}(\boldsymbol{S}_{\infty}=S)\leq\mathscr{E}_k\mathscr{P}_k(\widetilde{S})/\mathscr{P}(S)$$ for some error term $\mathscr{E}_k$ with $\mathscr{E}_k=O(1)$ for $k=O(1)$ and $\mathscr{E}_k\leq\alpha^{o(k)}$ for $k=O(M)$. Recall that $\widetilde{S}$ is a maximal FIS, such that $K=|V\setminus\overline{N}(\widetilde{S})|=n_k$ by Corollary~\ref{corollary_hereditary}. It thus suffices to show that $$\frac{\mathscr{P}_k(\widetilde{S})}{\mathscr{P}(S)}=\frac{\exp\left(\frac12P|R^*(\widetilde{S})|-\b(\widetilde{S})\right)\prod_{t=k+1}^\alpha\frac{t-k}{n_t-n_k}}{\exp\left(\frac12P|R^*(S)|-\b(S)\right)\prod_{t=1}^\alpha\frac{t}{n_t}}\leq e^{O(k)}\left(\frac{k^2}{\alpha}\right)^k.$$
		
		Firstly, since $|S\setminus\widetilde{S}|=3k$ and $|\widetilde{S}\setminus S|=2k$, we have
		\begin{align}
			|\b(S)-\b(\widetilde{S})|&\leq|\b(S\setminus\widetilde{S})|+|\b(\widetilde{S}\setminus S)|=O(kM/\alpha),
			\\||R^*(\widetilde{S})|-|R^*(S)||&\leq|R^*(\widetilde{S}\setminus S)|+|R^*(S\setminus\widetilde{S})|=O(kM).
		\end{align}
		It follows that $$\frac{\exp\left(\tfrac12P|R^*(\widetilde{S})|-\b(\widetilde{S})\right)}{\exp\left(\tfrac12P|R^*(S)|-\b(S)\right)}=e^{o(1)}.$$
		
		Secondly, since $n_t=\tfrac{\ell}{2}t(t+\tfrac{2}{\ell}-1)$ and thus $n_t-n_k=\tfrac{\ell}{2}(t-k)(t+k+\tfrac{2}{\ell}-1)$, telescoping gives
		\begin{align}
			\frac{\prod_{t=k+1}^\alpha\frac{t-k}{n_t-n_k}}{\prod_{t=1}^\alpha\frac{t}{n_t}}
			&=\frac{\prod_{t=1}^{\alpha}\frac{\ell}{2}(t+\frac{2}{\ell}-1)}{\prod_{t=k+1}^{\alpha}\frac{\ell}{2}(t+k+\frac{2}{\ell}-1)}
			\\&=\left(\frac{\ell}{2}\right)^k\frac{\prod_{i=1}^{2k}(i+\frac{2}{\ell}-1)}{\prod_{i=1}^{k}(\alpha+i+\frac{2}{\ell}-1)}
			\\&\leq\left(\frac{\ell}{2}\right)^k\frac{(2k)!}{\alpha^k}
			\\&\leq e^{O(k)}\left(\frac{k^2}{\alpha}\right)^k
		\end{align}
		by Stirling's approximation.
	\end{proof}
	
	For the proofs of Lemma~\ref{lemma_related_good} and Lemma~\ref{lemma_related_bad}, we first need some results on chordless $6$-cycles. First, note that any vertex $v$ outside some independent set $S$ is adjacent to at most two elements of $S$ by $2$-uniformity. Indeed, let $\overline{S}$ be a maximal independent set that contains $S$, such that $\overline{S}$ is an MIS by Proposition \ref{proposition_imis}. Then $v$ is either in $\overline{S}\setminus S$, or adjacent to exactly two elements of $\overline{S}$. For a chordless $6$-cycle $xuyvzw$ and an independent set $S$ with $x,y,z\in S$, it follows that $N(u)\cap S=\{x,y\}$, $N(v)\cap S=\{y,z\}$, and $N(w)\cap S=\{x,z\}$. We furthermore use the following two lemmas on chordless $6$-cycles.
	
	\begin{lemma}\label{lemma_sequences}
		For any related sets $S\in\mathcal{G}$ and $\widetilde{S}\in\mathcal{G}_k$, there are exactly $2^kk!$ cycle sequences relating $S$ to $\widetilde{S}$.
	\end{lemma}
	
	\begin{proof}
		For any cycle sequence $(x_iu_iy_iv_iz_iw_i)_{i=1,\ldots,k}$, we have $N(u_i)\cap S=\{x_i,y_i\}$ and $N(v_i)\cap S=\{y_i,z_i\}$ for all $i=1,\ldots,k$ by $2$-uniformity. For each vertex $u\in\widetilde{S}\setminus S$, there is thus a unique distinct vertex $v\in\widetilde{S}\setminus S$ with $N(u)\cap N(v)\cap S$ a singleton, say $N(u)\cap S=\{x,y\}$ and $N(v)\cap S=\{y,z\}$. Then $S':=S\setminus\{x,y,z\}\cup\{u,v\}$ is an independent set of size $\alpha-1$, so $V\setminus\overline{N}(S')$ has a single vertex $w$ by Corollary~\ref{corollary_hereditary}. It follows that any cycle sequence must include either $xuyvzw$ or its inversion $zvyuxw$. The cycle sequence is therefore unique up to $2^k$ inversions and $k!$ permutations.
	\end{proof}
	
	\begin{lemma}\label{lemma_cycle}
		Let $x,y,z\in V$ be pairwise non-adjacent. Then for any $u\in N(x)\cap N(y)$, there are exactly $\tfrac{\ell}{2}$ vertices $v\in N(y)\cap N(z)\setminus N(u)$, each yielding exactly one vertex $w\in V$, such that $xuyvzw$ is a chordless $6$-cycle.
	\end{lemma}
	
	\begin{proof}
		Let $S$ be a maximal independent set that contains $x$, $y$, and $z$, such that $S$ is an MIS by Proposition \ref{proposition_imis}. By Proposition~\ref{proposition_even}, there is a unique vertex $u'\in N(x)\cap N(y)\setminus\overline{N}(u)$. By $2$-uniformity, we find that $S':=S\setminus\{x,y\}\cup\{u,u'\}$ is an MIS. Then $u'$ is not adjacent to $z$, so by Proposition~\ref{proposition_even} again, $N(u')\cap N(z)$ consists of $\frac{\ell}{2}$ pairs of non-adjacent vertices.
		
		Consider such a pair $v,w$ of non-adjacent vertices in $N(u')\cap N(z)$. By $2$-uniformity, we find that $S'\setminus\{u',z\}\cup\{v,w\}=S\setminus\{x,y,z\}\cup\{u,v,w\}$ is an MIS. By $2$-uniformity, the induced graph on $\{x,y,z,u,v,w\}$ is $2$-regular and bipartite with parts $\{x,y,z\}$ and $\{u,v,w\}$, from which it follows that exactly one of $xuyvzw$ and $xuywzv$ is a chordless $6$-cycle.
		
		Conversely, it remains only to show for any chordless $6$-cycle $xuyvzw$ that the vertices $v,w$ are such a non-adjacent pair in $N(u')\cap N(z)$. By $2$-uniformity, we have $N(v)\cap S=\{y,z\}$, and $N(v)\cap S'=(N(v)\cap S\setminus\{x,y\})\cup(N(v)\cap\{u,u'\})$ has cardinality two, from which it follows that $N(v)\cap S'=\{z,u'\}$. By symmetry, we also have $N(w)\cap S'=\{z,u'\}$, so the result follows.
	\end{proof}
	
	Note that, by symmetry, the roles of $x,y,z$ and $u,v,w$ may be swapped in Lemma~\ref{lemma_cycle} by symmetry. We can now prove Lemma~\ref{lemma_related_good} and Lemma~\ref{lemma_related_bad}.
	
	\begin{proof}[Proof of Lemma~\ref{lemma_related_good}]
		We iteratively choose pairwise disjoint chordless $6$-cycles $x_iu_iy_iv_iz_iw_i$ for $i=1,\ldots,k$. We require that $u_i,v_i\in\widetilde{S}$ and $w_i\not\in R^*(\{u_i,v_i\})$ for each $i=1,\ldots,k$. We furthermore require that $$\widetilde{S}_r:=\widetilde{S}\setminus\{u_i,v_i:i=1,\ldots,r\}\cup\{x_i,y_i,z_i:i=1,\ldots,r\}$$ is an FIS for each $r=0,\ldots,k$. Then $(x_iu_iy_iv_iz_iw_i)_{i=1,\ldots,k}$ is a cycle sequence relating $\widetilde{S}_k$ to $\widetilde{S}$ by construction.
		
		For any $r=1,\ldots,k$, since $|\widetilde{S}_{r-1}|=\alpha-k+r-1$, by Corollary~\ref{corollary_hereditary}, there are $n_{k-r+1}$ ways to choose $w_r\in V\setminus\overline{N}(\widetilde{S}_{r-1})$. Each such choice of $w_r$ gives at least $(\alpha-k-2r+2-R_{\max}^{(a)})(\alpha-k-2r+2-R_{\max}^{(a)}-1)=\Omega(\alpha^2)$ ways to choose distinct $u_r,v_r\in\widetilde{S}\setminus\{u_i,v_i:i=1,\ldots,r-1\}\setminus R^*(w_r)$. Each such choice of $u_r,v_r$ gives exactly $\ell$ ways to choose $y_r\in N(u_r)\cap N(v_r)$. Finally, by Lemma~\ref{lemma_cycle}, each such choice of $y_r$ gives exactly $\tfrac{\ell}{2}$ ways to choose $x_r,z_r\in V$ such that $x_ru_ry_rv_rz_rw_r$ is a chordless $6$-cycle. We conclude that there are $\Omega(n_{k-r+1}\alpha^2)$ ways to choose a chordless $6$-cycle $x_ru_ry_rv_rz_rw_r$, such that $w_r\in V\setminus\overline{N}(\widetilde{S}_{r-1})\setminus R^*(\{u_r,v_r\})$ and $u_r,v_r\in\widetilde{S}\setminus\{u_i,v_i:i=1,\ldots,r-1\}$.
		
		By induction, we show that $\widetilde{S}_r$ and $\hat{S}_r:=\widetilde{S}\cup\{w_1,\ldots,w_r\}$ are independent sets with $\overline{N}(\widetilde{S}_r)=\overline{N}(\hat{S}_r)$, and all chosen $6$-cycles are disjoint. First, since $S_r':=\widetilde{S}_{r-1}\cup\{w_r\}$ is an independent set, by $2$-uniformity, we have $N(x_r)\cap S_r'=\{u_r,w_r\}$, $N(y_r)\cap S_r'=\{u_r,v_r\}$, and $N(z_r)\cap S_r'=\{v_r,w_r\}$. It follows that $\widetilde{S}_r=S_r'\setminus\{u_r,v_r,w_r\}\cup\{x_r,y_r,z_r\}$ is an independent set. Next, we have $w_r\in V\setminus\overline{N}(\widetilde{S}_{r-1})=V\setminus\overline{N}(\hat{S}_{r-1})$, which makes $\hat{S}_r$ an independent set. Then it follows that $T_r:=\{u_i,v_i,w_i:i=1,\ldots,r\}\subset\hat{S}_r$ is an independent set of $3r$ distinct vertices. Since each vertex $x_i$, $y_i$, and $z_i$ for $i=1,\ldots,r$ has a unique set of two adjacent vertices in $T_r$ by $2$-uniformity, it follows that all chosen $6$-cycles are disjoint. Finally, note that
		\begin{align}
			\overline{N}(\hat{S}_r)
			&=\overline{N}(\hat{S}_{r-1})\cup\overline{N}(w_r)
			\\&=\overline{N}(\widetilde{S}_{r-1})\cup\overline{N}(w_r)
			\\&=\overline{N}(\widetilde{S}_r\setminus\{x_r,y_r,z_r\}\cup\{u_r,v_r,w_r\}).
		\end{align}
		Let $\overline{S}_r$ be a maximal independent set containing $\widetilde{S}_r$, such that $\overline{S}_r$ is an MIS by Proposition \ref{proposition_imis}. By $2$-uniformity, we have $N(u_r)\cap\overline{S}_r=\{x_r,y_r\}$, $N(v_r)\cap\overline{S}_r=\{y_r,z_r\}$, and $N(w_r)\cap\overline{S}_r=\{z_r,x_r\}$, so $\overline{S}_r\setminus\{x_r,y_r,z_r\}\cup\{u_r,v_r,w_r\}$ is also an MIS. By $2$-uniformity, we find that $V\setminus\overline{N}(\hat{S}_r)$ and $V\setminus\overline{N}(\widetilde{S}_r)$ are both the set of vertices adjacent to two vertices in $\overline{S}_r\setminus\widetilde{S}_r$, together with those vertices themselves. This finishes the proof by induction.
		
		We analyze when the resulting set $\widetilde{S}_r$ is an FIS. Since $\widetilde{S}_r$ is an independent set, it is an FIS if $x_r,y_r,z_r\not\in F\cup R^*(\widetilde{S}_r)$, for which it is sufficient to check if $x_r,y_r,z_r\not\in F\cup R^*(\widetilde{S}_{r-1})\cup R^*(\{x_r,y_r,z_r\})$. We show that, for each possible choice of $w_r\in V\setminus\overline{N}(\widetilde{S}_{r-1})$, there are $O(\alpha M)$ possible chosen $6$-cycles $x_ru_ry_rv_rz_rw_r$ for which this fails.
		
		\begin{itemize}
			\item There are at most $F_{\max}+R_{\max}^{(b)}=O(M)$ ways to choose $x_r\in N(w_r)\cap(F\cup R^*(\widetilde{S}_{r-1}))$. Then, by $2$-uniformity, there is at most one possible choice for $u_r\in N(x_r)\cap\widetilde{S}_{r-1}\setminus\{w_r\}$. Next, there are at most $\alpha$ ways to choose $v_r\in\widetilde{S}_{r-1}$. Finally, by Lemma~\ref{lemma_cycle}, there are exactly $\tfrac{\ell}{2}$ ways to choose $y_r,z_r$, such that $x_ru_ry_rv_rz_rw_r$ is a chordless $6$-cycle. This gives $O(\alpha M)$ possible chosen $6$-cycles with $x_r\in F\cup R^*(\widetilde{S}_{r-1})$. By symmetry, the same holds for $z_r\in F\cup R^*(\widetilde{S}_{r-1})$.
			\item There are at most $|F\cup R^*(\widetilde{S}_{r-1})|=O(\alpha M)$ ways to choose $y_r\in F\cup R^*(\widetilde{S}_{r-1})$ by Lemma~\ref{lemma_terms}~(\ref{lemma_terms_f},\ref{lemma_terms_rs}). By $2$-uniformity, there are at most two possible choices for distinct $u_r,v_r\in N(y_r)\cap\widetilde{S}_{r-1}$. Then, by Lemma~\ref{lemma_cycle}, there are exactly $\tfrac{\ell}{2}$ ways to choose $x_r,z_r$, such that $x_ru_ry_rv_rz_rw_r$ is a chordless $6$-cycle. This gives $O(\alpha M)$ possible chosen $6$-cycles with $y_r\in F\cup R^*(\widetilde{S}_{r-1})$.
			\item There are $\ell(\alpha-1)$ ways to choose $x_r\in N(w_r)$ by Proposition~\ref{proposition_n}. Then there are at most $R_{\max}^{(a)}\leq M$ ways to choose $y_r\in R^*(x_r)$. By $2$-uniformity, there is at most one possible choice for $u_r\in N(x_r)\cap\widetilde{S}_{r-1}\setminus\{w_r\}$ and $v_r\in N(y_r)\cap\widetilde{S}_{r-1}\setminus\{u_r\}$. Finally, by Lemma~\ref{lemma_cycle}, there is exactly one vertex $z_r\in V$, such that $x_ru_ry_rv_rz_rw_r$ is a chordless $6$-cycle. This gives $O(\alpha M)$ possible chosen $6$-cycles with $x_rRy_r$. The same general argument applies to $x_rRz_r$ and $y_rRz_r$ as well.
		\end{itemize}
		
		We conclude that there are $\Omega(n_{k-r+1}\alpha^2)$ ways to choose a chordless $6$-cycle $x_ru_ry_rv_rz_rw_r$, such that $u_r,v_r\in\widetilde{S}\setminus\{u_i,v_i:i=1,\ldots,r-1\}$, $w_r\not\in R^*(\{u_r,v_r\})$, and $\widetilde{S}_r$ is an FIS. For $\alpha$ large enough, we get a fixed constant $C>0$, such that this gives at least $Cn_{k-r+1}\alpha^2$ possible choices for each $r=1,\ldots,k$. By Lemma~\ref{lemma_sequences}, it follows that $\widetilde{S}$ is related to at least
		\begin{align}
			&\frac1{2^kk!}\prod_{r=1}^kCn_{k-r+1}\alpha^2
			\\&=\frac{C^k\alpha^{2k}}{2^kk!}\prod_{i=1}^k\left(\tfrac{\ell}2i\left(i+\tfrac2{\ell}-1\right)\right)
			\\&=\left(\tfrac14\ell C\alpha^2\right)^k\prod_{i=1}^k\left(i+\tfrac{2}{\ell}-1\right)
			\\&\geq\left(\tfrac14\ell C\alpha^2\right)^k\cdot\tfrac{2}{\ell}(k-1)!
			\\&\geq e^{-O(k)}(\alpha^2k)^k
		\end{align}
		sets $S\in\mathcal{G}$ by Stirling's approximation, in particular, since $\log k\leq k$.
	\end{proof}
	
	\begin{proof}[Proof of Lemma~\ref{lemma_related_bad}]
		Consider a subset $I\subset\{1,\ldots,k\}$ and a corresponding sequence $(x_iu_iy_iv_iz_iw_i)_{i\in I}$ of $6$-cycles. We define $A((x_iu_iy_iv_iz_iw_i)_{i\in I})$ as the number of ways we can choose $6$-cycles $(x_iu_iy_iv_iz_iw_i)_{i\in\{1,\ldots,k\}\setminus I}$, such that $(x_iu_iy_iv_iz_iw_i)_{i=1,\ldots,k}$ is a cycle sequence relating $S$ to some set $\widetilde{S}\in\mathcal{G}_k$.
		
		Let $A(I)$ be the maximum of $A((x_iu_iy_iv_iz_iw_i)_{i\in I})$ over all possible sequences $(x_iu_iy_iv_iz_iw_i)_{i\in I}$ of $6$-cycles. Note that this only depends on the size of $I$ by symmetry, so it suffices to consider $A_r:=A(\{1,\ldots,r\})$ for $r=0,\ldots,k$. Then $A_0$ counts all possible cycle sequences relating $S$ to some set $\widetilde{S}\in\mathcal{G}_k$, so by Lemma~\ref{lemma_sequences}, $S$ relates to $\tfrac1{2^kk!}A_0$ such sets. By Stirling's approximation, it suffices to show $A_0\leq e^{O(k)}(\alpha^2M)^k$.
		
		Consider a cycle sequence $(x_iu_iy_iv_iz_iw_i)_{i=1,\ldots,k}$ relating $S$ to some set $\widetilde{S}\in\mathcal{G}_k$. Since $\widetilde{S}$ is a maximal FIS, we get $$w_1,\ldots,w_k\in V\setminus\overline{N}(\widetilde{S})\subset F\cup R^*(\widetilde{S})\subset F\cup R^*(S)\cup R^*(\{u_i,v_i:i=1,\ldots,k\}).$$ For $i=1,\ldots,k$, the vertex $w_i$ is thus either in $F\cup R^*(S)$ or in $R^*(\{u_j,v_j\})$ for some $j=1,\ldots,k$. Note that $i=j$ is not allowed, by the definition of a cycle sequence, so we only need to consider $j<i$ and $j>i$.
		
		We study the values $A_r$ for $r$ from $k$ down to $0$. Let us define $A_r=0$ for $r>k$ for convenience. Note that $A_k\leq1$, because if $I=\{1,\ldots,k\}$, then there is nothing left to choose.
		
		Next, consider choosing a $6$-cycle $x_iu_iy_iv_iz_iw_i$ for some $i=1,\ldots,k$ with $w_i$ already chosen. By $2$-uniformity, there are exactly two ways to choose distinct $x_i,z_i\in N(w_i)\cap S$. Next, there are at most $\alpha$ ways to choose $y_i\in S$. Finally, by Lemma~\ref{lemma_cycle}, there are exactly $\tfrac{\ell}{2}$ ways to choose $u_i,v_i$, such that $x_iu_iy_iv_iz_iw_i$ is a chordless $6$-cycle. So, with $w_i$ already chosen, there are at most $\ell\alpha$ ways to complete the cycle.
		
		For $r\leq k-1$, we bound $A_r$ in terms of $A_{r+1}$ and $A_{r+2}$ by distinguishing the three cases for $w_{r+1}$.
		
		\begin{itemize}
			\item Suppose $w_{r+1}\in F\cup R^*(S)$. There are at most $|F\cup R^*(S)|=O(\alpha M)$ ways to choose $w_{r+1}$ by Lemma~\ref{lemma_terms}~(\ref{lemma_terms_f},\ref{lemma_terms_rs}). Then there are at most $\ell\alpha$ ways to complete the cycle $x_{r+1}u_{r+1}y_{r+1}v_{r+1}z_{r+1}w_{r+1}$. Finally, we have at most $A(\{1,\ldots,r+1\})=A_{r+1}$ possible ways to choose the remaining cycles. In total, this gives $O(\alpha^2MA_{r+1})$ possible choices.
			\item Suppose $w_{r+1}\in R^*(\{u_j,v_j\})$ for some $j<r+1$. First, there are $r\leq k=O(M)$ ways to choose $j$. Then there are at most $2R_{\max}^{(a)}=O(M)$ ways to choose $w_{r+1}\in R^*(\{u_j,v_j\})$. The rest of the analysis is exactly the same as in the last case, giving $O(\alpha M^2A_{r+1})$ possible choices in total.
			\item Suppose $w_{r+1}\in R^*(\{u_j,v_j\})$ for some $j>r+1$. First, there are $k-r-1\leq k=O(M)$ ways to choose $j$. Then there are $n_{\alpha}=O(\alpha^2)$ ways to choose $w_j\in V$ by Lemma~\ref{lemma_terms}~(\ref{lemma_terms_n}), and at most $\ell\alpha$ ways to complete the cycle $x_ju_jy_jv_jz_jw_j$. Next, there are at most $2R_{\max}^{(a)}\leq2M$ ways to choose $w_{r+1}\in R^*(\{u_j,v_j\})$. Again, there are at most $\ell\alpha$ ways to complete the cycle $x_{r+1}u_{r+1}y_{r+1}v_{r+1}z_{r+1}w_{r+1}$. Finally, we have at most $A(\{1,\ldots,r+1,j\})=A_{r+2}$ possible ways to choose the remaining cycles. In total, this gives $O(\alpha^4M^2A_{r+2})$ possible choices.
		\end{itemize}
		
		We get the recurrence $A_r=O(\alpha^2MA_{r+1}+\alpha^4M^2A_{r+2})$. For $\alpha$ large enough, we thus get $A_r\leq C\alpha^2MA_{r+1}+C^2\alpha^4M^2A_{r+2}$ for some fixed constant $C>0$. By strong induction, it follows that $A_r\leq(2C\alpha^2M)^{k-r}$ for all $r=0,\ldots,k$. The result follows.
	\end{proof}
	
	This finishes the proof of Theorem~\ref{theorem_main}.
	
	\section*{Acknowledgments}
	IK gratefully acknowledges support from Netherlands Research Organisation (NWO), research program VIDI, project number VI.Vidi.213.108.
	
	\printbibliography

\end{document}